\documentclass{amsbook}

\usepackage{amsrefs}
\usepackage{euscript}
\usepackage{amsmath}
\usepackage{amscd}
\usepackage{amssymb}
\usepackage{enumerate}
\usepackage{stmaryrd}
 \usepackage{amsfonts}
\usepackage{graphicx}
\usepackage[all]{xy}
\usepackage[margin=1.5in]{geometry}
\usepackage{setspace}
\usepackage{eufrak}
\usepackage{tikz}
\usepackage{textcomp}

\newcommand{\bC}{{\mathbb C}}
\newcommand{\bP}{{\mathbb P}}
\newcommand{\bR}{{\mathbb R}}
\newcommand{\bZ}{{\mathbb Z}}

\newcommand{\cA}{\mathcal A}

\newcommand{\cN}{\mathcal N}

\newcommand{\cH}{\mathcal H}
\newcommand{\Tree}{\EuScript T}
\newcommand{\Treebar}{\overline{\Tree}}
\newcommand{\sL}{\EuScript L}

\newcommand{\Broke}{\EuScript B}
\newcommand{\Brokebar}{\overline{\Broke}}
\newcommand{\Brokecot}{\EuScript B'}
\newcommand{\Brokecotbar}{\overline{\Broke}'}
\newcommand{\Tq}[1][\!]{T_{q_{#1}}^{*}Q}
\newcommand{\TQ}{T^{*}Q}
\newcommand{\Q}{Q}
\newcommand{\SQ}[1][\!]{S^{*}_{#1}Q}
\newcommand{\DQ}[1][\!]{D^{*}_{#1}Q}

\newcommand{\Sigmabar}{\overline{\Sigma}}

\newcommand{\vq}[1][\!]{\vec{q}^{\, #1}}
\newcommand{\ro}{{\mathrm o}}

\newcommand{\CZ}{{\mathrm CZ}}

\newcommand{\U}[1][\!]{U^{#1}}

\newcommand{\sJ}{{\EuScript J}}

\newcommand{\Orbit}{{\EuScript O}}
\newcommand{\Chord}{{\EuScript X}}

\newcommand{\Hh}{h}

\newcommand{\bfdelta}{\mathbf{\delta}}

\newcommand{\coker}{\operatorname{coker}}

\newcommand{\ev}{\operatorname{ev}}
\newcommand{\geo}{\operatorname{geo}}
\newcommand{\id}{\operatorname{id}}
\newcommand{\im}{\operatorname{im}}

\renewcommand{\mod}{\operatorname{mod}}

\newcommand{\Disc}{{\mathcal R}}
\newcommand{\Pants}{{\mathcal P}}
\newcommand{\Pantsbar}{\overline{\Pants}}

\newcommand{\Cyl}{{\mathcal M}}
\newcommand{\Cylbar}{\overline{\Cyl}}
\newcommand{\Ann}{{\mathcal A}}
\newcommand{\Plane}{\mathcal{P}}
\newcommand{\Planebar}{\overline{\Plane}}
\newcommand{\Annbar}{\overline{\Ann}}

\newcommand{\Cont}{{\mathcal K}}
\newcommand{\Contbar}{\overline{\Cont}}

\newcommand{\End}{\operatorname{End}}

\newcommand{\Action}{\mathcal A}

\newcommand{\Path}{\mathcal P}

\newcommand{\Vit}{ {\mathcal V}}
\newcommand{\Fam}{ {\mathcal F}}
\newcommand{\Gam}{ {\mathcal G}}
\newcommand{\Strip}{B}
\newcommand{\Tria}{T}

\newcommand{\Sp}{\operatorname{Sp}}
\newcommand{\Pin}{\operatorname{Pin}}
\newcommand{\Spin}{\operatorname{Spin}}

\renewcommand{\dbar}{\overline{\partial}}
\renewcommand{\det}{\operatorname{det}}
\newcommand{\s}{\mathrm{\sigma}}
\newcommand{\ind}{\operatorname{ind}}
\newcommand{\grad}{\operatorname{grad}}
\newcommand{\Gr}{\operatorname{Gr}}
\newcommand{\fsp}{{\mathfrak{sp}}}
\newcommand{\fgl}{{\mathfrak{gl}}}
\newcommand{\vbar}{\, | \,}
\newcommand{\cont}{{\mathfrak{c}}}
\newcommand{\Para}{\mathrm{P}}
\def\co{\colon\thinspace}

\numberwithin{equation}{section}
\numberwithin{figure}{chapter}

\newtheorem{thm}{Theorem}[section]
\newtheorem{cor}[thm]{Corollary}
\newtheorem{lem}[thm]{Lemma}
\newtheorem{prop}[thm]{Proposition}
\newtheorem{defin}[thm]{Definition}
\newtheorem{def-lem}[thm]{Definition-Lemma}
\newtheorem{conj}[thm]{Conjecture}

\theoremstyle{remark}

\newtheorem{rem}[thm]{Remark}

\newtheorem{example}[thm]{Example}

\newtheorem{exercise}[thm]{Exercise}

\newcommand{\superscript}[1]{\ensuremath{^{\textrm{#1}}} }

\renewcommand{\th}[0]{\superscript{th}}
\newcommand{\st}[0]{\superscript{st}}

\newcommand{\comment}[1]{}

\title[Symplectic cohomology and Viterbo's theorem]{Symplectic cohomology and Viterbo's theorem}

\author[M.~Abouzaid]{Mohammed Abouzaid} \date{\today}
\begin{document}
\maketitle
\tableofcontents

\section*{Introduction}

In \cite{Floer-gradient,Floer-index, Floer-JDG}, Floer associated to a non-degenerate time-dependent Hamiltonian 
\begin{equation*}
  H \co \bR/\bZ \times M \to \bR
\end{equation*}
on a symplectic manifold $M$ (satisfying some technical hypotheses),  a cohomology group now called (Hamiltonian) Floer cohomology, which he showed to be independent of $H$ if $M$ is closed.

In these notes, we shall be concerned with a situation where $M$ is not closed. Since general open symplectic manifolds are too wild to allow for an interesting development of Floer theory, one usually restricts attention to those with controlled behaviour outside a compact set; a natural condition to impose is that a neighbourhood of infinity be modelled after the cone on a contact manifold. A key insight of Floer and Hofer \cite{FH-SH} is that there are, on such symplectic manifolds, natural classes of Hamiltonians whose Floer cohomology is related to the dynamics of the Reeb flow on the contact manifold at infinity. One such class, which admits a natural order with respect to the  ``rate of growth'' at infinity,  was introduced by Viterbo in  \cite{viterbo-99,viterbo-96}, and the \emph{symplectic cohomology} of such a manifold can be defined as a direct limit of Floer cohomology groups over this class of Hamiltonians. This is the cohomology group appearing in the title. These groups are extremely difficult to compute, except when they vanish, but they are known to satisfy good formal properties, including a version of the K\"unneth theorem \cite{oancea-K}.

Instead of considering such a general setting, we restrict ourselves to the first class of examples for which this invariant is both non-trivial and expressible in terms of classical topological invariants: the symplectic manifold $M$ which we shall consider will be the cotangent bundle $\TQ$ of a closed differentiable manifold. In this case, one naturally obtains a manifold equipped with a contact form by considering the unit sphere bundle with respect to a Riemannian metric on $\Q$, and it has been known for quite a long time that the Reeb flow on this contact manifold is related to the geodesic flow on the tangent bundle. Since the closed orbits of the geodesic flow are the generators of a Morse complex which computes the homology of the free loop space, a connection between the loop homology of $\Q$ and the symplectic cohomology of $\TQ$ is therefore to be expected.

In his ICM address \cite{viterbo-94}, Viterbo explained a strategy for showing that, for cotangent bundles of oriented manifolds, symplectic cohomology is isomorphic to the homology of the free loop space: the idea was to relate both  to an intermediate invariant called \emph{generating function homology.}  This strategy was implemented in \cite{viterbo-99}, and different approaches were later considered in \cites{AS,SW-06}. Surprisingly, the result stated  by Viterbo turns out to be true only if the base is $\Spin$; the key observation here is due to Kragh \cite{kragh-1}, who showed that, for oriented manifolds, generating function homology cannot be isomorphic to symplectic cohomology because it is not functorial under exact embeddings. Instead, Kragh proved the functoriality of a twisted version of generating function homology, which is isomorphic to the homology of a local system of rank $1$ on the free loop space that is trivial if and only the second Stiefel-Whitney class of $\Q$ vanishes on all tori. A corrected version of Viterbo's theorem for orientable base was, as a consequence, relatively easy to state and prove \cite{A-loops}.

These notes present a complete proof of Viterbo's theorem relating the (twisted) homology of the free loop space of a closed differentiable manifold to the symplectic cohomology of its cotangent bundle. In addition, they include the verification that  the primary operadic operations coming on one side from the count of holomorphic curves, and on the other from string topology agree. We pay particular attention to issues of signs and gradings, both because it turns out in the end that the answer is unexpected and because even some experts still consider them to be too mysterious to address.

The original intent was that the account given would be complete as well as accessible to a reader familiar with basic concepts in symplectic topology, but not necessarily an expert.  We do not quite succeed in this goal in three respects: 
\begin{enumerate}
\item The model for the homology of the free loop space that we use is the  direct limit of the Morse homology of spaces of piecewise geodesics. This model introduces even more signs conventions that one has to choose and verify are compatible.  The choice of was made in order to avoid having to reference  or prove the fact, well-known to all experts, but with no accessible proof available in the literature, that higher dimensional moduli space of Floer trajectories and their generalisations form manifolds with corners. With such a result at hand, and the additional knowledge that the evaluation map at a fixed point defines a smooth map from such moduli spaces to the ambient symplectic manifold, one would be able to avoid using Morse homology, and rely instead on a more classical theory.
\item While a complete account is given for the construction of a chain map implementing Viterbo's isomorphism, including a verification of the signs in the proof that it is a chain map (see Lemma \ref{lem:Vit_is_chain_map}), the reader who wants to see every detail of the proof that the structure maps coming from Floer theory and string topology are intertwined by this isomorphism will have to do quite a bit of sign checking beyond what is included. Natural orientations are constructed on all moduli spaces that are used to show that the isomorphism preserves operations, but beyond that, one needs to perform some symbol pushing to check that the relations  hold as stated, rather than up to an overall sign depending only on discrete invariants (the dimension of $\Q$, the degree of the inputs, ...).
\item The construction of a map from Floer theory to loop homology is given in Chapter \ref{cha:from-sympl-homol} and one can reasonable hope enough background has been provided that the diligent reader can follow the argument up to that point without being necessarily equipped with expertise in these matters. However, Chapters \ref{cha:viterbos-theorem} and \ref{cha:viterbos-theorem-II}, in which this map is proved to be an isomorphism, will likely prove to be more challenging because they rely on an essentially new technique using parametrised moduli spaces of pseudoholomorphic curves with Lagrangian boundary conditions.
\end{enumerate}

Beyond the results on the connection between symplectic cohomology and loop homology that have already appeared in the literature (see in particular \cite{viterbo-96,AS-product}), several new results are proved. First, statements and proofs are systematically generalised from the orientable to the non-orientable case, including the construction of a natural $\bZ$ grading on  symplectic cohomology, the definition of string topology operations, and the construction of the isomorphism between (twisted) loop homology and symplectic cohomology.

However, the most important new results are contained in Chapters \ref{cha:viterbos-theorem} and \ref{cha:viterbos-theorem-II}, which introduce two new mutually inverse maps  between loop homology and symplectic cohomology. These maps in a sense explain that Viterbo's theorem holds because
\begin{equation*}
  \parbox{35em}{the family of cotangent fibres $\{ \Tq \}_{q \in \Q}$ defines a Lagrangian foliation of $\TQ$.}
\end{equation*}
The motivation for introducing these maps comes from Fukaya's ideas on \emph{family Floer homology}. Moreover, the verification that the maps are mutually inverse uses degenerations of moduli spaces of discs with multiple punctures, which are related to recent work in Floer theory that uses moduli spaces of annuli \cites{FOOO-sign,BC,A-generate} (see, in particular Figures \ref{fig:infinite-thin-annuli} and  \ref{fig:moduli-disc-two-punctures-0}). The key point is to
\begin{equation*}
  \parbox{35em}{verify that maps in Floer theory are isomorphisms by considering degenerations of Riemann surfaces, rather than degenerations of Floer equations on a fixed surface.}
\end{equation*}
The idea of degenerating the Floer equation goes back to Floer who used it to prove that certain Floer cohomology groups are isomorphic to ordinary cohomology \cite{Floer-JDG}. Such degenerations usually give rise to \emph{isomorphisms} of chain complexes, but at the cost of requiring very delicate analytic estimates. The method we adopt usually gives a weaker result (only a chain homotopy equivalence), but tends to be more flexible, and requires arguments of a more topological nature.

These notes are organised as follows: symplectic cohomology, with coefficients in a local system over the free loop space,  is defined for cotangent bundles in Chapter \ref{cha:sympl-cohom-viterb}, and three operations on it are constructed in Chapter \ref{cha:oper-sympl-cohom} under the assumption that the local system is \emph{transgressive}. These operations give rise to a (twisted) Batalin-Vilkovisky structure. Chapter \ref{cha:finite-appr-loop} is independent of the first two, and provides a construction of a Batalin-Vilkovisky structure on the twisted homology of the loop space of a closed manifold. This structure is constructed from the Morse homology of finite dimensional approximations.  A map from symplectic cohomology to loop homology is constructed in Chapter \ref{cha:from-sympl-homol}, which also includes the verification that this map intertwines the operations on the two sides. A left inverse to this map is constructed in Chapter \ref{cha:viterbos-theorem}, and Chapter \ref{cha:viterbos-theorem-II} provides the proof that this left inverse is an isomorphism.

\subsection*{Acknowledgments}
I would like to thank Thomas Kragh for sharing his insights about Section \ref{ex:homotopy_different_inclusions}, Joanna Nelson for catching some typographical errors, and Janko Latschev, Dusa McDuff,  Alex Oancea, and an anonymous referee for extensive and helpful comments.

The author was partially supported by NSF Grant DMS-1308179, and by the Simons Center for Geometry and Physics.

\chapter{Symplectic cohomology of cotangent bundles}
\label{cha:sympl-cohom-viterb}

\section{Introduction}
In this chapter, we define the symplectic cohomology of a cotangent bundle, with coefficients in a local system $\nu$ over the free loop space; we denote this graded abelian group by 
\begin{equation}
  SH^{*}(\TQ; \nu).
\end{equation}
\begin{rem}
The main justification for considering non-trivial local systems will be explained in Chapter \ref{cha:from-sympl-homol}, where we compare symplectic cohomology to the homology of the free loop space.  
\end{rem}

In order to keep the construction of symplectic cohomology to a reasonable length, we shall focus on the aspects of the theory which distinguish it from Hamiltonian Floer theory on compact symplectic manifolds;  in particular, the reader will be occasionally advised to consult one of two references: (1) Salamon's notes on Floer theory \cite{salamon-notes} (2) the textbook on Floer and Morse homology by Audin and Damian \cite{AD}. The main differences are as follows:
\begin{enumerate}
\item For closed manifolds, the Floer complex is defined for a generic Hamiltonian and almost complex structure, and the cohomology of this complex is independent of these choices. This is not the case for cotangent bundles: one must impose additional conditions both on the Hamiltonian and on the almost complex structure in order to ensure that the differential is well-defined. Moreover, having imposed these restrictions, Floer cohomology still depends on the choice of Hamiltonian. 
\item Most discussions of the $\bZ$-grading in Floer theory are usually restricted to contractible orbits, under the assumption that the first Chern class vanishes. While the cotangent bundle of an orientable manifold has vanishing first Chern class, this is not true in general, e.g. for the cotangent bundle of $\bR \bP^{2}$. Moreover, there are interesting dynamical aspects in the study of non-contractible orbits, so we must  understand gradings for such orbits as well.
\item We shall define operations on symplectic cohomology in Chapter \ref{cha:oper-sympl-cohom}. In order to keep track of the signs in various equations, we shall give a treatment of signs in the construction of Floer theory which is superficially different from the usual accounts that appear in the literature.
\end{enumerate}

\section{Basic notions}
\subsection{The cotangent bundle as a symplectic manifold}
The construction of a symplectic form on the cotangent bundle $\TQ$ of a smooth manifold $\Q$ essentially goes back to Liouville:  Given local coordinates $(q_1, \ldots, q_n)$ on $Q$, let us write $p_i$ for the coefficient of $dq_i$ in a cotangent vector, so that $(q_1, \ldots, q_n, p_1, \ldots, p_n)= (q,p)$ define local coordinates on $\TQ$.  
\begin{defin}
 The \emph{canonical form} $\lambda$ on $\TQ$ is the $1$-form which assigns to a tangent vector $v$ at $(p,q)$
 \begin{equation}
   p(q_{*}(v))
 \end{equation}
where $q_{*}$ is the map induced on tangent vectors by projection to the base.

\end{defin}
\begin{exercise} \label{ex:Liouville}
Compute that $\lambda$ is given in local coordinates by
 \begin{equation}
   \label{eq:Liouville_form}
   \lambda = \sum_{i=1}^{n}p_{i} dq_{i}.
 \end{equation}
\end{exercise}

The differential of $\lambda$ is the \emph{canonical symplectic form} given in local coordinates by
\begin{equation}
  \label{eq:symplectic_form}
  \omega = \sum_{i=1}^{n}dp_{i} \wedge dq_{i}.
\end{equation}
To verify that $\omega$ is indeed symplectic, one checks that (1) $d \omega =0$ (which follows from $d^2 \equiv 0$) and  (2) that $\omega^{n}$ is a volume form.  Note that a direct consequence of Exercise \ref{ex:Liouville} is that our expression for $\omega$ is invariant under changes of coordinates.

\begin{rem}
 It will be convenient to identify the cotangent bundle of $\bR^{n}$ with $\bC^{n}$. Writing $(p,q)$ for the coordinates of $T^{*} \bR^{n}$ the map
 \begin{equation} \label{eq:map_cotangent_to_complex}
   (q,p) \mapsto q - i p
 \end{equation}
has the property that it takes the canonical symplectic form on the cotangent bundle to the standard symplectic form on $\bC^{n}$:
\begin{equation}
  \sum_{i=1}^{n} dx_i \wedge dy_i.
\end{equation}
\end{rem}

We shall also consider the \emph{Liouville} vector field
\begin{equation}
  \label{eq:Liouville_vector_field}
  X_{\lambda} = \sum_{i=1}^{n} p_{i} \partial_{p_i}
\end{equation}
which integrates to the flow
\begin{equation} 
  \label{eq:Liouville_flow}
  \psi^{\rho}(q_1, \ldots, q_n, p_1, \ldots, p_n) = (q_1, \ldots, q_n, e^{\rho} p_1, \ldots, e^{\rho} p_n)
\end{equation}
\begin{exercise}
  Define $X_{\lambda}$ invariantly in terms of $\omega$ and $\lambda$.
\end{exercise}
\subsection{Hamiltonian orbits} \label{sec:hamiltonian-orbits}
A \emph{Hamiltonian} is a smooth function $H$  on $\bR/\bZ  \times \TQ$, which we will think of as a family of functions $H_{t}$ on $\TQ$ parametrised by $t \in \bR/\bZ$.    Whenever $H_{t}$ is independent of $t$, we say that the Hamiltonian is \emph{autonomous}.
\begin{defin}
  The \emph{Hamiltonian vector field} of $H_{t}$ is the unique vector field $X_{H_t}$ on $\TQ$ satisfying
\begin{equation}
  \label{eq:Hamiltonian_vector}
  \omega(X_{H_t} , \_ ) = - dH_{t}( \_ )
\end{equation}
\end{defin}
We shall write $X_{H}$ for the time-dependent vector field whose value at $t$ is $X_{H_t}$. As with any vector field, one can try to understand the dynamical properties of the flow by considering the closed flow lines, which we call \emph{orbits}:
\begin{defin}
  A \emph{time-$1$ Hamiltonian orbit} of $H $ is a map
  \begin{equation}
    x \co  \bR/\bZ   \to \TQ
  \end{equation}
such that
\begin{equation}
  \label{eq:Hamiltonian_orbit_equation}
  \frac{d x}{dt} = X_{H}.
\end{equation}
\end{defin}
The set of time-$1$ Hamiltonian orbits for a given family $H$ will be denoted $\Orbit(H)$. The key idea in Floer theory is that, under suitable genericity properties, the elements of $\Orbit(H)$ label a basis for a cochain complex (the Floer complex defined in Section \ref{sec:floer-cohom-line}) whose cohomology is invariant under compactly supported perturbations of $H$.

Let us now fix a metric on $\Q$.  We write 
\begin{equation} \label{eq:radial_function}
 \rho(q,p) = \langle p , p \rangle^{1/2}
\end{equation}
for the norm of the covector $p$, which we think of as a \emph{radial} coordinate on $\TQ$.  For each positive real number $\rho$, we obtain a disc bundle
\begin{equation}
  \DQ[\rho] \subset \TQ
\end{equation}
 consisting of those points $(p,q)$ such that $\langle p , p \rangle^{1/2} \leq \rho$; the boundary of $\DQ[\rho]$ is the sphere bundle, which we denote $\SQ[\rho]$. When $\rho=1$, we omit the subscript from the notation of the unit disc and sphere bundles.

Rescaling the fibres allows us to identify the complement of $\DQ$ with the product of $\SQ$ with a ray; we obtain a decomposition
\begin{equation}
  \TQ = \DQ \cup_{\SQ} \SQ \times [1,+\infty)
\end{equation}
into the disc bundle and a \emph{conical end}.
\begin{defin} \label{def:linear}
Let $b$ be a real number. A  Hamiltonian $H$ is \emph{linear of slope $b$} if 
\begin{equation}
H| \SQ \times [1,+\infty)  \equiv   b \cdot \rho.
\end{equation}
We define a preorder on the set of linear Hamiltonians:
\begin{equation} \label{eq:preorder_Hamiltonians}
  \parbox{30em}{$H \preceq K$ if the slope of $H$ is less than or equals that of $K$.}
\end{equation}
\end{defin}
Unless otherwise mentioned, all Hamiltonians considered from now on will be linear.  The Hamiltonian flow of a linear function is connected to the geodesic flow: we remind the reader that a loop $\gamma$ in $Q$ is a (non-constant) geodesic if and only if the lift $\tilde{\gamma} = (\gamma, \frac{d\gamma}{dt})$ of $\gamma$ to $T \Q$ is always tangent to the horizontal distribution defined by the metric. A loop $x = (q(t),v(t))$ in $T \Q$ is therefore the lift of a geodesic if and only if it is tangent to the horizontal distribution and the projection to the base of the tangent vector to $x$ satisfies
\begin{equation} \label{eq:lift_geodesic}
    q_{*}\left(\frac{dx}{dt}\right) = v. 
\end{equation}
Using the metric, we may identify the cotangent and tangent bundle; we write $g(p)$ for the vector dual to a covector $p$, and 
\begin{equation}
  \tilde{g} \co \TQ \to T \Q
\end{equation}
for the induced map on total spaces.
\begin{exercise} \label{ex:horizontal-lag}
Show that the image of the horizontal distribution under $\tilde{g}^{-1}$ defines a Lagrangian distribution in $\TQ$ (Hint: use normal geodesic coordinates). Conclude that the Hamiltonian flow of the function $\frac{\rho^{2}}{2}$ is identified by $\tilde{g}$ with the geodesic flow.
\end{exercise}

\begin{lem} \label{lem:orbit_is_closed_geodesic}
Let  $x$ be an orbit of a linear Hamiltonian $H$ of slope $b$. If $x$ intersects the conical end, then the loop
\begin{equation}
\begin{aligned}
\bR/b\bZ  & \to \Q \\
  t & \mapsto q (x(t/b))  
\end{aligned}
\end{equation}
is a geodesic parametrised by unit speed.
\end{lem}
\begin{proof}
Since $dH(X_H)=0$, any Hamiltonian orbit which intersects the complement of $\DQ$ lies entirely in one of the level sets of $\rho$; in particular it lies entirely in the complement of $\DQ$. We claim that
\begin{equation}
\frac{  \tilde{q} (x(t/b))}{\rho}  
\end{equation}
satisfies Equation \eqref{eq:lift_geodesic}, and has tangent vector lying in the horizontal distribution.

To prove this, we first reduce to the case $x$ lies on the unit cotangent bundle. The key point is that dilating the fibres preserves $X_H$, because it scales $\omega$ and $\rho$ by the same amount; in particular if $x(t) = (q(t),p(t))$ is an orbit of $X_{\rho}$, so is $(q(t), p(t)/  \langle p , p \rangle^{1/2})$.

Next, we show that if $p$ has norm $1$, then
\begin{equation}  \label{eq:projection_dual_to_fibre}
  q_{*}(X_\rho) = \tilde{g}(p).
\end{equation}
This is a straightforward computation: identify the vertical tangent vectors at $(q,p)$ with $\Tq$, and observe that, for such a covector $p'$:
\begin{equation}
 p'(q_{*}(X_\rho))  = \omega( p', X_\rho) = d\rho(p') =  \langle p , p' \rangle.
\end{equation}

From the discussion preceding Exercise \ref{ex:horizontal-lag}, and the fact that $\tilde{g} $ commutes with projection to the base, the result follows once we show that  the image of $X_{\rho}(q,p)$ under $\tilde{g}$ lies in the horizontal distribution. Since parallel transport with respect to the connection induced by the metric is an isometry, $d\rho$ vanishes on the horizontal distribution, hence $\omega(\_, X_\rho )$ also vanishes on this Lagrangian subspace (see Exercise \ref{ex:horizontal-lag}). Since a Lagrangian subspace is its own symplectic orthogonal complement, we conclude that $\tilde{g}_{*} X_\rho$ lies in the horizontal distribution.
\end{proof}
\begin{cor}
If $\Q$ does not admit any closed geodesic of length $b$, and $H$ is linear of slope $b$, then all elements of $\Orbit(H)$ have image contained in the interior of $\DQ$. \qed
\end{cor}

\subsection{Non-degeneracy of orbits}
In order to define Floer complexes, we need the set of orbits be well behaved: in particular, we would like the number of orbits to be invariant under small perturbations.  To state the genericity condition which implies this, we integrate the Hamiltonian vector field $X_{H}$ to obtain a family of Hamiltonian symplectomorphisms 
\begin{equation}
 \phi^{t} \co \TQ \to \TQ
\end{equation}
such that $\phi^{t}(x(0)) = x(t)$ for every flow line $x$ of $X_{H}$.  In particular, if $x$ is a time-$1$ orbit, then $x(1) = x(0)$, hence $x(0)$  is a fixed point of $\phi^{1}$, so we obtain an induced \emph{Poincar\'e return map}
\begin{equation} \label{eq:Poincare_return}
  d \phi^{1}  \co T_{x(0)} \TQ \to T_{x(0)} \TQ .
\end{equation}

\begin{defin} \label{def:non-degen-orbits}
A Hamiltonian  orbit $x$ is \emph{non-degenerate} if $1$ is not an eigenvalue of $  d \phi^{1}|x(0) $.
\end{defin}

\begin{example}
 If $x_{i}$ is a sequence of distinct orbits such that $\lim_{i} x_i(0) = x (0) $, show that $  d \phi^{1}|x(0) $ has an eigenvector with eigenvalue $1$.
\end{example}

\begin{lem} \label{lem:generic_Ham_non-deg}
Let $H$ be a Hamiltonian on $\TQ$. If $U \subset \TQ $ is an open set, there is a countable intersection $\cH(U) $ of open dense subsets in the space of compactly supported smooth functions on $U \times S^1$, such that all orbits of $H + K$ which pass through $U$ are non-degenerate whenever $K \in \cH(U)$.  \qed
\end{lem}
\begin{proof}[Sketch of proof:]
For a detailed proof, see the appendix to \cite{ABW}. The general idea is as follows:  consider the graph of $\phi^{1}$ as a submanifold of $\TQ \times \TQ$. If we reverse the symplectic form on the second factor, this is a Lagrangian submanifold. An orbit is non-degenerate if and only if the corresponding intersection point between the diagonal and the graph is transverse. Since every $C^{2}$ small Hamiltonian perturbation of the graph corresponds to the graph of a perturbed Hamiltonian function, the result follows from the fact that transversality for Lagrangians can be achieved by such perturbations. 
\end{proof}

In practice, we shall be working with linear Hamiltonians: the first step is therefore to choose a slope $b$ such that $\Q$ admits no closed geodesic of length $b$; since the lengths of geodesics form a closed set of measure $0$, there are arbitrarily large choices of $b$ satisfying this property. As an immediate consequence of Lemma \ref{lem:generic_Ham_non-deg}, we conclude
\begin{cor}
If $b$ is not the length of any geodesic on $Q$, and $H$ is a generic Hamiltonian of slope $b$, all elements of $\Orbit(H)$ are non-degenerate.
\end{cor}

\section{A first look at Floer cohomology}
In this section, we define an ungraded Floer group over $\bZ/2 \bZ$. The proper construction of a graded Floer groups over the integers is relegated to Section \ref{sec:floer-cohom-line}.

Let $\SQ$ be the unit cotangent bundle of $Q$ with respect to some Riemannian metric, and let $H$ be a Hamiltonian all of whose orbits are non-degenerate, which agrees with $b \cdot \rho$ whenever $\rho \geq 1$ (i.e. $H$ is linear of slope $b$ in the sense of Definition \ref{def:linear}).

The goal of this section is to construct the \emph{Hamiltonian Floer cochain complex} of $H$ which is generated by basis elements $\langle x \rangle$ labelled by the elements of $\Orbit(H)  $:
\begin{equation}
CF(H ; \bZ/2\bZ) \equiv \bigoplus_{x \in \Orbit(H)}  \bZ/2\bZ \cdot \langle x \rangle.
\end{equation}

The differential will be obtained by counting pseudo-holomorphic cylinders in $\TQ$, which requires choosing a compatible almost complex structure.  Recall that such an almost complex structure satisfies
\begin{align}
  \omega(v,Jv) & > 0  \\
\omega(Jv,Ju) & = \omega(v,u) 
\end{align}
for every pair of tangent vectors.

Since $\TQ$ is not compact, we must impose additional conditions away from a compact set:
\begin{defin} \label{def:convex_almost_complex}
A compatible almost complex structure $J$ is said to be \emph{convex} near  $\SQ[\rho]$ if the restriction to a neighbourhood of this hypersurface satisfies
  \begin{equation*}
    d\rho\circ J = - e^{f}\lambda
  \end{equation*}
for some smooth function $f$.
\end{defin}

\begin{exercise}
On the plane, let $\rho = \pi r^{2}$, and consider the $1$-form $\lambda =  \rho d \varphi$. Show that the standard complex structure is convex near every circle centered at the origin.
\end{exercise}

\subsection{Moduli spaces of cylinders} \label{sec:moduli-spac-cylind}
For the purpose of defining Floer cohomology, choose a family $J_t$ of almost complex structures on $\TQ$, parametrised by $t \in S^1$, which are compatible with $\omega$,  and consider smooth maps
\begin{equation}
  u \co Z = (-\infty, +\infty) \times S^1 \to \TQ
\end{equation}
satisfying Floer's equation
\begin{equation} \label{eq:dbar_equation_s-indep}
    J_{t} \partial_{s} u = \left(  \partial_{t} u - X_{H_{t}} \right).
\end{equation}
Note that there is an $\bR$-action on the space of such maps, given by pre-composing with translation in the $s$-coordinate.
\begin{defin}
For $i \in \{0,1\}$, let $x_{i} \in \Orbit(H)$ be time-$1$ Hamiltonian orbits.  The moduli space $\Cyl(x_0;x_1)$ is the quotient by $\bR$ of the space of maps from $Z$ to $\TQ$, satisfying Equation \eqref{eq:dbar_equation_s-indep}, and converging to $x_1$ in the limit $s \to +\infty$, and to $x_0$ in the limit $s \to -\infty$.
\end{defin}

\begin{figure}[h] \label{fig:Floer_cylinder}
  \centering
\includegraphics[scale=1.33]{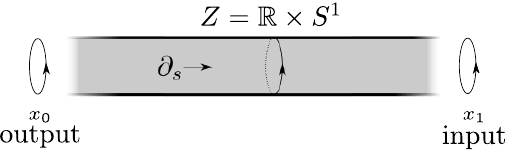}
  \caption{ }
\end{figure}

One can set up this problem as a solution to an elliptic problem on the space of all smooth maps from the cylinder to $\TQ$; the pseudo-holomorphic curve equation
\begin{equation} \label{eq:CR-operator-Floer}
  u \mapsto \partial_{s} u  + J_{t}\left(  \partial_{t} u - X_{H_{t}} \right)
\end{equation}
involves taking exactly one derivative. On the space of maps which converge exponentially to the orbits $x_{0}$ and $x_1$, this expression defines a  $L^{p}$ section of the pullback of $T \TQ$ under $u$.  At a solution to the Floer equation, we can take the differential of this section with respect to vector fields along $u$. In an early paper \cite{Floer-index}, Floer observed that this differential is a Fredholm operator
\begin{equation}
D_{u} \co   W^{1,p}(Z, u ^{*}  T\TQ ) \to  L^{p}(Z, u ^{*}  T\TQ ).
\end{equation}
\begin{rem}
 One can invariantly write the pseudo-holomorphic curve equation as a section of the bundle of $(0,1)$ forms on $Z$ valued in $TM$. However, this bundle is trivial, and Equation \eqref{eq:CR-operator-Floer} is the result of writing the invariant operator in one of the possible trivialisations.
\end{rem}
\begin{defin}
The \emph{virtual dimension} of $u$ is:
\begin{equation} \label{eq:virtual_dimension}
\dim(  \ker(D_{u})) - \dim(  \coker(D_{u})) -1.
\end{equation}
\end{defin}
The presence of the constant $-1$ term in Equation \eqref{eq:virtual_dimension} is due to the fact that we are interested in the dimension of $\Cyl(x_0,x_1)$ near $u$, and this moduli space was defined to be the quotient by $\bR$ of the space of solution to Floer's equation.
\begin{exercise}\label{ex:translation_in_kernel}
 Show that, if $u$ is a solution of the Floer equation with asymptotic conditions $x_0 \neq x_1$, then the kernel of $D_{u}$ is at least $1$-dimensional, with $\partial_{s} u$ defining an element of the kernel. 
\end{exercise}

\subsection{Action and Energy} \label{sec:orbits-crit-pts}
In order to control the moduli spaces $\Cyl(x_0,x_1) $, it is useful to recall that Floer defined his theory as a Morse theory for the action functional
\begin{align}
\sL (\TQ) & \to \bR \\
  \label{eq:Action-orbit}
  \Action(x) & = \int  - x^{*}(\lambda) + H_{t}  \circ x \, dt.
\end{align}

\begin{rem}
  There are four different conventions for the action of a Hamiltonian orbit: first, one must decide whether to define the Hamiltonian flow $X_H$ to satisfy $\iota_{X_H} \omega = dH$, or $\iota_{X_H} \omega =- dH  $. We opt for the second convention, which leads to the two terms in Equation \eqref{eq:Action-orbit} having opposite signs. Then, one can either consider the action as we have defined it, or its negative. 
\end{rem}

\begin{exercise}
The critical points of $\Action$ are exactly the time-$1$ Hamiltonian orbits of $H$. For help, see for example the first paragraph of \cite{salamon-notes}*{Section 1.5}.
\end{exercise}
One can in fact show that the moduli space of negative gradient flow lines of the action functional, starting at  $x_0$ and ending at $x_1$,  is the moduli space of cylinders $\Cyl(x_0;x_1)$, i.e. if $u$ is such a cylinder, the family of loops $u(s,\_)$ defines a negative  gradient flow line. 

To see that the action decreases with the $s$-coordinate along a solution to Floer's equation, we introduce a local notion of energy
\begin{equation}
  \|du - X_{H_{t}} \otimes dt \|^{2} =  \omega\left(\partial_{s} u , J_{t}\partial_{s} u \right) + \omega\left( \partial_{t} u- X_{H_{t}},   J_{t} \left( \partial_{t} u- X_{H_{t}} \right) \right)
\end{equation}
using the family of metrics which are induced by the almost complex structure $J_{t}$ and the symplectic form $\omega$. The integral over the cylinder is the energy of a Floer trajectory:
\begin{equation} \label{eq:energy_formula}
  E(u) = \frac{1}{2} \int \|  du - X_{H_{t}} \otimes dt  \|^2 ds \wedge dt.
\end{equation} 

One of the reasons for considering this energy is the following result which asserts that finiteness of the energy implies convergence to Hamiltonian orbits at the ends, see \cite{salamon-notes}*{Proposition 1.21}:
\begin{lem}
If $u$ is a solution to Floer's equation, then $E(u)$ is finite if and only if there exist orbits $x_0$ and $x_1$ such that $u \in \Cyl(x_0,x_1)$.  \qed
\end{lem}
\begin{exercise}
Under the assumption that $u$ is a solution to Floer's equation, show that $E(u)$ vanishes if and only if $u(s,t)$ is independent of $s$, hence $u \in \Cyl(x;x)  $.  Show that this \emph{stationary} solution is the unique element of $ \Cyl(x;x) $.
\end{exercise}
More generally, if $a < b$, we consider the energy of the restriction of $u$ to the annulus $(a,b) \times S^1$. We can compute this energy as 

\begin{align} \notag 
E(u| (a, b) \times S^1 ) & = \frac{1}{2} \int_{(a,b) \times S^1 }  \omega\left(\partial_{s} u , J_{t}\partial_{s} u \right) + \omega\left( \partial_{t} u- X_{H_{t}},   J_{t} \left( \partial_{t} u- X_{H_{t}} \right) \right) ds \wedge dt \\ \notag
& = \frac{1}{2} \int_{(a,b) \times S^1 } \omega \left(\partial_{s} u , \partial_{t} u -  X_{H_{t}} \right) + \omega \left( \partial_{t} u- X_{H_{t}},  -\partial_{s} u  \right) ds \wedge dt \\  \label{eq:energy_is_topological}
& =  \int_{(a,b) \times S^1 }  u^{*}(\omega) - u^* dH_{t} \wedge dt.
\end{align}    
Applying Stokes's theorem, we see that the right hand side agrees with the difference between the actions of the boundary curves.
\begin{lem}
The restriction of the integral in \eqref{eq:energy_formula} to a finite annulus $(a,b) \times S^1$ satisfies: 
\begin{equation}
 E(u| (a, b) \times S^1 ) = \Action(u(a, \_)) - \Action(u(b, \_)).
\end{equation}
 \qed
\end{lem}
\begin{cor} \label{cor:action_decreases}
   The function $\Action(u(s, \_))$ decreases monotonically with $s$. In particular, $\Cyl(x_0;x_1)$  is empty unless $\cA(x_0) > \cA(x_1)$ or $x_0 = x_1$.  \qed
\end{cor}

\subsection{Positivity of energy and Compactness }
In the case of Hamiltonian Floer theory on closed aspherical symplectic manifolds, Floer constructed a compactification $\Cylbar(x_0;x_1)$ for each pair of time-$1$ orbits of a given Hamiltonian. By construction, this space admits a natural stratification by products of moduli spaces of cylinders:
\begin{equation}\label{eq:boundary_moduli_space}
\Cylbar(x_0;x_1) \equiv  \bigcup_{x'_i \in \Orbit(H)}   \Cyl(x_0;x_1') \times \Cyl(x'_1;x'_2) \times \cdots \times \Cyl(x'_{d-1};x'_{d}) \times \Cyl(x'_d;x_1).
\end{equation}
For more general closed symplectic manifolds, one must take into account, as well, the possibility of bubbling arising from holomorphic spheres. The space  $  \Cylbar(x_0;x_1)  $ is therefore called the \emph{Gromov-Floer} compactification.

On a general open symplectic manifold, the Gromov-Floer procedure may not produce a compactification of the moduli space of holomorphic curves.  The issue is that  a sequence of such curves could escape to infinity, and hence not converge to anything in the Gromov-Floer sense.  In order to exclude this, we shall prove that the images of all elements of $\Cyl(x_0;x_1) $ lie in $\DQ$. This can be shown using a standard version of the maximum principle, but it is useful for later arguments to introduce the \emph{integrated maximum principle} of \cite{ASeidel}. 

Let $u$ be an element of $\Cylbar(x_0;x_1)  $. We start by choosing a regular value $1 + \epsilon$ of $  \rho \circ u$. Let $\Sigma \subset Z$ denote the inverse image of $[1 + \epsilon,+\infty)  $ under $\rho \circ u  $. Let $v$ denote $u| \partial \Sigma$.  We define the \emph{geometric energy} of $v$
\begin{equation} \label{eq:geometric_energy_v}
     E(v) = \int_{\Sigma} \| dv - b X_{\rho} \otimes dt  \|^{2},
\end{equation}
where we have used the fact that $H_{t} = b \rho$ away from the unit disc bundle. 
\begin{exercise} 
  Show that $E(v)$ is non-negative, and vanishes if and only if the image of $v$ is contained in a level set of $\rho$.
\end{exercise}
\begin{exercise}
 Generalising Equation \eqref{eq:energy_is_topological}, show that
 \begin{equation} \label{eq:energy_topological_v}
   E(v) =  \int_{\partial \Sigma} v^{*}(\lambda) - b \rho \circ v  \cdot dt
 \end{equation}
\end{exercise}

We shall use the above two exercises to prove compactness.
\begin{lem} \label{lem:maximum_principle_cylinders}
If $J_{t}$ is convex near $\SQ$, then the image of every element of $\Cyl(x_0;x_1) $ is contained in $\DQ$.
\end{lem}
\begin{proof}
Assume (by contradiction) that there is an element $u$ of $\Cylbar(x_0;x_1)  $ whose image intersects the complement of $\DQ$.  Combining non-negativity of energy with Equation \eqref{eq:energy_topological_v}, we find that
\begin{equation}
  0 < \int_{\partial \Sigma} v^{*}(\lambda) - b \rho \circ v  \cdot dt.
\end{equation}
 We will derive a contradiction by proving that the opposite inequality holds as well. First, we rewrite the right hand side as
\begin{equation}
  \int_{\partial \Sigma} \lambda \left( dv  - bX_{\rho}  \otimes dt \right).
\end{equation}
Equation \eqref{eq:dbar_equation_s-indep} implies that the integrand is equal to
\begin{equation}
 -  \lambda \circ J_{t} \left( dv  - X_{H}  \otimes dt \right) \circ j = - e^{-f} d\rho \left( dv \circ j  +  b X_{\rho}  \otimes ds \right)  = - e^{-f} d(\rho \circ v) \circ j.
\end{equation}
The last equality follows from the fact that $X_{\rho}$ is tangent to the level sets of $\rho$.  Recall that a tangent vector $\xi$ to $\partial \Sigma $ is positively oriented if $j \xi $ points inwards. In this case, $ 0 \leq  d(\rho \circ v) (j \xi)  $, since $\rho \circ v$ reaches its global minimum on $\partial \Sigma$. We conclude that
\begin{equation}
   \int_{\partial \Sigma}  - e^{-f}d(\rho \circ v) \circ j \leq 0.
\end{equation}
We have reached the desired contradiction, which implies that the image of $u$ is contained in a disc bundle of radius $1+\epsilon$. Since $\epsilon$ can be arbitrarily small, we see that the image is in fact contained in the unit disc bundle. We  conclude that the image of elements of $\Cyl(x_0;x_1)$ is contained in this set.
\end{proof}

\begin{cor} \label{cor:compactness}
The moduli space  $\Cylbar(x_0; x_1)  $  is compact for any pair $x_0$ and $x_1$. \qed
\end{cor}

\subsection{Transversality}
\label{sec:transversality}
Consider the space of almost complex structures on $\TQ$ which are convex near $\SQ$;  this space admits a natural topology as a subset of the Fr\'echet space of sections of the bundle of endomorphisms of the tangent space of $\TQ$:
\begin{equation}
  C^{\infty}(\TQ, \End(T \TQ) ).
\end{equation}
A natural  Fr\'echet manifold structure is provided by the following result:
\begin{exercise}
Prove that, given an almost complex structure which is convex near $\SQ$,  the space of nearby almost complex structures admits  a local chart modelled after the linear subspace of $ C^{\infty}(\TQ, \End(T \TQ) ) $ consisting of elements $K$ such that 
\begin{equation}
  K J + J K = 0,
\end{equation}
and the restriction to a neighbourhood of $\SQ$ satisfies
\begin{equation}
  d \rho \circ K  \in \bR \cdot \left( d \rho \circ J \right),
\end{equation}
where both sides are co-vector fields on this neighbourhood.
\end{exercise} 
One can prove the desired transversality results in this  Fr\'echet setting as in \cite{FHS}.  The original approach of Floer instead bypassed Fr\'echet manifolds, and used a Banach manifold of families of almost complex structures on $\TQ$ parametrised by a space $P$, which is modelled after a Banach space
\begin{equation}
  C^{\infty}_{\mathbf{\epsilon}}(P \times \TQ, \End(T \TQ) ) 
\end{equation}
of sections whose covariant derivatives decay sufficiently fast. Let $\sJ_{S^1}$ denote the Banach submanifold of those almost complex structures parametrised by $S^1$, which are, in addition, convex near $\SQ$. 

The following result is the cornerstone of Floer theory, and goes back to \cite{Floer-gradient}*{Section 5}. For the statement, we fix a Hamiltonian $H$ such that all orbits are non-degenerate.
\begin{thm} \label{lem:transversality}
There is a dense set $\sJ^{reg}_{S^1} \subset \sJ_{S^1}$ such that the following holds whenever $J_{t} \in \sJ^{reg}_{S^1}$:
\begin{equation} \label{eq:regular_J}
  \parbox{34em}{for every pair $(x_0,x_1)$ of orbits, and every cylinder $u \in \Cyl(x_0;x_1)$, the operator $D_{u}$ is surjective. }
\end{equation}
In this case, $\Cyl(x_0;x_1) $ is a smooth manifold of dimension equal, at every point, to the virtual dimension.
\qed
\end{thm}
\begin{rem}
It is more common to fix the almost complex structure, and vary the Hamiltonian instead. This is the method adopted, for example in \cite{salamon-notes}*{Theorem 1.24} and \cite{AD}*{Chapitre 8}.
\end{rem}

We say that our data $(H,\{J_{t}\})$ are regular if all orbits are non-degenerate, and Condition \eqref{eq:regular_J} holds. From now on, such data will be assumed to be regular. We shall be particularly interested in the situation when  $\Cyl(x_0;x_1) $ has virtual dimension equal to $0$:
\begin{defin}
  An element $u \in \Cyl(x_0;x_1)$ is \emph{rigid} if it is regular, and the Fredholm index of $D_{u}$ is equal to $1$.
\end{defin}
We temporarily denote the subset of rigid elements by $\Cyl^{0}(x_0;x_1)  \subset  \Cyl(x_0;x_1)$.  It shall follow from Theorem \ref{lem:cylinder_orientation_lines} that, for cotangent bundles, all elements of  $\Cyl(x_0;x_1)   $  have the same virtual dimension. In particular, $ \Cyl^{0}(x_0;x_1)   $  is either empty, or consists of the whole of $  \Cyl(x_0;x_1) $.

\begin{exercise}
  Using Corollary \ref{cor:compactness}, show that $ \Cyl^{0}(x_0;x_1) $ is a finite set.
\end{exercise}

\subsection{The Floer complex} \label{sec:floer-complex}
Given regular Floer data $(H,\{J_{t}\}) $, consider the endomorphism of the Floer cochain complex
\begin{align}
\partial \co   CF(H ; \bZ/2\bZ)  & \to CF(H ; \bZ/2\bZ) \\ \label{eq:differential_no_sign}
\langle x_1 \rangle & \mapsto \sum_{x_0} \# \Cyl^{0}(x_0;x_1)   \cdot \langle x_0 \rangle,
\end{align}
where $\# \Cyl^{0}(x_0;x_1)   $  is the number of rigid elements of $\Cyl(x_0;x_1) $, counted modulo $2$. 

We shall now argue that $\partial^2=0$, i.e. that $\partial$ defines a differential. First, we observe that, the coefficient of $\langle x_0 \rangle $ in $\partial^{2} \langle x_1 \rangle $ agrees with the number of elements of
\begin{equation} \label{eq:square_differential_count}
  \bigcup_{x_1 \in \Orbit(H) }\Cyl^{0}(x_0;x_1) \times \Cyl^{0}(x_1;x_2).
\end{equation}
To show that this set has an even number of elements, it suffices to prove that it is the boundary of a closed, $1$-dimensional manifold. To this end, let
\begin{equation}
  \Cyl^{1}(x_0;x_2)  \subset  \Cyl(x_0;x_2)
\end{equation}
denote the $1$-dimensional submanifold consisting of solutions to Floer's equation whose virtual dimension is $1$. We omit the proof of the following fact, which may be found in standard references in Floer theory (e.g. \cite{salamon-notes}*{Theorem 3.5}).
\begin{lem}
The closure of $  \Cyl^{1}(x_0;x_2) $ in $\Cylbar(x_0;x_2)$  is a $1$-dimensional manifold whose boundary is given by Equation \eqref{eq:square_differential_count}. \qed
\end{lem}

With this in mind, we conclude that the square of $\partial$ indeed vanishes, which allows us to define Floer cohomology as the quotient:
\begin{equation} \label{eq:ungraded_Floer_group}
  HF(H ; \bZ/2\bZ) \equiv \frac{\ker(\partial)}{\im(\partial)}.
\end{equation}

\section{Towards gradings and orientations}
The Floer group constructed in Section \ref{sec:floer-complex} is ungraded, and defined only over $ \bZ/2\bZ$. To obtain a graded group, we need to assign an integral degree to each orbit, such that a solution to Floer's equation on the cylinder is rigid if and only if the difference in degree between the asymptotic conditions is $1$.  After presenting the necessary material at the linear level, we define this degree in Section \ref{sec:conley-zehnder-index}.

One way to produce a Floer group defined over the integers is to define orientations of all moduli spaces of solutions to Floer's equation, which are consistent with the breaking of Floer trajectories. With this data at hand, one can replace the differential in Equation \eqref{eq:differential_no_sign} with a \emph{signed count} of rigid elements. This is the strategy pioneered by Floer and Hofer in \cite{FH}.

We shall construct the differential using a superficially different approach: in Section \ref{sec:conley-zehnder-index}, we assign to each orbit an \emph{orientation line} which is a free abelian group of rank $1$, and adapt  the ideas of  Floer and Hofer in Section \ref{sec:floer-cohom-line} to construct a canonical map on orientation lines associated to each rigid Floer trajectory. In order to recover the original approach, it suffices to choose generators for these orientation lines.

\subsection{Invariants for paths of unitary matrices} \label{sec:invar-paths-sympl}
In this section, we describe the construction of an analytic index and a determinant line associated to a path $\Psi_{t}$ of symplectomorphisms of $\bC^{n}$ starting at the identity, and such that $\Psi_{1}$ does not have $1$ as an eigenvalue. Writing $\fsp_{2n}$ for the Lie algebra of the group of symplectomorphisms of $\bC^{n}$, we can write such a path uniquely as
\begin{equation}
   \Psi_{t} = \exp(A_{t})
\end{equation}
for a path of matrices $A_{t} \in \fsp_{2n}$. We shall be interested in paths all of whose higher derivatives at $0$ and $1$ agree:
\begin{equation} \label{eq:derivative_agree_at_0_and_1}
  \frac{ d^{k} A_t}{dt^{k}}|_{t=0} = \frac{ d^{k} A_t}{dt^{k}}|_{t=1}.
\end{equation}
\begin{exercise} \label{ex:derivative_agrees_at_beginning and end}
Show that any path $\Psi_{t}$ may be reparametrised so that Equation \eqref{eq:derivative_agree_at_0_and_1} holds.
\end{exercise}
\begin{exercise}
Let $I$ denote the real $2n \times 2n$ matrix corresponding to complex multiplication. Show that the Lie algebra $ \fsp_{2n} $ consists of $2n \times 2n$ real matrices $A$ such that $IA$ is symmetric.
\end{exercise}

Equip $\bC$ with negative cylindrical polar coordinates
\begin{align}
(-\infty,+\infty) \times S^1 & \to \bC \\
(s,t) & \mapsto e^{-s-2 \pi it}.
\end{align}
We say that a metric on $\bC$ is \emph{cylindrical} if it agrees with the product metric on $ (-\infty,+\infty) \times S^1 $ for $s \ll 0$.

Assuming  Equation \eqref{eq:derivative_agree_at_0_and_1}, fix any map
\begin{equation} \label{eq:extend_asymptotic_conditions}
B \in C^{\infty}(\bC , \bR^{2n \times 2n})
\end{equation}
such that 
\begin{equation} \label{eq:condition_inhomogeneous_negative}
  B(e^{-s-2 \pi it}) = \frac{ d A_t}{dt}
\end{equation}
 if $s \ll 0$.  Writing $I$ for the standard complex structure on $\bC^{n}$, we define an operator 
\begin{align} \label{eq:operator_orbits}
D_{\Psi} \co W^{1,p}(\bC, \bC^{n}) & \to   L^{p}(\bC, \bC^{n}) \\ \label{eq:operator_orbits-formula}
D_{\Psi} (X)  & =  \partial_{s} X + I \left( \partial_{t} X  - B \cdot X \right);
\end{align}
where $p > 2$. Because we have assumed that $\Psi_{1}$ does not have $1$ as an eigenvalue, this is a Fredholm operator with finite dimensional kernel and cokernel (for expository accounts, see e.g. \cite{Schwarz}*{Theorem 3.1.9} or \cite{AD}*{Section 8.7}).

Since all the choices that have gone into the construction of $D_{\Psi}$ are canonical (up to contractible choice), any object that is constructed from $D_{\Psi}$, and that is invariant in families, will be an invariant of the loop $\Psi$. In particular, Fredholm theory implies that if $B_{0}$ and $B_{1}$ are two choices of maps in Equation \eqref{eq:extend_asymptotic_conditions}, with associated operators $D_{\Psi}^{0}$ and $D_{\Psi}^{1}$, we have an isomorphism between the determinant lines
\begin{equation}
    \det( \coker^{\vee}(D_{\Psi}^{0})) \otimes  \det( \ker(D_{\Psi}^{0})) \cong  \det( \coker^{\vee}(D_{\Psi}^{1})) \otimes  \det( \ker(D_{\Psi}^{1}))
\end{equation}
where $ \coker^{\vee}$ is the dual of the cokernel, and  $\det(V)$ is the top exterior power of a vector space $V$, which is naturally a $\bZ$-graded real line supported in degree $\dim_{\bR}(V)$  (see Section \ref{sec:aside-orient-lines}).  Such an isomorphism is produced by choosing a path connecting $B_0$ and $B_1$, and noting that the determinant lines of the interpolating family define a real line bundle over the interval, with the above fibres at the two endpoints. Since the space of such paths is contractible, we conclude that this isomorphism is canonical up to multiplication by a positive real number: 

\begin{defin} \label{def:determinant_line-CZ_path}
The \emph{determinant line} of $\Psi$ is the $1$-dimensional $\bZ$-graded real vector space
\begin{equation} \label{eq:determinant_line_definition}
\det(D_{\Psi}) \equiv \det( \coker^{\vee}(D_{\Psi})) \otimes  \det( \ker(D_{\Psi})).
\end{equation}
\end{defin}
By the usual conventions in graded linear algebra, the degree of the determinant line is  the Fredholm index of $D_{\Psi}$:
\begin{equation}
 \ind(\Psi)  =  \dim_{\bR}(\ker(D_{\Psi})) - \dim_{\bR}(\coker(D_{\Psi})) .  
\end{equation}
We shall call this integer the \emph{cohomological Conley-Zehnder index} of $\Psi$.
\begin{rem}
This variant of the Conley-Zehnder index is called \emph{cohomological} because it naturally leads to the construction of a cochain complex associated to Hamiltonian functions, which computes Floer cohomology. Much of the literature which studies the dual theory called Floer homology uses a variant that is related by the formula
\begin{equation}
  \ind(\Psi) = n - \CZ(\Psi).  
\end{equation}
One can in fact compute either of these indices using topological methods as explained in \cite{salamon-notes}*{Section 2.4}.
\end{rem}

Note that the definition of $ \det(D_{\Psi}) $ as a graded vector space is slightly pedantic; a graded vector space $V$ consists of a collection of vector spaces $V^{i}$ for each integer $i$ which are its \emph{graded components}; in our situation, $ \det(D_{\Psi})$ has rank $1$, and hence all but one of these must vanish. We say that $ \det(D_{\Psi}) $ is \emph{supported in degree} $\ind(\Psi)$.

\subsection{Gluing of operators and determinant lines} \label{sec:gluing-oper-determ}

Before venturing into the study of the (non-linear) Floer equation, we discuss the linear analogue: Let $Z$ denote the cylinder $\bR \times S^1$, which we shall equip with coordinates $(s,t)$. Let $\Psi_{\pm}$ be a pair of paths of symplectic matrices, both satisfying Equation \eqref{eq:derivative_agree_at_0_and_1}, and which do not have $1$ as an eigenvalue, with associated loops of matrices $B_{\pm,t} \in \fsp_{2n}$. Consider any matrix-valued function $B$ on $Z$ which agrees at the positive and negative ends with $B_{\pm,s}$, i.e. a map
\begin{align} 
    B \co Z & \to \fgl(2n,\bR) \\
B_{s,t} & = B_{+,t} \textrm{ if } 0 \ll s \\
B_{s,t} & = B_{-,t} \textrm{ if } 0 \gg s.
\end{align}

Such a matrix defines a Cauchy-Riemann operator $D_{B}$ on the cylinder, which gives a Fredholm operator on Sobolev spaces with respect to the standard metric:
\begin{align} 
D_{B} \co W^{1,p}(Z, \bC^{n}) & \to   L^{p}(Z, \bC^{n}) \\ \label{eq:operator_CR-linearised}
D_{B} (X)  & =  \partial_{s} X + I \left( \partial_{t} X  - B \cdot X \right).
\end{align}

It is useful at this stage to note that we have simply repeated Equation \eqref{eq:operator_orbits}, replacing the domain by $Z$, and relabelling the name of the operator. We shall presently see that \emph{gluing} relates the operators $D_{B}$ and $D_{\Psi_{\pm}}$, more precisely, we shall relate the determinant line
\begin{equation}
\det(D_{B})  \equiv \det( \coker^{\vee}(D_{B})) \otimes  \det( \ker(D_{B})).
\end{equation}
to the determinant lines of $\Psi_{\pm}$.

\begin{figure}[h]
  \centering
\includegraphics[scale=1.33]{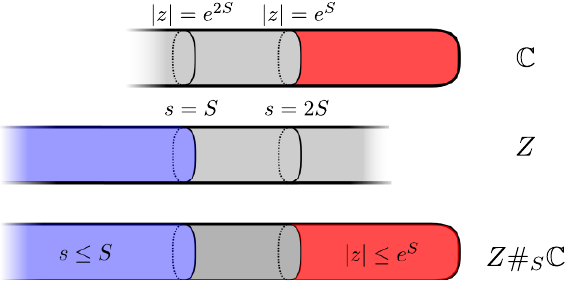}
\caption{} \label{fig:gluing_plane}
\end{figure}

For each positive real number $S$, we obtain a Riemann surface by gluing the disc  $\{ z  \vbar |z| \leq e^{2S}\}$ in $\bC$ to the half-cylinder $(-\infty,2S] \times S^1 \subset Z$ along the identification of the common  closed subsets 
\begin{align} \label{eq:glue_plane_to_cylinder}
  [S,2S] \times S^{1} & \to \{ z \vbar e^{S} \leq |z| \leq e^{2S}\} \\
(s,t) & \mapsto e^{3S -s - 2 \pi i t}.
\end{align}
See Figure  \ref{fig:gluing_plane}. Even though the Riemann surface obtained by this gluing is naturally bi-holomorphic to the plane, we pedantically write $Z \#_{S} \bC  $ for it.

Note that, in the situation at hand, the plane $\bC$ carries an operator $D_{\Psi_+}$, while $Z$ carries the operator $D_{B}$. Whenever $S$ is large enough, the restrictions of the inhomogeneous term $B$ to the two sides of Equation  \eqref{eq:glue_plane_to_cylinder} agree with $B_{+,t}$.  We therefore obtain an operator on  $ Z \#_{S} \bC   $, denoted $ D_{B} \#_{S} D_{\Psi_+} $, which we refer to as the \emph{glued operator}.

The properties of $ D_{B} \#_{S} D_{\Psi_+}  $ as a Fredholm operator can be reduced to those of $D_{B}$ and $D_{\Psi_+}$ as follows: choose a partition of unity on the $ Z \#_{S} \bC$, consisting of two functions, respectively supported away from the disc of radius $e^{S}$ in $\bC$ (the red region in Figure  \ref{fig:gluing_plane}) and away from the half cylinder $(-\infty,S] \times S^1  $ in $Z$ (the blue region in Figure  \ref{fig:gluing_plane}).  By multiplying a function on $Z \#_{S} \bC  $ by the two elements of this partition, we obtain functions on $\bC$ and $Z$; this yields the splitting map
\begin{equation}
 W^{1,p}(Z \#_{S} \bC, \bR^{n}) \to W^{1,p}(Z, \bR^{n}) \oplus W^{1,p}(\bC, \bR^{n}).
\end{equation}
In the other direction, there is a gluing map
\begin{equation}
  L^{p}(Z, \bR^{n}) \oplus L^{p}(\bC, \bR^{n}) \to L^{p}(Z \#_{S} \bC, \bR^{n})
\end{equation}
as follows: given two functions on $\bC$ and $Z$, we  first multiply them by functions which respectively vanish away from the disc of radius $e^{2S}$ and the cylinder $ (-\infty,2S] \times S^1   $ to obtain functions on these domains which vanish on the boundary. These domains are naturally included in $ Z \#_{S} \bC$, so we can obtain a function on $Z \#_{S} \bC  $ by taking the sum of the extensions  by $0$.

Whenever $D_{B}$ and $D_{\Psi_+}$ are both surjective, one can use gluing and splitting to show that the kernel of $ D_{B} \#_{S} D_{\Psi_+}  $ is up to homotopy canonically isomorphic to the direct sum of the kernels of $D_{B}$ and $D_{\Psi_+}$. More generally, we \emph{stabilise} the problem by choosing finite dimensional vector spaces $V_{B}$ and $V_{\Psi_+}$ which surject onto the respective cokernels; we obtain surjective operators
\begin{align}
\tilde{D}_{B} \co    W^{1,p}(Z, \bR^{n}) \oplus V_{B} & \to L^{p}(Z, \bR^{n})\\
\tilde{D}_{\Psi_+} \co W^{1,p}(\bC, \bR^{n}) \oplus V_ {\Psi_+}& \to L^{p}(\bC, \bR^{n}) .
\end{align}
Using the composition
\begin{equation}
  V_{B} \oplus V_ {\Psi_+}  \to L^{p}(Z, \bR^{n}) \oplus L^{p}(\bC, \bR^{n}) \to L^{p}(Z \#_{S} \bC, \bR^{n}),
\end{equation}
where the second map is gluing, we also obtain a stabilised operator on the glued surface:
\begin{equation}
  \tilde{D}_{B}  \#_{S} \tilde{D}_{\Psi_+ } \co    W^{1,p}(Z \#_{S} \bC , \bR^{n}) \oplus V_{B} \oplus V_ {\Psi_+} \to L^{p}(Z \#_{S} \bC, \bR^{n}).
\end{equation}

At the level of graded lines, elementary linear algebra shows that we have, up to multiplication by a positive scalar, canonical isomorphisms
\begin{align}
\det(D_{B}) \otimes \det(V_B) & \cong   \det( \tilde{D}_{B} )  \\
\det(D_{\Psi_+}) \otimes \det(V_{\Psi_+})   & \cong   \det( \tilde{D}_{\Psi_+})\\
\det(  D_{B}  \#_{S} D_{\Psi_+ }  ) \otimes \det(V_B \oplus V_{\Psi_+})  & \cong \det(  \tilde{D}_{B}  \#_{S} \tilde{D}_{\Psi_+ }  )
\end{align}
In particular, an isomorphism
\begin{equation} \label{eq:iso_stablised_glued}
  \det(\tilde{D}_{B}) \otimes \det(\tilde{D}_{\Psi_+})  \cong \det(  \tilde{D}_{B}  \#_{S} \tilde{D}_{\Psi_+ }  )
\end{equation}
induces an isomorphism of the original determinant lines, because the determinant lines of $ V_B $ and $V_{\Psi_+}$ appear once on each side. This will allow us to prove the following key result (in the proof, we use the conventions outlined in Section \ref{sec:aside-orient-lines}):

\begin{lem}[Proposition 9 of \cite{FH}] \label{lem:isomorphism_glued_det_lines}
  There exists an isomorphism of graded lines
  \begin{equation} \label{eq:iso_glued_det_0}
    \det( D_{B} \#_{S} D_{\Psi_+}   )  \cong \det(D_{B}) \otimes \det(D_{\Psi_{+}})
  \end{equation}
which is canonical up to multiplication by a positive real number.
\end{lem} 
\begin{proof}[Sketch of proof]
As discussed above, it suffices to construct instead the isomorphism in Equation \eqref{eq:iso_stablised_glued}. Choose right inverses
\begin{align}
Q_{B} \co    L^{p}(Z, \bR^{n}) & \to W^{1,p}(Z, \bR^{n}) \oplus V_{B}\\
Q_{\Psi_+} \co L^{p}(\bC, \bR^{n})  & \to W^{1,p}(\bC, \bR^{n}) \oplus V_{\Psi_+}.
\end{align}
 The main estimate required is that the composition
\begin{equation}
  \xymatrix{ L^{p}(Z \#_{S} \bC, \bR^{n})  \ar[d]  & W^{1,p}(Z \#_{S} \bC, \bR^{n}) \oplus V_B \oplus V_ {\Psi_+}  \\
 L^{p}(Z, \bR^{n}) \oplus L^{p}(\bC, \bR^{n})  \ar[r]^-{Q_{B}  \oplus Q_{\Psi_+}  } &   W^{1,p}(Z, \bR^{n}) \oplus V_B \oplus W^{1,p}(\bC, \bR^{n}) \oplus V_ {\Psi_+}  \ar[u]
}
\end{equation}
 which we denote $ Q_{B} \#_{S} Q_{\Psi_+ }  $, is an approximate right inverse to $ \tilde{D}_{B} \#_{S} \tilde{D}_{\Psi_+ } $, in the sense that
\begin{equation}
\lim_{S \to +\infty}  \| \id -   \tilde{D}_{B} \#_{S} \tilde{D}_{\Psi_+ }   \circ  Q_{B} \#_{S} Q_{\Psi_+ }  \|  = 0.
\end{equation}
Proving this requires a careful choice of the cutoff function, compatible with the  exponential decay of  elements of the kernels and cokernels of $D_{B}$ and $D_{\Psi_+}$, see, e.g \cite{salamon-notes}*{Lemma 2.11}.  Using the expansion
\begin{equation}
  \frac{1}{1-x} = \sum  x^{i}
\end{equation}
we obtain a unique right inverse to $ \tilde{D}_{B} \#_{S} \tilde{D}_{\Psi_+ }$, denoted $Q_{B \#_{S} \Psi_+ }   $  with the same image as $ Q_{B} \#_{S} Q_{\Psi_+}  $. In particular, we conclude that $\tilde{D}_{B} \#_{S} \tilde{D}_{\Psi_+ }  $  is surjective, and hence that
\begin{equation}
  \det(\tilde{D}_{B} \#_{S} \tilde{D}_{\Psi_+ }  )  \cong \det(\ker( \tilde{D}_{B} \#_{S} \tilde{D}_{\Psi_+ }   )  ).
\end{equation}

It remains therefore to construct an isomorphism
\begin{equation}
   \ker(\tilde{D}_{B}) \oplus \ker(\tilde{D}_{\Psi_+}) \cong \ker(\tilde{D}_{B} \#_{S} \tilde{D}_{\Psi_+ }),
\end{equation}
which will immediately imply Equation \eqref{eq:iso_glued_det_0} by the discussion of  Section \ref{lem:iso_inverse_signs}.

We start with the projection 
\begin{equation}
 \id -   Q_{B \#_{S} \Psi_+ } \circ \left( \tilde{D}_{B} \#_{S} \tilde{D}_{\Psi_+ }\right)   \co W^{1,p}(Z \#_{S} \bC, \bR^{n}) \to   \ker(\tilde{D}_{B} \#_{S} \tilde{D}_{\Psi_+ }),
\end{equation}
which allows us to define the map
\begin{equation} \label{eq:isomorphism_gluing_kernels}
\xymatrix{  \ker(\tilde{D}_{B}) \oplus \ker(\tilde{D}_{\Psi_+}) \ar[d] & \ker(\tilde{D}_{B} \#_{S} \tilde{D}_{\Psi_+ }) \\
W^{1,p}(Z, \bR^{n}) \oplus V_{B} \oplus W^{1,p}(\bC, \bR^{n}) \oplus V_{\Psi_+} \ar[r] & W^{1,p}(Z \#_{S} \bC, \bR^{n})  \oplus V_{B} \oplus V_{\Psi_+} \ar[u]}
\end{equation}
The injectivity of this map is easy to see from the construction, since the norm of $  Q_{B} \#_{S} Q_{\Psi_+ } $ will be extremely small on functions on $ Z \#_{S} \bC $ obtained by gluing elements of the kernel of $\tilde{D}_{B}$ and $\tilde{D}_{\Psi_+}  $. 

To prove surjectivity, one notes that the restriction of $D_{B} \#_{S} D_{\Psi_+ }$ to the cylinder $[S,2S] \times S^1 \subset  Z \#_{S} \bC$ (the grey region in Figure \ref{fig:gluing_plane}) is  $s$-independent. In this setting, there is an exponential decay property for solutions to the Cauchy-Riemann equation (see \cite{MS}*{Lemma 4.7.3} for a non-linear analogue), which shows that the $C^0$ norm in this region is bounded exponentially in $S$ by the $C^0$ norms in the red and blue regions. For $S$ sufficiently large, this implies that elements of the kernel of $D_{B} \#_{S} D_{\Psi_+ }  $ are extremely close to elements obtained by gluing; finite dimensionality implies that this is only possible if the map is indeed a surjection.
\end{proof}

The assertion that there is an isomorphism of lines encapsulates the statements that (1) an orientation of $    \det( D_{B} \#_{S} D_{\Psi_+ }   )   $ is induced by orientations of $  \det(D_{B})$ and $ \det(D_{\Psi_{+}}) $, and (2) that the Fredholm index of the glued operator is given by the sum of the two Fredholm indices.

The next step is to observe that $D_{\Psi_-}$ agrees, away from a compact set, with the operator $ D_{B} \#_{S} D_{\Psi_+}$. Here, we use the natural idenfication $ Z \#_{S} \bC  \cong \bC$. We may therefore choose a family of operators, parametrised by the interval, and which are constant away from a compact set, interpolating between $D_{\Psi_-}$  and $ \det( D_{B} \#_{S} D_{\Psi_+ } ) $. The invariance of the Fredholm index for such families shows that there is an induced isomorphism between the determinant lines 
\begin{equation} \label{eq:iso_det_+_to_gluing}
  D_{\Psi_-} \cong \det( D_{B} \#_{S} D_{\Psi_+ } ). 
\end{equation}
This isomorphism depends, a priori, on the chosen path. However, the space of choices is contractible, so the induced isomorphism is in fact canonical up to multiplication by a positive constant.

We can now combine the isomorphims in Equations \eqref{eq:iso_glued_det_0} and \eqref{eq:iso_det_+_to_gluing} to obtain an isomorphism where the glued operator does not appear. It is however useful to consider a slightly more general setting, where we relax the conditions on the inhomogeneous term. We say that a Cauchy-Riemann operator has \emph{asymptotic conditions} which agree with $D_{\Psi_{\pm}}$ if it is given by Equation \eqref{eq:operator_CR-linearised}, and 
\begin{align}
  \lim_{s \to + \infty} |B_{s,t} - B_{+,t}| & = 0 \\
 \lim_{s \to - \infty} |B_{s,t} - B_{-,t}| & = 0.
\end{align}

\begin{prop} \label{lem:gluing_det_bundles_linear}
If $D_{B}$ is a Cauchy-Riemann operator with asymptotic conditions $ D_{\Psi_{\pm}}$, then, up to multiplication by a positive scalar, there is a canonical isomorphism
  \begin{equation} 
  D_{\Psi_-} \cong \det( D_{B} \#_{S} D_{\Psi_+} )
  \end{equation}
induced by gluing.
\end{prop}
\begin{proof}[Sketch of proof]
Consider another operator  $B'$ with the same asymptotic conditions and which agrees with $B_{\pm,t}$ near $s = \pm \infty$. We have a canonical isomorphism
  \begin{equation}
    \det(B) \cong \det(B')
  \end{equation}
induced by any path of operators with the same asymptotic conditions, which is independent of the choice of path because the space of such choices is contractible. We obtain the desired result by combining this isomorphism with the result of applying Equations \eqref{eq:iso_glued_det_0} and \eqref{eq:iso_det_+_to_gluing} to $B'$. 
\end{proof}

\subsection{Inverse paths and dual lines} \label{sec:inverse-paths-dual}

Our construction of determinant lines is based on the choice of a negative cylindrical end on the plane. One could build the entire theory, instead, by considering the plane equipped with the positive cylindrical end
\begin{align}
(-\infty,+\infty) \times S^1 \to \bC \\
(s,t) & \mapsto e^{s+2 \pi it}.
\end{align}
Given a path $A$ in $\fsp_{2n}$ satisfying Equation \eqref{eq:derivative_agree_at_0_and_1}, which exponentiates to a path of matrices $\Psi$, and a map
\begin{equation} 
B \in C^{\infty}(\bC , \bR^{2n \times 2n})
\end{equation}
such that 
\begin{equation} \label{eq:condition_inhomogeneous_positive}
  B(e^{s+2 \pi it}) = \frac{ d A_t}{dt} \textrm{ if } 0 \ll s,
\end{equation}
we obtain an operator
\begin{equation}
  D_{\Psi}^{+} \co W^{1,p}(\bC, \bC^{n})  \to   L^{p}(\bC, \bC^{n})
\end{equation}
given by Equation \eqref{eq:operator_orbits-formula}. The only difference between $  D_{\Psi}^{+}  $ and $D_{\Psi}$ is the negative sign in Equation \eqref{eq:condition_inhomogeneous_negative} which is lacking in Equation \eqref{eq:condition_inhomogeneous_positive}. The class of operators that we obtain via this construction is exactly the same as before.
\begin{exercise}
Let $\Psi^{-1}$ denote the path
\begin{equation}
\Psi^{-1} =   \exp(A_{1-t} - A_{1}).
\end{equation}
Show that, for the appropriate choice  in Equations \eqref{eq:condition_inhomogeneous_negative}  and  \eqref{eq:condition_inhomogeneous_positive}, 
\begin{equation}
  D_{\Psi}^{+} = D_{\Psi^{-1}}.
\end{equation}
\end{exercise}

The main reason for introducing these operators is that the proof of the following result is transparent:

\begin{lem}
  Up to multiplication by a positive real number, there is a canonical isomorphism
  \begin{equation} \label{eq:glue_operator_negative_positive_end_Cn}
   \det(D_{\Psi}^{+}) \otimes  \det(D_{\Psi}) \cong \det_{\bR}(\bC^{n}). 
  \end{equation}
\end{lem}
\begin{proof}
By construction, the restrictions of the operators $ D_{\Psi}^{+}$ and $ D_{\Psi} $ to the cylindrical ends of the respective copies of $\bC$ agree; we can therefore glue these two copies of $\bC$ along the ends to obtain an operator:
\begin{equation}
D_{\Psi}^{+} \#  D_{\Psi} \co W^{1,p}(\bC \bP^{1} , \bC^{n})  \to   L^{p}(\bC \bP^{1}, \bC^{n}).
\end{equation}
Deforming this operator to the standard Cauchy-Riemann operator on $\bC \bP^{1} $, Equation \eqref{eq:glue_operator_negative_positive_end_Cn} follows from the fact that the only holomorphic functions on $\bC \bP^{1}$ are constant.
\end{proof}

\subsection{Change of trivialisations and gluing} \label{sec:change-triv-gluing}
In Section \ref{sec:invar-paths-sympl}, we defined the index of a path $\Psi$ of symplectic matrices. We shall now describe the behaviour of this invariant under a change of trivialisation: let $\Phi \co S^1 \to U(n)$ be a loop in the unitary group of $n \times n$ complex matrices. By the usual clutching construction, we can associate to $\Phi$ a unitary bundle $E_{\Phi}$ over $\bC \bP^{1}$: this bundle is obtained by gluing the trivial bundle $(z,e)$ on the unit disc (centered at $z_{0}$) to the trivial bundle on its complement (centered at $z_{\infty}$) via the map
\begin{equation}
  (e^{i \theta},e) \mapsto (e^{i \theta}, \Phi_{\theta}^{-1}(e)).
\end{equation}
By choosing a connection, we obtain an operator:
\begin{equation}
  D^{-}_{\Phi} \co W^{1,p}(\bC \bP^{1} , E_{\Phi}) \to L^{p}(\bC \bP^{1}, E_{\Phi} \otimes \Omega^{0,1} \bC \bP^{1}).
\end{equation}
Unlike the operator in Equation \eqref{eq:operator_orbits}, this  is a complex linear operator. In particular, both the kernel and cokernel are complex vector spaces, hence have natural orientations as real vector spaces. We define $\det_{\bR}(D^{-}_{\Phi})$ to be the determinant line of $D^{-}_{\Phi}$ over the real numbers. The degree of this line is equal to the real index of $D^{-}_{\Phi} $:
\begin{equation}
  \ind_{\bR}(D^{-}_{\Phi}) = \dim_{\bR}(\ker(D^{-}_{\Phi})) - \dim_{\bR}(\coker(D^{-}_{\Phi})).
\end{equation}
The following is a standard result (see, e.g. \cite{hatcher}*{Section 1.2}, in particular Example 1.10):
\begin{lem} \label{lem:pi_1_U(n)}
The map  
\begin{align}
\rho \co \pi_{1}(U(n)) & \to \bZ  \\
  \rho(\Phi) & =  \frac{\ind_{\bR}(D^{-}_{\Phi})}{2} -n  
\end{align}
is an isomorphism which agrees with the homomorphism
  \begin{equation}
    \Phi \mapsto \deg(\det \circ \Phi)
  \end{equation}
where $\det \co U(n) \to S^1$  is the determinant homomorphism. \qed
\end{lem}

We can now describe the behaviour of the index under change of trivialisations: If $\Phi$ is a loop in $U(n)$, based at the identity, and $\Psi$ a path of symplectomorphisms, we  denote by $ \Psi \circ \Phi $ the path of symplectomorphisms $\Psi_{t}  \circ \Phi_{t} $:
\begin{prop} \label{lem:orient_line_indep_path}
Gluing induces an isomorphism
\begin{equation} \label{eq:change_determinant_trivialisation}
  \det(D_{\Psi \circ \Phi })  \otimes \det_{\bR}(\bC^{n})  \cong   \det_{\bR}(D^{-}_{\Phi^{-1}}) \otimes  \det(D_{\Psi}).
\end{equation}
In particular, using the complex orientations of $\det_{\bR}(\bC^{n})   $ and $  \det_{\bR}(D^{-}_{\Phi^{-1}}) $, we obtain an induced isomorphism:
\begin{equation}
    \det(D_{\Psi \circ \Phi })  \cong  \det(D_{\Psi}).
\end{equation}
\end{prop}
\begin{proof}[Sketch of proof:]
Consider the case where the operators $ D_{\Psi} $ on $\bC$ and $ D^{-}_{\Phi^{-1}} $ on $\bC \bP^1$ are both surjective, and the evaluation map at $0 \in \bC $ and $z_{0}$ defines a surjective map:
\begin{equation} \label{eq:evaluation_map_point_sphere_plane}
  \ker(D^{-}_{\Psi}) \oplus \ker(D^{-}_{\Phi^{-1}}) \to \bC^{n}.
\end{equation}
The general case can be recovered by the same method using stabilisation, i.e. by adding a finite rank vector space which surjects onto the cokernel.

Fix a holomorphic identification  
\begin{equation}
 \bC \bP^1  \setminus \{z_{0},z_{\infty}\} \cong \bR \times S^{1},
\end{equation}
with the negative end converging to $0$. For each positive real number, we obtain a new Riemann surface $ \bC \bP^1  \#_{S} \bC$ by gluing the complements of  the (open) disc of radius $e^{-S}$ in $\bC$, and the image of the cylinder $(-\infty, -S) \times S^{1}$ in $ \bC \bP^1 $; this is the same construction as in Section \ref{sec:gluing-oper-determ}. This surface is equipped with a natural biholomorphism to $\bC$ mapping $z_{\infty}$ to $0$. Moreover, we can glue the operators $ D^{-}_{\Phi^{-1}} $ and $ D_{\Psi}  $ to obtain an operator $D^{-}_{\Phi^{-1}} \#_{S} D_{\Psi}   $ on sections of a vector bundle $E_{\Phi^{-1}} \#_{S} \bC^{n}$ over $\bC$; the key result of gluing theory is that if $S$ is large enough this is a surjective operator, whose kernel can be identified with the kernel of the operator in Equation \eqref{eq:evaluation_map_point_sphere_plane}. The choice of identification is unique up to contractible choice. In particular, up to a positive real number, we have a canonical isomorphism
\begin{equation}
  \det(D^{-}_{\Phi^{-1}} \#_{S} D_{\Psi}  )  \otimes \det_{\bR} (\bC^{n}) \cong   \det_{\bR}(D^{-}_{\Phi^{-1}}) \otimes  \det(D_{\Psi}),
\end{equation}
via the convention for orienting the middle term of a short exact sequence (see Section \ref{lem:iso_inverse_signs}).
\begin{figure}[h] \label{fig:change_trivialise}
  \centering
\includegraphics[scale=1.33]{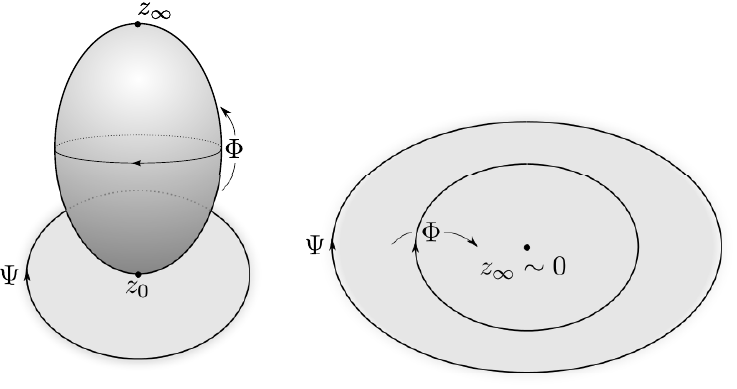}
  \caption{ }
\end{figure}

Comparing with Equation \eqref{eq:change_determinant_trivialisation}, we see that it remains to construct an isomorphism
\begin{equation} \label{eq:isomorphism_gluing_operator_composition}
    \det(D^{-}_{\Phi^{-1}} \#_{S} D_{\Psi} ) \cong   \det(D_{\Psi \circ \Phi}).
\end{equation}

To prove this, we note that the trivialisation of $E_{\Phi^{-1}}$ at $z_{\infty}$ induces a trivialisation of  $E_{\Phi^{-1} } \#_{S} \bC^{n}$  at $0$. We extend the trivialisation to $\bC$; such a choice is unique up to homotopy  since $\bC$ is contractible. In this trivialisation, the asymptotic conditions at infinity are given by $\Psi \circ \Phi$ (see Figure \ref{fig:change_trivialise} where we keep track of the orientation of the loops as well as the direction of the gluing). Equation \eqref{eq:isomorphism_gluing_operator_composition} follows from the fact that the determinant depends only on the asymptotic data at infinity.
\end{proof}
\begin{exercise}
By computing the degree in which both sides in Equation \eqref{eq:change_determinant_trivialisation} are supported, show that
\begin{equation} \label{eq:degree_change_trivialisation}
  \ind(\Psi \circ \Phi) = \ind( \Psi) - 2 \rho(\Phi).
\end{equation}
\end{exercise}

\subsection{The index of Hamiltonian orbits} \label{sec:conley-zehnder-index}

All methods for defining gradings in Floer theory rely on assigning to each orbit $x$, a homotopy class of trivialisations of $x^{*}(\TQ)$, and considering the Conley-Zehnder index of an associated path of matrices. In order for such a construction to make sense, one must be able to relate the trivialisations assigned to different orbits: this can be done for contractible orbits by choosing the trivialisation coming from \emph{capping discs}, or, whenever $\Q$ is orientable, by using the trivialisation of $\Lambda^{n}_{\bC} T \TQ$ induced from the choice of a volume form in the base.  In this section, we explain how to obtain such gradings without assuming that $\Q$ is orientable.

The following is the main result about vector bundles that we shall use:
\begin{lem} \label{lem:trivialisation_det}
  If $E_{\bC}$ is a complex vector bundle over the circle, there is a bijective correspondence between the homotopy classes of trivialisations of $E_{\bC}$ and those of $\det(E_{\bC})$.
\end{lem}
\begin{proof}
First, we recall why every complex vector bundle over $S^1$ is trivial: decomposing the circle as the union of two intervals, we can trivialise the restriction of every bundle to the two sides. A bundle over the circle is then determined by a choice of clutching function at the endpoints: i.e. a pair of matrices in $GL(n,\bC)$ which can be used to glue the bundles on the two intervals. Since $ GL(n,\bC) $ is connected, this choice is unique up to homotopy.

Next, we claim that any trivialisation of $\det(E)$ induces a trivialisation of $E$. Note that any two trivialisations of a bundle over $S^1$ differ by a map from $S^1$ to $GL(n,\bC)$, i.e. an element of $\sL GL(n,\bC)$, and that two trivialisations are homotopic if and only if the corresponding maps lie in the same component. Since $GL(n,\bC)$ is a group, its free loop space splits as a product $GL(n,\bC) \times \Omega GL(n,\bC)$, where $\Omega GL(n,\bC)$ is the set of loops based at the identity. Since $GL(n,\bC) $ is connected, there is a canonical identification between the components of $\sL GL(n,\bC)$ and those of $\Omega GL(n,\bC)$. The connected components of $\Omega GL(n,\bC)$ correspond to elements of $\pi_{1}(GL(n,\bC))$, which is a free abelian group of rank $1$. Since, the determinant map
\begin{equation}
    \det \co GL(n,\bC) \to \bC^{*}
\end{equation}
induces an isomorphism on fundamental groups, we conclude that the map which assigns to a trivialisation of $E$ the corresponding trivialisation of $\det(E)$ is a bijection on homotopy classes.
\end{proof}
Assume that $E_{\bC}$ is the complexification of a real bundle $E_{\bR}$ over $S^1$. In this case, $\det(E_{\bC})$ is the complexification of the bundle $\det(E_{\bR})$. In particular, any trivialisation of $E_{\bC}$ induces a map
\begin{equation} \label{eq:classifying_map_grassmannian}
S^{1} \to \bR \bP^{1}
\end{equation}
which assigns to a point in $S^1$ the image of the real line $\det(E_{\bR}) $ in $\bC$. We fix the standard  orientation of $\bR \bP^{1}$, for which the positive direction corresponds to moving a line through the origin counter-clockwise.  We call this map the \emph{Gauss map}.
\begin{rem}
The degree of the Gauss map is an incarnation of the \emph{Maslov index for loops}. We shall discuss the Maslov index further in Section \ref{sec:maslov-index-loops}.
\end{rem}
\begin{lem} \label{lem:fixed_trivialisation}
Up to homotopy, there is a unique trivialisation of $E_{\bC}$ such that the Gauss map has degree $0$ if $E_{\bR}$ is orientable or $1$ if $E_{\bR}$ is not orientable.
\end{lem}
\begin{proof}[Sketch of proof:]
By Lemma \ref{lem:trivialisation_det}, it suffices to show the existence of unique trivialisations of $\det(E_{\bC})$ with these properties. In particular, we should understand how the map in Equation \eqref{eq:classifying_map_grassmannian} behaves under a change of trivialisation.  The key point is that the natural map $U(1) \to \bR \bP^{1}$ has degree $2$: this map  assigns to a unitary transformation $\phi$ the real line $\phi(\bR)$. If we use the degree to identify the homotopy classes of maps from $S^1$ to $\bR \bP^{1}$ with $\bZ$, we see that the group of trivialisations (which is also $\bZ$) acts by adding an element of $2 \bZ$. In particular, assuming that $ E_{\bR} $ is orientable, there is a unique choice of trivialisation such that the degree is $0$, while  if $ E_{\bR} $ is not orientable, there is a unique choice of trivialisation such that the degree is $1$.
\end{proof}
\begin{rem}
In the orientable case, one can be slightly more explicit in the construction: choose a trivialisation of $\det(E_{\bR})$, and consider the induced trivialisation of $ \det(E_{\bC})$. By Lemma  \ref{lem:trivialisation_det}, there is a unique trivialisation of $ E_{\bC} $, up to homotopy, which is compatible with this. Note that the choice of orientation does not change the homotopy class of trivialisation of $  \det(E_{\bC})$, because the two trivialisations associated to opposite orientations differ by a constant element $-1 \in U(1)$; a choice of homotopy from $-1$ to $1$ gives a homotopy between the trivialisations.
\end{rem}

\begin{exercise} \label{ex:complexify_real_bundle_symplectic}
If $E$ is a real vector bundle over a space $X$, and $E^{\vee}$ the dual bundle, show that $E \oplus E^{\vee}$ admits a canonical symplectic structure. Using the fact that the space of metrics on $E$ is contractible, conclude that there is a symplectic structure on $E_{\bC} = E \oplus \sqrt{-1} E$, which is canonical up to contractible choice and such that the complex structure on $E_{\bC}$ is compatible with it.  Show that the natural symplectic structure on $T \TQ  $ is isomorphic  to $q^{*}( T^{*} Q )_{\bC} $.
\end{exercise}

Given a loop $x$ in $\TQ$, consider the pullback
\begin{equation}
 x^{*} \left( T \TQ  \right)
\end{equation}
which is a symplectic vector bundle over the circle.  Choosing a family $J_{t}$ of almost complex structures on $\TQ$ which are compatible with the natural symplectic structure, this becomes a unitary bundle.  The sub-bundle which consists of tangent vectors to $T^{*}_{q\circ x(t)}Q$ is Lagrangian and naturally isomorphic to $ (q \circ x)^{*} T\Q $, so Exercise \ref{ex:complexify_real_bundle_symplectic} implies that we can identify this as a unitary bundle with
\begin{equation}
   (q \circ x)^{*}\left( T^{*} \Q \right)_{\bC}.
\end{equation}
Applying Lemma \ref{lem:fixed_trivialisation}, we conclude:
\begin{lem} \label{lem:trivalisation_up_to_htpy}
If $J_{t}$ is a family of almost complex structures on $\TQ$, the vector bundle $  x^{*} \left(T \TQ \right)$ admits a  trivialisation as a unitary vector bundle which is canonical up to homotopy. \qed
\end{lem}
We now extend this result to cylinders:
\begin{exercise} \label{ex:extend_trivialisation_cylinder}
  Given an annulus $u \co S^{1} \times [0,1] \to \TQ$, show that there is a trivialisation of $ u^{*}\left(T \TQ\right) $, unique up to homotopy, whose restriction to $ u^{*}\left(T \TQ \right)| S^{1} \times \{t\}$ for $t \in [0,1]$  is the one provided by Lemma \ref{lem:trivalisation_up_to_htpy}. 
\end{exercise}

Let us now assume that $x$ is a Hamiltonian orbit. By definition, we have a path $\phi^{t}$ of Hamiltonian symplectomorphisms such that
\begin{equation}
 x(t) = \phi^{t}(x(0))
\end{equation}
The differential of $\phi^{t}$, defines a linear symplectic map
\begin{equation} 
  \phi^{t}_{*} \co T_{x(0)} \TQ  \to T_{x(t)} \TQ .
\end{equation}
On the other hand, Lemma \ref{lem:trivalisation_up_to_htpy} provides a trivialisation of the corresponding unitary vector bundle over the circle which identifies $J_{t}$ and $\omega)$ with the standard complex and symplectic structures on $\bC^{n}$. In particular, we obtain a family   $\Psi_{x}(t) = \exp(A_{x}(t))$ of symplectomorphisms of $\bC^{n}$ as the composition
\begin{equation} \label{eq:path_matrices_global}
\xymatrix{\bC^{n} \ar[r]^-{\cong} & T_{x(0)} \TQ  \ar[r]^-{ \phi^{t}_{*}} & T_{x(t)} \TQ  \ar[r]^-{\cong} &\bC^{n}.}
\end{equation}
We can now assign an integer to each orbit:

\begin{defin} \label{def:determinant_line-CZ}
 The \emph{cohomological Conley-Zehnder index} of a non-degenerate orbit $x$ is 
\begin{equation} \label{eq:definition_degree_not_shifted}
  |x| = \ind\left(\Psi_x \right).
\end{equation}
The \emph{determinant line} is the $1$-dimensional  $\bZ$-graded real vector space
\begin{equation} \label{eq:determinant_line_orbit}
\delta_{x}  \equiv \det(D_{\Psi_{x}}).
\end{equation}
The \emph{orientation line} $  \ro_{x} $ is the $\bZ$-graded abelian group with two generators corresponding to the two orientations of $\delta_{x}$, and the relation that the sum vanishes.
\end{defin}
The Floer complex we shall study will have the property that the degree of certain orbits is shifted by $1$.  To this end, we introduce the integer $w(x)$ given by
\begin{equation} \label{eq:first_sw_class_pullback}
w(x) \equiv \begin{cases}  0 & \textrm{ if } x^{*}(T\Q) \textrm{ is orientable} \\
-1 & \textrm{ otherwise.}
\end{cases}  
\end{equation}
\begin{defin}
  The \emph{cohomological degree} of an orbit is
\begin{equation} \label{eq:definition_degree_shifted_1}
  \deg(x) = |x| - w(x).
\end{equation}
\end{defin}

\begin{rem} \label{eq:apology_to_the_reader}
Choosing a different trivialisation would change the index by an even number. It is therefore impossible to choose a different trivialisation in the non-orientable case to eliminate the constant term $1$ in the definition of the cohomological degree. We give two justifications for this term, in hope that the reader will find one of them agreeable:
\begin{itemize}
\item In Section \ref{sec:pair-pants-product}, we shall define a product on Floer cohomology; in fact, we shall define a product on the Floer chain complex. The shift is required in order for this product to be homogeneous of degree $0$; if one does not shift, the product of two generators corresponding to loops along which $T\Q$ is non-orientable does not have the correct degree.
\item In Section \ref{cha:from-sympl-homol}, we shall construct a map from symplectic cohomology to the homology of the free loop space using moduli spaces of holomorphic discs. This map does not preserve the grading unless we introduce the above shift.

\end{itemize}
\end{rem}

\begin{rem}
The original construction of Floer cohomology with $\bZ$ coefficients \cite{FH}, relied on fixing orientations for the determinant line $\det(D_{\Psi})$ for any path $\Psi$ of symplectomorphisms which induces a choice of generator of $\ro_{x}$ for every orbit. This is called \emph{a coherent choice} of orientations. We find the more abstract point of view presented here slightly more convenient for defining operations in Floer theory.
\end{rem}

It will sometimes be convenient to consider the variant of the determinant and orientation lines obtained by taking an operator on a plane equipped with a positive end as in Section \ref{sec:inverse-paths-dual}. Given an orbit $x$, we introduce the notation:
\begin{align}
  \delta_{x}^{+} & \equiv \det(D_{\Psi_x}^{+}) \\
  \ro_{x}^{+}  & \equiv  |   \delta_{x}^{+} |.
\end{align}
\begin{exercise}
Show that there is a canonical isomorphism
\begin{equation} \label{eq:dual_orientation_lines_orbit}
  \ro_{x}^{+} \otimes \ro_{x} \cong |\bC^{n}|.
\end{equation}
\end{exercise}

\section{Floer cohomology of linear Hamiltonians} \label{sec:floer-cohom-line}

Recall that a local system $\nu$ of rank $1$ on a space $X$ is the assignment of a free abelian group $\nu_{x}$ of rank $1$ for each point $x \in X$, together with a map 
\begin{equation} \label{eq:local_system_parallel}
  \nu_{\gamma} \co \nu_{\gamma(0)} \to \nu_{\gamma(1) }
\end{equation}
for each path $\gamma \co [0,1] \to X$, which only depends on the homotopy class of $\gamma$ relative its endpoints, and so that
\begin{equation}
  \nu_{\gamma^{0} \# \gamma^{1}} = \nu_{\gamma^{1}} \circ \nu_{\gamma^{0}}
\end{equation}
whenever  the initial point of $\gamma^{1}$ agrees with the final point of $\gamma^{0}$, and $\gamma^{0} \# \gamma^{1} $ is the concatenation. Moreover, we require that  $\nu_{x}$ is the identity if $x$ is the constant path at $x$. These conditions imply that every map $\nu_{\gamma}$ is an isomorphism with inverse provided by traversing $\gamma$ in the opposite direction.

Given a local system $\nu$ on the free loop space of $\TQ$, the goal of this section is to construct the \emph{Hamiltonian Floer cochain complex} of a linear Hamiltonian $H$ as the cohomology of a graded abelian group
\begin{equation}
CF^{i}(H ; \nu) \equiv \bigoplus_{\substack{x \in \Orbit(H) \\  \deg( x)=i }}  \ro_{x}[w(x)] \otimes \nu_{x}
\end{equation}
which in degree $i$, is given by the direct sum of the orientation lines of orbits of degree $i$. The notation $\ro_{x}[w(x)]$ indicates that we shift the degree of the graded line $  \ro_{x}$  by $w(x)$ (i.e. up by $1$ whenever $(q \circ x)^{*}(T\Q)$ is not orientable).  The differential is obtained by defining maps on orientation lines associated to rigid pseudo-holomorphic cylinders in $\TQ$.

In order to construct the differential in this version of the Floer complex, we return to the setting of Section \ref{sec:transversality}: $H$ is a linear Hamiltonian all of whose orbits are non-degenerate, and $J_t$ is a family of almost complex structures on $\TQ$ which are compatible with the symplectic form, and which are convex near $\SQ$.

\subsection{Orientations}
Let $x_0$ and $x_1$ be Hamiltonian orbits, and consider an element $u \in \Cyl(x_0;x_1)$.  By Exercise \ref{ex:extend_trivialisation_cylinder}, we have a canonical trivialisation of $u^{*}(T \TQ)$. Using this trivialisation, we can identify the linearisation $D_u$ of the Floer equation at $u$ with an operator
\begin{equation}
   W^{1,p}(Z, \bC^{n}) \to  L^{p}(Z, \bC^{n}).
\end{equation}

Since the trivialisation of $u^{*}(T \TQ)$ restricts along the ends to the trivialisations of $x_{i}^{*}(T \TQ)$ used to define the path $\Psi_{x_i}$, we are in the setting of Proposition  \ref{lem:gluing_det_bundles_linear}. Applying that result,  we conclude:
\begin{thm} \label{lem:cylinder_orientation_lines}
 If $u \in \Cyl(x_0;x_1)$ is a solution to Floer's equation, then there is an isomorphism of graded lines
 \begin{equation} \label{eq:map_orientation_lines_indep_triv}
   \det(D_{u}) \otimes  \delta_{x_0} \cong \delta_{x_1}
 \end{equation} 
which is canonical up to multiplication by a positive real number. In particular the virtual dimension of $u$ is
 \begin{equation} \label{eq:virtual_dimension_formula}
  \deg(x_0) - \deg(x_1) - 1.
 \end{equation}
\end{thm}

\begin{exercise} 
Use Exercise \ref{ex:translation_in_kernel} to produce an isomorphism
\begin{equation} \label{eq:det_D_u_orientation}
  |\det(D_u) | \cong   |\bR \cdot \partial_{s}| \otimes |\Cyl(x_0;x_1)|
\end{equation}
where $\bR \cdot \partial_{s}$ corresponds to the subspace of $\ker(D_u)$ spanned by translation.
\end{exercise}

Let us now consider the situation where the Floer data $(H,J_{t})$ are regular in the sense of Equation \eqref{eq:regular_J}. First, we consider the cases where the moduli space will necessarily be empty:

\begin{lem} \label{lem:negative_index_empty}
If $\deg(x_0) \leq \deg(x_1)$, and $x_1 \neq x_0 $, then the moduli space $ \Cyl(x_0;x_1) $ is empty.
\end{lem}
\begin{proof}
If $ x_1 \neq x_0  $, and $u$ is an element of $ \Cyl(x_0;x_1) $, then $\partial_{s}u$ defines a non-zero element of the kernel of $D_{u}$. Since it is surjective, the index of $D_{u}$ is therefore greater than or equal to $1$. Using Equation \eqref{eq:virtual_dimension_formula}, we conclude that $ \deg(x_0) > \deg(x_1)  $.
\end{proof}

Next, we assume that $\deg(x_0) = \deg(x_1) +1$, and refine the count of elements in $ \Cyl(x_0;x_1)  $ to maps on orientation lines:
\begin{lem} \label{lem:rigid_cylinder_orientation_lines}
 Every rigid cylinder $u \in \Cyl(x_0;x_1)$ determines a canonical isomorphism of orientation lines
 \begin{equation} \label{eq:map_orientation_lines}
   \partial_{u} \co \ro_{x_1} \to \ro_{x_0}.
 \end{equation}
 \end{lem}
 \begin{proof}
Recall that $ \ro_{x_i} $ is the orientation line of $\delta_{x_i}  $. Equation \eqref{eq:map_orientation_lines_indep_triv}  yields an isomorphism
\begin{equation} \label{eq:standard_iso_orientation_lines_cylinder}
    |\det(D_u) | \otimes \ro_{x_1}  \cong \ro_{x_0} .
\end{equation}
Using Equation \eqref{eq:det_D_u_orientation}, we obtain an isomorphism
\begin{equation}
  \ro_{x_1}  \cong |\Cyl(x_0;x_1)| \otimes | \bR \cdot \partial_{s} | \otimes  \ro_{x_0}.
\end{equation}
We fix the orientation of $  \bR \cdot \partial_{s} $ corresponding to the generator $\partial_{s}$. Moreover, since $u$ is rigid, $T_{u} \Cyl(x_0;x_1)$ has dimension $0$, and hence is canonically trivial. Using these two trivialisations, we produce the desired map.
 \end{proof}

\subsection{Gluing theory}
In verifying that the square of the differential on the Floer complex over $\bZ/2\bZ$ vanishes, the key step was to ensure that the Gromov-Floer compactification of the $1$-dimensional moduli space of trajectories is a compact manifold, whose boundary points correspond to the matrix coefficients of $\partial^{2}$.

In order to prove the analogous result over the integers, or more generally with twisted coefficients, we give a more careful description of this compactified moduli space. First, we observe that the  transversality result in Theorem \ref{lem:transversality} has strong consequences for the nature of the compactified moduli spaces of solutions to Floer's equation if the data are regular.  In particular, Lemma \ref{lem:negative_index_empty} implies that the only strata which contribute to the Gromov-Floer compactification in Equation \eqref{eq:boundary_moduli_space}, are those satisfying
\begin{equation}\label{eq:boundary_moduli_space-transverse}
\begin{aligned}
& \cA(x_0) > \cA(x'_0) > \cdots > \cA( x'_{d}) > \cA(x_1) \\
& \deg(x_0) > \deg(x'_0) > \cdots > \deg( x'_{d}) > \deg(x_1).
\end{aligned}
\end{equation}
In other words, the cohomological Conley-Zehnder index of the orbits, in addition to their action, must increase. 
\begin{exercise}
Using Theorem \ref{lem:transversality}, show that each space in the right hand side of Equation \eqref{eq:boundary_moduli_space} is a smooth manifold whose dimension is $\deg(x_0) - \deg(x_1) - d - 1$. 
\end{exercise}

We shall not prove the following standard result whose proof for an appropriate class of closed symplectic manifolds appears in \cite{AD}*{Theorem 9.2.1}:
\begin{thm} \label{thm:moduli_manifold_boundary}
If $\deg(x_0) = \deg(x_1) +2 $,  the compactified moduli space  $\Cylbar(x_0;x_1)$ is a $1$-dimensional manifold with boundary 
\begin{equation}
  \coprod_{\deg(x) = \deg(x_1) +1 } \Cyl(x_0;x) \times \Cyl(x;x_1).
\end{equation} \qed
\end{thm}

Two refinements of this result will be needed.  First, if $\beta$ is a homotopy class of cylinders connecting $x_0$ and $x_1$, the corresponding component $ \Cyl_{\beta}(x_0;x_1) $ of the moduli space compactifies to $\Cylbar_{\beta}(x_0;x_1)   $, whose boundary is
\begin{equation}
  \coprod_{\stackrel{\deg(x) = \deg(x_1) +1}{\beta_1 \# \beta_2 = \beta }} \Cyl_{\beta_1}(x_0;x) \times \Cyl_{\beta_2}(x;x_1).
\end{equation}
Here, $ \beta_1$ and $ \beta_2 $ are homotopy classes of cylinders, and $ \# $ stands for the operation of gluing along the common end, which in this case is $x$.  The compatibility of the Gromov compactification with the decomposition into homotopy classes follows immediately from the topology on $ \Cylbar(x_0;x_1) $.

The second refinement concerns tangent spaces. Assume that a cylinder $w \in \Cyl(x_0;x_1)$ lies sufficiently close to a pair of curves $(u,v) \in  \Cyl(x_0;x) \times \Cyl(x;x_1) $. In this case, a non-linear analogue of the gluing construction discussed in Section \ref{sec:gluing-oper-determ}, defines a map
\begin{equation}
  \ker(D_{u}) \oplus \ker(D_{v}) \to \ker(D_{w}).  
\end{equation}
Each of the factors in the left hand side is $1$-dimensional, and is generated by the vector fields $\partial_{s} u$ and $\partial_{s}v$, while the right hand side is two dimensional, and is the middle term of the exact sequence:
\begin{equation}
 \bR \cdot \partial_{s} w \to \ker(D_{w}) \to T_{w} \Cyl(x_0;x_1).
\end{equation}
Since $w$ is close to the boundary stratum $(u,v)$, it makes sense to say that an element of $ T_{w} \Cyl(x_0;x_1) $ either points towards the boundary, or away from it. 
\begin{lem} \label{sec:translation_vfield_in_and_out}
  The image of $\partial_{s} u$ in $  T_{w} \Cyl(x_0;x_1) $  points away from the boundary, while the image of $\partial_{s} v $ points towards the boundary.
\end{lem}
\begin{figure}[h]
  \centering
  \includegraphics[scale=1.33]{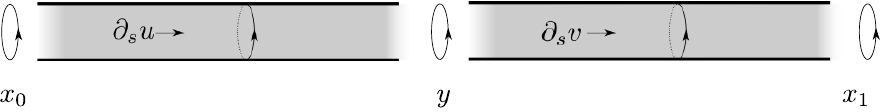} 
  \caption{ }\label{fig:cylinder_Floer-glue}
\end{figure}
\begin{proof}[Sketch of proof]
Figure \ref{fig:cylinder_Floer-glue} summarises the proof. If $w$ is close to $(u,v)$, then most of the energy $E(w)$ is supported in two annuli separated by a large cylinder where the $L^{2}$ integral of the energy approximately vanishes. A holomorphic vector field along $w$ will point towards the boundary if it integrates to a family of holomorphic maps for which the distance between the two regions where the energy is supported grows.  The image of $\partial_{s} u$ is a vector field which is close to $\partial_{s} w$ near $-\infty$, and approximately vanishes near $+\infty$, which implies that it pushes the two regions where the energy is supported closer together, hence points away from the boundary. The image of $\partial_{s} v$ is a vector field which is close to $\partial_{s} w$ near $+\infty$, and approximately vanishes near $-\infty$; it points outwards.
\end{proof}

\subsection{Floer cohomology}
The differential in Floer cohomology is defined as a sum of contributions of all holomorphic cylinders.  In the presence of a local system $\nu$, we think of every holomorphic cylinder $u$ as giving us a path in the free loop space of $\TQ$, and hence an isomorphism
\begin{equation}
  \nu_{u} \co \nu_{x_1} \to  \nu_{x_0}.
\end{equation}
On the other hand, the map constructed in  Lemma \ref{lem:rigid_cylinder_orientation_lines}, induces, after shift by $w(x_1)$, a map
\begin{equation}
  \partial_{u}  \co  \ro_{x_1}[w(x_1)]  \to  \ro_{x_0}[w(x_0)],
\end{equation}
which still has degree $1$ because $w(x_0) = w(x_1)$ whenever the moduli space $ \Cyl(x_0;x_1) $ is not empty.  We are working with cohomological conventions as a consequence of grading each orbit by $n - \CZ(x)$, and would have obtained a map that lowers degree by $1$ if we had chosen different conventions.
\begin{rem}
If we choose a generator $\alpha$ for $ \ro_{x_1} $, and write $S^{w(x_1)} \alpha$ for the corresponding generator of $  \ro_{x_1}[w(x_1)] $, then by definition
  \begin{equation} \label{eq:koszul_twist_differental}
    d S^{w(x_1)} \alpha \equiv (-1)^{ w(x_1) } S^{w(x_0)} d \alpha.
  \end{equation}
\end{rem}

We have the necessary ingredients to define the differential:
\begin{align} \label{eq:differential_SC}
  d \co CF^{i}( H ; \nu) & \to CF^{i+1}( H ; \nu) \\
d |  \ro_{x_1}[w(x_1)] \otimes \nu_{x_1}  & \equiv
 \sum_{x_{0}}  \sum_{u \in \Cyl(x_0;x_1)} d_{u} \otimes \nu_{u}.
\end{align}
\begin{rem}
In the original paper which considered orientations in Floer homology \cite{FH}, one chooses \emph{consistent orientations} on all determinant lines $ \det(D_{\Psi}) $, i.e. generators for $ |\det(D_{\Psi})| $. This gives a more concrete description of the signs in Floer theory than the one we shall use whereby one assigns $\pm 1$ to a curve depending on whether certain orientations are preserved. 
\end{rem}

In order to obtain a homology group associated to $d$, we prove:
\begin{prop} \label{prop:d_2_is_0}
  The map $d$ defined in Equation \eqref{eq:differential_SC} squares to $0$.
\end{prop}
\begin{proof}
The proof follows the mold of similar results in Floer theory: by decomposing the Floer complex into its constituent lines, the vanishing of $d^{2}$ is equivalent to the vanishing of the map
\begin{equation}
 \ro_{x_2}[w(x_2)] \otimes \nu_{x_2}  \to  \ro_{x_0}[w(x_0)] \otimes \nu_{x_0}
\end{equation}
which is the sum
\begin{equation} \label{eq:sum_terms_d^2}
 \sum_{\substack{ u \in \Cyl(x_0; x_1) \\ v \in \Cyl(x_1; x_2)} } \left( d_{u} \circ d_{v} \right)  \otimes \left( \nu_{u} \circ \nu_{v} \right).
\end{equation}
Let $[u]$ denote the homotopy class of  $u$ as a map from the cylinder to $\TQ$ with asymptotic conditions $x_0$ at $s= - \infty$ and $x_1$ at $s=+\infty$, and similarly for $[v]$. By gluing along the end converging to $x_1$, we can associate to such a pair  a homotopy class of cylinders with asymptotic conditions $x_0$ and $x_{2}$, which we denote $[u] \# [v]$. 

Since $\nu$ is a local system, the composition $\nu_{u} \circ \nu_{v}$ depends only on the homotopy class $[u] \# [v]$. In particular, the vanishing of Equation \eqref{eq:sum_terms_d^2} follows from the vanishing of the sum of all terms which glue to a given homotopy class $\beta$ of cylinders with asymptotic conditions $x_0$ and $x_{2}$:
\begin{equation} \label{eq:sum_terms_d^2-fixed-htpy}
\sum_{[u] \# [v] = \beta  } d_{u} \circ d_{v} .
\end{equation}
By passing to a fixed homotopy class of cylinders, we have therefore reduced the verification of $d^2=0$ to the case of the local system $\bZ[-w(x)]$.

The vanishing of Equation \eqref{eq:sum_terms_d^2-fixed-htpy} now follows from the familiar fact the signed number of points on the boundary of an oriented manifold with boundary vanishes. Concretely, we fix an orientation $\mu$ of $ \Cylbar_{\beta}(x_0; x_2) $. Applying Theorem \ref{lem:cylinder_orientation_lines} as in the proof of Lemma \ref{lem:rigid_cylinder_orientation_lines}, we find that at each point of the moduli space, such an orientation induces an isomorphism
\begin{equation}
\ro(\mu) \co  \ro_{x_2}\to \ro_{x_0}.
\end{equation}
At the boundary of the moduli space, we can compare this isomorphism with $d_{u} \circ d_{v}$. The key point is that $d_{u} \circ d_{v} = \ro(\mu) $ if and only if $\mu$ evaluates positively on a tangent vector to $   \Cylbar_{\beta}(x_0; x_2) $ pointing outwards at $(u,v)$.  We conclude therefore that the vanishing of Equation \eqref{eq:sum_terms_d^2-fixed-htpy} is equivalent to the vanishing of the signed count of points on the boundary of the moduli space.
\end{proof}

\begin{defin}
 The Floer cohomology  $HF^{*}(  H; \nu )  $ of $H$ is the cohomology of $CF^{*}( H; \nu )$ with respect to the differential $d$.
\end{defin}
\begin{rem}
  Note that the definition of $d$ depends on the choice of the family of almost complex structures $\{ J_{t} \}_{t \in S^{1}}$. Our notation for the Floer complex does not record the data of this choice because the cohomology is in fact independent of it. This follows from the results of Section \ref{sec:comp-cont-maps-1} below, in particular the proof of Corollary \ref{cor:HF-independent_of_slope}.
\end{rem}

\section{Symplectic cohomology as a limit}
Let  $H^{{+}}$ and $H^{{-}}$  be linear Hamiltonians, which are non-degenerate in the sense of Definition \ref{def:non-degen-orbits}.  Assuming that $H^{{+}} \preceq  H^{{-}}$, where $\preceq$ is the preorder defined in Equation \eqref{eq:preorder_Hamiltonians}, we shall construct a continuation map
\begin{equation} \label{eq:continuation}
 \cont \co HF^{*}( H^{{+}}; \nu) \to HF^{*}( H^{{-}}; \nu).
\end{equation}

We define
\begin{equation} \label{eq:SH_defin_limit}
  SH^{*}(\TQ; \nu) \equiv \lim_{\cont} HF^{*}( H; \nu)
\end{equation}
to be the direct limit of all Floer cohomology groups of linear Hamiltonians with respect to these maps.

\subsection{Energy for pseudo-holomorphic maps}
The notion of energy for solutions to Floer's equation has the following variant which we find useful: the cylinder may be replaced by an arbitrary compact Riemann surface $(\Sigma,j)$ with boundary. We also choose a family $H_{z}$ of linear Hamiltonians on $\TQ$, which are parametrised by $z \in \Sigma$, and a $1$-form $\gamma$ on $\Sigma$, and study maps from $\Sigma$ to $\TQ$ which, at every point $z \in \Sigma$,  satisfy the  equation:
\begin{equation}
 \left( du(z) -  X_{H_z} \otimes \gamma_{z}  \right)  \circ j = J_{z} \circ \left( du(z) -  X_{H_z} \otimes \gamma_{z} \right).
\end{equation}
Note that Equation \eqref{eq:dbar_equation_s-indep} is the special case where $H_{z}$ depends only on the $t$-coordinate, and $\gamma = dt$.

We define the energy of a map $u$ by the formula
\begin{equation}
  E(u) = \int_{\Sigma} \| du - X_{H_{z}} \otimes \gamma_{z}  \|^{2}.
\end{equation}
\begin{exercise} \label{ex:no_energy_in_level_set}
Assume that $E(u)$ vanishes. Show that the image of every tangent vector under $du$ is parallel to $X_{H_z}$. If the image of $u$ intersects the complement of $\DQ$, conclude  that it  is contained in a level set of $\rho$.
\end{exercise}

We say that  $  d\gamma$ is non-positive if
\begin{equation} \label{eq:negativity}
  d  \gamma(\xi,j\xi) \leq 0
\end{equation}
for every tangent vector along $\Sigma$. Assume moreover that the family $\{ H_{z} \}_{z \in \Sigma}$ is constant, and write $H$ for the corresponding autonomous Hamiltonian. The proof of the following Lemma is left as an exercise to the reader, who should go through Equation \eqref{eq:energy_formula}, and keep track of the extra contribution due to the fact that $ \gamma$ may not be closed.
\begin{lem} \label{lem:monotone_implies_positive_energy}
If $d  \gamma$ is  non-positive, then
\begin{equation}
 0 \leq E(u) \leq \int_{d \Sigma}  u^{*}(\lambda ) - H \circ u \cdot \gamma.
\end{equation} 
Moreover, the second inequality is strict unless $d \gamma \equiv 0$.
\qed
\end{lem}
\begin{exercise} \label{ex:continuation}
Consider a $1$-form $\gamma$ on the cylinder given by $\rho(s) dt$, where $\rho$ is a non-increasing function of $s$. Show that $d \gamma \leq 0$.
\end{exercise}

\subsection{Continuation maps} \label{sec:continuation-maps}

As in the construction of the differential, several choices must be fixed to define the continuation map; these choices interpolate between those made for the source and target of Equation \eqref{eq:continuation}.  First, we choose a family of monotonically decreasing slopes $b_{s}$, which agree with the slope of $H^{-}$ whenever $s \ll 0$ and with that of $H^{+}$ if  $s \gg 0$.  Next, we choose a family $H_{s,t}$ of time-dependent Hamiltonians such that
\begin{align}
  H_{s,t} = b_{s} \rho  & \textrm{ in a neighbourhood of } \rho^{-1}[1,+\infty) \\
H_{s,t} = H_{t}^{\pm}& \textrm{ whenever $s \ll 0$ or $0 \ll s$.}
\end{align}
\begin{figure}[h] \label{fig:cylinder_continuation}
  \centering
\includegraphics[scale=1.33]{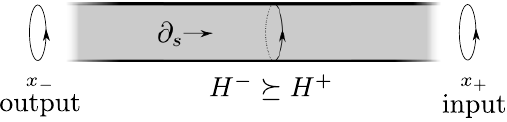}
  \caption{ }
\end{figure}

In addition, let $ J^{\pm}_{t} $  denote the family of almost complex structures used to define the Floer cohomology groups of $H^{\pm}$, and choose a  family $J_{s,t}$ of almost complex structures, parametrised by $(s,t) \in Z$, satisfying the following conditions:
\begin{equation}
  \parbox{34em}{$J_{s,t}$ is convex near $\SQ$, and agrees with $ J^{\pm}_{t} $ whenever $s \ll 0$ or $0 \ll s$.}
\end{equation}

In the usual coordinates on the cylinder, the continuation equation for maps from $Z$ to $M$ is
\begin{equation} \label{eq:continuation_equation}
 J_{s,t}   \frac{\partial u}{\partial s} =  \frac{\partial u}{\partial t} - X_{s,t}(u(s,t)).
\end{equation}

Note that this equation interpolates between Floer's equation for the Hamiltonians $H^{+}$ and  $H^{-}$.  In particular, the natural asymptotic conditions are time-$1$ orbits of $ X_{H^{{\pm}} } $ near $s = \pm \infty$.  
\begin{defin}
Given a pair of orbits $x_{\pm} \in \Orbit(H^{\pm})$, the \emph{continuation moduli space} $\Cont(x_{-}; x_{+})$ is the space of solutions to Equation \eqref{eq:continuation_equation} which exponentially converge to $x_{\pm}$ at $ s = \pm \infty$.
\end{defin}

The moduli space  $\Cont(x_{-}; x_{+})$ can be described as the solution set of a Fredholm section on a Banach space; by linearisation, we can associate to every such solution $u$ an operator
\begin{equation} \label{eq:linearisation}
 D_{u} \co W^{1,p}(Z, u^{*}(\TQ)) \to L^{p}(Z, u^{*}(\TQ))   
\end{equation}
whose kernel and cokernel control the structure of the moduli space near $u$. 

\begin{exercise} 
Prove the analogue of Theorem \ref{lem:rigid_cylinder_orientation_lines} for continuation maps, i.e. show that, associated to each element of $\Cont(x_{-}; x_{+})  $ there is a canonical isomorphism
\begin{equation} \label{eq:isomorphism_det_lines_continuation}
  \det(D_u) \otimes \delta_{x_+} \to  \delta_{x_-}.
\end{equation}
Conclude that  the index of $D_{u}$  is $ \deg(x_+) - \deg(x_-)$. (Hint: the proof is simpler than that for Floer's equation, since there is no quotient procedure in the definition of the moduli space of continuation maps).
\end{exercise}

Given that $J_{s,t}$ and $H_{s,t}$ are allowed to vary arbitrarily in a large open set containing all the orbits, the analogue of Lemma \ref{lem:transversality} holds:
\begin{lem}
  For generic data  $J_{s,t}$ and $H_{s,t}$, all moduli spaces $  \Cont(x_{-}; x_{+})$ are regular. \qed
\end{lem}

The moduli space $  \Cont(x_{-}; x_{+}) $ is, under this condition, a smooth manifold of dimension $ \deg(x_+) - \deg(x_-)$. Whenever 
\begin{equation}
  \deg(x_-) =   \deg(x_+),
\end{equation}
 $\Cont(x_{-}; x_{+})$ is therefore a $0$-dimensional manifold consisting of regular elements. As for the case of Floer's equation, we call such solutions to the continuation map equation \emph{rigid curves}.
\begin{rem}
 The constant term $-1$ disappears because $  \Cont(x_{-};x_{+})$ is not defined as the quotient by $\bR$ of a space of maps. Indeed, there is no a priori reason for the vector field $u_{*}(\partial_s)$ to give an infinitesimal automorphism of Equation \eqref{eq:continuation_equation}. 
\end{rem}

From Equation \eqref{eq:isomorphism_det_lines_continuation} and the fact that $\det(D_u)$ is canonically trivial whenever $u$ is rigid, we conclude the analogue of Lemma \ref{lem:rigid_cylinder_orientation_lines}:
\begin{lem} \label{lem:rigid_cylinder_continuation_orientation_lines}
 Every rigid solution to the continuation equation $u \in \Cont(x_{-};x_{+})$ determines a canonical isomorphism of orientation lines
 \begin{equation}
   \cont_{u} \co \ro_{x_+}[w(x_+)]  \to \ro_{x_-}[w(x_-)].
 \end{equation} \qed
 \end{lem}

We now define a map
\begin{equation}
\cont \co CF^{*}(H^{{+}}; \nu)  \to  CF^{*}(H^{{-}}; \nu)  
\end{equation}
by adding the contributions of all elements of $ \Cont(x_{-};x_{+})$:
\begin{equation}
\cont = \sum \cont_{u} \otimes \nu_{u}.
\end{equation}

The sum is finite by Gromov compactness, whose proof relies on showing that solutions to Equation \eqref{eq:continuation_equation} remain in a compact set.   The key point is that we have assumed that the slopes $b_{s}$ are decreasing with $s$, which implies that the differential of the $1$-form
\begin{equation}
    b_{s} dt
\end{equation}
is a non-positive $2$-form on $Z$. Using Lemma \ref{lem:monotone_implies_positive_energy}, the proof of the following result is then essentially the same as that of Lemma \ref{lem:maximum_principle_cylinders}:
\begin{exercise}
Show that all elements of $\Cont(x_{-};x_{+}) $ have image lying in the unit disc bundle. 
\end{exercise}

With this in mind, we conclude that the Gromov-Floer construction yields a compact manifold with boundary $ \Contbar(x_{-},x_{+}) $.  Assuming that this space is $1$-dimensional, we have a decomposition of its boundary into two types of boundary strata
\begin{equation}
\begin{aligned}
&\coprod_{x_0} \Cont(x_{-};x_0) \times \Cyl(x_0;x_{+})  \\
&\coprod_{x_1} \Cyl(x_{-};x_1) \times \Cont(x_1;x_{+}) .
\end{aligned}
\end{equation}

Note that the elements of the product of manifolds on the first line exactly correspond to the terms in $\cont \circ d $ which involve $x_-$ and $x_+$, while those on the second line are the corresponding terms in $d \circ \cont$.

\begin{prop}
The map $\cont$ is a chain homomorphism:
\begin{equation} \label{eq:chain_map_equation}
  \cont \circ d  - d \circ \cont = 0.
\end{equation}
\end{prop}
\begin{proof}[Sketch of proof:]
The reader may repeat, essentially word by word, the argument of Proposition \ref{prop:d_2_is_0}. We shall focus on the non-trivial point, which is the fact that the two terms in Equation \eqref{eq:chain_map_equation} have opposite sign.

The composition $   \cont \circ d   $  is induced by the two canonical isomorphisms
\begin{align}
  \det(D_{u_-}) \otimes \delta_{x_0} & \cong \delta_{x_-} \\
  \det(D_{v}) \otimes \delta_{x_+} & \cong \delta_{x_0}
\end{align}
defined by gluing whenever $u_{-} \in \Cont(x_{-};x_0) $ and $v \in \Cyl(x_0;x_{+}) $. Composing these two isomorphisms, we obtain an isomorphism
\begin{equation}
  \det(D_{u_-}) \otimes   \det(D_{v}) \otimes   \delta_{x_+}  \cong  \delta_{x_-}  .
\end{equation}
At this stage, we recall that $\det(D_{u_-})$ is canonically trivial because this curve is rigid. On the other hand, the construction of the differential  relied on fixing a trivialisation of $ \det(D_{v}) $ corresponding to $ \partial_{s}v $. As in Lemma \ref{sec:translation_vfield_in_and_out}, $ \partial_{s}v $ gives rise to a tangent vector to $\Cylbar(x_-,x_+)$ which points inwards. For this reason, $  \cont \circ d $ appears with a positive sign in Equation \eqref{eq:chain_map_equation}. 

If, on the other hand, we have a pair of curves $u \in \Cyl(x_{-};x_0) $ and $v_{+} \in \Cont(x_0;x_{+}) $, gluing theory yields an isomorphism
\begin{equation}
  \det(D_{u}) \otimes    \det(D_{v_+}) \otimes \delta_{x_+}  \cong  \delta_{x_-}  .
\end{equation}
The vector field $ \partial_{s}u $ now gives rise to a tangent vector to $\Contbar(x_-;x_+)$ which points outwards; the term $  d \circ \cont$ appears with a negative sign in Equation \eqref{eq:chain_map_equation}. 
\end{proof}
\begin{rem}
  The above proof highlights one of the origins of signs in Floer theory (once all moduli spaces have been coherently oriented), arising from the fact that orienting moduli spaces of rigid trajectories requires choosing an orientation on the $\bR$ factor of $Z$; we make the usual choice by choosing  $\partial_{s}$ as the positive generator. Two other sources of signs are (1) permuting factors in product decompositions of the boundary strata of a moduli space and (2) fixing orientations of abstract moduli spaces of curves. We shall encounter the first phenomenon in the discussion of the product structure, but the second is more relevant when discussing higher (infinity) structures.
\end{rem}

A particular case of interest occurs when two of the Hamiltonians are equal; one can then choose the Floer data for the continuation map to be the same as the Floer data. 
\begin{exercise} \label{ex:constant_continuation}
Assume that the Floer data $(H_{t},J_{t})$ is regular in the sense of Theorem \ref{lem:transversality}. Show that the solutions to the continuation map given by the data $H_{s,t}= H_{t}$ and $J_{s,t}= J_{t}$ are all regular and that the only rigid solutions are independent of $s$. Conclude that the corresponding continuation map is the identity at the cochain level, hence on cohomology.

\end{exercise}

Finally, we discuss the proof of the invariance of the continuation map:
\begin{lem} \label{lem:continuation_independent_choices}
The continuation map
\begin{equation}
\cont \co HF^{*}(H^{+}; \nu) \to  HF^{*}(H^{-}; \nu)  
\end{equation}
does not depend on the choice of family $(H_{s,t},J_{s,t})$.
\end{lem}
\begin{proof}[Sketch of proof]
Starting with choices $(H^{i}_{s,t},J^{i}_{s,t}) $ for $i \in \{0,1\}$, define continuation maps
\begin{equation}
  \cont_{i} \co CF^{*}( H^{+}; \nu) \to  CF^{*}( H^{-}; \nu).
\end{equation}

Consider a family of Floer data $(H^{r}_{s,t},J^{r}_{s,t})  $, for $r \in [0,1]$, which agree with  $ (H^{+}_{t}, J^+_{t}) $ if $ 0 \ll s$, and with $ (H^{-}_{t}, J^{-}_{t}) $ if $s \ll 0$. We claim that this choice defines a chain homotopy $\tilde{\cont}$ between $   \cont_{0} $ and $\cont_{1}$. Concretely, we let  $\Cont_{r}(x_-,x_+)  $ denote the moduli space of continuation maps for Floer data $ (H^{r}_{s,t},J^{r}_{s,t}) $, and let   $\tilde{\Cont}(x_-,x_+)$ denote the union of the moduli spaces over $r \in [0,1]$:
\begin{equation}
  \tilde{\Cont}(x_-;x_+) = \coprod_{r \in [0,1]} \Cont_{r}(x_-;x_+).
\end{equation}
The key point is that $  \tilde{\Cont}(x_-;x_+) $ is equipped with a natural topology, arising from its embedding in
\begin{equation}
  [0,1] \times C^{\infty}(Z,\TQ)
\end{equation}
as the zero locus of a section of the bundle whose fibre at $(r,u)$ is  $C^{\infty}(Z,u^{*}(T\TQ)) $. Note that this bundle is pulled back by projection to the second factor, but that the section we are considering depends on the first factor; its restriction to $\{r \} \times  C^{\infty}(Z,\TQ) $ is  the continuation map operator
\begin{equation}
u \mapsto     \frac{\partial u}{\partial s} + J_{s,t}^{r}  \left( \frac{\partial u}{\partial t} - X_{s,t}^{r}(u(s,t)) \right),
\end{equation}
where $X_{s,t}^{r}  $ is the Hamiltonian flow of $H^{r}_{s,t}$. 

For a generic choice of families $ (H^{r}_{s,t},J^{r}_{s,t})  $, the moduli space $  \tilde{\Cont}(x_-;x_+)  $ is regular, and hence a smooth manifold with boundary:
\begin{equation}
  \partial  \tilde{\Cont}(x_-;x_+) = \Cont_{0}(x_-;x_+) \cup \Cont_{1}(x_-;x_+).
\end{equation}
In this situation, regularity of the moduli space at a point $(r,u)$ is equivalent to the surjectivity of the Fredholm map
\begin{equation} \label{eq:extended_linearised_operator}
 T_{r} [0,1] \oplus W^{1,p}(Z,u^{*}(T \TQ)) \to L^{p}(Z,u^{*}(T \TQ)),
\end{equation}
where the second factor is the linearised operator $D_{u}$, and the first factor is obtained by taking the derivative for the linearised operator with respect to the $r$ variable:
\begin{equation} \label{eq:differential_parametrising_direction}
  \partial_{r} \mapsto  \frac{\partial J_{s,t}^{r}}{\partial r}  \left( \frac{\partial u}{\partial t} - X_{s,t}^{r}(u(s,t)) \right) + J_{s,t}^{r} \frac{\partial X_{s,t}^{r}}{\partial r}.
\end{equation}
 This implies the existence of  a canonical isomorphism
\begin{equation} \label{eq:tangent_space_parametrised}
 |T_{u}\tilde{\Cont}(x_-;x_+)  |  \cong    |T_{r} [0,1]| \otimes |\det(D_{u})|.
\end{equation}
We now fix an orientation of the interval $[0,1]$, which allows us to combine Equation \eqref{eq:tangent_space_parametrised} with the isomorphism in Equation \eqref{eq:isomorphism_det_lines_continuation} coming from gluing theory, to obtain an isomorphism
\begin{equation} \label{eq:canonical_maps_orientation_lines_parametrised}
  |T_{u} \tilde{\Cont}(x_-;x_+)   |  \otimes \ro_{x_+}  \cong \ro_{x_-}.
\end{equation}

We now restrict attention to the situation where $\deg(x_+) = \deg(x_-) -1$; this implies that  $\tilde{\Cont}(x_-;x_+) $ is $0$-dimensional, and hence that $  T \tilde{\Cont}(x_-;x_+)  $ is canonically trivial. In this case, Equation \eqref{eq:canonical_maps_orientation_lines_parametrised} induces an isomorphism
\begin{equation}
  \tilde{\cont}_{u} \co \ro_{x_+}[w(x_+)]  \to \ro_{x_-}[w(x_-)] .
\end{equation}
The map $\tilde{\cont}$ is obtained by taking the sum, over all elements of $ \tilde{\Cont}(x_-;x_+) $, of the tensor product of $   \tilde{\cont}_{u}  $ with the map on local systems $\nu_{x_+} \to \nu_{x_-}$ induced by $u$.

The proof that $ \tilde{\cont} $ is a chain homotopy between $\cont_0$ and $\cont_1$ now follows from analysing the boundary of the compactification of $  \tilde{\Cont}(x_-;x_+)  $ when $\deg(x_+) = \deg(x_-)$. Having provided a definition of all the maps over the integers, one can now directly lift the familiar argument in the case of a field of characteristic $2$; this is discussed e.g. in \cite{salamon-notes}*{Lemma 3.12}.
\end{proof}
\begin{rem} \label{rem:general_continuation_map}
The construction of continuation maps can be performed in more generality, breaking the assumption that the slope depends only on $s$: choose a $1$-form $\alpha$ on the cylinder, a function $w  \co Z \to \bR$, and a family of Hamiltonian functions $H_{s,t}$, parametrised by points in $Z$, such that the slope of $H_{s,t}$ is $b(s,t)$.  At infinity, we assume that these data are $s$-independent, and that $\alpha = dt$. We impose the condition that $d(b \cdot \alpha) \leq 0$, so that the maximum principle applies to this equation.  We can then use it to define a continuation map.

To see that continuation maps for these more general data are still independent of such a choice, it suffices to show that the space is convex; this is indeed the case because the equation $d(b \cdot \alpha) \leq 0$ is local, and evidently convex in $b$ and $\alpha$. 
\end{rem}

\subsection{Composition of continuation maps} \label{sec:comp-cont-maps-1}
In this section, we discuss the compatibility of continuation maps with compositions. 
\begin{prop} \label{prop:comp-cont-maps}
Given a triple of linear Hamiltonians such that $H^{+} \preceq H^{0} \preceq H^{-}$,  we have a commutative diagram
\begin{equation} \label{eq:compatibility_continuation}
  \xymatrix{HF^{*}(H^{-}; \nu) \ar[r] \ar[dr] & HF^{*}(H^{0}; \nu) \ar[d] \\
& HF^{*}(H^{+}; \nu) }
\end{equation}
for any choice of continuation maps.
\end{prop}
\begin{proof}[Sketch of proof]
Let $H^{+}_{s,t}  $ and $H^{-}_{s,t}$ be Hamiltonians which define continuation maps
\begin{align}
 \cont_{+ \to 0} \co  HF^{*}(H^{+}; \nu) & \to HF^{*}(H^{0}; \nu) \\
 \cont_{0 \to -} \co HF^{*}(H^{0}; \nu) & \to HF^{*}(H^{-}; \nu).
\end{align}
By construction, whenever $s \ll 0$, $ H^{+}_{s,t} =H^{0}_{t}$, while whenever $0 \ll s$,   $ H^{-}_{s,t} =H^{0}_{t}$. Given a negative real number $S$ of large absolute value, we can define a family of Hamiltonians by concatenation:
\begin{equation}
H^{S}_{s,t} = \begin{cases}   H^{+}_{s-2S,t} & \textrm{ if } s \leq S \\
H^{-}_{s,t} & \textrm{ if } S  \leq s.
\end{cases}
\end{equation}
Note that each such choice defines a continuation map
\begin{equation}
\cont^{S}_{+ \to -} \co    HF^{*}(H^{+}; \nu) \to HF^{*}(H^{-}; \nu). 
\end{equation}
In order to compare this map to the composition $   \cont_{+ \to 0} \circ \cont_{0 \to -}$, we consider a triple of orbits  $x_+$, $x_0$, and $x_-$ for the three Hamiltonians we are considering, and the space of pairs of solutions to the continuation map equation:
\begin{equation}
  \Cont(x_-,x_0) \times  \Cont(x_0,x_+).
\end{equation}
For $S \ll 0$, we can glue pairs of such solutions, and obtain a solution to the continuation map equation defined using $ H^{S}_{s,t} $; this is the same procedure used in proving that $\cont$ is a chain map, see e.g. \cite{salamon-notes}*{Section 3.4}. We denote  by $  \Cont_{S}(x_-;x_+)$ the space of such solutions. If $\deg(x_-) = \deg(x_0) = \deg(x_+)$, these spaces are all regular,  and we obtain a bijection
\begin{equation}
 \Cont_{S}(x_-;x_+) \cong \coprod_{x_0}    \Cont(x_-,x_0) \times  \Cont(x_0,x_+).
\end{equation}
The product of moduli spaces on the right hand side defines the composition of continuation maps.  This proves that
\begin{equation}
  \cont^{S}_{+ \to -}  =  \cont_{+ \to 0}  \circ \cont_{0 \to -}.
\end{equation}
\end{proof}
\begin{cor} \label{cor:HF-independent_of_slope}
If $H$ and $H'$ have the same slope, then there is a canonical isomorphism
\begin{equation}
    HF^{*}(H; \nu) \to HF^{*}(H'; \nu)
\end{equation}
induced by continuation maps.
\end{cor}
\begin{proof}
First, we show that such continuation maps induce isomorphisms by considering the diagram
\begin{equation}
  \xymatrix{HF^{*}(H; \nu) \ar[r] \ar[dr] & HF^{*}(H'; \nu) \ar[d] \\
& HF^{*}(H; \nu) }
\end{equation}
where the diagonal arrow is induced by the Floer data that is independent of $s$. Exercise \ref{ex:constant_continuation} implies that this map is the identity, proving that the first continuation map is injective, and the second surjective. Reversing the r\^oles of $H$ and $H'$, we conclude that the continuation map is indeed an isomorphism.

To show that the isomorphism is independent of choices, consider the diagram
\begin{equation}
  \xymatrix{HF^{*}(H; \nu) \ar[r] \ar[dr] & HF^{*}(H'; \nu) \ar[d] \\
& HF^{*}(H'; \nu), }
\end{equation}
where the vertical map is now the identity.
\end{proof}

In Equation \eqref{eq:SH_defin_limit}, we define symplectic cohomology as a limit over all continuation maps.  The following Lemma  gives a more concrete approach to computing it:
\begin{lem} \label{lem:directed_system_H}
   If $H^{i}$ is any sequence of Hamiltonians on $\TQ$ whose slope is unbounded, the natural map
  \begin{equation} \label{eq:directed_system}
    \lim_{i} HF^{*}( H^i; \nu) \to SH^*(\TQ; \nu)
  \end{equation}
is an isomorphism.
\end{lem}
\begin{proof}
  By definition, the direct limit in Equation \eqref{eq:SH_defin_limit} is the quotient of the direct sum
  \begin{equation}
    \bigoplus_{H} HF^{*}( H; \nu)
  \end{equation}
by the subspace generated by $a - \cont(a)$ for each element $a \in HF^{*}( H; \nu)$. The assertion that the natural map is surjective follows immediately from the assumption that the sequence of slopes is unbounded, since this implies that, for each Hamiltonian $H$,s there is a well defined map
\begin{equation}
 HF^{*}( H; \nu) \to HF^{*}( H^{i}; \nu)
\end{equation}
for some sufficiently large $i$.

The assertion that the map is injective is equivalent to the claim that every relation in $SH^{*}(\TQ; \nu) $ is detected by the sequence $H^{i}$, i.e. that two classes $a_i \in HF^{*}(H^{i}; \nu)$ and $a_j \in HF^{*}(H^{j}; \nu)$ are equivalent in $  SH^{*}(\TQ; \nu)$  if and only if there exists an integer $k$  so that the images of $a_i$ and $a_j$ in $ HF^{*}(H^{k}; \nu) $ agree. This follows immediately from the fact that, for any continuation map from $H$ to $K$, we may choose $H^{k}$ of slope larger than both, so that we have a commutative diagram
\begin{equation}
  \xymatrix{HF^{*}(H; \nu) \ar[r] \ar[dr] & HF^{*}(K; \nu) \ar[d] \\
& HF^{*}(H^{i}; \nu). }
\end{equation}
\end{proof}

\section{Aside on orientation lines} \label{sec:aside-orient-lines}
In the construction of the Floer complex, we use various rank-$1$ graded $\bR$-vector spaces which arise as determinants: given a real vector space $V$, we define $\det(V)$ to be the top exterior power of $V$, which is naturally graded in degree $\dim_{\bR}(V)$.  Associated to a rank-$1$ graded $\bR$-vector space $\delta$ is a graded rank-$1$ free abelian group which we call the \emph{orientation line}, and denote $|\delta|$, and which is generated by the two orientations of $\delta$, modulo the relation that the sum vanishes. When $\delta = \det(V)$, we also write $|V|$ for the orientation line.  We need to repeatedly manipulate orientation lines associated to short exact sequences, and to dual vector space, keeping in mind  the following general principle:
\begin{equation}
  \parbox{34em}{in comparing two operations which differ by permuting operations or generators, one must introduce the appropriate Koszul sign.}
\end{equation}

We briefly describe the origin of two of the signs we shall encounter.  Start with the familiar fact that an orientation of two vector spaces induces an orientation of the direct sum. At the level of orientation lines, this can be restated as the existence of a canonical isomorphism:
\begin{equation}
 |V_1| \otimes |V_2| \cong |V_1 \oplus V_2|.
\end{equation}
Let us consider the composition 
\begin{equation}
  |V_1| \otimes |V_2| \cong | V_{1} \oplus V_{2} | \cong  | V_{2} \oplus V_{1}| \cong  |V_2| \otimes |V_1|.
\end{equation}
The Koszul sign is the sign difference between this map, and the transposition. It is equal to
\begin{equation}
  (-1)^{\dim(V_1) \cdot \dim(V_2)}.
\end{equation}
The next convention we fix is that of splitting every short exact sequence of vector spaces
\begin{equation}
  0 \to U \to W \to V \to 0
\end{equation}
as $U \oplus V \cong W$, which yields an isomorphism
\begin{equation}
  |U| \otimes |V| \cong |W|.
\end{equation}

Finally, given a $\bZ$-graded orientation line $\ell$, we define $\ell^{-1}$ to  be the dual line to $\ell$, carrying a canonical map
\begin{equation} \label{eq:inverse_line}
  \ell^{-1} \otimes \ell \to \bZ.
\end{equation}
The lines $\ell$ and  $\ell^{-1}$ are by definition supported in opposite degrees. 

\begin{lem} \label{lem:iso_inverse_signs}
  If $\ell$ is a graded orientation line, then there is a canonical isomorphism $\ell \cong \ell^{-1}$ as $\bZ_{2}$ graded lines. If $\ell \cong \ell_{1} \otimes \ell_{2} \otimes \ell_{3}$, then the isomorphism induced via the composition
  \begin{equation}
    \ell \cong \ell_{1} \otimes \ell_{2} \otimes \ell_{3} \cong \ell_{1}^{-1} \otimes \ell_{2}^{-1} \otimes \ell_{3}^{-1} \cong \ell_{3}^{-1} \otimes \ell_{2}^{-1} \otimes \ell_{1}^{-1} \cong \ell^{-1}
  \end{equation}
differs from the canonical isomorphism by $\sum_{i < j} \deg(\ell_i) \deg(\ell_j) $. \qed
\end{lem}

\section{Guide to the literature}

\subsection{Definitions of symplectic cohomology}
Hamiltonian Floer (co)-homology was introduced by its eponymous inventor in \cite{Floer-JDG} at the same time as the Lagrangian analogue; it was significantly developed in the middle of the 1990's, with most work dedicated to proving the well definedness of this group for increasingly general classes of closed symplectic manifolds. The technical difficulties which appear in this setting are completely orthogonal to those we encounter for cotangent bundles.

Symplectic cohomology was first introduced by Floer and Hofer in \cite{FH-SH} for convex domains in $\bC^{n}$, and their definition was soon generalised in their joint work \cite{CFH} with Cieliebak to symplectic manifolds with \emph{restricted contact boundary;} this class of manifolds is now more commonly referred to as Liouville domains \cite{seidel-biased}. Most notions introduced in this chapter and the next can be generalised with no difficulty to Liouville domains; we have restricted ourselves to cotangent bundles only for the sake of concreteness.

The definition of symplectic cohomology as a limit of Floer cohomology groups for Hamiltonians which are linear at infinity is due to Viterbo \cite{viterbo-94}, who was motivated by applications to the study of the symplectic topology of cotangent bundles.  A different definition, as the Floer cohomology of a single Hamiltonian, is given by Abbondandolo and Schwarz in \cite{AS}, and Salamon and Weber in \cite{SW-06}. For general Liouville domains, such a definition was also given by Seidel in \cite{seidel-biased}.  At the cohomological level, the comparison between different versions of Floer cohomology is performed by Oancea in \cite{oancea}, and for a different class of Hamiltonians by Seidel in \cite{seidel-biased}.

There is a completely different definition of symplectic cohomology using methods from symplectic field theory due to Bourgeois, Ekholm, and Eliashberg \cite{BEE}, generalising the connection between contact homology and equivariant symplectic cohomology due to Bourgeois and Oancea \cite{BO}. Under the assumption that all transversality problems in  symplectic field theory  can be resolved by virtual fundamental chain methods, an isomorphism between these theories over the rational numbers is constructed in \cite{BO2}.

The main deficiency of the definition of symplectic cohomology as a limit is that we have not defined it as the homology of a chain complex; or, rather, that the only natural chain complex that we could use is the homotopy direct limit of the Floer cochain complexes, which can be rather unwieldy. Once can adapt the model defined in \cite{ASeidel} to build a relatively explicit chain complex computing symplectic cohomology.  

\subsection{Homology versus cohomology}

In the closed setting, Floer homology and cohomology are isomorphic via an isomorphism that implements Poincar\'e duality. In comparing between two different papers, it is often the case that the theories called homology and cohomology are interchanged. The reason is that in order to compare the conventions of two different authors, one must fix several choices which can be made independently: (1) the sign in the definition of the Hamiltonian flow, (2) the sign of the inhomogeneous term in the Floer equation, and (3) which asymptotic end of a Floer trajectory will be considered the input. If one starts with a set of conventions, having decided that one group is to be called homology and the other cohomology, one finds that changing any one of these choices reverses the identifications.

For manifolds which are not closed, the two different Floer groups one can assign to a Hamiltonian are not in general isomorphic; in fact, they need not have the same rank. A priori, this should make it easy to determine which group is called homology and which is cohomology. Unforunately, the conventions in the literature are still not consistent: up until Seidel's work \cite{seidel-biased}, the invariant which we call symplectic cohomology was called homology. One justification for this is that, for cotangent bundles, it is isomorphic in the $\Spin$ case to the homology of the free loop space of the base, not its cohomology.

The reason for adopting Seidel's conventions is that the use of homological gradings would lead to the operations constructed in Chapter \ref{cha:oper-sympl-cohom} having bizarre gradings (e.g. the unit of multiplication would lie in degree $n$, and the product would not preserve grading). From the point of view of studying algebraic structures, the cohomological conventions are ultimately more convenient.

\subsection{Gradings and signs}

The definition of relative gradings in Hamiltonian Floer theory for manifolds $M$ whose first Chern class vanishes goes back to Floer in \cite{Floer-index}; by relative, we mean that the index \emph{difference} between two orbits is well-defined. A relative grading determines an absolute grading up to global shift. Under the assumption that $c_{1}(M)$ vanishes on every sphere, there is a natural  absolute grading on the set of contractible orbits.

In Lagrangian Floer theory, it seems that  Kontsevich was first to observe that an absolute grading arises from a choice of quadratic volume form with respect to a compatible almost complex structure; the theory was developed by Seidel in \cite{seidel-Graded}. The vanishing of  $2 c_{1}(M) =0$ is the obstruction to the existence of such a volume form. Because of what is known in the Lagrangian case, the  fact that we have defined gradings without assuming that $Q$ is oriented should not be too surprising: while the technical details do not seem to be written anywhere, it is known that the symplectic cohomology of a manifold $M$ is isomorphic to the (wrapped) Lagrangian Floer cohomology of the diagonal in $M \times M$, and this Lagrangian Floer cohomology group is equipped with a grading by \cite{ASeidel}. The construction of gradings that we give is an attempt to give the simplest possible definition of the grading, rather than the one which generalises most conveniently.

In order to define a Floer cohomology group over the integers, one must understand how to orient moduli spaces of Floer trajectories. The key work here is due to Floer and Hofer \cite{FH}, who described the signs in Floer theory in terms of coherent orientations. Most of the (rather limited) literature which concerns itself with Floer cohomology over $\bZ$ uses these conventions. Here, we use a less concrete but ultimately equivalent point of view which is directly modelled after the discussion of signs in Lagrangian Floer theory as explained in \cite{seidel-Book}. The more abstract approach that we adopt is convenient for understanding the signs arising in the operations which we discuss in Chapter \ref{cha:oper-sympl-cohom}.

\chapter{Operations in Symplectic cohomology} \label{cha:oper-sympl-cohom}

\section{Introduction}

Symplectic cohomology  is equipped with operations coming from counts of Riemann surfaces with at least one negative end and an arbitrary number of positive ends; the ends are required to map to Hamiltonian orbits.  Since we are considering time-dependent Hamiltonians, each orbit has a canonical starting point, and the Riemann surfaces we study  carry an \emph{asymptotic marker at each puncture}.

Associated to ``counting'' Riemann surfaces of genus $0$ with one negative end and no positive ends, i.e. planes, we shall produce in Section \ref{sec:unit} an element
\begin{equation}
  e \in SH^{0}(\TQ;  \bZ)
\end{equation}

In Section \ref{sec:bv-operator}, we shall define  a map
\begin{equation} \label{eq:BV_on_SH}
  \Delta \co SH^{*}(\TQ;  \bZ ) \to SH^{*-1}(\TQ;  \bZ) 
\end{equation}
induced by Riemann surfaces of genus $0$ with one negative end and one positive end (cylinders), where the asymptotic marker moves in an $S^1$ family at the negative end. 

Finally, in Section \ref{sec:pair-pants-product}, we shall define a map associated to Riemann surfaces of genus $0$ with one negative end and two positive ends (pairs of pants):
\begin{equation}
  \star \co SH^{*}(\TQ;  \bZ) \otimes SH^{*}(\TQ; \bZ)  \to SH^{*}(\TQ; \bZ).
\end{equation}

The compatibility between these three structures is summarised by $6$  equations, which are variants of the ones for a $BV$ algebra (see \cite{CV} for an expository account). To state them, recall that in addition to the $\bZ$ grading $\deg$ we have a $\bZ_{2}$ grading $w$. The first five equations are:
\begin{alignat}{4}  \label{eq:square_Delta_0}
\Delta^{2} & = 0 && \textrm{ Section \ref{sec:square-bv-operator}} \\ 
\Delta(e) & = 0  && \textrm{ Section \ref{sec:properties-unit}}\\
(a \star b) \star c & = a \star (b \star c) && \textrm{ Section \ref{sec:associativity}}\\ \label{eq:twisted_commutative}
 a \star b & = (-1)^{ \deg(a)\deg(b) + w(a)w(b)} b \star a && \textrm{ Section  \ref{sec:associativity}} \\
e \star a & = a \star e = a && \textrm{ Section \ref{sec:properties-unit}}
\end{alignat}
The last equation is the $BV$ equation
\begin{multline} \label{eq:BV-equation}
  \Delta(  a \star b \star c) + \Delta(a) \star b \star c + (-1)^{\deg(a)}  a \star \Delta(b)  \star c + (-1)^{\deg(a) + \deg(b)}   a \star b \star \Delta(c) = \\
\Delta(a \star b) \star c + (-1)^{\deg(a)}   a \star \Delta( b \star c) +  (-1)^{(1 + \deg(a))\deg(b) + w(a)w(b)}    b \star \Delta( a \star c) 
\end{multline}
and will be discussed in Section \ref{sec:bv-equation}.

\begin{rem} 
The equations we have written follow, by a computation of Koszul signs, from the usual $BV$ equations for the $\bZ_{2}$ gradings on $SH^{*}(\TQ; \bZ)$ induced by assigning $|x| \mod 2$ to every orbit $x$. In particular, with such gradings, one can replace Equation \eqref{eq:twisted_commutative} with the usual skew-commutativity, and the $BV$ equation with:
\begin{multline} \label{eq:BV-equation-z_2-grading}
  \Delta(  a \star b \star c) +  \Delta(a) \star b \star c + (-1)^{|a|}  a \star \Delta(b)  \star c + (-1)^{|a| + |b|}   a \star b \star \Delta(c)  \\
= \Delta(a \star b) \star c + (-1)^{|a |}   a \star \Delta( b \star c) +  (-1)^{(1 + |a|)|b|}    b \star \Delta( a \star c).
\end{multline}

  We have already observed in Remark \eqref{eq:apology_to_the_reader} that the grading $|x|$ does not in general lift to an integral grading of $SH^{*}(\TQ; \bZ)$ with the property that the product has cohomological degree $0$. This was the reason for introducing the grading $\deg(x)$.

Having introduced the shift in Equation \eqref{eq:definition_degree_shifted_1}, we now reap the reward that we do not quite have a $\bZ$ graded $BV$ structure with respect to $\deg$ either. Rather, we obtain what might be called a  \emph{twisted} graded commutative structure, where all signs have to be corrected by the difference between our two notions of degree. Since this is a lesser transgression of mathematical conventions than a product whose degree is not homogeneous, we find our choice justified.
\end{rem}

\subsection{$BV$ structure with twisted coefficients}
Since the comparison with string topology involves twisted coefficients, it is important to understand how to produce a $BV$ structure for local systems which are not necessarily trivial.  For an arbitrary local system, there is no natural $BV$ structure on symplectic cohomology, so we consider instead those local systems that are obtained by \emph{transgressing} a vector bundle on the base: given a real vector bundle $E$ on $\Q$, define $\s^{E}$ to be the local system on $\sL Q$ whose fibre at a loop $x$ is the rank-$1$ free abelian group generated by the set of trivialisations of $x^{*}(E \oplus \det(E)^{\oplus 3}) $ up to homotopy; there are two such homotopy classes, and we impose the relation that their sum vanish. More explicitly, note that this bundle is naturally oriented, and that any two trivialisations which are compatible with this orientation differ by the action of the set of  maps from $S^1$ to $SO(n+3)$. Since $\pi_{1}(SO(n+3)) = \bZ_{2}$, there are exactly two such trivialisations up to homotopy.  

\begin{rem}
  The set of trivialisations of $ x^{*}(E \oplus \det(E)^{\oplus 3})  $ compatible with its natural orientation is in canonical bijection with the set of $\Pin^{+}$ structures of $x^{*}(E)$, where $\Pin^{+}$ is one of the two central $\bZ/2\bZ$ extensions of the orthogonal group  (see \cite{KT}; the other extension which is called $\Pin^{-}$ corresponds to trivialisations of $ x^{*}(E \oplus \det(E))  $). For a concrete description of this local system in terms of monodromy, see Remark \ref{rem:BV_structure_monodromy}. If we were only considering orientable manifolds, it would probably make sense to restrict the discussion to orientable vector bundles $E$; in which case it would suffice to consider the local system of trivialisations of $ x^{*}(E) $, i.e. the local system of $\Spin$ structures.  
\end{rem}

There are three structure maps associated to the local system $\s^{E}$ that will be of interest: first, a retraction of $x$ to a point induces a canonical map
\begin{equation} \label{eq:unit_local_system}
  \bZ \to \s^{E}_{x}
\end{equation}
 since such a retraction determines an extension of $x^{*}(E)$ to the disc, and a vector bundle over the disc admits a unique trivialisation up to homotopy.

Next, for any loop $x$, define $x_{\theta}(t) = x(t + \theta)$.   Since a trivialisation of $x^{*}(E \oplus \det(E)^{\oplus 3})$ is equivalent to a trivialisation of $ x^{*}_{\theta}(E \oplus \det(E)^{\oplus 3}) $ we have a canonical isomorphism
\begin{equation} \label{eq:rotate_circle_local_system}
   \s^{E}_{x} \to   \s^{E}_{x_{\theta}}  
\end{equation}
which is the identity when $\theta = 1$.

Finally, to a map $u$ from a pair of pants to $\TQ$ with asymptotic conditions $x_1$, $x_2$ and $x_0$ at the three punctures, we associate a map
\begin{equation} \label{eq:product_local_system}
\s^{E}_{u} \co  \s^{E}_{x_1} \otimes \s^{E}_{x_2} \to \s^{E}_{x_0}
\end{equation}
by noting that a trivialisation of a vector bundle over a pair of pants is determined by its restriction to two of the three circles at infinity.

We shall use these additional structures at the level of local systems to define structure maps on symplectic cohomology:
\begin{align}
   e \co \s^{E} & \to SH^{0}(\TQ;  \s^{E}) \\
  \Delta \co SH^{*}(\TQ;  \s^{E} ) & \to SH^{*-1}(\TQ;  \s^{E})  \\
  \star \co SH^{*}(\TQ;  \s^{E}) \otimes SH^{*}(\TQ; \s^{E})  & \to SH^{*}(\TQ; \s^{E}).
\end{align}

The following Theorem summarises the compatibility of these operations:
\begin{thm}
 Symplectic cohomology with coefficients in $\s^{E}$ is naturally equipped with the structure of a twisted Batalin-Vilkovisky algebra ($BV$-algebra).
\end{thm}

\begin{rem} \label{rem:BV_structure_monodromy}
One can construct a $BV$ structure on symplectic cohomology with twisted coefficients in more generality than described above, starting with a class in $b \in H^{2}(\Q, \bZ_{2})$.  The case we have discussed corresponds to the case were $b$ is the second Stiefel-Whitney class of the bundle $E$.

To such a class, one can assign, by transgression, a class in $H^{1}(\sL \Q, \bZ_{2})$, which is the group classifying local systems of rank $1$.  The corresponding local system can most explicitly be described by its monodromy: a loop of loops corresponds to a torus in $\Q$, and the monodromy is trivial if and only if the given class evaluates trivially on the corresponding torus.

For this class of local systems, one can define the twisted version of symplectic cohomology in terms of intersection numbers of Riemann surfaces with a Poincar\'e dual cycle $D_{b}$. Generically,  all time-$1$ orbits of a Hamiltonian $H$  are disjoint from the inverse image of $D_{b}$ in $\TQ$.  Each holomorphic cylinder with endpoints on such orbits has a well-defined intersection number with this inverse image; one can therefore define a twisted differential on $CF^*(H; \s^{E})$ by multiplying the contribution of each curve by a sign which is non-trivial if and only if this intersection number is odd. Applying the same procedure to continuation maps yields a twisted symplectic cohomology group  $  SH^{*}(\TQ; \s^{E}) $. The product and $BV$ operator are defined via the same procedure, twisting the sign of the contribution of each curve by its intersection number with $D_{b}$.

Note that there are local systems on the free loop space which do not arise from this construction;  e.g.  local systems which are pulled back from the base. The most important such local system is the local system of orientations of $\Q$.  

\end{rem}

\section{The $BV$ operator}\label{sec:bv-operator}

\subsection{The $BV$ operator on Floer homology} 
The cylinder $Z$ has natural cylindrical ends near both $+\infty$ and $-\infty$.  In order to construct the $BV$ operator $\Delta$, we fix the cylindrical end near  $+\infty$ and consider a family of cylindrical ends parametrised by $\theta \in \bR / \bZ = S^{1}$ near $-\infty$.

Given a linear Hamiltonian $H = H_{t}$, we then choose
\begin{equation} \label{eq:hamiltonian_family_for_Delta}
  \parbox{35em}{a family of functions $H^{\theta}_{s,t}$ parametrised by $(\theta,s,t) \in S^{1} \times Z$ which agree with $H_{t}$ when $0 \ll s$, with $H_{t + \theta}$ when $s \ll 0$. Moreover, we require that the restriction of $H^{\theta}_{s,t}  $ to the complement of $\DQ $ be linear of slope independent of $(\theta,s,t)$.} 
\end{equation}
Choosing a family $J_{s,t}^{\theta}$ of convex almost complex structures (in the sense of Definition \ref{def:convex_almost_complex}) which also agree with $J_{t}$ when  $0 \ll s$ and with $J_{t + \theta}$ when $s \ll 0$, we define the moduli space
\begin{equation} 
  \Cyl_{\Delta}(x_0,x_1)
\end{equation}
to be the space of pairs $(\theta, u)$ where $\theta$ is a point on $S^{1}$, and $u$ is a map from a cylinder to $\TQ$ which solves the differential equation
\begin{equation}
J_{s,t}^{\theta}   \frac{\partial u}{\partial s} =  \frac{\partial u}{\partial t} - X_{H_{s,t}}^{\theta} 
\end{equation}
and such that 
\begin{equation}
  \label{eq:asymptotic_conditions_Delta}
\begin{aligned} 
  \lim_{s \to +\infty}u(s,t) & = x_{1}(t) \\
 \lim_{s \to -\infty}u(s,t+\theta) & = x_{0}(t+\theta).
\end{aligned}
\end{equation}

There are no new phenomena in the study of the moduli space $  \Cyl_{\Delta}(x_0,x_1) $ which do not already appear in the study of continuation maps. To each map $u \in \Cyl_{\Delta}(x_0,x_1)$, we can assign a linearised operator $D_{u}$ on  sections  of the pullback of $T\TQ$.  The direct sum with the derivative of this operator with respect to the parametrising variable $\theta$ yields a Fredholm map
\begin{equation} \label{eq:parametrised_CR_equation_circle}
 T_{\theta} S^{1}  \oplus W^{1,p}(Z,u^{*}(T \TQ)) \to L^{p}(Z,u^{*}(T \TQ));
\end{equation}
this is the same construction which yields the operator in Equation \eqref{eq:extended_linearised_operator}.

The Floer data  $ H^{\theta}_{s,t} $ and $ J_{s,t}^{\theta} $ are said to be regular if the map in Equation  \eqref{eq:parametrised_CR_equation_circle} is surjective, which is equivalent to the assumption that the map
 \begin{equation}
  T_{\theta} S^{1}   \to \coker(D_u),
\end{equation}
is surjective.  As in Theorem \ref{lem:transversality}, this is achieved for generic choices.

Assuming regularity, the moduli space $ \Cyl_{\Delta}(x_0,x_1) $ is a smooth manifold of dimension
\begin{equation} \label{eq:dimension_moduli_parametrised_equation_S^1}
  \deg(x_0) - \deg(x_1) +1.
\end{equation}
The tangent space of $  \Cyl_{\Delta}(x_0,x_1)  $ at a point $(\theta,u)  $ is the kernel of the operator in Equation \eqref{eq:parametrised_CR_equation_circle}.  This tangent space is therefore the first term in a short exact sequence
\begin{equation}
  T_{u}  \Cyl_{\Delta}(x_0,x_1)  \to  T_{\theta} S^{1}  \oplus \ker(D_u) \to \coker(D_u),
\end{equation}
which induces an isomorphism
\begin{equation} 
   \det(T_{u}  \Cyl_{\Delta}(x_0,x_1)  ) \otimes \det(\coker(D_u))  \cong   \det(T_{\theta}  S^{1}) \otimes   \det(\ker(D_u)).
\end{equation}

Recalling the definition of $\det(D_u)$, we conclude that we have a canonical isomorphism
\begin{equation}\label{eq:isomorphism_parametrised_equation_circle}
   \det(T_{u}  \Cyl_{\Delta}(x_0,x_1)  )  \cong   \det(D_{u})  \otimes \det(T_{\theta}  S^{1})
\end{equation}
at every point $(\theta,u)$ of this moduli space.

To obtain an orientation of $\Cyl_{\Delta}(x_0,x_1)   $, one must orient both the determinant line of $D_{u}$ and $T_{\theta}S^{1} $. The second choice will be fixed once and for all, while an orientation of the determinant line of $D_{u}$ arises from gluing theory:
\begin{equation}
  |\det(D_u)| \otimes \ro_{x_1} \cong \ro_{x_0}.
\end{equation}
Note that this is Equation \eqref{eq:standard_iso_orientation_lines_cylinder} verbatim.

Since $\s^{E}$ is a local system, the cylinder $u$ induces a map from $   \s^{E}_{x_1}  $  to $  \s^{E}_{x_{\theta,0}} $. Composing this with the map in  Equation \eqref{eq:rotate_circle_local_system}
yields a canonical isomorphim
\begin{equation}
  \s^{E}_{x_1} \to  \s^{E}_{x_0}.
\end{equation}

Applying this discussion to the case of rigid solutions, we conclude:
\begin{lem}
If all Floer data is regular, the moduli space $  \Cyl_{\Delta}(x_0,x_1)   $ has dimension $0$ whenever $  \deg(x_0) =  \deg(x_1) - 1 $ and every element $u$ of $  \Cyl_{\Delta}(x_0,x_1)   $ induces a map
\begin{equation} \label{eq:Circle_operator_one_map}
  \Delta_{u} \co \ro_{x_1}[w(x_1)] \otimes \s^{E}_{x_1} \to \ro_{x_0}[w(x_0)]\otimes \s^{E}_{x_0}
\end{equation}
which is canonically defined upon choosing an orientation for $S^{1}$. \qed
 \end{lem} 
By taking the sum of all the maps $\Delta_{u}$, we obtain a map
\begin{equation}
  \Delta \co CF^{*}(H; \s^{E}) \to CF^{*-1}(H; \s^{E}).
\end{equation}
To check that this defines a chain map, consider a pair $x_{0}$ and $x_2$ of orbits with the same Conley-Zehnder index.  The moduli space $\Cyl_{\Delta}(x_0,x_2)  $ has dimension $1$, and Gromov-Floer compactness implies that it may be compactified to a manifold whose boundary is given by two types of strata: (i)  pairs of cylinders $u  \in   \Cyl_{\Delta}(x_0,x_1)  $ and $v   \in   \Cyl(x_1,x_2) $ and (ii) pairs of cylinders $u  \in   \Cyl(x_0,x_1)  $ and $v   \in   \Cyl_{\Delta}(x_1,x_2) $.

\begin{lem}
At the level of homology, the count of solutions to Equation \eqref{eq:continuation_equation} defines a degree $-1$ map
\begin{equation}
\Delta  \co  HF^{*}(H; \s^{E}) \to HF^{*-1}(H; \s^{E}).
\end{equation} \qed
\end{lem}

\subsection{The square of the $BV$ operator}
\label{sec:square-bv-operator}

In this section, we prove that $\Delta^{2} =0$. The key idea is that the map $\Delta^{2}$ is associated to a family of  equations parametrised by $S^1 \times S^1$. We shall construct a family of equations parametrised by an oriented $3$-manifold with boundary $S^1 \times S^1$ (in fact, a solid torus); the counts of rigid elements in this new parametrised space give a null homotopy for $\Delta^{2}$ at the cochain level.

We now implement this strategy: fix the family of Hamiltonians  $ H^{\theta}_{s,t} $ for the definition of the $BV$ operator, and choose a sufficiently large real number $S$ so that $ H^{\theta}_{s,t}  $ is independent of $s$ along the ends $S \leq s$ and $s \leq -S$. By gluing, we define a family of Hamiltonians $H^{\theta_1,\theta_2}_{s,t}$ parametrised by $(\theta_{1},\theta_{2}) \in S^{1} \times S^{1}$:
\begin{equation}
  H^{\theta_1,\theta_2}_{s,t} = \begin{cases} H^{\theta_1}_{s-S,t} & \textrm{ if } 0 \leq s \\
H^{\theta_2}_{s+S,t+\theta_{1}} & \textrm{ if } s \leq 0.
\end{cases}
\end{equation}
\begin{exercise}
  Show that $H^{\theta_1,\theta_2}_{0,t} = H_{t+\theta_{1}}$, and that 
  \begin{equation}
    H^{\theta_1,\theta_2}_{s,t} = H_{t+\theta_{1}+\theta_{2}} \textrm{ if } s \ll 0.
  \end{equation}
\end{exercise}
We may construct a parametrised family $J^{\theta_{1},\theta_{2}}_{s,t}$ of almost complex structures on $\TQ$ in the same way.

For a pair $(x_0,x_2)$ of orbits, let $\Cyl_{\Delta^{2}}(x_0,x_2)$ denote the moduli space of solutions to this family of equations for a fixed $S$ which we omit from the notation.  Gluing theory implies the following:
\begin{lem}
If $S$ is large enough, all elements of the moduli space $ \Cyl_{\Delta^{2}}(x_0,x_2) $ are regular. Moreover, if $\deg(x_0) = \deg(x_2) -2$, there is a canonical bijection
\begin{equation}
  \Cyl_{\Delta^{2}}(x_0,x_2) \cong \coprod_{x_1}   \Cyl_{\Delta}(x_0,x_1) \times  \Cyl_{\Delta}(x_1,x_2).
\end{equation} \qed
\end{lem}

The above result implies that $\Delta^{2}$ is obtained by counting rigid elements of $\Cyl_{\Delta^{2}}(x_0,x_2)  $. The next step is to build a null cobordism for this moduli space.

To this end, we choose a family of Hamiltonians $H^{\theta_1,\theta_2,\tau}_{s,t}$, with $\tau \in [0,1]$, subject to the following constraints:
\begin{enumerate}
\item For $\tau = 1$, $H^{\theta_1,\theta_2,1}_{s,t} =H^{\theta_1,\theta_2}_{s,t} $.
\item When $\tau = 0$, $ H^{\theta_1,\theta_2,0}_{s,t}  = H^{\theta_1 + \theta_2}_{s,t} $ .
\item The restriction of $ H^{\theta_1,\theta_2,\tau}_{s,t} $  to $0 \ll s$ agrees with $H_{t}$, and its restriction to $s \ll 0$ agrees with  $H_{t + \theta_1 + \theta_2}$.
\end{enumerate}
Choosing almost complex structures $J^{\theta_1,\theta_2,\tau}_{s,t} $ satisfying similar conditions, we obtain a family of Floer equations parametrised by $S^{1} \times S^{1} \times [0,1]$. Since the equation at $\tau=0$ depends only the sum $\theta_1 + \theta_2$, we can collapse the boundary  component $S^{1} \times S^{1} \times \{ 0 \}$ to a circle, and produce a family of equations parametrised by the solid torus.

Imposing the asymptotic conditions in Equation \eqref{eq:asymptotic_conditions_Delta}, we obtain a parametrised moduli space of solutions $\Cyl^{\cN}(x_0,x_3)  $ for every pair $(x_0,x_3)$ of orbits. For generic choices of data, $\Cyl^{\cN}(x_0,x_3)   $ is a manifold with boundary of dimension
\begin{equation}
  \deg(x_0) - \deg(x_3) + 3.
\end{equation}
Whenever $\deg(x_0) = \deg(x_3) - 3 $, we conclude that this manifold is therefore $0$-dimensional. Choosing an orientation of the parametrising solid torus, we can associate to every element of this space a map on the lines which generate the Floer complex. Define
\begin{equation}
    \cN \co CF^{*}(\TQ, H; \s^{E} ) \to CF^{*-3}(\TQ, H; \s^{E} )
\end{equation}
to be sum of all maps induced by elements of $ \Cyl^{\cN}(x_0,x_3)  $. The proof of the following result is left to the reader.
\begin{lem}
  The map $\cN$ is a null-homotopy for $\Delta^{2}$.  \qed
\end{lem}

\subsection{Continuation maps commute with the $BV$ operator} \label{sec:cont-maps-comm}

In order to show that $\Delta$ descends to an operation on symplectic cohomology, we must show that it commutes with continuation maps. The argument for Proposition \ref{prop:comp-cont-maps} can be repeated essentially verbatim in this context by  fixing the following: 
\begin{enumerate}
\item Floer data $(H^{\pm}_{t}, J^{\pm}_{t}  )$  such that  $H^{+} \preceq H^{-}$ (see Equation \eqref{eq:preorder_Hamiltonians} for the definition of $\preceq$). These define the Floer cohomology groups $HF^{*}(H^{\pm}; \s^{E})   $.
\item Families $(H^{\pm,\theta}_{s,t}, J_{s,t}^{\pm,\theta})$ parametrised by $S^1$ defining degree $-1$ operators $\Delta^{\pm}$ on the Floer groups $HF^{*}(H^{\pm}; \s^{E})  $.
\item Continuation map data $(H_{s,t}, J_{s,t})$ defining a map $\cont \co HF^{*}(H^{+}; \s^{E})  \to HF^{*}(H^{-}; \s^{E}) $.
\end{enumerate}

We now choose a family of Hamiltonians $H^{r,\theta}_{s,t}$ parametrised by $(\theta,r) \in S^{1} \times \bR$ and $(s,t) \in Z$ subject to the following constraints:
\begin{alignat}{2}\label{eq:properties_Hamiltonian_htpy_k_delta-1}
 &\parbox{16em}{For each $r$, there exists $S_{r}$ such that} 
\, \, H^{\theta,r}_{s,t}  = \begin{cases} H^{+}_{t} & \textrm{ if } S_{r} \leq s \\
  H^{-}_{s,t+\theta} & \textrm{ if } s \leq - S_{r}.
\end{cases} \\
 &\parbox{35em}{The slope of the Hamiltonians $ H^{\theta,r}_{s,t}  $ decreases monotonically in $s$.} \\
& \parbox{35em}{If $r \ll 0$, then $ H^{\theta,r}_{s,t}  $ is obtained by concatenating $ H^{+,\theta}_{s,t} $ and $ H_{s,t+\theta} $.} \\ \label{eq:properties_Hamiltonian_htpy_k_delta-4}
& \parbox{35em}{ If $0 \ll r$, then $ H^{\theta,r}_{s,t}  $ is obtained by concatenating $ H_{s,t} $ with $ H^{-,\theta}_{s,t} $.}
\end{alignat}
The notion of concatenating Hamiltonians is the same as the one discussed in the proof of Proposition \ref{prop:comp-cont-maps}. We give an explicit construction when $r \ll 0$: recall that the restriction of $H_{s,t} $  to the region where $s \ll 0$ is independent of $s$, and agrees with $H^{+}_{t}  $, which also agrees with the restriction of $H^{+,\theta}_{s,t-\theta}$ to the region where $0 \ll s $. We can therefore construct a new Hamiltonian by the formula:
\begin{equation} \label{eq:floer_data_negative_r}
  H^{\theta,r}_{s,t} = \begin{cases}  H_{s -2r,t+\theta} & \textrm{ if } s \leq r \\
  H^{+,\theta}_{s,t} & \textrm{ if } r \leq s.
 \end{cases}
\end{equation}
 In order to know that a family of Hamiltonian satisfying these properties exist, the key step is to ensure that they are consistent:
\begin{exercise}
Check that the family of Hamiltonians defined by Equation \eqref{eq:floer_data_negative_r} satisfies Equation \eqref{eq:properties_Hamiltonian_htpy_k_delta-1}.
\end{exercise}
Since Conditions \eqref{eq:properties_Hamiltonian_htpy_k_delta-1}-\eqref{eq:properties_Hamiltonian_htpy_k_delta-4} are convex and local in $Z$, we can in fact produce an explicit family as follows:  define $H^{\theta,r}_{s,t}  $ by gluing whenever $r \leq -R$ or $R \leq r$. Next, choose a partition of unity $(\chi_{+}, \chi_{-})$ on $\bR$ subordinate to the cover $(-R,+\infty)$ and $(-\infty,R)$. We  define
\begin{equation}
 H^{\theta,r}_{s,t} = \chi_{+}(r) H^{\theta,R}_{s,t} + \chi_{-}(r) H^{\theta,-R}_{s,t} \textrm{ whenever } r \in [-R,R],
\end{equation}
and use the gluing definition away from this region.  Note that this produces a continuous family, which is everywhere smooth except possibly where $r = \pm R$; this is sufficient for our purpose, but the punctilious reader is invited to use a further cutoff construction at these points to construct a smooth family.

Next, we choose a  family $J^{r,\theta}_{s,t}$ of almost complex structures on $\TQ$ (which are convex near $\SQ$) subject to the constraints:
\begin{alignat}{2}\label{eq:properties_complex_structure_htpy_k_delta-1}
 &\parbox{16em}{For each $r$, there exists $S_{r}$ such that } 
\, \, J^{r,\theta}_{s,t}  = \begin{cases} J^{+}_{t} & \textrm{ if } S_{r} \leq s \\
  J^{-}_{s,t+\theta} & \textrm{ if } s \leq - S_{r}.
\end{cases} \\
& \parbox{35em}{If $r \ll 0$, then $ J^{r,\theta}_{s,t}  $ is obtained by concatenating $ J^{+,\theta}_{s,t} $ and $ J_{s,t+\theta} $.} \\ \label{eq:properties_complex_structure_htpy_k_delta-3}
& \parbox{35em}{ If $0 \ll r$, then $ J^{r,\theta}_{s,t}  $ is obtained by concatenating $ J_{s,t} $ with $ J^{-,\theta}_{s,t} $.}
\end{alignat}

Given the families $(H^{r,\theta}_{s,t}, J^{r,\theta}_{s,t})$, we define $  \tilde{\Cyl}_{\Delta}(x_-;x_+) $ to be the moduli of triples $(\theta, r,u)$, solving the differential equation
\begin{equation}
  J_{s,t}^{\theta,r}  \frac{\partial u}{\partial s} = \frac{\partial u}{\partial t} - X_{H_{s,t}}^{\theta,r} 
\end{equation}
and satisfying the asymptotic conditions
\begin{align}
  \lim_{s \to +\infty}u(s,t) & = x_{+}(t) \\
 \lim_{s \to -\infty}u(s,t) & = x_{-}(t+\theta).
\end{align}

By construction, the Gromov compactification of  $   \tilde{\Cyl}_{\Delta}(x_-;x_+)  $ admits two boundary strata, corresponding to the limits $r \to \pm \infty$, which are the disjoint union of products
\begin{alignat}{2}
  & \coprod_{x \in \Orbit(H^{-})} \Cyl_{\Delta}(x_-;x) \times \Cont(x;x_+)  \\
  & \coprod_{x \in \Orbit(H^{+}) } \Cont(x_-;x) \times \Cyl_{\Delta}(x;x_+).
\end{alignat}
If we assume that $\deg(x_-) = \deg(x_+) -1 $, these strata respectively correspond to the compositions $\cont \circ \Delta$ and $\Delta \circ \cont$. Applying the same method as in the proof of Proposition \ref{prop:comp-cont-maps}, we conclude
\begin{lem} \label{lem:cont_map_commute_Delta}
Whenever $H^{+} \preceq H^{-}$, we have a commutative diagram:
\begin{equation}
  \xymatrix{  HF^{*}(H^{+}; \s^{E}) \ar[r]^-{\Delta} \ar[d]^{\cont} &  HF^{*-1}(H^{+}; \s^{E})  \ar[d]^{\cont}\\
HF^{*}(H^{-}; \s^{E}) \ar[r]^-{\Delta}  &  HF^{*-1}(H^{-}; \s^{E}).   }
\end{equation}
\qed
\end{lem}
This result completes the construction of the operator $\Delta$ on symplectic cohomology (Equation \eqref{eq:BV_on_SH}): given an element $a \in SH^{*}(\TQ; \s^{E})$, we choose a representative $a_{H} \in HF^{*}(H; \s^{E})  $, and define $\Delta(a)$ to be the image of $\Delta(a_{H})$ under the natural map
\begin{equation}
  HF^{*}(H; \s^{E}) \to  SH^{*}(\TQ; \s^{E}).
\end{equation}
Since any two representatives are related by applying the continuation map, $ \Delta(a) $ is in fact independent of this choice.

\section{The pair of pants product} \label{sec:pair-pants-product}
In this section, we use moduli spaces of pairs of pants to define a map
\begin{equation} \label{eq:star-product}
  SH^{i}(\TQ; \s^{E}) \otimes SH^{j}(\TQ; \s^{E}) \to SH^{i+j}(\TQ; \s^{E})
\end{equation}
which we denote $\star$. The following result summarises the properties proved in this section:
\begin{prop} \label{prop:star-Leibnitz}
The  product $\star$ is associative and twisted graded-commutative.  

\end{prop}

\subsection{Moduli spaces of pairs of pants} \label{sec:moduli-spaces-pairs}
Let $z_0$, $z_1$, and $z_2$ be three distinct points on the Riemann sphere $\bC \bP^{1}$.   Write $P$ for the open Riemann surface (a \emph{pair of pants}) obtained by removing these three points from $\bC \bP^{1}$.  We shall study maps
\begin{equation}
  P \to \TQ
\end{equation}
which converge at the ends to Hamiltonian orbits.  In order to make this notion of convergence more precise, we consider the  \emph{positive half cylinder} 
\begin{equation}
  Z^{+} \equiv  [0,+\infty) \times S^{1}
\end{equation}
as a Riemann surface with boundary equipped with coordinates $(s,t)$ and complex structure
\begin{equation}
  j \partial_{s} = \partial_{t} .
\end{equation}

The following notion will allow us to specify  a convenient family of Cauchy-Riemann equations on maps with source $P$:
\begin{defin}
A \emph{positive cylindrical end} near $z_i$ is a holomorphic embedding
\begin{equation}
  \epsilon_{i} \co Z^{+} \to P
\end{equation}
which extends to a map from the disc to $\bC \bP^{1}$ taking the origin to $z_i$. A \emph{negative cylindrical end} is one whose source is the negative half cylinder $Z^{-} \equiv  (-\infty,0 ] \times S^{1} $.
\end{defin}

Let us once and for all fix a negative end near $z_0$ and positive ends near $z_1$ and $z_2$, and equip $P$ with a metric which restricts to the product (cylindrical) metric on the ends whenever $0 \ll |s| $. Whatever equation is imposed on maps on $P$  should have the property that its finite energy solutions converge  along the cylindrical ends of $P$ to Hamiltonian orbits.   The natural way to ensure this is to require that its pullback under $\epsilon_{i}$ agree with Floer's equation (Equation \eqref{eq:dbar_equation_s-indep}), whenever $|s|$ is sufficiently close to infinity.  There are three data that are needed in order to define Floer's equation:
\begin{enumerate}
\item a   $1$-form on $Z^{+}$ or $Z^{-}$ which we fix to be $dt$,
\item a linear  Hamiltonian $H^{i}$ of slope $b_i$, and
\item an $S^{1}$ family of almost complex structures $J^{i}$ which are convex near $\SQ$.
\end{enumerate}

To define a Cauchy-Riemann equation on the thrice-punctured sphere, we choose:
\begin{enumerate}
\item a $1$ form $\alpha$ whose pullback under $\epsilon_{i}$ is $dt$,
\item a family $H^{P}$ of linear Hamiltonians, parametrised by $P$, such that  $ H^{P}_{\epsilon_{i}(s,t)} = H^{i}_{t}$, and
\item  a family $J^{P}$ of almost complex structures which are convex near $\SQ$, parametrised by $P$, such that $J^{P}_{\epsilon_{i}(s,t)}  = J^{i}_{t}$.
\end{enumerate}
We write $b^{P} \co P \to \bR$ for the function whose value at $z$ is the slope of $H^{P}_{z}$. We require that the form $\alpha$ and the family of slopes satisfy the following property:
\begin{equation} \label{eq:subclosed_form-pop}
  d \left(b^{P} \cdot \alpha \right) \leq 0.
\end{equation}
\begin{rem}
  The negativity in the above equation has the same meaning as in Equation \eqref{eq:negativity}, and we shall use it for the same purpose: proving compactness for the moduli space of cylinders.
\end{rem}
\begin{exercise}
Show that Condition \eqref{eq:subclosed_form-pop} implies that
\begin{equation} \label{eq:sum_hamiltonians_slope_inequality}
H^{1} + H^{2} \preceq H^{0},
\end{equation}
where $\preceq$ is defined in Equation \eqref{eq:preorder_Hamiltonians}. (Hint: use Stokes's theorem on the complement of a small neighbourhood of the punctures in $P$). 

Conversely, show that, if this condition is satisfied, one can find a $1$-form $\alpha$ and a family $H^{P}$ of Hamiltonians so that Equation \eqref{eq:subclosed_form-pop} holds. 
\end{exercise}

With these choices, we can define the moduli spaces whose count determines the product on symplectic cohomology.
\begin{defin}
 The \emph{moduli space} of holomorphic pairs of pants $\Pants(x_0; x_1,x_2)$ with input a pair of Hamiltonian orbits $x_1$ and $x_2$ of $H^{1}$ and $H^{2}$ and output a Hamiltonian orbit $x_0$ of $H^{0}$ is the space of finite energy maps
 \begin{equation}
   u \co P \to \TQ
 \end{equation}
which satisfy the Cauchy-Riemann equation
\begin{equation} \label{eq:CR-pair-pants}
  (du - \alpha \otimes X_{H^{P}})^{0,1} =0
\end{equation}
and converge along the end $\epsilon_{i}$ to the orbit $x_i$, i.e.
\begin{equation}
  \lim_{|s| \to +\infty} u (\epsilon_{i}(s,t)) = x_{i}(t).
\end{equation}
\end{defin}
\begin{rem}
Equation \eqref{eq:CR-pair-pants} is the coordinate-free way of writing the type of Cauchy-Riemann equations, with inhomogeneous term, which usually appear in Floer theory. It is equivalent to the equation
  \begin{equation}
   J^{P}_{z} \left(du(\xi) - \alpha(\xi) \cdot X_{H^{P}_{z} }\right) = du(j\xi) - \alpha(j\xi) \cdot X_{H^{P}_{z} }
  \end{equation}
for every tangent vector $\xi$ of $P$ at a point $z$. The finite energy condition is the requirement that the integral
\begin{equation}
  \int_{P} \| du - \alpha \otimes X_{H^{P}} \|^{2}
\end{equation}
be finite.
\end{rem}
\begin{exercise}
 Compute the restriction of Equation \eqref{eq:CR-pair-pants} to the ends in terms of the $(s,t)$ coordinates.
\end{exercise}

\subsection{Compactness of the moduli space of pairs of pants} \label{sec:comp-moduli-space}
In order to produce operations from the moduli space $\Pants(x_0; x_1,x_2)$, we need to construct a compactification. The candidate for such a compactification is:
\begin{equation}
   \Pantsbar(x_0;x_1,x_2) = \coprod_{x'_0,x'_1,x'_2}   \Cylbar(x_0;x'_0) \times  \Pants(x'_0;x'_1,x'_2) \times \Cylbar(x'_1;x_1) \times \Cylbar(x'_2;x_2). 
\end{equation}
We shall not give a proof of the following result, which is a minor variant of Gromov and Floer's original compactness results. In a sense, the result we need is a fibre product of the case of Floer's equation on the cylinder, discussed in \cite{AD}*{Chapitre 6.6}, and pseudo-holomorphic equations on general Riemann surfaces, which is the subject of  \cite{MS}*{Chapter 4}:
\begin{lem}
  If $u_i$ is a sequence of elements of $  \Pants(x_0;x_1,x_2)  $, and there is a compact set of $\TQ$ which contains the image of all $u_i$, then there is a subsequence converging to an element of $   \Pantsbar(x_0;x_1,x_2)  $.  \qed
\end{lem}
Our main task in this section is therefore to show that all elements of $\Pantsbar(x_0;x_1,x_2)  $ are contained in a fixed compact set:
\begin{lem}
The image of every element of $ \Pantsbar(x_0;x_1,x_2)   $ is contained in the unit disc cotangent bundle.
\end{lem}
\begin{proof}[Sketch of proof:]
  This is the same argument as for Lemma \ref{lem:maximum_principle_cylinders}, so we shall quickly repeat it: given  $u \in \Pantsbar(x_0;x_1,x_2)$, let $\Sigma$ denote the inverse image of the complement of the (open) disc cotangent bundle of  radius $1+\epsilon$ such that the corresponding sphere bundle is transverse to $u$; note that $\epsilon$ can be chosen to be arbitrarily small. Let $v$ denote the restriction of $u$ to $\Sigma$.

Lemma \ref{lem:monotone_implies_positive_energy} implies that the energy $E(v)$ is non-negative. Note that the hypothesis of this Lemma holds because Equation \eqref{eq:subclosed_form-pop} yields Equation \eqref{eq:negativity}. On the other hand, an application of Stokes's theorem as in Lemma \ref{lem:maximum_principle_cylinders} shows that this integral is strictly negative if $\Sigma$ is not empty. Having reached a contradiction, we conclude the desired result.
\end{proof}

\subsection{Virtual dimension of the moduli space of pairs of pants}
Given an element $u \in \Pantsbar(x_0;x_1,x_2)$, the linearisation of Equation \eqref{eq:CR-pair-pants} gives an operator
\begin{equation} \label{eq:linearisation_pair_of_pants}
  D_{u} \co W^{1,p}(P, u^{*}( T \TQ) ) \to  L^{p}(P, u^{*}( T \TQ) \otimes \Omega^{0,1}(P) ),
\end{equation}
where $\Omega^{0,1}(P)  $ is the vector bundle of anti-holomorphic $1$-forms on $P$, and the Sobolev spaces are defined with respect to a metric on $P$ which is cylindrical on the ends (see \cite{Schwarz}*{Chapter 2}).
\begin{rem}
As a vector bundle, $\Omega^{0,1}(P)  $ is trivial, so we can choose a trivialisation of this bundle, and rewrite the above linearisation as an operator
\begin{equation}
  W^{1,p}(P, u^{*} T \TQ ) \to  L^{p}(P, u^{*} T \TQ ).
\end{equation}
In this form, this is the same as Equation \eqref{eq:linearisation}, but, unlike the case of the cylinders, there is no natural choice of trivialisation of $\Omega^{0,1}(P)  $, so we prefer the invariant form given in Equation \eqref{eq:linearisation_pair_of_pants}.
\end{rem}
\begin{defin}
The \emph{virtual dimension} of $\,\Pantsbar(x_0;x_1,x_2) $  is the index of the operator $D_{u}$.
\end{defin}

In order to express the virtual dimension in terms of the index of $x_i$, choose a trivialisation $F$ of $u^{*} T \TQ $. By restricting to neighbourhoods of the three punctures, we obtain trivialisations of $x_{i}^{*}(T \TQ)$; these do not necessarily agree with the trivialisations fixed in Lemma \ref{lem:trivalisation_up_to_htpy}. We write $D_{\Psi^{F}_{x_i}}$ for the operator on $\bC$ (with values in $\bC^{n}$) associated in Section \ref{sec:invar-paths-sympl} to the path of unitary matrices defined by the differential of the Hamiltonian flow (see Equation \eqref{eq:path_matrices_global}).

Since the pullback of Equation \eqref{eq:CR-pair-pants} by $\epsilon_{i}$ agrees, whenever  $|s|$ is sufficiently large, with the Floer equation (Equation \eqref{eq:CR-operator-Floer}), the pullback of its linearisation $D_{u}$ agrees with the restriction of $D_{\Psi^{F}_{x_i}}$   to the complement of a compact set in the plane, as in Equation \eqref{eq:operator_orbits}.  
\begin{figure}[h]
  \centering
  \includegraphics{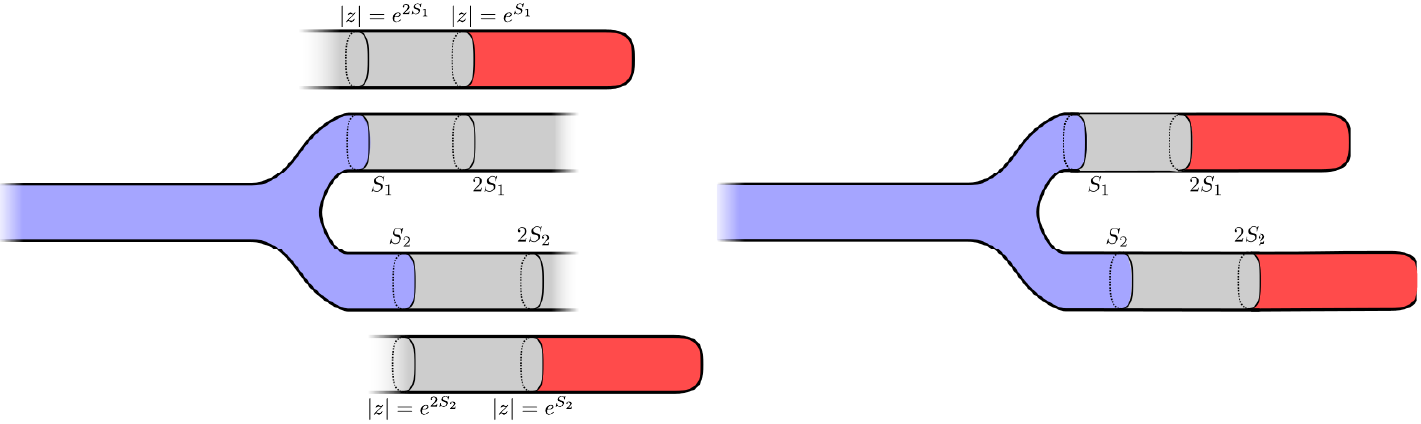}
  \caption{ }
  \label{fig:gluing_pop}
\end{figure}

We shall now show that the result of \emph{gluing} $D_{\Psi^{F}_{x_1}}$ and $D_{\Psi^{F}_{x_2}}$ to $D_{u}$ defines an operator on the disc which is homotopic to $D_{\Psi^{F}_{x_0}}$.  The first step is to consider two positive real numbers $S_{1}$ and $S_{2}$, and define a Riemann surface $P_{S_1,S_2}$ by removing $\epsilon_{1}\left([2S_1,+\infty) \times S^{1}\right)$ and $\epsilon_{2}\left([2S_2,+\infty) \times S^{1}\right)$ and gluing the complement to discs in $\bC$ of radii $e^{2S_1}$ and $e^{2S_2}$ along the identification shown in Figure \ref{fig:gluing_pop}. The formula is: 
\begin{align}
\{ z | e^{S_j} \leq |z| \leq e^{2S_j} \} & \cong \epsilon_{j}([S_j,2S_j] \times S^{1}) \\
e^{s + 2 \pi i \theta} & \mapsto  \left( 3 S_j - s, -\theta \right).
\end{align}
By construction, we have natural inclusions in $ P_{S_1,S_2} $ of the discs of radii $e^{2S_1}$ and $e^{2S_2}$ (coming from the domains of definition for $D_{\Psi^{F}_{x_1}}$ and $D_{\Psi^{F}_{x_2}}$) as well as of the complement of $\epsilon_{1}([2S_1,+\infty) \times S^{1}$ and $\epsilon_{2}([2S_2,+\infty) \times S^{1}$ in $P$.

Since the operators $D_{u}$ and $D_{\Psi^{F}_{x_i}}$ are local, the fact that $D_u$ agrees with $D_{\Psi^{F}_{x_i}}$ near the gluing region implies that we obtain a glued operator
\begin{equation}
D_{u } \#_{S_1} D_{\Psi^{F}_{x_1}} \#_{S_2}  D_{\Psi^{F}_{x_2}}  \co    W^{1,p}(P_{S_1,S_2}, \bC^{n}) \to L^{p}(P_{S_1,S_2}, \bC^{n} \otimes \Omega_{P_{S_1,S_2}}^{0,1})
\end{equation}
which is a Cauchy-Riemann operator.
\begin{lem}
There is a canonical isomorphism of determinant lines
\begin{equation} \label{eq:iso_glued_det}
  \det( D_{u } \#_{S_1} D_{\Psi^{F}_{x_1}} \#_{S_2}  D_{\Psi^{F}_{x_2}}  ) \cong \det(D_u) \otimes \det(D_{\Psi^{F}_{x_1}}) \otimes \det(D_{\Psi^{F}_{x_2}}) .
\end{equation}
\end{lem}
\begin{proof}[Sketch of proof]
Just as one glues two planes to $P$ along cylindrical like ends, one may use cutoff functions to glue  $\bC^{n}$-valued functions.  If we assume, for simplicity, that $ D_{u }  $, $D_{\Psi^{F}_{x_1}}$and $D_{\Psi^{F}_{x_2}}$ are all surjective, then the key point is to check that gluing elements in the kernel of these operators produces functions on $ P_{S_1,S_2} $  which are approximate zeroes of $ D_{u } \#_{S_1} D_{\Psi^{F}_{x_1}} \#_{S_2}  D_{\Psi^{F}_{x_2}} $. This defines a map
\begin{equation} \label{eq:glue_kernels_pop}
  \ker(D_{u}) \oplus \ker( D_{\Psi^{F}_{x_1}} ) \oplus \ker( D_{\Psi^{F}_{x_2}} ) \to \ker( D_{u } \#_{S_1} D_{\Psi^{F}_{x_1}} \#_{S_2}  D_{\Psi^{F}_{x_2}}  ).
\end{equation}
 In the opposite direction, one may start with an element of the kernel of  $ D_{u } \#_{S_1} D_{\Psi^{F}_{x_1}} \#_{S_2}  D_{\Psi^{F}_{x_2}}   $, and, using cutoff functions, break it into maps on $\bC$ and on $P$, which are approximate zeroes for the respective operators  $ D_{u }$, $ D_{\Psi^{F}_{x_1}}$ and $ D_{\Psi^{F}_{x_2}}  $.

The gluing and breaking operators can be shown to be approximate inverses, which implies (using an implicit function theorem) the existence of an isomorphism between the kernels of $ D_u \oplus D_{\Psi^{F}_{x_1}} \oplus D_{\Psi^{F}_{x_2}} $ and that of $ D_{u } \#_{S_1} D_{x_1} \#_{S_2}  D_{\Psi^{F}_{x_2}}  $, and hence an isomorphism of determinant lines.  The case in which the operators are not surjective is treated by adding large enough finite dimensional subspaces to the source so that surjectivity holds.
\end{proof}

Note that the surface   $P_{S_1,S_2}$  is diffeomorphic to the plane, and carries a cylindrical end coming from the end $\epsilon_{0}$.  In this region,  $ D_{u } \#_{S_1} D_{\Psi^{F}_{x_1}} \#_{S_2}  D_{\Psi^{F}_{x_2}}  $  agrees with $D_{x_0}$ since both share the same asymptotic conditions as $D_{u}$.  We conclude that $ D_{u } \#_{S_1} D_{\Psi^{F}_{x_1}} \#_{S_2}  D_{\Psi^{F}_{x_2}}  $  and $D_{\Psi^{F}_{x_0}}$ may be joined by a path of Fredholm operators.  There is a unique such path up to homotopy which induces an isomorphism
\begin{equation} \label{eq:glued_pop_operator_htpic_output}
  \det( D_{u } \#_{S_1} D_{\Psi^{F}_{x_1}} \#_{S_2}  D_{\Psi^{F}_{x_2}} ) \cong \det (D_{\Psi^{F}_{x_0}}   ).
\end{equation}
From Equation \eqref{eq:glued_pop_operator_htpic_output}, we obtain an isomorphism
\begin{equation} \label{eq:gluing_isomorphism_det_line_pop}
  \det(D_{u}) \otimes  \det(D_{\Psi^{F}_{x_1}}) \otimes \det(D_{\Psi^{F}_{x_2}}) \cong \det (D_{\Psi^{F}_{x_0}}   ).
\end{equation}
Using the additivity of index under gluing, we conclude that
\begin{equation} \label{eq:index_formula}
  \ind(D_{u}) = \ind( \Psi^{F}_{x_0})  -   \ind( \Psi^{F}_{x_2} ) - \ind( \Psi^{F}_{x_1} ) ,
\end{equation}
which implies:
\begin{prop}
The virtual dimension of  the moduli space $ \Pants(x_0;x_1,x_2)   $  is 
\begin{equation} \label{eq:virtual_dimension_moduli_pants}
  \deg(x_0) - \deg(x_1) - \deg(x_2).
\end{equation}
\end{prop}
\begin{proof}
Recall that the definition of $\deg(x)$ depends on whether $x^{*}(T\Q)$ is orientable. Moreover, since the monodromy of $x_{0}^{*}(\det(T\Q) )$ is the product of the monodromies of  $x_{1}^{*}( \det(T\Q))$ and $x_{2}^{*}( \det(T\Q))$, we see that there are three cases to consider: (1)  $x_{1}^{*}( \det(T\Q))$ and $x_{2}^{*}( \det(T\Q))$ are both non-orientable (2) exactly one of $x_{1}^{*}( \det(T\Q))$ and $x_{2}^{*}( \det(T\Q))$  is orientable, or (3) $x_{i}^{*}( \det(T\Q))$ is orientable for every $i$.

Up to a constant term, the index of $x_i$ is defined to be $  \ind( \Psi^{F}_{x_i} ) $ for a preferred choice of trivialisation coming from Lemma \ref{lem:trivalisation_up_to_htpy}. In cases  (2) and (3)  one can in fact choose the trivialisation $F$ of $ u^{*} T \TQ  $ so that its restriction to each orbit $x_i$ is the preferred trivialisation; in these cases, Equation \eqref{eq:virtual_dimension_moduli_pants} is an immediate consequence of Equation \eqref{eq:index_formula}.

In case (1), recall that, for $i=\{1,2\}$,  $\deg(x_i)$ is defined by choosing a trivialisation of $x_{i}^{*}(T \TQ)$ with the property that the induced trivialisation of $ x_{i}^{*}\det( \TQ \otimes_{\bR} \bC ) $ gives rise to a map $S^1 \to \bR \bP^{1}$ of degree $1$; this map records the image of the real subspace $\det(\TQ) \subset  \det_{\bC}( \TQ \otimes_{\bR} \bC) $.

A trivialisation of $u^{*}(\TQ \otimes_{\bR} \bC)$ is determined by its restrictions to trivialisations of $ x_{1}^{*}\det( \TQ \otimes_{\bR} \bC ) $ and $ x_{2}^{*}\det( \TQ \otimes_{\bR} \bC ) $. To see this, recall that $P$ has the homotopy type of a $1$-dimensional complex, so every complex bundle over it is trivial. Moreover, the restriction to two loops surrounding the first and second punctures yields an isomorphism $H^{1}(P,\bZ ) \cong H^{1}(S^{1},\bZ ) \oplus H^{1}(S^{1},\bZ) $. Since the space of trivialisations of a complex bundle over $P$  is an affine space over $H^{1}(P,\bZ )$, we conclude that we can indeed choose a unique trivialisation of a bundle over $P$ which restricts to the desired ones over the two circles at infinity.

Let $F$ denote the trivialisation of $u^{*}(\TQ \otimes_{\bR} \bC)  $ which restricts to the preferred trivialisations of $ x_{1}^{*}(\TQ \otimes_{\bR} \bC)  $ and $ x_{2}^{*}(\TQ \otimes_{\bR} \bC)  $. Equation \eqref{eq:index_formula} can be re-written as
\begin{equation} \label{eq:index_in_inputs_trivialisation}
 \ind(D_{u}) =  \ind( \Psi^{F}_{x_0}) - 2 -   \deg(x_2) - \deg( x_1). 
\end{equation}

The main observation, at this point, is that  we can associate a map
\begin{equation}
  P \to \bR \bP^{1}
\end{equation}
to a trivialisation of $ u^{*}(\TQ \otimes_{\bR} \bC) $, which assigns to a point in $P$ the image of the real line $u^{*}(\TQ)$. In particular, the degree of the restriction of this map to a loop surrounding the puncture $z_0$ is the sum of the degrees of the restriction to loops surrounding $z_1$ and $z_2$.  We conclude that the restriction of $F$ to $x_0$ gives a Gauss map of degree $2$. On the other hand, the Gauss map has degree $0$ for the preferred trivialisation of $x_{0}^{*}(\TQ \otimes_{\bR} \bC) $. In particular,  the loop $\Phi$ of unitary transformations which takes the restriction of $F$ to $x_0$ to the preferred trivialisation satisfies $\rho(\Phi) = -1$.

From Equation \eqref{eq:degree_change_trivialisation}, we conclude that
\begin{equation}
\deg(x_0) = \ind( \Psi^{F}_{x_0}) + 2.
\end{equation}
This, together with Equation \eqref{eq:index_in_inputs_trivialisation}, implies that the virtual dimension of $  \Pants(x_0;x_1,x_2) $ is indeed given by Equation \eqref{eq:virtual_dimension_moduli_pants}.
\end{proof}
At the level of determinant lines, we conclude:
\begin{lem}
 There is a canonical isomorphism of $\bZ$-graded lines:
 \begin{equation} \label{eq:iso_determinant_lines}
\det(D_{u}) \otimes \delta_{x_1}[w(x_1)] \otimes \delta_{x_2}[w(x_2)] \cong \delta_{x_0}[w(x_0)].
 \end{equation}
\end{lem}
\begin{proof}
The fact that the two graded lines in Equation \eqref{eq:iso_determinant_lines} are supported in the same integral degree is a straightforward computation from  Formula \eqref{eq:virtual_dimension_moduli_pants} for the virtual dimension of the moduli space.

To prove the existence of a canonical isomorphism, recall that $\delta_{x}$ is defined to be the determinant line associated to the path of symplectomorphisms obtained from the canonical trivialisation of $x^{*}(T\TQ)$ constructed in Lemma \ref{lem:trivalisation_up_to_htpy}.  Lemma \ref{lem:orient_line_indep_path},  implies that the complex orientations on $\bC^{n}$ and on the determinant of a Cauchy-Riemann operator on a complex bundle over $\bC \bP^{1}$ induce a canonical isomorphism of determinant lines for paths associated to different trivialisations. We obtain an isomorphism
\begin{equation}
  \delta_{x_i} \cong \det(D_{\Psi^{F}_{x_i}}).
\end{equation}
Equation \eqref{eq:iso_determinant_lines} therefore follows from Equation \eqref{eq:gluing_isomorphism_det_line_pop}.
\end{proof}

The standard transversality techniques (see \cite{FHS}) imply the following result:
\begin{lem} \label{lem:pop-gen-regular}
For generic data  $H^{P}$ and $J^{P}$, the moduli spaces $ \Pants(x_0;x_1,x_2)   $  are regular for all triples of orbits $\{x_i\}$, and this space is a smooth manifold of dimension
\begin{equation}
  \deg(x_0) - \deg(x_1) - \deg(x_2).
\end{equation} \qed
\end{lem}

\subsection{Definition of the product}
For $i \in \{0,1,2\}$, let $(H^{i},J^{i})$ be Floer data such that
\begin{equation}
H^{1} +  H^{2} \preceq H^{0}.
\end{equation}
Assume moreover that the Floer data is generic, so that $CF^{*}(H^{i}; \nu)$,  equipped with the Floer differential, is a cochain complex, for every local system $\nu$. We shall assume that $\nu = \s^{E}$, where $E$ is a vector bundle on $Q$.

Choose Floer data $(H^{P},J^{P})$ over the pair of pants, as in Section \ref{sec:moduli-spaces-pairs}, which are regular in the sense of Lemma \ref{lem:pop-gen-regular}.  In particular, whenever $\deg(x_0) =\deg(x_1) + \deg(x_2) $, we have a canonical trivialisation of $\det(D_{u})$ for each element $u$ of $ \Pants(x_0;x_1,x_2) $. Applying Equation \eqref{eq:iso_determinant_lines}, we obtain an induced map
\begin{equation}
  \star_{u} \co  \ro_{x_1} [w(x_1)] \otimes \ro_{x_2}[w(x_2)]  \to \ro_{x_0}[w(x_0)] .
\end{equation}
\begin{rem}
 Using the fact that $w(x_0) = w(x_1) + w(x_2) \mod 2$, the Koszul conventions introduce a sign of parity
 \begin{equation} \label{eq:koszul_sign_shift_product}
   w(x_2) \cdot |x_1|
 \end{equation}
in the above map.
\end{rem}

Tensoring with the local system $\s^{E}$, and using the structure maps from Equation \eqref{eq:product_local_system}, we obtain a map
\begin{equation}
   \star_{u} \otimes \s_{u}^{E} \co \left( \ro_{x_1} \otimes \s_{x_1}^{E}  \right) \otimes \left( \ro_{x_2}\otimes \s_{x_2}^{E}  \right) \to \ro_{x_0}\otimes \s_{x_0}^{E}. 
\end{equation}
Recall that $CF^{*}( H; \s^{E})  $ is the direct sum of the lines $\ro_{x}\otimes \s_{x}^{E} $ over all orbits $x$ of $H$. Abusing notation,  we write
\begin{equation}
 \star_{u} \co CF^{*}( H^{1};  \s^{E}) \otimes CF^{*}( H^{2}; \s^{E}) \to CF^{*}( H^{0}; \s^{E}) 
\end{equation}
for the composition of $  \star_{u} \otimes \s_{u}^{E}$ on the right with the projection of the Floer complexes of $H^1$ and $H^2$ into the lines associated to $x_1$ and $x_2$, and on the left with the inclusion of the line associated to $x_0$ into the Floer complex of $H^0$.

\begin{defin}
  The \emph{pair of pants product}
  \begin{equation}
     \star  \co CF^{*}( H^{1};  \s^{E}) \otimes CF^{*}( H^{2}; \s^{E}) \to CF^{*}( H^{0}; \s^{E}) 
  \end{equation}
is the sum of the contributions of all rigid pairs of pants:
\begin{equation}
  \star = \sum_{\substack{ \deg(x_0) = \deg(x_1) + \deg(x_2) \\ u \in \Pants(x_0;x_1,x_2) }  } \star_{u}.
\end{equation}
\end{defin}
We shall write $\alpha \star \beta$ for the image of $\alpha \otimes \beta$ under $\star$.

\subsection{The cohomological product}
\label{sec:properties-product}

In this section, we sketch proofs of important properties of the pair of pants product.
\begin{prop}
The product $\star$ satisfies
  \begin{equation} \label{eq:twisted_chain_map_equation}
    d (a_1\star a_2) = \left(d a_1 \right) \star a_2 + (-1)^{\deg(a_1)} a_{1} \star \left( d a_{2} \right)
  \end{equation}
\end{prop}
\begin{proof}
There are two parts to the argument: first, one introduces the Gromov-Floer compactification $\Pantsbar(x_0;x_1,x_2)   $. Under our regularity assumption, this is a manifold with boundary of dimension $1$ whenever
\begin{equation}
  \deg(x_0) = \deg(x_1) + \deg(x_2) +1. 
\end{equation}
The boundary strata are
\begin{align}
& \coprod_{\deg(x'_0) +1 = \deg(x_0) } \Cyl(x_0,x'_0) \times \Pants(x'_0;x_1,x_2) \\
&  \coprod_{\deg(x_1) +1 = \deg(x'_1) } \Pants(x_0;x'_1,x_2) \times \Cyl(x'_1,x_1) \\
&  \coprod_{\deg(x_2) +1 = \deg(x'_2) } \Pants(x_0;x_1,x'_2) \times \Cyl(x'_2,x_2).
\end{align}
We algebraically interpret the first of these moduli spaces as the composition $d (x_1 \star x_2) $, and the other two respectively as $(d x_1) \star x_2$ and $x_1 \star (d x_2)$.  Since the number of elements on the boundary of a $1$-dimensional manifold vanishes modulo $2$, we conclude that $\star$ defines a chain map if we reduce coefficients to $\bZ_{2}$.

To prove the result without reducing coefficients, we use the same strategy as in Proposition \ref{prop:d_2_is_0}. The starting point is to keep track of the signs from Equations \eqref{eq:koszul_twist_differental} and \eqref{eq:koszul_sign_shift_product}. We record the parity of the sign below (i.e. if the integer is odd, the sign is $-1$):
\begin{equation} \label{eq:table_signs_switching}
  \begin{array}{ccc}
w(x_1) + w(x_2) (1+ |x_1|) & w(x_1) + w(x_2) (1+ |x_1|) & w(x_2) (1+ |x_1|) \\
 d (a_1\star a_2) & \left(d a_1 \right) \star a_2 &  a_{1} \star \left( d a_{2} \right)
\end{array}
\end{equation}

Next, we compare all compositions of Floer-theoretic operations with the natural map
\begin{equation} \label{eq:map_determinant_line_w}
    \det(D_{w}) \otimes \det (D_{\Psi_{x_1}}   )  \otimes \det (D_{\Psi_{x_2}}   )    \cong  \det (D_{\Psi_{x_0}}   ) 
\end{equation}
associated to any element $w \in \Pants(x_0;x_1,x_2)  $.

We discuss the sign on the term $a_{1} \star \left( d a_{2} \right)  $: recall that the map $\star$ is induced by the sum of isomorphisms
\begin{equation} \label{eq:gluing_determinant_lines}
  \det(D_{u}) \otimes \det (D_{\Psi_{x_1}}   )  \otimes \det (D_{\Psi_{x'_2}}   )    \cong  \det (D_{\Psi_{x_0}}   ) 
\end{equation}
associated to elements $u  \in  \Pants(x_0;x_1,x'_2) $. At the same time, the differential is induced by the sum of maps
\begin{equation}
  \det(D_{v}) \otimes  \det (D_{\Psi_{x_2}}   )    \cong  \det (D_{\Psi_{x'_2}}   ).
\end{equation}
Substituting this in Equation \eqref{eq:gluing_determinant_lines}, we obtain an isomorphism
\begin{equation}
  \det(D_{u}) \otimes \det (D_{\Psi_{x_1}}   )  \otimes   \det(D_{v}) \otimes  \det (D_{\Psi_{x_2}}   )   \cong  \det (D_{\Psi_{x_0}}   ) 
\end{equation}

We can compare this map to Equation \eqref{eq:map_determinant_line_w}, where $w$ is obtained by gluing $u$ and $v$ for some large gluing parameter. The key step is to transpose $ \det (D_{\Psi_{x_1}}   )  \otimes   \det(D_{v}) $, which introduces a sign whose parity is given by
\begin{equation}
  |x_1| ,
\end{equation}
using the fact that this is the index of $D_{\Psi_{x_1}} $.  A similar analysis in the other two cases yields the parities:
\begin{equation}
  \begin{array}{ccc}
1  &  0   & |x_1| \\
d (a_1\star a_2) & \left(d a_1 \right) \star a_2 &  a_{1} \star \left( d a_{2} \right),
\end{array}
\end{equation}
where the non-trivial parity in the case of $  d (a_1\star a_2)  $ comes from the fact that the translation vector field gives rise to an \emph{inward} pointing vector in this situation.  Combining this with Equation \eqref{eq:table_signs_switching}, and adding $w(x_1) + w(x_2) (1+ |x_1|)$ to all entries, we find that the parities of the signs are
\begin{equation}
  \begin{array}{ccc}
1  &  0   & |x_1|  + w(x_1)\\
 d (a_1\star a_2) & \left(d a_1 \right) \star a_2 &  a_{1} \star \left( d a_{2} \right).
\end{array}
\end{equation}
These are exactly the signs in Equation \eqref{eq:twisted_chain_map_equation}: the sign on $  d (a_1\star a_2)  $ corresponds to the fact that it appears on the left hand side. 
\end{proof}

The next property concerns the well-definedness of $\star$. It is a straightforward cobordism argument, along the lines of the proof for continuation maps given in Lemma \ref{lem:continuation_independent_choices}. The proof is completely omitted:
\begin{lem}
  The map induced by $\star$ on Floer cohomology 
  \begin{equation}
    HF^{*}( H^{1};  \s^{E}) \otimes HF^{*}( H^{2}; \s^{E}) \to HF^{*}( H^{0}; \s^{E}) 
  \end{equation}
is independent of the choice of Floer data $(H^{P},J^{P})$. \qed
\end{lem}

\subsection{The product on symplectic cohomology}

In order to know that the pair of pants product descends to symplectic cohomology, we need to understand its compatibility with continuation maps. The proof of the following result uses the same method as the discussion of Section \ref{sec:cont-maps-comm}:
\begin{lem} \label{lem:product_commutes_with_continuation}
Assume that $H^{1} \preceq K^{1}$,  $H^{2} \preceq K^{2}$, and $ K^{1} +  K^{2} \preceq H^{0} $ . We have commutative diagrams
\begin{equation} \label{eq:commute_continuation_multiply}
  \xymatrix{  HF^{*}( H^{1};  \s^{E}) \otimes HF^{*}( H^{2};  \s^{E})  \ar[r]^{\id \otimes \cont} \ar[d]^{\cont \otimes \id} \ar[dr]^{\star} &  HF^{*}( H^{1};  \s^{E}) \otimes HF^{*}( K^{2};  \s^{E}) \ar[d]^{\star}  \\
HF^{*}( K^{1};  \s^{E}) \otimes HF^{*}( H^{2};  \s^{E})  \ar[r]^-{\star} & HF^{*}( H^{0};  \s^{E}) 
}
\end{equation}
If, in addition,  $H^{0} \preceq K^{0}$, we also have a commutative diagram
\begin{equation} \label{eq:commute_multiply_continuation}
    \xymatrix{  HF^{*}( H^{1};  \s^{E}) \otimes HF^{*}( H^{2};  \s^{E})  \ar[dr]^-{\star} \ar[r]^-{\star} &  HF^{*}( H^{0};  \s^{E}) \ar[d]^{\cont}   \\
 & HF^{*}( K^{0};  \s^{E}) .}
\end{equation} \qed
\end{lem}

Equation \eqref{eq:commute_multiply_continuation} implies that the composition
\begin{equation}
HF^{*}(H^{1}; \s^{E}) \otimes HF^{*}(H^{2}; \s^{E})  \to HF^{*}( H^{0};  \s^{E}) \to SH^{*}(\TQ; \s^{E})
\end{equation}
is independent of all choices. 

On the other hand, the top right triangle in Equation \eqref{eq:commute_continuation_multiply} implies that this map factors as
\begin{equation}
  HF^{*}(H^{1}; \s^{E}) \otimes HF^{*}(H^{2}; \s^{E}) \to   HF^{*}(H^{1}; \s^{E}) \otimes SH^{*}(\TQ; \s^{E}) \to SH^{*}(\TQ; \s^{E}), 
\end{equation}
while the bottom left triangle implies a similar factorisation if one applies the continuation map to the first factor. We conclude:
\begin{lem}
The map $\star$ induces a product
\begin{equation}
 SH^{*}(\TQ; \s^{E}) \otimes SH^{*}(\TQ; \s^{E}) \to SH^{*}(\TQ; \s^{E}).
\end{equation} \qed
\end{lem}

\subsection{Associativity and commutativity}
\label{sec:associativity}

We now prove that the product $\star$ on symplectic cohomology is, up to the appropriate sign twists, associative and commutative. Given the use of a direct limit to define both $SH^{*}(\TQ; \s^{E})$ and the product, it should not be surprising that the verification of these properties essentially reduces to a computation involving linear Hamiltonians.

Let us therefore consider four linear Hamiltonians $H^{i}$, labelled by $i \in \{ 0,1,2,3\}$, such that
\begin{equation}
  H^{1} + H^{2} + H^{3} \preceq H^{0}.
\end{equation}
Choose, in addition Hamiltonians $H^{1,2}$ and $H^{2,3}$ such that
\begin{align}
  H^{1} + H^{2} + H^{3} &  \preceq H^{1,2} + H^{3} \preceq   H^{0} \\
  H^{1} + H^{2} + H^{3} &  \preceq H^{1} + H^{2,3} \preceq   H^{0}.
\end{align}

\begin{lem}
  The following diagram commutes:
  \begin{equation} \label{eq:associativity}
    \xymatrix{
HF^{*}( H^{1};  \s^{E}) \otimes HF^{*}( H^{2}; \s^{E}) \otimes HF^{*}( H^{3}; \s^{E})  \ar[d] \ar[r] & HF^{*}( H^{1};  \s^{E}) \otimes HF^{*}( H^{2,3 }; \s^{E})  \ar[d] \\
 HF^{*}( H^{1,2}; \s^{E}) \otimes HF^{*}( H^{3}; \s^{E}) \ar[r] & HF^{*}( H^{0}; \s^{E}) .}
  \end{equation}
\end{lem}
\begin{proof}[Sketch of proof:]
\begin{figure}[h]
  \centering
 \includegraphics{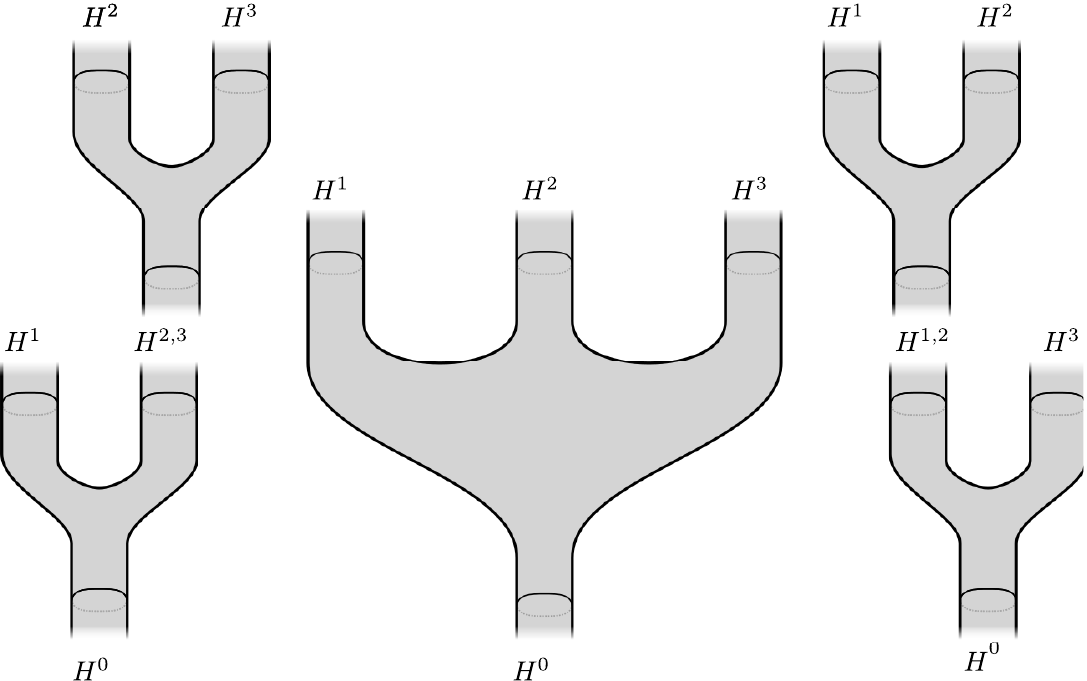}
\caption{}
 \label{fig:associativity}
\end{figure}

This is again a standard cobordism argument: going through the top right corner in Equation \eqref{eq:associativity} corresponds to counting configurations of two pairs of pants as on the left of Figure \ref{fig:associativity}, while going through the bottom left corner corresponds to the right of Figure \ref{fig:associativity}. The sources of these two configurations lie on boundary strata of the  moduli space of complex structures on the complement of $4$ marked points in $S^{2}$. Since this space is connected, we may choose an interval connecting these two boundary points, and choose a family of Floer data interpolating between them. The resulting cobordism shows that the maps induced at the two ends are homotopic. 
\end{proof}
\begin{exercise}
  Use Equation \eqref{eq:associativity} to prove that the product on symplectic cohomology is associative.
\end{exercise}

Next, we study the commutativity of $\star$:
\begin{lem} \label{lem:product_not_quite_commutative}
 The product induced by $\star$ on symplectic cohomology satisfies:
  \begin{equation} \label{eq:commutativity}
    a_1 \star a_2  = (-1)^{\deg(a_1) \deg(a_2) + w(a_1)w(a_2)} a_{2} \star a_{1}.
  \end{equation}
\end{lem}
\begin{proof}[Sketch of proof:]
The product $\star$ is defined by fixing $3$ marked points $(z_0,z_1,z_2)$ on $S^2$, and counting maps from the complement of these points to $\TQ$ with the appropriate asymptotic conditions with respect to fixed cylindrical ends. Since $S^{2}$ admits a biholomorphism fixing $z_0$, and mapping $z_{1}$ to $z_2$, we  pull back the Floer data under this biholomorphism, and conclude that the product on cohomology is commutative up to a sign.

To determine the sign, we  return to the definition of  the product via the gluing isomorphism in Equation \eqref{eq:gluing_isomorphism_det_line_pop}. The main  Koszul sign associated to permuting $ \det(D_{\Psi^{F}_{x_1}})$ and $ \det(D_{\Psi^{F}_{x_2}}) $  is
\begin{equation}
  (-1)^{|x_1||x_2|}.
\end{equation}
Note that this would be the sign without the shift by the $\bZ_{2}$ grading $w$. Using Equation \eqref{eq:koszul_sign_shift_product}, we find the appearance of signs of parity $|a_1| w(a_2)$ and $|a_2| w(a_1)$ on the two sides due to this shift. We leave the reader to compute that, modulo $2$, we have
\begin{equation}
  \deg(a_1) \deg(a_2) + w(a_1)w(a_2) = |a_1||a_2| + |a_1| w(a_2) + |a_2| w(a_1),
\end{equation}
which yields the sign in Equation \eqref{eq:commutativity}.
\end{proof}

\section{The unit} \label{sec:unit}

\subsection{Construction of the unit}
Let $(H,J)$ be generic data defining a  Floer complex $CF^{*}(H; \nu)$.  Choose a family $H^{z}$ of linear Hamiltonians on $\TQ$ parametrised by $z \in \bC \bP^{1} \setminus \{ 0 \}$  which vanishe near $\infty \in \bC \bP^{1}$, and are given by $H^{e^{s+it}} = H_{t}$ whenever $s \ll 0$.  We require that the monotonicity property hold:
\begin{equation} \label{eq:monotonicity_plane}
  H^{z} \preceq H^{z'} \; \; \textrm{whenever } |z'| \leq |z|.
\end{equation}
In addition, choose a family $J^{z}$ of almost complex structures such that $J^{e^{s+it}} = J_{t}$ for $s \ll 0 $.   Given an orbit $x$ of $H$, let $\Plane(x)$ denote the space of maps
\begin{align}
  u \co  \bC \bP^{1} \setminus \{ 0 \} & \to \TQ \\
\lim_{s \to \infty} u(e^{s+it}) & = x(t)
\end{align}
satisfying the Cauchy-Riemann equation
\begin{equation} \label{eq:CR-equation-plane}
  \left( du - X_{H^z} \otimes dt \right)^{0,1} = 0.
\end{equation}
\begin{rem}
  While $dt$ does not extend smoothly to $z = \infty$, the fact that $H^{z}$ is required to vanish in a neighbourhood of this point implies that Equation \eqref{eq:CR-equation-plane} makes sense.
\end{rem}

\begin{exercise}
Using Equation \eqref{eq:monotonicity_plane}, show that the images of all elements of $\Plane(x)$ are contained in $\DQ$ (Hint: review the results of Section \ref{sec:comp-moduli-space}).
\end{exercise}

 Note that $\Plane(x)  $   is empty unless $x$ is contractible; in which case $x^{*}(\det(T\Q))$ is orientable.  To compute the virtual dimension in this case, observe that the unique trivialisation of $x^{*}(T \TQ)$ which extends to $u^{*}(T \TQ)$ is the preferred trivialisation from Lemma  \ref{lem:trivalisation_up_to_htpy}.  Moreover, the linearisation of Equation \eqref{eq:CR-equation-plane} in such a trivialisation takes the form given in Equation \eqref{eq:operator_orbits}. We conclude that the virtual dimension is $\deg(x)$, and that we have a canonical isomorphism
 \begin{equation}
   \det(D_{u}) \cong \delta_{x}.
 \end{equation}
 Applying the usual transversality arguments implies:
\begin{lem}
For generic choices of data $(H^{z},J^{z})$, the moduli space $\Plane(x)  $ is a smooth manifold of dimension $\deg(x)$. \qed
\end{lem}

Assume now that $\deg(x)=0$. In this case, $\det(D_{u})$ is canonically trivial for every map $u \in   \Plane(x) $, so we obtain a map
\begin{equation}
e_{u} \co  \bZ \to \ro_{x}.
\end{equation}
Tensoring this with the map in Equation \eqref{eq:unit_local_system}, and composing with the inclusion of $ \ro_{x} \otimes \s^{E}_{x} $ into the Floer complex, we define the unit
\begin{equation}
  \bZ \to CF^{*}(  H; \s^{E}).
\end{equation}
By taking the sum over all rigid curves $u$, we obtain a map
\begin{equation}
  e = \sum_{ \stackrel{ \deg(x) = 0 }{u \in \Plane(x)} } e_{u}.
\end{equation}
In order to show that $e$ descends to cohomology, let $\Planebar(x)$ denote the  Gromov-Floer space, with codimension $1$ strata
\begin{equation} \label{eq:boundary_moduli_plane}
 \bigcup_{x'}  \Cylbar(x,x') \times \Planebar(x').
\end{equation}

Whenever $\deg(x) =1$, we conclude that  the moduli space $\Planebar(x)$ is a $1$-dimensional manifold with boundary given by
\begin{equation}
  \coprod_{\deg(x') = \deg(x) +1}    \Cyl(x,x') \times \Plane(x').
\end{equation}
This product corresponds to the composition $d \circ e$, and we conclude that  $e$ is a chain map.
\begin{exercise} \label{ex:unit_indep_choices}
Use a parametrised moduli space, as in Lemma \ref{lem:continuation_independent_choices}, to show that the map 
\begin{equation}
  \bZ \to HF^{*}(  H; \s^{E}).
\end{equation}
induced by $e$ on cohomology is independent of the choice of Cauchy-Riemann equation on the plane.
\end{exercise}

\begin{rem}
 We shall abuse notation and write $e$ for the map into $ CF^{*}( H; \s^{E}) $, for the induced element $e(1)$ of the Floer complex, and for its image in cohomology. 
\end{rem}

\subsection{Properties of the unit} \label{sec:properties-unit}
In this section, we sketch the proofs that $e$ satisfies the desired properties. In order to obtain a well-defined element in $SH^{*}(\TQ; \s^{E}) $, we must show that, if $H^{+} \preceq H^{-} $,  we have a commutative diagram
\begin{equation}
 \xymatrix{ \bZ  \ar[r] \ar[rd] & HF^{*}(H^{+}; \s^{E}) \ar[d]^{\cont} \\
& HF^{*}(H^{-}; \s^{E}).  }
\end{equation}
The proof is essentially the same as that of Lemma \ref{lem:cont_map_commute_Delta}: we can glue elements of $\Plane(x')$ and $\Cont(x,x')$ to obtain solutions to a Cauchy-Riemann equation on the plane, with asymptotic condition $x$, which defines a map
\begin{equation}
  \bZ  \to CF^{*}(H^{-}; \s^{E})
\end{equation}
that agrees with the composition $\cont \circ e$. Exercise \ref{ex:unit_indep_choices} implies that this map, on cohomology, agrees with the unit of $HF^{*}(H^{-}; \s^{E})  $.

Next, we consider the composition of $\Delta $ with $ e  $. This is controlled by the moduli space
\begin{equation}
 \Cyl_{\Delta}(x,x') \times   \Plane(x')
\end{equation}
which is a parametrised moduli space over $S^1$. We can glue elements of these two moduli spaces to obtain a family of Cauchy-Riemann equations on the plane, also parametrised by $S^1$. This family is homotopic to a family which is independent of the parameter $S^1$. In this way, we obtain a cobordism to $S^1 \times   \Plane(x)$.  The same argument as in Section \ref{sec:square-bv-operator} shows that the composition
\begin{equation}
  \Delta \circ e   \co \bZ \to HF^{*}(H; \s^{E}) 
\end{equation}
vanishes. The result for symplectic cohomology follows by using the compatibility of $e$ and $\Delta$ with continuation maps.

Finally, we justify our terminology by proving that $e$ is a unit for the product $\star$ on symplectic cohomology. This follows from proving the commutativity of the diagram
\begin{equation} \label{eq:commutativity_product_unit_continuation}
  \xymatrix{HF^{*}(H^{1}; \s^{E}) \ar[r]^-{\id \otimes e} \ar[dr]^{\cont} & HF^{*}(H^{1}; \s^{E})  \otimes HF^{*}(H^{2}; \s^{E}) \ar[d]^{\star} \\
&  HF^{*}(H^{0}; \s^{E}).
}
\end{equation}
whenever $H^{1} + H^{2}  \preceq H^{0}$. Given orbits $x_i$ of $H^{i}$ for $i = \{ 0,1\}$ of the same index, we can define a cobordism between
\begin{equation}
  \Contbar(x_0,x_1) \textrm{ and } \coprod_{\substack{ x_2 \in \Orbit(H^{2}) \\ \deg(x_2) = 0}}  \Pants(x_0; x_1, x_2) \times \Plane(x_2) .
\end{equation}
The key point is that, by gluing a plane to a pair of pants, we obtain a cylinder; at the level of maps, this means that we can glue elements of $  \Plane(x_2)  $ and $  \Pants(x_0; x_1, x_2) $ to obtain solutions to an equation on a cylinder, which at the outgoing end converge to $x_0 \in \Orbit(H^0)$, and at the incoming end converge to $x_1 \in \Orbit(H^1)  $. The choice of data for the definition of the product and the unit were made so that, after gluing, we have a function $b$ (measuring the slope of the Hamiltonian) and a $1$-form $\alpha$ on the cylinder, so that $b \cdot \alpha$ is subclosed. Using Remark \ref{rem:general_continuation_map}, we see that the moduli space we obtain by gluing is cobordant to the map defining continuation, which implies that Diagram \eqref{eq:commutativity_product_unit_continuation} commutes.

At this stage, we pass to symplectic cohomology, and conclude that the map
\begin{equation}
 \_  \star e \co SH^{*}(\TQ; \s^{E}) \to SH^{*}(\TQ; \s^{E})
\end{equation}
agrees with the map induced by continuation, which is the identity.
\begin{exercise}
Show that left multiplication by $e$ also induces the identity on symplectic cohomology.
\end{exercise}

\section{The $BV$ equation}
\label{sec:bv-equation}
In this section, we prove Equation \eqref{eq:BV-equation}, which corresponds to a relation in the homology of the moduli space of spheres with $4$ punctures, with a choice of cylindrical end at each puncture.

In order to prove that this equation holds, it is particularly convenient to be able to associate a picture to the datum of a Riemann surface with cylindrical ends. Let us therefore consider a Riemann surface  $\Sigma$  obtained by removing points $\{ z_{i} \}$ from a closed Riemann surface $\bar{\Sigma}$:
\begin{defin}
An \emph{asymptotic marker} at the $i$\th end of $\Sigma$ is a tangent direction in $T_{z_i} \bar{\Sigma}$.
\end{defin}

\begin{lem} \label{lem:cylindrical_end_is_asymptotic}
  The space of cylindrical ends near $z_i$ is weakly homotopy equivalent to the space of choices of asymptotic markers at $z_i$.
\end{lem}
\begin{proof}[Sketch of proof:]
First, recall that a cylindrical end extends, by definition, to a biholomorphism from the unit disc to $ \bar{\Sigma} $, mapping the origin to $z_i$. We assign to a cylindrical end the tangent direction in $ T_{z_i} \bar{\Sigma} $ which corresponds to the image of the real line $[0,+\infty) \times \{1 \}  $ or $(-\infty, 0] \times \{1 \} $, depending on whether the end at $z_i$ is positive or negative. We shall show that this map gives a homotopy equivalence between the space of cylindrical ends and the space of asymptotic makers. In order to show this, we prove the contractibility of  the space of cylindrical ends which give rise to the same asymptotic marker.

We fix a cylindrical end $\epsilon_{i}$ which will serve as reference. The space of cylindrical ends retracts to the subset of those whose image is contained in that of $\epsilon_i$; these correspond to biholomorphisms from the disc to an open subset thereof fixing the origin, which extend holomorphically to the closed unit disc. The asymptotic marker condition corresponds to requiring that the derivative at the origin be a positive real number.

In terms of analytic expansions, we are now considering power series with vanishing constant order term, and first order term a positive real number, which uniformly converge in the unit disc. By dilation, this space is homotopy equivalent to those series which, in addition, satisfy an absolute bound
\begin{equation}
  \sum_{i=1}^{\infty} |a_i| \leq 1.
\end{equation}
By conjugating every such biholomorphism with a dilation, we obtain a flow
\begin{equation}
  f_{\lambda}(z) = \lambda^{-1}  f(\lambda z)
\end{equation}
on this space, parametrised by $\lambda \in (0,1]$, which retracts it to those power series which have trivial higher order terms. This proves the desired result.
\end{proof}

We are studying the case in which $\bar{\Sigma} = S^{2}$, and there is a unique marked point $z_0$ corresponding to the output. The choice of an asymptotic marker at this end induces one at every other  end $z_i$ as follows: up to $\bR$ translation, there is a unique biholomorphism
\begin{equation} \label{eq:biholomorphism_z_0_z_i}
  \bR \times S^{1} \to S^{2} \setminus \{z_0, z_i \}
\end{equation}
so that our chosen asymptotic marker at $z_0$ corresponds to the negative end which is the restriction of this map to  $(-\infty,0] \times \{1 \}$.  By restricting instead to the positive half-cylinder $[0,\infty)  \times \{1 \}  $, we obtain a positive end at $z_i$, and hence an asymptotic marker.

\begin{rem} 
  In the figures appearing in this section, an arrow emanating from a marked point on a surface will indicate the choice of a fixed asymptotic marker at that point, and hence, by Lemma \ref{lem:cylindrical_end_is_asymptotic}, of a cylindrical end uniquely determined up to homotopy. A circle labeled by a marked point will correspond to a family of marked points moving along that circe on the surface, and arrows emanating from the circle indicate the choice of asymptotic markers in that family. Many arrows emanating from the same marked point indicate that we take the $S^1$-family consisting of all choices of asymptotic markers at that marked point.  When the marked point is at infinity, these markers will appear as arrows lying on the outer edge of the drawing.
\end{rem}

\subsection{The pair of pants bracket}
The proof of the $BV$ equation is easier to understand if one first introduces a Lie bracket:
\begin{equation}
  \{ \; , \, \} \co SH^{*}(\TQ; \s^{E}) \otimes SH^{*}(\TQ; \s^{E}) \to SH^{*}(\TQ; \s^{E}).
\end{equation}
We shall define the bracket in this section, and show that it can be expressed in terms of $\Delta$ and $\star$.

We start by considering the family of Riemann surfaces with $3$ marked points and asymptotic markers shown in the left of  Figure \ref{fig:bracket-definition}.
\begin{figure}[h]
  \centering
  \begin{tabular}{ccc}
\includegraphics[scale=2]{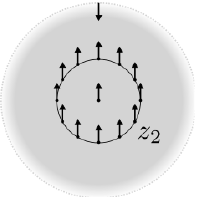} &\quad &  \includegraphics[scale=2]{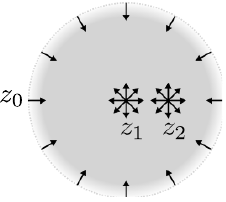} 
\end{tabular}
\caption{The family of Riemann surfaces defining the bracket.}
 \label{fig:bracket-definition}
\end{figure} 
To be more precise, we set $z_0 = \infty$ and $z_1 = 0$ as marked points on $S^{2} = \bC \bP^{1}$ and choose asymptotic markers at both of these marked points; for aesthetic reasons, we choose the asymptotic direction at $z_0$ to be tangent to the (positive) imaginary axis, which corresponds to the arrow pointing downwards appearing in the left of Figure \ref{fig:bracket-definition}.  Given any other marked point $z_2$ in the punctured plane, the asymptotic marker induced by identifying the complement of $z_0$ and $z_2$ with a cylinder as in Equation \eqref{eq:biholomorphism_z_0_z_i} is the one which points ``upwards.''  We then define 
\begin{equation}
  \vec{\Pants}
\end{equation}
to be the $1$-dimensional family of pairs of pants with asymptotic markers obtained by removing $z_0$, $z_1$, and a point $ z_2 = e^{i \theta}$ on the unit circle from $\bC \bP^{1}$. By the above discussion, we have canonical up to homotopy choices, for each $\Sigma \in  \vec{\Pants}$, of cylindrical ends $\epsilon_{i,\Sigma}$ for $i \in \{0,1,2\}$ at the three marked points. We choose these ends to be positive for $i=\{1,2\}$, and negative for $i=0$.

Given linear Hamiltonians $H^{i}$ with non-degenerate orbits so that
\begin{equation}
  H^{1} + H^{2} \preceq H^{0},
\end{equation}
we now choose a Cauchy-Riemann equation on the space of maps to $\TQ$, for each curve in $\vec{\Pants}$, whose pullback under the $i$\th cylindrical ends agrees with the Floer equation for $H^{i}$. For each triple $(x_0,x_1,x_2)$ of Hamiltonian orbits of $(H^{0}, H^{1}, H^{2})$, we obtain a parametrised moduli space
\begin{equation}
   \vec{\Pants}(x_0; x_1,x_2)
\end{equation}
of finite energy maps with asymptotic conditions $x_i$ on the $i$\th end whose virtual dimension is
\begin{equation}
  \deg(x_0) - \deg(x_1) - \deg(x_2) +1.  
\end{equation}
\begin{exercise}
Show that the family of Cauchy-Riemann equations above can be chosen so that all elements of $ \vec{\Pants}(x_0; x_1,x_2) $ have image contained in the unit cotangent bundle (Hint: for each $\Sigma \in \vec{\Pants}$, choose the Floer data as in Section \ref{sec:moduli-spaces-pairs}. Then repeat the argument of Section \ref{sec:comp-moduli-space}).
\end{exercise}
Taking care to ensure transversality as in the definition of the product, the rigid elements of $ \vec{\Pants}(x_0; x_1,x_2)  $ define a map
\begin{equation}
  \{ \, , \} \co HF^{i}( H^{1}; \s^{E}) \otimes HF^{j}( H^{2}; \s^{E})  \to HF^{i+j-1}( H^{0}; \s^{E}) .
\end{equation}

\begin{lem}
For any choice of Floer data we have
\begin{equation} \label{eq:bracket_commutator}
   \{a  , b\} = (-1)^{\deg(a)} \Delta(a \star b) -(-1)^{\deg(a)} \Delta(a) \star b - a \star \Delta(b).
\end{equation}
\end{lem}
\begin{proof}[Sketch of proof:]
The key point is that there is a unique biholomorphism of $S^{2}$ fixing $0$ and $\infty$ and mapping $z_2$ to $1$, which is given by multiplying by $z_{2}^{-1}$. Applying this biholomorphism to every element of the $1$-parameter family of Riemann surfaces from the left Figure \ref{fig:bracket-definition}  gives rise to a $1$-parameter family of asymptotic markers at the three marked points $\{0,\infty,1\}$ on $S^2$ as shown on the right. The space of such markers can be identified with $(S^1)^{3}$, with each factor corresponding to a marked point, and the family we obtain is the diagonal, because all the asymptotic markers rotate once around the circle. Decomposing this diagonal into the three directions, we obtain the three families of asymptotic markers shown in Figure \ref{fig:compositions_Delta_star}; these correspond, term by term, to the right hand side of Equation \eqref{eq:bracket_commutator}.
\begin{figure}[h] 
  \centering
  \begin{tabular}{ccccc}
$ \Delta(a \star b)  $ & & $ \Delta(a) \star b $ & & $ a \star \Delta(b)$ \\
\includegraphics[scale=2]{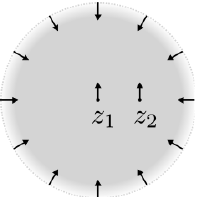} &\quad &  \includegraphics[scale=2]{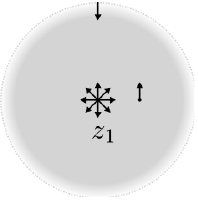} & \quad &  \includegraphics[scale=2]{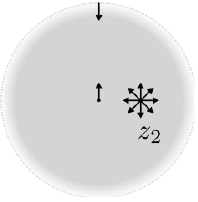}  
\end{tabular}
\caption{}
\label{fig:compositions_Delta_star}
\end{figure}
\end{proof}
\begin{exercise}
  Use Equation \eqref{eq:bracket_commutator} to show that the bracket commutes with continuation maps, and hence descends to an operation on symplectic cohomology.
\end{exercise}

We leave, as exercises to the reader, the verification that the bracket satisfies the expected identities (we shall not use these identities).

\begin{exercise}
Using Equation \eqref{eq:BV-equation}, associativity, and the (twisted) commutativity of the product, check that $\{ \, , \}$ defines a (twisted) Lie bracket on symplectic cohomology, i.e. that the Jacobi identity holds.

Moreover, show that the Poisson identity holds:
\begin{equation}
  \{a ,b \star c\} = \{ a, b \}  \star b,c  + (-1)^{\deg(a) \deg(b) +w(a) w(b) + (1+w(a) \deg(b) } b \star \{ a,c \}. 
\end{equation}

\end{exercise}

\subsection{Proof of the $BV$ relation}
\begin{figure}[h] 
  \centering
  \begin{tabular}{ccc}
\includegraphics[scale=1.5]{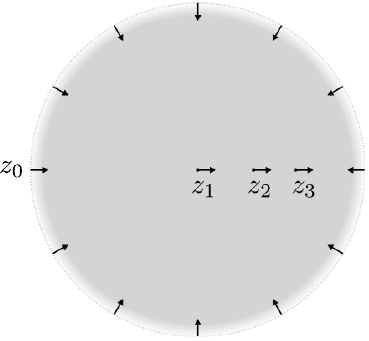}&\quad &  \includegraphics[scale=1.5]{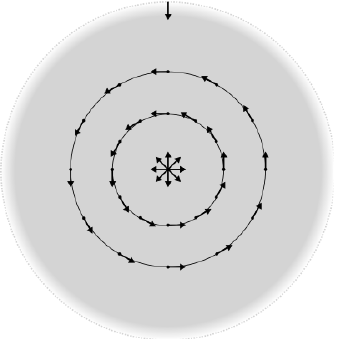} 
\end{tabular}
\caption{Two representations of the $S^{1}$ family of Riemann surfaces with asymptotic markers giving rise to the operation $\Delta(a \star b \star c) $}
\label{fig:composition_Delta_product}
\end{figure}

We now explain the proof of Equation \eqref{eq:BV-equation}. Starting with the left hand side, we can express the composition $\Delta(a \star b \star c) $ as counting elements of a parametrised Floer problem on the complement of any $4$ points $\{z_0, z_1, z_2, z_3 \}$  on $S^2$, with asymptotic markers fixed at $z_{i}$ for $i = \{1 ,2 ,3\}$, and moving once around a circle for $z_{0}$. Since any set of points gives rise to the same operation (at the level of cohomology), we fix $z_0 = \infty$ and $z_1 = 0$, and in addition require that $z_{2}$ and $z_{3}$ lie on the positive real axis, as shown in the left of Figure \ref{fig:composition_Delta_product}.

If we rotate the plane, we obtain an equivalent family, shown on the right of Figure \ref{fig:composition_Delta_product},  where the asymptotic marker at infinity is constant, and the points $z_2$ and $z_3$ each rotate once around the origin. In this family, each of the asymptotic markers at the inputs rotates once.

\begin{figure}[h]
  \centering
  \begin{tabular}{ccc}
$\Delta(a) \star b \star c$ & $a \star \Delta(b) \star c$  & $a \star b \star \Delta(c)$ \\
\includegraphics{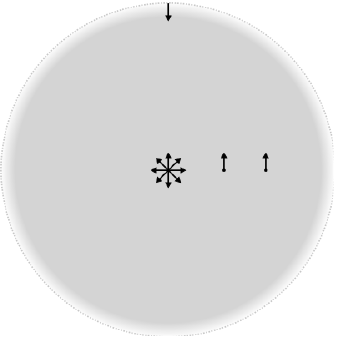}& \includegraphics{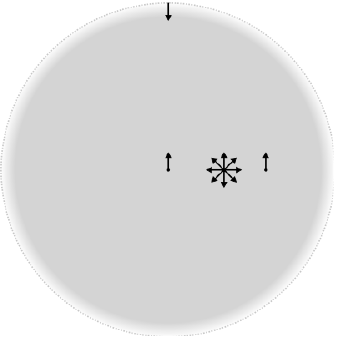} &  \includegraphics{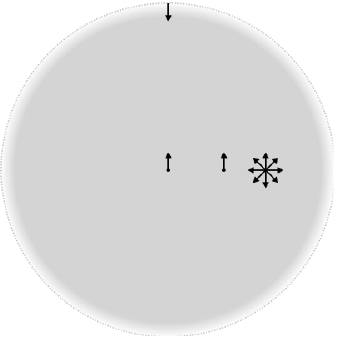} \\
$\{ a , b \} \star c$ & $a \star \{ b,  c \}$  & $ b \star  \{ a, c \}$\\
\includegraphics{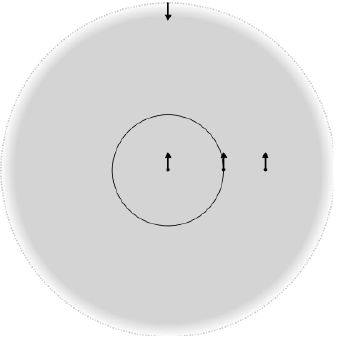}&  \includegraphics{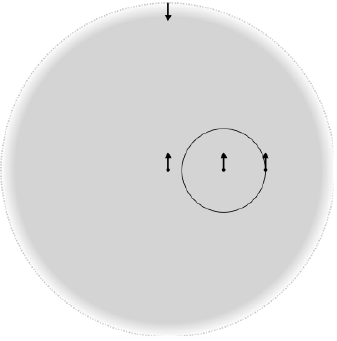} &  \includegraphics{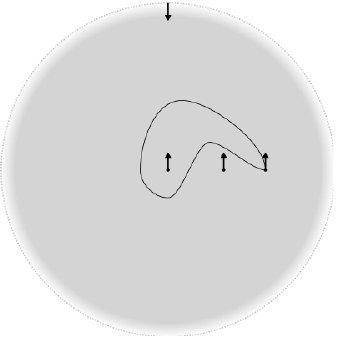}
\end{tabular}
\caption{}
 \label{fig:BV-relation}
\end{figure}

The corresponding $1$-cycle in the space of Riemann surfaces with asymptotic markers is homologous to the sum of the $6$ cycles shown in Figure \ref{fig:BV-relation}. This implies that
\begin{multline} \label{eq:BV-relation-different}
  \Delta(a \star b \star c) =   \Delta(a) \star b \star c + (-1)^{\deg(a)}  a \star \Delta(b) \star c + (-1)^{\deg(a) + \deg(b)}   a \star b \star \Delta(c) + \\
(-1)^{\deg(a)} \{ a, b\} \star c + (-1)^{\deg(a) + \deg(b)}   a \star \{ b ,  c \}  + (-1)^{\deg(b)(1+\deg(a)) + \deg(a) + w(a) w(b)} b \star \{ a, c \}.
\end{multline}

Note that the first three terms above are exactly the first three of Equation \eqref{eq:BV-relation-different}. It remains to push symbols.
\begin{exercise}
  Using Equation \eqref{eq:bracket_commutator}, expand the last three terms in Equation \eqref{eq:BV-relation-different}, and conclude that Equation \eqref{eq:BV-equation} holds.
\end{exercise}

\section{Guide to the literature}

\subsection{The product}
The existence of a product on Floer homology was observed by Donaldson soon after the introduction of these groups; a rather complete account in the Hamiltonian setting appears in Schwarz's thesis \cite{Schwarz}. In \cite{viterbo-99}, Viterbo noticed that this construction yields a product on symplectic cohomology. The construction also appears in \cite{seidel-biased}, where it is observed that the same methods can be used to define operations parametrised by the homology of moduli spaces of curves of arbitrary genus (see also \cite{ritter}).  

In the specific case of cotangent bundles, the product in the symplectic cohomology was compared to the loop product in \cite{AS-product}.

\subsection{$BV$ operator}
The $BV$ operator is simply the shadow of the existence of a natural $S^1$ action on the free loop space of a symplectic manifold. In the case of closed manifolds, the isomorphism with the cohomology of $M$ implies that the $BV$ operator is trivial; it has therefore not been studied in this setting.

The importance of studying this operator in the open setting was first noticed by Viterbo in \cite{viterbo-99}, who in fact constructed higher analogues of this operator, parametrised by spaces of flow lines in $\bC \bP^{\infty}$, and used it to construct an $S^1$-equivariant version of Floer cohomology. The first of these higher analogues of $\Delta$ is the chain-level operator introduced in Section \ref{sec:square-bv-operator} to show that the square of the $BV$ operator vanishes. This theory was also studied by  Seidel in \cite{seidel-biased}.

A slight variant of the construction of $S^1$-equivariant cohomology is given by Bourgeois and Oancea in \cite{BO1}; the comparison with the approach closer to the spirit of the construction given here appears in the recent preprint \cite{BO2}, which is a good reference for detailed discussion of $S^1$-equivariant symplectic cohomology.

\subsection{$BV$ relation} \label{sec:bv-relation}

In a very abstract setting, the fact that the $BV$ equation holds goes back to Getzler \cite{getzler} who considered the moduli space of punctured genus $0$ Riemann surfaces with asymptotic markers at each marked point. By distinguishing one marked point on each such surface as outgoing, the union of all these moduli spaces (for arbitrary number of punctures) forms an \emph{operad} under gluing the (unique) output to any of the outputs. It turns out that this operad is homotopy equivalent (as an operad) to the framed little disc operad $E_{2}^{f}$, which is well-studied by algebraic topologists \cite{may}.  Getzler proved that the cohomology of this operad controls $BV$ algebras: concretely, there are unique elements of degree $0$ in the homology of the moduli space of genus $0$ surfaces with $1$ and $3$ punctures which respectively give rise to the unit $e$ and to the product in our language, and a unique element of degree $1$ in  the homology of the moduli space of genus $0$ surfaces with $2$ punctures which gives rise to $\Delta$.

By gluing Riemann surfaces along the appropriate ends (i.e. applying the operad maps), one can generate homology classes for moduli spaces of surfaces with more punctures. Getzler's result is that all cohomology classes arise via this construction, and that the only relations they satisfy are those corresponding to Equations \eqref{eq:square_Delta_0}-\eqref{eq:BV-equation}.

The relevance of Getzler's result for Floer theory was first observed by Seidel in \cite{seidel-biased};  there seem to be no other written accounts. The proof of the $BV$ equation seems to be one of those facts that are too trivial for words; the corresponding pictures were therefore provided in Section \ref{sec:bv-equation}.

\subsection{What is missing: Chain level structure}
Let $SC^{*}(M)$ denote a (natural) chain complex which computes symplectic cohomology of a manifold $M$; the main failure of our presentation of the subject is that we did not provide an explicit complex. As a result, we did not introduce chain level structures.

The most natural such structure is the refinement of the $BV$ structure:
\begin{conj}
Let $E_{2}^{f}$ be the framed little disc operad. The complex $SC^{*}(M)$ is an algebra over the operad in chain complexes $C_{*}(  E_{2}^{f}) $.
\end{conj}

The notion of an algebra over the operad $C_{*}(  E_{2}^{f}) $ is equivalent to any of the various notions of $BV_{\infty}$ structure which appear in the literature, and which  give \emph{homotopy} analogues of $BV$ algebras, e.g. \cite{GTV}.

\begin{rem}
Lest the reader fall into unnecessary confusion, we digress into a discussion of the literature on the framed little disc operad: this operad is formal, i.e.  $C_{*}(  E_{2}^{f})$ and $H_{*}(  E_{2}^{f})$ are equivalent as operads, and hence the categories of algebras over these operads are equivalent. This does not mean that the algebra structure of $ SC^{*}(M) $ over $ C_{*}(  E_{2}^{f}) $ encodes the same amount of information as the $BV$ operations on $SH^{*}(M)$.

Rather, the meaning of formality needs to be understood in the appropriate homotopical sense: there is a chain complex, homotopy equivalent to $SC^{*}(M)$, which can be equipped with a $BV$ structure in such a way that the structure of $ SC^{*}(M) $ as an algebra over $ C_{*}(  E_{2}^{f})  $ can be recovered from this $BV$ structure together with the homotopy equivalence of $ SC^{*}(M) $ with this chain complex. This chain complex is not in general the cohomology $SH^*(M)$.

There are relatively explicit universal constructions of such a chain complex in \cite{GTV}, but none are intrinsic to symplectic topology. 
\end{rem}

We are of course also missing the generalisation of the above chain level statement to moduli spaces with multiple outputs, and to higher genera; the two generalisations are tightly connected since the result of gluing genus $0$ Riemann surfaces at multiple nodes has higher genus. One can of course simply consider the obvious generalisation of operads to PROP's studied by algebraic topologists \cite{may}, and show that $SC^{*}(M)$ is an algebra over the chains of the PROP associated to Riemann surfaces with asymptotic markers at the punctures.

The problem of extracting interesting invariants from higher genus operation is still not completely settled. One candidate for encoding is the formalism used  in symplectic field theory \cite{EGH}, though only the genus $0$ part of the non-equivariant theory has been studied so far \cite{BEE}. Another is Wahl's formulation in terms of universal operations \cite{wahl}.

\chapter{String topology using piecewise geodesics} \label{cha:finite-appr-loop}

\section{Introduction} 
In this section, we shall use piecewise geodesics as a finite dimensional model of the free loop space. As a result, the homology of the free loop space, with coefficients in a graded local system, will be expressed as a (direct) limit of the homology of these finite dimensional manifolds.  The most important local system, which we shall call $\eta$, is needed in order for the results of Chapter \ref{cha:from-sympl-homol} to be valid integrally; there is of course an isomorphism with coefficients in a field of characteristic $2$ without this assumption. Moreover, as shall be clear in the discussion below, the local system we introduce will be trivial, supported in degree $-n$, whenever $Q$ admits a $\Spin$ structure.

The local system $\eta$ is the tensor product of three local systems: the first is
\begin{equation}
  \sigma \equiv \sigma^{T\Q},
\end{equation}
the local system of trivialisations of the pullback of  $T\Q \oplus \det(T\Q)^{\oplus 3}$ (see the Introduction to Chapter \ref{cha:oper-sympl-cohom}).  We define the fibre of $\s$ at a loop $\gamma$ to be the quotient of the free abelian group generated by these two stable trivialisations, by the relation that their sum vanishes. We can now define the fibre of $\eta$ at a loop $\gamma$:
\begin{equation} \label{eq:define_eta}
  \eta_{\gamma} \equiv \s_{\gamma} \otimes |\Q|^{-1}_{\gamma(0)} \otimes  \left(|\Q|_{\gamma(0)}[n]\right)^{\otimes - w(\gamma)}
\end{equation}
where $ |\Q| $ is the space of orientations of $\Q$ at $\gamma(0)$, the symbol $[n]$ indicates shifting the degree \emph{down} by $n$, and $w(\gamma)$ is $0$ or $-1$ depending on whether $\gamma^{*}(\TQ)$ is orientable (see Equation \eqref{eq:first_sw_class_pullback}). Since $|Q|$ is naturally supported in degree $n$, the above local system is naturally supported in degree $-n$ for both orientable and non-orientable loops.
\begin{rem}
In order to compare with the literature, the following point may be useful: we have defined $\s$ to be the local system of trivialisations of the pullback of $ T\Q \oplus \det(T\Q)^{\oplus 3} $; i.e. the space of $\Spin$ structures. By \cite{KT}, this is the same as the space of $\Pin^{+}$ structures on the pullback of $T\Q$, where $\Pin^{+}$ is one of the two possible central extensions of the orthogonal group by $\bZ_{2}$. 
\end{rem}

For technical reasons explained in Remark \ref{rem:struct-comp-moduli}, we shall construct all homology groups  using Morse theory. In order to use standard versions of Morse theory, we shall ensure that the spaces of piecewise geodesics that we study are manifolds with corners. Moreover, we shall construct a $BV$ structure on the homology of the free loop space by constructing maps on the  Morse homology of its finite dimensional approximations. 

\begin{rem}
The results of Chapter \ref{cha:oper-sympl-cohom}, together with the isomorphism we shall construct in Chapter \ref{cha:from-sympl-homol}, will imply that the operations we construct in this section define a twisted $BV$ structure. In order to verify this independently, it would seem useful to relate our Morse theoretic model to more standard constructions of operations in string topology \cite{CS,Laudenbach}; this task is left to the reader. 
\end{rem}

\section{Construction}

We fix, once and for all, a metric on $\Q$ such that
\begin{equation}
  \parbox{35em}{the injectivity radius is larger than $4$.}
\end{equation}

For each integer $r$, and collection of strictly positive real numbers $\bfdelta^{r} = \{\delta^{r}_{1}, \ldots, \delta^{r}_{r} \} $ bounded by $2$, we shall define a finite-dimensional approximation of the loop space as follows: whenever $1 \leq i \leq r$, consider the function
\begin{align}
\rho_{i} \co \Q^{r} & \to \bR \\
\rho_{i}(q_0, \ldots, q_{r-1}) & \equiv  d(q_{i},q_{i+1}),
\end{align}
where $d$ is the distance, using the convention that $q_{r} \equiv q_{0}$. We define $\sL^{r}_{\bfdelta^{r}} \Q$ to be the subset of $\Q^{r}$ defined by the equations
\begin{align}
\rho_{i} \leq \delta^{r}_{i} \textrm{ for } i =1, \ldots, r
\end{align}
\begin{rem}
 At the end of Section \ref{sec:smooth-structures}, we shall fix a sequence $ \delta^{r}_{i}  $ subject to certain bounds and genericity constraints, and then elide the choice from the notation.
\end{rem}
We have a natural map
\begin{equation}
\geo \co  \sL^{r}_{\bfdelta^{r}} \Q \to \sL \Q
\end{equation}
which sends a collection of points to the piecewise geodesic connecting them, parametrised at unit speed. 

Whenever $\delta_{i}^{r} \leq \delta_{i+1}^{r+1}$ for every integer $i$ between $1$ and $r$,  we also have an inclusion
\begin{align} \label{eq:map_successive_r}
  \iota \co  \sL^{r}_{\bfdelta^{r}} \Q & \to \sL^{r+1}_{\bfdelta^{r+1}} \Q \\
\iota (q_0, \ldots, q_{r-1}) & = ( q_0, q_0, \ldots, q_{r-1}).
\end{align}
\begin{exercise} \label{ex:homotopy_different_inclusions}
Assume that $ \delta_{j}^{r} \leq \delta_{i}^{r+1} $ for every pair $(i,j)$. Show that the map
\begin{equation}
  (q_0, \ldots, q_{r-1})  \mapsto (q_0, \ldots, q_{k-1}, q_{k}, q_{k}, q_{k+1}, \ldots ,  q_{r-1})
\end{equation}
also defines an inclusion $   \sL^{r}_{\bfdelta^{r}} \Q \to \sL^{r+1}_{\bfdelta^{r+1}} \Q $. By induction on $k$, show that this inclusion is homotopic to $\iota$.
\end{exercise}

\begin{exercise}
   Show that we have a commutative diagram
  \begin{equation} \label{eq:finite_approx_maps_commute}
    \xymatrix{  \sL^{r-1}_{\bfdelta^{r-1}} \Q \ar[r]^{\iota} \ar[dr]^{\geo} & \sL^{r}_{\bfdelta^{r}} \Q  \ar[d]^{\geo}\\
& \sL \Q.}
  \end{equation}
\end{exercise}
Let us define $|\bfdelta^{r}| = \sum_{i=1}^{r} \delta_{i}$.  Choose, for each positive integer $r$ a sequence $\bfdelta^{r} = \{\delta^{r}_{1}, \ldots, \delta^{r}_{r} \} $ such that
\begin{align} \label{eq:length_goes_to_infty}
  \lim_{r \to +\infty} |\bfdelta^{r}| & = \infty \\ \label{eq:inclusion_well_defined}
\delta_{i}^{r} & \leq \delta_{i+1}^{r+1} \textrm{ if } 1 \leq i \leq r.
\end{align}
 The main justification for calling the spaces of piecewise geodesics \emph{finite approximations} to the loop space is the following result:
\begin{prop}
Assuming Equations \eqref{eq:length_goes_to_infty} and \eqref{eq:inclusion_well_defined},  the inclusion of the finite dimensional approximations induces a homotopy equivalence
\begin{equation} \label{eq:equivalence_finite_approx}
  \bigcup_{r=1}^{\infty} \sL^{r}_{\bfdelta^{r}} \Q \to \sL \Q.
\end{equation}
\end{prop}
\begin{proof}[Sketch of proof:]
Let $\sL_{ |\bfdelta^{r}|  } \Q$ denote the subset of piecewise smooth loops whose length is bounded by $|\bfdelta^{r}| $; by Equation \eqref{eq:length_goes_to_infty}, $\sL \Q$ is a union of the subsets $\sL_{|\bfdelta^{r}|  } \Q$. For each $r$, the image of $\geo$ is contained in $\sL_{ |\bfdelta^{r}|  } \Q$, and is a homotopy equivalence. To see this observe that there is a projection
\begin{equation}
\pi_{r} \co  \sL_{ |\bfdelta^{r}|  }  \Q \to \sL^{r}_{\bfdelta^{r}} \Q
\end{equation}
which sends a loop of length bounded by $r$ to the evaluation of its arc-length parametrisation at the points
\begin{equation} \label{eq:partial_sums_delta}
\Big{\{}0, \frac{ \delta_{1}^{r}}{ |\bfdelta^{r}|}, \cdots, \frac{ \delta_{1}^{r} + \ldots +  \delta_{r-1}^{r}}{ |\bfdelta^{r}|} \Big{\}}. 
\end{equation}
The fibre of $\pi_{r}$ is the product of spaces of paths of length bounded by $ \delta_{i}^{r} $ with fixed endpoints, each of which retracts to the unique local geodesic; i.e. to the section of $\pi_{r}$ defined by $\geo$.  By showing the compatibility of these homotopies for increasing $r$, we conclude the desired result.
\end{proof}
By restriction, we associate to every local system $\nu$ on $\sL \Q$, a local system on $ \sL^{r}_{\bfdelta^{r}} \Q $ which we still denote $\nu$. The homotopy equivalence in Equation \eqref{eq:equivalence_finite_approx} implies that the inclusion map
\begin{equation} \label{eq:isomorphism_direct_limit}
  \lim_{r} H_{*}(\sL^{r}_{\bfdelta^{r}} \Q ; \nu) \to H_{*}(\sL \Q ; \nu)
\end{equation}
is an isomorphism. This implies, in particular, that operations on $ H_{*}(\sL^{r}_{\bfdelta^{r}} \Q ; \nu) $ which are compatible with the inclusion maps give rise to operations on $H_{*}(\sL \Q ; \nu)  $; this is the point of view we adopt in the remainder of this Chapter.
\subsection{Smooth structures} \label{sec:smooth-structures}
We would like to prove that $\sL^{r}_{\bfdelta^{r}} \Q$ is a submanifold with corners of $\Q^{r}$ for a generic choice of $\bfdelta^{r}$; to this end, we shall use a general result about manifolds defined as the intersection of sublevel sets of smooth functions:
\begin{exercise}
  Let $X$ be a smooth manifold, and $f_{i} \co X \to \bR$ a finite collection of smooth functions labelled by a set $I$. Use Sard's theorem to show that, for a dense open subset of real numbers $\{ \delta_{i} \}$, the intersection
  \begin{equation}
    \bigcap_{i} f_{i}^{-1}(\infty,\delta_{i}]
  \end{equation}
 is a submanifold with corners of $X$. Hint: proceed by induction on $i$, ensuring, at each step, that the hypersurface $ f_{i}^{-1}(\delta_{i}) $ is transverse to all iterated intersections of previously defined hypersurfaces.
\end{exercise}

As an immediate consequence, we conclude
\begin{lem}
For a generic choice of constants $ \bfdelta^{r} $, $\sL^{r}_{\bfdelta^{r}} \Q$ is a smooth submanifold with corners of $ \Q^{r} $. \qed
\end{lem}
The following exercise explains why one must choose the constants $\bfdelta^{r}$ generically:
\begin{exercise}[T. Kragh]
  Show that $ \sL^{4}_{(1,1,1,1)} \bR^{2}$ is not a manifold with corners. 
\end{exercise}

\begin{defin} \label{defin:finite_approx}
  Assume that a sequence $ \bfdelta^{r} $ is fixed so that $ \sL^{r}_{\bfdelta^{r}}  $ is a smooth manifold with corners, and 
  \begin{equation} \label{eq:distance_increases_with_r}
  0 <    \delta_{i}^{r}  \leq \delta_{j}^{r+1}  < 2 \textrm{ if } 1 \leq i \leq r \textrm{ and } 1 \leq j \leq r+1.
  \end{equation}
Define
\begin{equation}
  \sL^{r} \Q \equiv \sL^{r}_{\bfdelta^{r}} \Q.
\end{equation}
\end{defin}
\begin{rem}
In this entire Chapter, there shall be no constraints imposed on $\delta_{i}^{r} $ beyond those imposed by Equation \eqref{eq:distance_increases_with_r}. However, it will later be useful to introduce a constant $\delta$ such that
\begin{equation} \label{eq:piecewise_geodesic_delta}
  \frac{\delta}{2} < \delta_{i}^{r} < \delta.
\end{equation}
This condition is evidently compatible with Equation \eqref{eq:distance_increases_with_r}.  For example, in Section \ref{sec:cauchy-riem-equat}, we shall require that $\delta \leq 1$, while in Section \ref{sec:moduli-space-triangl-+}, we shall assume that $\delta$ is smaller than a constant implicitly depending on the geometry of $\Q$. It is safe to assume that we simply fix a sufficiently small such constant for the remainder of the discussion, and impose Equation \eqref{eq:piecewise_geodesic_delta}. 
\end{rem}
\section{Morse theory}
Given a local system $\nu$, we shall now study the Morse complex which computes $H_{*}(\sL^{r} \Q ; \nu)$: start with a Morse function $f^{r}$ on $\sL^{r} \Q$ whose gradient flow points outwards at the boundary, i.e. so that
\begin{equation}
  df^{r}(\grad(\rho_{i}) )| \rho_i^{-1}(\delta_{i}^{r})  > 0. 
\end{equation}
This implies that $f^{r}$ has no critical point on $\partial \sL^{r} \Q$.  In order to study the Morse complex, we introduce the notation $\psi^{r}_{t}$ for the negative gradient flow of $f^{r}$; this flow is globally defined whenever $ 0 \leq t$, and defined on a closed subset of $\sL^{r} \Q $ for negative $t$, because a positive gradient flow line may escape to the boundary, but a negative one does not.

For each critical point of $f^{r}$, we define the unstable and stable manifolds:
\begin{align}
W^{u}(y) & = \{ \vq  \vbar \lim_{t \to -\infty}  \psi_{t}^{r}(\vq) = y \} \\
W^{s}(y) & = \{ \vq  \vbar \lim_{t \to +\infty}  \psi_{t}^{r}(\vq) = y \}.
\end{align}
\begin{rem}
We think of elements of $W^{u}(y)$ as negative gradient flow trajectories, parametrised by $(-\infty,0]$ which converge to $y$ at $-\infty$, and of  elements of $W^{s}(y)$ as negative gradient flow trajectories, parametrised by $[0,\infty)$ converging to $y$ at $+\infty$. In particular, for every point $q$ in the stable manifold, we have $f^{r}(q) \geq f^{r}(y) $, and the opposite inequality for points in the unstable manifold.
\end{rem}

\begin{exercise} \label{ex:descending_manifold_away_boundary}
Show that the closure of $W^{u}(y)$ is disjoint from $\partial \sL^{r} \Q$. 
\end{exercise}
\begin{defin}
The \emph{Morse index} of a critical point is
\begin{equation}
  \ind(y) = \dim_{\bR}(   W^{u}(y)  ),
\end{equation}
and the orientation line  is
\begin{equation}
  \ro_{y} \equiv | W^{u}(y) |.
\end{equation}
\end{defin}
The Morse assumption implies that $ W^{u}(y)  $  and $ W^{s}(y) $ intersect transversely at a single point, which implies that we have a decomposition of the tangent space
\begin{equation}
  T_{y}\sL^{r} Q \cong T_{y} W^{u}(y) \oplus T_{y} W^{s}(y).
\end{equation}
We conclude
\begin{lem}
  There is a canonical isomorphism:
\begin{equation} \label{eq:spliting_tangent_at_crit_point}
    \ro_{y} \otimes |W^{s}(y)| \cong |T_{y}\sL^{r} Q|.
\end{equation} \qed
\end{lem}

The Morse-Smale assumption is that, for each pair of critical points $y_0$ and $y_1$ the intersection
\begin{equation}
 W^{u}(y_1)   \cap W^{s}(y_0)
\end{equation}
is transverse. 
\begin{exercise} \label{ex:intersection_stable_unstable_gradient}
Show that there is a bijective correspondence between element of this intersection and  negative gradient flow lines 
\begin{align}
  \gamma \co \bR & \to \sL^{r} \Q  \\
\frac{d \gamma}{ds} & = -\grad(f^r)
\end{align}
which converge at $s=-\infty$ to $y_1$   and at $s=+\infty$ to $y_0$.
\end{exercise}
From Exercise \ref{ex:intersection_stable_unstable_gradient}, we conclude that there is a natural $\bR$ action on $ W^{u}(y_1) \cap  W^{s}(y_0) $ by reparametrisation:
\begin{defin}
If $y_0 \neq y_1$, the moduli space of gradient trajectories $\Tree(y_0;y_1)$ is the quotient 
  \begin{equation} \label{eq:define_moduli_gradient}
   \left( W^{u}(y_1) \cap W^{s}(y_0) \right)  / \bR.
  \end{equation}
\end{defin}
\begin{exercise} \label{ex:moduli_precompact}
Use Exercise \ref{ex:descending_manifold_away_boundary} to show that $ W^{u}(y_1) \cap W^{s}(y_0) $ is disjoint from $\partial \sL^{r} \Q$. 
\end{exercise}
\begin{rem}
  The order of the intersection in Equation \eqref{eq:define_moduli_gradient} is not consistent with our conventions in Floer theory, but is chosen to reduce the number of Koszul signs that appear in later constructions.
\end{rem}

The following result is the Morse-theoretic model for Theorem \ref{lem:transversality}. 
\begin{lem}
If $f^{r}$ is Morse Smale, $\Tree(y_0;y_1)$ is a smooth manifold of dimension
\begin{equation}
  \ind(y_1) - \ind(y_0) - 1.
\end{equation} 
\end{lem}
\begin{proof}
  The dimension of $W^{u}(y_1)$ is $\ind(y_1)$, and the codimension of $W^{s}(y_0)$ is $\ind(y_0)$, so the dimension of the intersection is $\ind(y_1) - \ind(y_0)$. Taking the quotient by $\bR$, we subtract $1$.
\end{proof}
We shall be interested in the case $\Tree(y_0;y_1)$ is a $0$-dimensional manifold, i.e.
\begin{equation}
  \ind(y_0) = \ind(y_1) -1.
\end{equation}
In this case we say that trajectories in $\Tree(y_0;y_1)  $ are rigid.
\begin{lem}
Every rigid trajectory $\gamma \in \Tree(y_0;y_1) $ induces a canonical isomorphism
\begin{equation}
\partial_{\gamma} \co  \ro_{y_1} \to \ro_{y_0}.
\end{equation}
\end{lem}
\begin{proof}
At every point along $\gamma$, we have an isomorphism
\begin{equation} \label{eq:splitting_tangent_along_gradient}
| \bR \cdot   \partial_{s}\gamma | \otimes |T_{\gamma(s)}\sL^{r} Q| \cong   \ro_{y_1} \otimes |W^{s}(y_0)|
\end{equation}
induced by the exact sequence
\begin{equation}
\bR \cdot   \partial_{s}\gamma  \to T_{\gamma(s)}  W^{u}(y_1)  \oplus  T_{\gamma(s)}  W^{s}(y_0)  \to T_{\gamma(s)} \sL^{r} Q.
\end{equation}
Combining Equations \eqref{eq:spliting_tangent_at_crit_point} and \eqref{eq:splitting_tangent_along_gradient}, and using the identification of orientations of $T_{y_0}\sL^{r} Q $ and $T_{\gamma(s)}\sL^{r} Q  $ by parallel transport along $\gamma$, we obtain the isomorphism
\begin{equation}
 | \bR \cdot   \partial_{s}\gamma |^{-1} \otimes    \ro_{y_1} \cong     \ro_{y_0}.
\end{equation}
Fixing the \emph{opposite} of the usual orientation on $\bR$, we obtain the desired isomorphism.
\end{proof}
\begin{rem}
 Note that the isomorphism above depends, up to a global sign, only on the choice of an orientation on $\bR$, which induces a direction for the flow line $\gamma$. The choice we make corresponds to considering  positive gradient flow lines with their natural orientation, which makes it compatible with the Floer theoretic conventions in Section \ref{sec:floer-cohom-line}.
\end{rem}

Given an arbitrary local system $\nu$ on  $\sL^{r} \Q$, we define the Morse complex of $f^{r}$ with coefficients in $\nu$ to be the direct sum
\begin{equation}
  CM_{k}(f^{r}; \nu) \equiv \bigoplus_{\ind(y_0) = k} \ro_{y_0} \otimes \nu_{y_0}.
\end{equation}
The differential is given by the expression
\begin{align}
  \partial \co CM_{k}(f^{r}; \nu) & \to CM_{k-1}(f^{r}; \nu) \\
\partial |  \ro_{y_1} \otimes \nu_{y_1} & = \sum_{\substack{ \ind(y_0) = \ind(y_1)-1 \\   \gamma \in   \Tree(y_0;y_1)} } \partial_{\gamma} \otimes  \nu_{\gamma},
\end{align}
where $\nu_{\gamma}$ is the parallel transport map along $\gamma$. In order to prove that  $\partial^{2}=0$, we recall that there is a natural compactification $\Treebar(y_0;y_1)$ of the moduli space of trajectories from $y_1$ to $y_0$. In the usual setting (of Morse theory on closed manifolds) the proof of compactness is standard (see, e.g. \cite{AD}).   Exercise \ref{ex:intersection_stable_unstable_gradient} ensures that no gradient trajectory escapes  to the boundary of $\sL^{r} \Q$, so we conclude that $ \Treebar(y_0;y_1) $ is also compact. 
\begin{exercise}
Imitating the case of closed manifolds (see, e.g. \cite{AD}), show that $\partial^{2}=0 $ by considering $1$-dimensional moduli spaces $\Treebar(y_0;y_1)$.  
\end{exercise}
We write 
\begin{equation}
    HM_{*}(\sL^{r} Q; \nu) \equiv H_{*} \left(CM_{k}(f^{r}; \nu), \partial   \right)
\end{equation}
for the resulting homology groups.

\subsection{From geometric chains to Morse chains}
\label{sec:from-geom-chains}

The following result essentially goes back to Morse and Baiada \cite{MB}, who studied Morse theory on manifolds with boundary. 
\begin{prop} \label{prop:Morse_standard}
Morse homology is independent of $f^{r}$, and canonically isomorphic to the (ordinary) homology of $ \sL^{r} Q $ with coefficients in $\nu$. \qed
\end{prop}

Even though the above result is not usually stated in this generality in the literature, it can be proved using many of the methods that have been used for a trivial local system, including the original proof using   the long exact sequence on homology associated to adding an additional cell (see \cite{milnor}), or proofs which equip the unstable manifold with a fundamental chain in a more geometric theory of chains, e.g. pseudo-cycles in \cite{Schwarz-equivalence}, or cubical chains in \cite{BH}.  Over the reals, one can also use the comparison with de Rham cohomology \cite{witten}.

There is an alternative proof of Proposition \ref{prop:Morse_standard} which relies on the pairing between geometric chains and the ascending manifolds of critical points, e.g. \cite{A-HMS-toric}. The main idea is to introduce, for each critical point $y$ of $f^{r}$, the space
\begin{equation} \label{eq:definition_compactification}
  \bar{W}^{s}(y) \equiv \bigcup_{\substack{0 \leq d \\ y_1, \ldots, y_d}} W^{s}(y_d) \times \Tree(y_{d-1};y_d) \times \cdots \times \Tree(y;y_1) 
\end{equation}
which includes $W^{s}(y)$ as an open subset. The topology on this space can be thought of as the Morse-theoretic analogue of the Gromov-Floer topology on pseudo-holomorphic curves: recall that an element of $W^{s}(y)  $ is a negative gradient flow trajectory parametrised by the interval $[0,+\infty)$. A sequence $\gamma_k$ of such flow lines converges to an element of $ W^{s}(y_d) \times \Tree(y_{d-1};y_d) \times \cdots \times \Tree(y;y_1)   $ if, for each $k$ one can decompose $[0,+\infty)$ into $d$ disjoint intervals whose length goes to infinity with $k$ so that the following two properties hold:
\begin{enumerate}
\item every point in $[0,+\infty)$ lies in the union of the intervals for $k$ large enough, and
\item the restriction of $\gamma_k$ to the $i$\th interval converges (up to translation in the source) to an element of $ \Tree(y_{i-1};y_i) $ if $i \neq d$, and to an element of $W^{s}(y_d)   $  otherwise.
\end{enumerate}
The natural evaluation map
\begin{equation}
   \bar{W}^{s}(y) \to \sL^{r} \Q,
\end{equation}
whose restriction to the stratum in the right of Equation \eqref{eq:definition_compactification} is given by composing the  projection to $ W^{s}(y_d) $ with the inclusion of this space in the ambient manifold, is  continuous  and proper in this topology. Moreover, it is know that this space admits the structure of a smooth manifold with corners, such that this map is in fact smooth (see, e.g.  \cite{Latour}*{D\'efinition 2.7} where such a construction is considered for the generalisation of Morse theory to closed $1$-forms).

We shall use a weaker variant, which can be obtained by constructing a smooth structure near the codimension $1$ strata of $   \bar{W}^{s}(y)  $, and allows one not to have to worry about what is happening in higher codimension. Let $\bar{P}$ be a compact topological space stratified by smooth manifolds of bounded dimension, and let $P$ denote the union of the top dimensional strata and  $ \partial^{1} \bar{P}  $ the strata of codimension one. Assume that
\begin{equation}
  P \cup \partial^{1} \bar{P}
\end{equation}
admits the structure of a smooth manifold with boundary.  We say that a map
\begin{equation}
  \bar{P} \to \sL^{r} \Q
\end{equation}
is smooth if the restriction to each stratum and to $P \cup \partial^{1} \bar{P}  $ is smooth.
\begin{lem} \label{lem:geom_chain_to_Morse}
Let $\bar{P} \to \sL^{r} \Q$ be a smooth map whose restriction to the strata of $\bar{P}$ is transverse to all the strata of $   \bar{W}^{s}(y)   $. If 
\begin{equation}
  \dim(P) = \ind(y) \textrm{ or }   \dim(P) = \ind(y) +1,
\end{equation}
then $     \bar{P} \times _{\sL^{r} \Q}  \bar{W}^{s}(y) $  is a compact manifold with boundary. The boundary is covered by two strata:
\begin{equation}
     \partial^{1} \bar{P} \times _{\sL^{r} \Q} W^{s}(y)    \textrm{ and } P \times _{\sL^{r} \Q}    W^{s}(y_1) \times   \Tree(y;y_1)  .
\end{equation} 
\qed
\end{lem}

Whenever $  \dim(P) = \ind(y)  $, the transverse fibre product $   \bar{P} \times _{\sL^{r} \Q}  \bar{W}^{s}(y) $  consists of finitely many points which all lie in $  P \times _{\sL^{r} \Q}    W^{s}(y)  $, so we have a canonical isomorphism
\begin{equation}
  T P \oplus T W^{s}(y)  \cong T \sL^{r} \Q.
\end{equation}
Passing to orientation lines, and using the isomorphism in Equation \eqref{eq:spliting_tangent_at_crit_point}, we obtain an isomorphism
\begin{equation} \label{eq:map_transverse_intersection_Morse_trivial}
  |T_{p}P| \to \ro_{y}.
\end{equation}
If we assume that $P$ is oriented, we define an element
\begin{equation}
  [\bar{P}]  \in CM_{*}(f^{r}; \bZ)
\end{equation}
as the sum of the images of $1 \in \bZ \cong |T_{p}P|$ under all maps induced by Equation \eqref{eq:map_transverse_intersection_Morse_trivial}.
\begin{exercise}
Show that $\partial [\bar{P}] = [\partial^{1} \bar{P}]$.
\end{exercise}
At this stage, there are many ways to proceed in order to produce a map from ordinary homology to Morse homology. One way is to choose a simplicial triangulation with the property that all cells are transverse to all ascending manifolds, and apply the above result to obtain a chain map from simplicial homology to Morse homology, see, e.g. \cite{A-HMS-toric}.

More generally, if $\nu$ is a local system on $\sL^{r} \Q$,  combining Equation \eqref{eq:map_transverse_intersection_Morse} with parallel transport along the geodesic induces a map
\begin{equation} \label{eq:map_transverse_intersection_Morse}
  |T_{p}P| \otimes \nu_{p} \to \ro_{y} \otimes \nu_{y}.
\end{equation}
In this way, we obtain a map from homology twisted by $\nu$, to Morse homology twisted by $\nu$.

The proof that these maps are isomorphisms requires more care. One approach is to choose a specific Morse function for which all flow lines can be explicitly computed; e.g. a function whose critical points correspond to the barycenters of a simplicial triangulation. Alternatively, one can use the space 

\begin{equation} \label{eq:definition_compactification_descending}
  \bar{W}^{u}(y) \equiv \bigcup_{\substack{0 \leq d \\ y_{-1}, \ldots, y_{-d}}}   \Treebar(y_{-1};y)  \times \cdots \times  \Treebar(y_{-d};y_{-d+1})\times W^{u}(y_{-d}) 
\end{equation}
to construct an inverse map.

\subsection{Inclusion maps in Morse homology}
In this section, we construct a Morse theoretic model for the map on homology induced by the inclusion in Equation \eqref{eq:map_successive_r}. 

Given critical points $y_0$ of $f^{r}$, and $y_1$ of $f^{r-1}$, we define
\begin{equation} \label{eq:definition_morse_cont_moduli}
  \Tree_{\iota}(y_0;y_1) =  \iota( W^{u}(y_1) ) \cap  W^{s}(y_0) .
\end{equation}
We can think of every element of $  \Tree_{\iota}(y_0;y_1)  $ as a \emph{piecewise trajectory}:
\begin{equation}  \gamma \co \bR \to \sL^{r} \Q
\end{equation}
which maps $(-\infty,0]$ to the composition of $\iota$ with a negative gradient flow line of $f^{r-1}$ converging at $-\infty$ to $y_1$, and $[0,+\infty)$ to a negative gradient flow line of $f^{r}$ converging at $\infty$ to $y_0$, and with a matching condition at $0$. While this point of view is useful for understanding the compactification of $  \Tree_{\iota}(y_0;y_1)$, the original definition readily yields the following result:
\begin{lem}
  If $f^{r-1}$ is fixed, then for a generic choice of function $ f^{r} $, $ \Tree_{\iota}(y_0;y_1)$ is a smooth manifold of dimension
\begin{equation}
  \ind(y_1) - \ind(y_0).
\end{equation}
\end{lem}
\begin{proof}
The transversality statement follows from Sard's theorem. To compute the dimension, observe that the dimension of $W^{u}(y_1)  $ is $\ind(y_1)$, and that the codimension of $W^{s}(y_0)$ is $ \ind(y_0) $.
\end{proof}
\begin{rem} \label{rem:transversality_more_general}
  Even if $f^{r-1}$ and $ f^{r} $ are both fixed, one can achieve transversality by a slight tweaking of the definition: choose a family $X_{s}$ of vector fields on $\sL^{r} \Q $ parametrised by $s \in [0,1]$, which point outwards along the boundary of $ \sL^{r} \Q  $. Write $\psi_{X}$ for the diffeomorphism obtained by integrating this family, and define
  \begin{equation}
    \Tree_{\iota}^{X}(y_0;y_1) \equiv \iota( W^{u}(y_1) ) \cap  \psi^{-1}_{X}( W^{s}(y_0)). 
  \end{equation}
By choosing $X$ generically, we may ensure that this is a transverse intersection. We can replace all future uses of $ \Tree_{\iota}(y_0;y_1) $ by this perturbed moduli space. While the definition of $ \Tree_{\iota}^{X}(y_0;y_1) $  only depends on the diffeomorphism $\psi_{X}$, the choice of vector field is needed, up to homotopy, to construct induced maps on local systems.
\end{rem}
Let us now assume that
\begin{equation}
    \ind(y_1) = \ind(y_0),
\end{equation}
which implies that $   \Tree_{\iota}(y_0;y_1)  $ is a $0$-dimensional manifold. We claim that every element $\gamma \in \Tree_{\iota}(y_0;y_1)  $    induces a canonical map
\begin{equation} \label{eq:continuation_map_trajectory}
  \iota_{\gamma} \co \ro_{y_1} \to \ro_{y_0}.
\end{equation}
To define this map, we observe that the transversality assumption implies that the natural map
\begin{equation}
T_{\gamma(0)} W^{u}(y_1)  \oplus  T_{\gamma(0)} W^{s}(y_0) \cong T_{\gamma(0)} \sL^{r} \Q
\end{equation}
is an isomorphism. This yields a canonical map
\begin{equation}
  \ro_{y_0} \otimes \ro_{y_1} \cong |\sL^{r} \Q |.
\end{equation}
The map in Equation \eqref{eq:continuation_map_trajectory} is then induced by comparing this with the isomorphism in Equation \eqref{eq:spliting_tangent_at_crit_point}.

Given a local system $\nu$ on $\sL \Q$, with pullbacks to $\sL^{r-1} \Q $  and $\sL^{r} \Q $ which we also denote by $\nu$, the homotopy commutativity of Diagram \eqref{eq:finite_approx_maps_commute} implies that $\gamma$ induces a canonical isomorphism
\begin{equation}
 \nu_{\gamma} \co \nu_{y_1} \to  \nu_{y_0}. 
\end{equation}
We now define a map
\begin{align}
  \iota \co CM_{k}(f^{r-1}; \nu) & \to CM_{k}(f^{r}; \nu) \\
\partial |  \ro_{y_1} \otimes \nu_{y_1} & = \sum_{\substack{ \ind(y_0) = \ind(y_1) \\   \gamma \in   \Tree_{\iota}(y_0;y_1)} } \iota_{\gamma} \otimes  \nu_{\gamma}.
\end{align}

In order to prove that $\iota$ is a chain map, we consider the space
\begin{align}
   \Treebar_{\iota}(y_0;y_1)  \equiv  \iota( \bar{W}^{u}(y_1) ) \cap  \bar{W}^{s}(y_0).
\end{align}
\begin{exercise}
  Show that $  \Treebar_{\iota}(y_0;y_1) $ is compact.
\end{exercise}
As in Lemma \ref{lem:geom_chain_to_Morse}, the codimension $1$ strata of $   \Treebar_{\iota}(y_0;y_1) $ are:
\begin{equation}
  \iota( \bar{W}^{u}(y_1) ) \cap  W^{s}(y_0') \times  \Tree(y_0;y_0') \textrm{ and }  \Tree(y'_1;y_1) \times \iota( \bar{W}^{u}(y_1') ) \cap  W^{s}(y_0'). 
\end{equation}
\begin{exercise}
By considering moduli spaces $ \Treebar_{\iota}(y_0;y_1) $ which have dimension $1$, show that $\iota$ is a chain map.
\end{exercise}
As usual, we write 
\begin{equation}
  \iota \co HM_{*}(\sL^{r-1} \Q; \nu)  \to HM_{*}(\sL^{r} \Q; \nu)
\end{equation}
for the induced map on homology. Using the isomorphism of Morse and ordinary homology, this map is the one naturally induced by inclusion. In particular, Equation \eqref{eq:isomorphism_direct_limit} yields an isomorphism
\begin{equation}
  \lim_{\iota} HM_{*}(\sL^{r} \Q; \nu) \cong H_{*}(\sL \Q; \nu).
\end{equation}

\section{Operations on loop homology}

\subsection{The unit}
The first space in our finite dimensional approximation is $\sL^{1} \Q \cong Q$. Assume that we are given 
\begin{equation}
  \parbox{35em}{a local system $\nu$ on $\sL \Q$ with an isomorphism of the pullback to $\Q$  with $|\Q|^{-1}$.}
\end{equation}
From this, we obtain a map
\begin{equation}
  HM_{*}(\Q; |Q|^{-1}) \to \lim_{r} HM_{*}(\sL^{r} \Q; \nu)  \cong H_{*}(\sL \Q; \nu).
\end{equation}
\begin{exercise}
Show that the restriction of the local system $\eta$ to constant loops is equipped with such a canonical isomorphism (see the discussion in the introduction to Chapter \ref{cha:from-sympl-homol}).
\end{exercise}
The homology of $\Q$ with coefficients in the local system of orientations $|\Q|^{-1}$ is equipped with a natural fundamental cycle
\begin{equation}
   \bZ \to HM_{0}(\Q; |\Q|^{-1}).
\end{equation}
In the Morse setting we are considering, this fundamental cycle arises as follows: if $f^{1}$ is a Morse function on $\Q$, and $y$ is a maximum, we have a canonical isomorphism:
\begin{equation}
  \ro_{y} \cong |\Q|.
\end{equation}
From this, we obtain the isomorphism
\begin{equation}
  e_{y} \co \bZ \cong  \ro_{y} \otimes |\Q|^{-1}.
\end{equation}

At the chain level, we define
\begin{align}
 e \co \bZ & \to CM_{0}(\Q; |\Q|^{-1})  \\
1 & \mapsto \sum_{\ind(y) = n}  e_{y}(1).
\end{align}
Finally, we recall that we defined a local system $\eta$ in Equation \eqref{eq:define_eta}:
\begin{defin}
  The unit $e$ of the loop space homology with coefficients in $\eta$ is the composition
  \begin{equation}
    \bZ \to HM_{0}(\Q; |\Q|^{-1})  \to  H_{0}(\sL \Q; \eta).
  \end{equation}
\end{defin}

\subsection{The $BV$ operator}
Let us now assume that we have a local system $\nu$ which is $S^1$ equivariant. This consists of an isomorphism between  the two local systems on
\begin{equation}
  S^{1} \times \sL \Q \to \sL \Q
\end{equation}
obtained by pulling back $\nu$ under (i) the projection to the second factor or (ii)  the action of the circle on the free loop space by reparametrisation.  Equivalently, if $\gamma_{\theta}$ is obtained from a loop $\gamma$ by precomposing with a rotation by $\theta$, we assume the existence of an isomorphism
\begin{equation} \label{eq:isomorphism_equivariant_local_system}
  \nu_{\gamma_{\theta}} \cong \nu_{\gamma}
\end{equation}
varying continuously both in $\gamma$ and $\theta$, and which is the identity when $\theta=1$.

\begin{exercise} \label{ex:eta-s-1-equivariant}
If $\Q$ is non-orientable, show that $\ev_{0}^{*}|Q|$ does not admit an $S^1$-equivariant structure. By considering separately the components of the loop space consisting of orientable and non-orientable loops, show that the local system $\eta$ naturally admits such a structure.
\end{exercise}

Under this assumption, we obtain a $BV$ operator
\begin{equation}
  H_{*}(\sL \Q; \nu) \to H_{*+1}(\sL \Q; \nu)
\end{equation}
which is induced by the circle action.  In this section, we construct this operator from the finite dimensional point of view. 
\begin{exercise}
Construct a natural action of $S^1$ on $ \sL^{r}_{(1, \ldots, 1)} \Q $. 
\end{exercise}

As explained in Section \ref{sec:smooth-structures}, the choice of equal successive lengths does not yield in general a manifold with corners. We shall remedy this problem by  constructing instead a family of embeddings parametrised by $S^1$:
\begin{equation} \label{eq:family_circle_embeddings}
a \co  S^{1} \times \sL^{r} \Q \to \sL^{r+1} \Q.
\end{equation}
Identifying $S^1$ with $\bR/\bZ$, we construct this family of embeddings in two steps:
\begin{enumerate}
\item If $\theta \in [0,1/r)$, define $q_{i}^{\theta}$ to be the point on the shortest geodesic between $q_{i}$ and $ q_{i+1}$ satisfying
  \begin{equation}
    \frac{d(q_{i}, q_{i}^{\theta})}{ d(q_{i}, q_{i+1})} = r \theta. 
  \end{equation}
\item Write a general element of the circle as $\theta= \theta_{0} + \frac{i}{r}$ with $\theta_{0} \in [0,1/r)$, and define:
\begin{equation} \label{eq:action_general_theta}
     a(\theta, q_0, \ldots, q_{r-1}) =  (q_i^{\theta},q_i^{\theta}, q_{i+1}^{\theta}, \ldots , q_{i-1}^{\theta}).
\end{equation}
\end{enumerate}
The fact that the first two coordinates in the right hand side of Equation \eqref{eq:action_general_theta} are equal is consistent with our definition of the map $\iota$ in Equation (\ref{eq:map_successive_r}), so that $a$ corresponds to composing $\iota$ with the partially defined $S^1$ action.
\begin{exercise}
Use the triangle inequality and Equation \eqref{eq:distance_increases_with_r} to show that 
\begin{equation}
  d(q_{i}^{\theta}, q_{i+1}^{\theta}) \leq \delta_{j}^{r+1}
\end{equation}
for  $1 \leq j \leq r+1$ whenever $ d(q_{i}, q_{i+1}) \leq \delta_{i}^{r}$ and   $ d(q_{i+1}, q_{i+2}) \leq \delta_{i+1}^{r}$. Conclude that whenever $\vq$   lies in $\sL^{r} \Q$ the right hand side of Equation \eqref{eq:action_general_theta}  lies in $  \sL^{r+1} \Q$.
\end{exercise}
\begin{exercise}
Show that Equation \eqref{eq:action_general_theta} defines a continuous family of embeddings parametrised by $S^1$ (Hint: the key point is to check that the two possible ways of defining rotation by $1/r$ agree).
\end{exercise}

We now define the map induced at the level of Morse complexes:  given critical points $y_+$ of $f^{r}$ and $y_-$ of $f^{r+1}$, consider the fibre product
\begin{equation}
  \Tree_{\Delta}(y_-;y_+)  \equiv  \left( S^{1} \times W^{u}(y_+)   \right) \times_{\sL^{r+1} \Q}  W^{s}(y_-)
\end{equation}
where the evaluation map on the first factor is
\begin{align}
 a \co S^{1} \times W^{u}(y_+) & \to  \sL^{r+1} \Q \\
(\theta, \vq) & \mapsto a(\theta,\vq).
\end{align}

\begin{lem}
Fix a Morse function $f^{r}$. For a generic function $f^{r+1}$, the moduli space $\Tree_{\Delta}(y_-;y_+)$ is a manifold of dimension $\ind(y_+) - \ind(y_-)+1 $.
\end{lem}
\begin{proof}
If the  map $a$ were smooth, this would follow immediately from Sard's theorem. We require an additional step because we constructed the family of embeddings in a piecewise way.

For a generic function $f^{r+1}$, the fibre products
\begin{equation}
  \left( {\Big \{}  \frac{i}{r} {\Big \}} \times  W^{u}_{y_+} \right)  \times_{ \sL^{r} \Q} W^{s}(y_-) 
\end{equation}
are transverse. 

Using Sard's theorem again, we see that the fibre products
\begin{equation}
 \left(\left[ \frac{i}{r}, \frac{i+1}{r} \right] \times   W^{u}(y_+) \right)\times_{ \sL^{r} \Q} W^{s}(y_-) 
\end{equation}
are also transverse, and define cobordisms between the manifolds at the endpoints. Decomposing $S^1$ as the union of these intervals, we obtain the desired result.
\end{proof}

We now repeat the same strategy as for the construction of the inclusion map: assume that
\begin{equation}
   \ind(y_-) =  \ind(y_+) +1, 
\end{equation}
so that the moduli space $ \Tree_{\Delta}(y_-;y_+)  $ is $0$-dimensional.

By transversality, we have an isomorphism
\begin{equation} \label{eq:splitting_tangent_space_circle_moduli}
  T_{\theta} S^{1} \oplus T_{a(\theta,\gamma_1(0))} W^{u}(y_+) \oplus   T_{\gamma_2(0)} W^{s}(y_-)  \cong T_{\gamma_2(0)} \sL^{r+1} \Q  
\end{equation}
whenever $\gamma= (\gamma_1,\gamma_2)$ is an element of $ \Tree_{\Delta}(y_-;y_+)  $, which we think of as consisting of pairs of gradient flow lines in $\sL^{r} \Q$ and $\sL^{r+1} \Q$ matched at the end points.
\begin{exercise}
Choosing an orientation of the circle, use Equations \eqref{eq:spliting_tangent_at_crit_point} and \eqref{eq:splitting_tangent_space_circle_moduli} to construct a canonical isomorphism:
\begin{equation}
  \Delta_{\gamma} \co \ro_{y_+} \to \ro_{y_-}
\end{equation} 
associated to any rigid element of $ \Tree_{\Delta}(y_-;y_+)  $.
\end{exercise}

Assume that $\nu$ is an $S^1$ equivariant local system equipped with maps as in Equation \eqref{eq:isomorphism_equivariant_local_system}. In this case, we can assign a map
\begin{equation}
  \nu_{\gamma}\co \nu_{y_+} \to \nu_{y_-}
\end{equation}
as the composition
\begin{equation}
  \nu_{y_+} \to \nu_{\gamma_1(0)} \to \nu_{a(\theta,\gamma_1(0))} \to \nu_{y_-},
\end{equation}
where the first map is defined by parallel transport along $\gamma_1$, the second by the equivariant structure, and the last by parallel transport along $\gamma_2$. 

For such a local system $\nu$, we then define
\begin{align}
\Delta \co CM_{k}(f^{r}; \nu) & \to CM_{k+1}(f^{r+1}; \nu)  \\
\Delta | \ro_{y_+}\otimes  \nu_{y_+} & = \sum_{\substack{ \ind(y_-) = \ind(y_+) +1 \\   \gamma \in   \Tree_{\Delta}(y_-;y_+)} } \Delta_{\gamma} \otimes  \nu_{\gamma}.
\end{align}
\begin{exercise}
Show that $\Delta$ is a chain map.
\end{exercise}
By abuse of notation, we use the same notation for the induced map on homology:
\begin{equation}
 \Delta \co HM_{*}(\sL^{r} \Q; \nu) \to  HM_{*+1}(\sL^{r+1} \Q; \nu).
\end{equation}
\begin{exercise} \label{ex:circle_action_commutes_inclusion}
Show that we have a homotopy commutative diagram
\begin{equation}
  \xymatrix{S^{1} \times \sL^{r-1} \Q  \ar[d]^{\id \times \iota} \ar[r]^{a} &  \sL^{r} \Q \ar[d]^{\iota} \\
S^{1} \times \sL^{r} \Q \ar[r]^{a}  &   \sL^{r+1} \Q. }
\end{equation}
\end{exercise}
\begin{exercise}\label{ex:Delta_commutes_inclusion}
Using Exercise \ref{ex:circle_action_commutes_inclusion}, show that $\Delta$ commutes with $\iota$.
\end{exercise}

Applying Exercise \ref{ex:eta-s-1-equivariant}, we conclude that we have a degree $1$ operator on loop homology with coefficients in $\eta$:
\begin{equation}
\Delta \co HM_{*}(\sL \Q; \eta) \to HM_{*+1}(\sL \Q; \eta).
\end{equation}

\subsection{The loop product with coefficients in $|Q|^{-1}$}

The loop product
\begin{equation}
  H_{i}(\sL \Q ; \ev^{*}_{0} |Q|^{-1}) \otimes H_{j}(\sL \Q ;  \ev^{*}_{0} |Q|^{-1}) \to H_{i+j}(\sL \Q ; \ev^{*}_{0} |Q|^{-1} )
\end{equation}
is defined  by composing the fibre product over the evaluation at the starting point with the concatenation of loops see, e.g. \cites{CS,Laudenbach}.

In this section, we define a map at the level of finite approximations:
\begin{equation}
HM_{*}(\sL^{r_1} Q; \ev^{*}_{0} |Q|^{-1} ) \otimes   HM_{*}(\sL^{r_2} Q; \ev^{*}_{0} |Q|^{-1} ) \to  HM_{*}(\sL^{r_1+r_2} Q; \ev^{*}_{0} |Q|^{-1} ).
\end{equation}

We start by introducing the analogue of the evaluation map at the starting point
\begin{equation}
  \ev_{0} \co  \sL^{r} \Q  \to Q ,
\end{equation}
which is the projection to the first coordinate $q_0$.  Next, we fix an embedding
\begin{align} \label{eq:map_product_to_more_points}
  \sL^{r_1} \Q \times_{ev_0} \sL^{r_2} \Q  & \to \sL^{r_1+r_2 } \Q \\
\left((q_0, q_1, \ldots, q_{r_1-1}) , (q_0, q_1', \ldots, q_{r_2-1}') \right) & \mapsto (q_0, q_1, \ldots, q_{r_1-1}, q_0, q_1', \ldots, q_{r_2-1}').
\end{align}
Using Equation \eqref{eq:distance_increases_with_r}, we see that the above map is well-defined, i.e. that the distance between the successive points on the right hand side satisfy the inequalities required of elements of $ \sL^{r_1+r_2} \Q  $.
\begin{exercise}
Show that the following diagram commutes up to homotopy:
  \begin{equation}
    \xymatrix{   \sL^{r_1} \Q \times_{ev_0} \sL^{r_2} \Q  \ar[r] \ar[d] & \sL^{r_1 +r_2} \Q  \ar[d] \\
 \sL^{r_1} \Q \times_{ev_0} \sL^{r_2 +1} \Q \ar[r] &  \sL^{r_1 +r_2 +1} \Q,} 
  \end{equation}
and strictly commutes if we switch the roles of the first and second factor (Hint:   using Exercise \ref{ex:homotopy_different_inclusions} define a map, homotopic to  Equation \eqref{eq:map_successive_r} so that the diagram strictly commutes.)
\end{exercise}

Given critical points $y_0$ of $f^{r_1+r_2}$, and $y_i$ of $f^{r_i}$, we define
\begin{equation} \label{eq:multiplication_moduli_space_2_inputs}
  \Tree(y_0; y_1, y_2) \equiv  \left( W^{u}(y_1) \times_{\ev_0} W^{u}(y_2)   \right) \times_{ \sL^{r_1+r_2} \Q   } W^{s}(y_0).
\end{equation}
This is a rather complicated definition, in the spirit of Equation \eqref{eq:definition_morse_cont_moduli}, which we unpack in the following result:
\begin{lem}
Assume $r_1 \neq r_2$. If $f^{r_1}$ is fixed, we may choose $f^{r_2}$ and $f^{r_1 +r_2}$ so that the fibre product in Equation \eqref{eq:multiplication_moduli_space_2_inputs} is transverse. In this case, $   \Tree(y_0; y_1, y_2)  $  is a smooth manifold of dimension
\begin{equation} \label{eq:dimension_tree}
  \ind(y_1) + \ind(y_2) - n - \ind(y_0).
\end{equation}
\end{lem}
\begin{proof}
  We first choose $f^{r_2} $ so that $ev_{0}| W^{u}(y_2)$ is transverse to $ev_{0}|W^{u}(y_1)$. Since the dimension of these manifolds is $\ind(y_1)$ and $\ind(y_2)$, the dimension of the fibre product $W^{u}(y_1) \times_{\ev_0} W^{u}(y_2)    $  is
  \begin{equation}
     \ind(y_1) + \ind(y_2) - n.
  \end{equation}
Using the map in Equation \eqref{eq:map_product_to_more_points}, we obtain an inclusion
\begin{equation}
  W^{u}(y_1) \times_{\ev_0} W^{u}(y_2)   \to  \sL^{r_1} \Q \times_{ev_0} \sL^{r_2} \Q  \to \sL^{r_1 +r_2} \Q.
\end{equation}
For a generic choice of the Morse function $f^{r_1+r_2}  $, the ascending manifold of $y_0$ is transverse to this inclusion. Since this ascending manifold has codimension $ \ind(y_0)  $, we conclude that the dimension of the intersection is given by Equation \eqref{eq:dimension_tree}.
\end{proof}
\begin{rem}
 Following Remark \ref{rem:transversality_more_general}, a small perturbation of the definition would allow us to achieve transversality even if the functions $f^{r_1}$, $f^{r_2}$, and $f^{r_1+r_2}$ are all fixed. This is particularly important because one might be interested in dropping the restrictive assumption that $r_1 \neq r_2$, in  which case the fibre product $ W^{u}(x) \times_{\ev_0} W^{u}(x)  $ is never transverse unless $x$ is a maximum.

The key point is to choose vector fields $X^{2}_{s}$ on $\sL^{r_2} \Q   $ and $X^{0}_{s}$ on $ \sL^{r_1+r_2} \Q $ , parametrised by the interval, which integrate to diffeomorphisms $\psi_{2,X}$ and  $\psi_{0,X}$. We then define
\begin{equation}
   \Tree^{X}(y_0; y_1, y_2) \equiv     \left( W^{u}(y_1) \times_{\ev_0} \psi_{2,X}( W^{u}(y_2))   \right) \times_{ \sL^{r_1 +r_2} \Q   } \psi_{0,X}^{-1}( W^{s}(y_0)).
\end{equation}
The interested reader is invited to construct the product in more generality using this moduli space.
\end{rem}

Let us now assume that
\begin{equation}
\ind(y_0) =  \ind(y_1) + \ind(y_2) - n
\end{equation}
 which implies that $  \Tree(y_0; y_1, y_2) $  is $0$-dimensional.
 \begin{lem}
   We can canonically associate to every element $\gamma \in   \Tree(y_0; y_1, y_2)$ a map
 \begin{equation} \label{eq:map_induced_gamma}
\star_{\gamma} \co \left(  \ro_{y_1} \otimes  \ev_{0}^{*}|Q|^{-1}_{y_1}  \right)  \otimes \left( \ro_{y_2} \otimes  \ev_{0}^{*}|Q|^{-1}_{y_2} \right)  \to \ro_{y_0}\otimes \ev_0^{*}|Q|^{-1}_{y_0}.
 \end{equation}
 \end{lem}
 \begin{proof}
By construction, we have a homotopy between $\geo(y_0)$ and the concatenations of $\geo(y_1)$ and $  \geo(y_2)$. In particular, we have a canonical isomorphism
\begin{equation}
  \ev_{0}^{*}|Q|^{-1}_{y_1} \cong \ev_{0}^{*}|Q|^{-1}_{y_2}  \cong \ev_0^{*}|Q|^{-1}_{y_0}.
\end{equation}
The isomorphism in Equation \eqref{eq:map_induced_gamma} can then be simplified as follows (we keep track of the Koszul sign at each step below):
\begin{alignat}{3} \label{eq:last_map_product_to_simplified}
  \ro_{y_1} \otimes  \ev_{0}^{*}|Q|^{-1}_{y_1}   \otimes  \ro_{y_2}  &  \cong \ro_{y_0}  &  & \quad 0  \\
\ro_{y_1}     \otimes  \ro_{y_2} \otimes  \ev_{0}^{*}|Q|^{-1}_{y_1} &  \cong \ro_{y_0}  & & \quad n \ind(y_2)\\ \label{eq:simplified_map}
\ro_{y_1}  \otimes  \ro_{y_2} & \cong \ro_{y_0}  \otimes  \ev_{0}^{*}|Q|_{y_1}   & & \quad  0
\end{alignat}
We now construct the last isomorphism as a canonical map associated to transverse fibre products.  First, we have a short exact sequence
\begin{equation}
  T \left( W^{u}(y_1) \times_{\ev_0} W^{u}(y_2)  \right) \to T W^{u}(y_1) \oplus  T W^{u}(y_2 ) \to T \Q,
\end{equation}
which induces an isomorphism
\begin{equation} \label{eq:first_fibre_product_loop}
  |  W^{u}(y_1) \times_{\ev_0} W^{u}(y_2)| \otimes \ev_{0}^{*}|Q|_{y_1}  \cong \ro_{y_1}  \otimes  \ro_{y_2} .
\end{equation}
Next, we use the isomorphism
\begin{equation}
 T \left( W^{u}(y_1) \times_{\ev_0} W^{u}(y_2) \right)  \oplus  T W^{s}(y_0 ) \to T \sL^{r_1+r_2} \Q,
\end{equation}
to produce the isomorphism
\begin{equation} \label{eq:second_fibre_product_loop}
 |  \sL^{r_1+r_2} \Q |  \cong   |  W^{u}(y_1) \times_{\ev_0} W^{u}(y_2)|  \otimes |W^{s}(y_0)|.
\end{equation}
Together with the isomorphism
\begin{equation}
   |  \sL^{r_1+r_2} \Q |  \cong  \ro_{y_0}  \otimes |W^{s}(y_0)|,
\end{equation}
this yields an identification
\begin{equation} \label{eq:identification_fibre_product_descending}
  |  W^{u}(y_1) \times_{\ev_0} W^{u}(y_2)|  \cong  \ro_{y_0}.
\end{equation}
Combining this with Equation \eqref{eq:first_fibre_product_loop}, arrive at  Equation \eqref{eq:simplified_map}, which we explained at the beginning of the proof yields the isomorphism in Equation \eqref{eq:map_induced_gamma}.
 \end{proof}

At this stage, we define the product
\begin{align}
  \star \co CM_{i}(f^{r_1}; \ev_{0}^{*}|Q|^{-1} ) \otimes CM_{j}(f^{r_2};\ev_{0}^{*}|Q|^{-1} ) & \to CM_{i+j}(f^{r_1+r_2}; \ev_{0}^{*}|Q|^{-1} ) \\
\star | \left(  \ro_{y_1} \otimes  \ev_{0}^{*}|Q|^{-1}_{y_1}  \right)  \otimes \left( \ro_{y_2} \otimes  \ev_{0}^{*}|Q|^{-1}_{y_2} \right)  & \equiv \sum_{ \substack{\ind(y_0) = \ind(y_1) + \ind(y_2) -n  \\ \gamma \in  \Tree(y_0; y_1, y_2)  }  } \star_{\gamma}.
\end{align}
\begin{rem}
  Note that the product $\star$ preserves degree because the line $|Q|^{-1}$ is supported in degree $-n$, so that $ \ro_{y} \otimes   |Q|^{-1}_{y} $ is supported in degree
  \begin{equation}
    \ind(y) - n. 
  \end{equation}
Simple arithmetic implies that
\begin{equation}
  (   \ind(y_1) - n ) + (   \ind(y_2) - n )  =   \ind(y_0) - n 
\end{equation}
whenever the moduli space $ \Tree(y_0; y_1, y_2) $ has dimension $0$.
\end{rem}
The next step is to prove that the differential satisfies the Leibnitz rule with respect to this product. The key point is that, whenever $  \Tree(y_0; y_1, y_2)  $ has dimension $1$, it admits a natural compactification with boundary
\begin{align} \label{eq:zero_boundary_product_tree}
  &\coprod_{\ind(y_0) = \ind(y_0') -1 } \Tree(y_0;y_0') \times  \Tree(y_0'; y_1, y_2) \\ \label{eq:one_boundary_product_tree}
&\coprod_{\ind(y_1') = \ind(y_1) -1 }  \Tree(y_0; y_1', y_2) \times \Tree(y_1';y_1) \\ \label{eq:two_boundary_product_tree}
& \coprod_{\ind(y_2') = \ind(y_2) -1 }  \Tree(y_0; y_1, y_2') \times \Tree(y_2';y_2).
\end{align}
The elements of the first set correspond to applying the differential to the product, and the second two sets correspond to multiplying after applying the differential. In order to check that the signs are correct, one must compare the natural orientations induced at the boundary of the moduli space with the product orientation: 
\begin{lem}
Given $a_{i} \in  \ro_{y_i} \otimes  \ev_{0}^{*}|Q|^{-1}_{y_i}   $ for $i\in \{ 1,2 \}$,  we have
  \begin{equation} \label{eq:loop_product_chain_map}
    \partial \left( a_1 \star a_2 \right) =  \partial a_1 \star a_2 + (-1)^{\deg(a_1)} a_{1} \star \partial a_2.
  \end{equation}
\end{lem}
\begin{proof}
Given an element of $ \Tree(y_0; y_1, y_2)  $, we compare the boundary orientation with that induced by the product orientation; the cases of the strata in Equation \eqref{eq:zero_boundary_product_tree} and \eqref{eq:one_boundary_product_tree} are easier, and left to the reader, and we discuss only the stratum in Equation \eqref{eq:two_boundary_product_tree}. Given an element of this boundary stratum, let $\gamma_{2}$ denote the flow line in $ \Tree(y_2';y_2) $.

We start with the fact that the vector $-\frac{d \gamma_{2}}{ds}$ corresponds to an outward normal vector field at the boundary of $W^{u}(y_2)$. This yields an isomorphism
\begin{equation}
 | \bR \cdot   \partial_{s} |^{-1} \otimes | W^{u}(y_1) \times_{\ev_0} W^{u}(y_2')| \cong  | W^{u}(y_1) \times_{\ev_0} W^{u}(y_2)|.
\end{equation}
Applying Equation \eqref{eq:first_fibre_product_loop}, we obtain an isomorphism
\begin{equation}
  | \bR \cdot   \partial_{s} |^{-1}  \otimes  |  W^{u}(y_1) \times_{\ev_0} W^{u}(y_2)| \otimes \ev_{0}^{*}|Q|_{y_1}  \cong \ro_{y_1}  \otimes  \ro_{y_2}.
\end{equation}
Using Equation \eqref{eq:identification_fibre_product_descending}, and introducing a Koszul sign of 
\begin{equation} \label{eq:Koszul_sign_ind_y_1}
  \ind(y_1),
\end{equation}
we obtain the isomorphisms
\begin{align}
   \ro_{y_0}  \otimes  \ev_{0}^{*}|Q|_{y_1} & \cong \ro_{y_1} \otimes   | \bR \cdot   \partial_{s} |^{-1}   \otimes  \ro_{y_2} \\ \label{eq:product_dy_2_y_1}
  \ro_{y_0}  \otimes  \ev_{0}^{*}|Q|_{y_1} & \cong \ro_{y_1} \otimes  \ro_{y_2'},
\end{align}
where we use the isomorphism $ \ | \bR \cdot   \partial_{s} |^{-1}   \otimes  \ro_{y_2} \cong \ro_{y_2'}$ which defines the differential, and the negative orientation of $\bR$ in the last step.

In order to go from Equation \eqref{eq:product_dy_2_y_1} to the product, we must reverse the steps in Equations \eqref{eq:last_map_product_to_simplified}-\eqref{eq:simplified_map}. This introduces a Koszul sign of $ n (\ind(y_2)-1)$ in the definition of  $a_1 \star \partial a_2$, and $n \ind(y_2) $  in the definition of $\partial (a_1 \star a_2) $. The sum is $-n$, and the sum with Equation \eqref{eq:Koszul_sign_ind_y_1} yields the desired sign $\ind(y_1)-n$, which is the degree of the generator $a_1$ in Equation \eqref{eq:loop_product_chain_map}. 
\end{proof}
\begin{exercise}
Show that $\star$ commutes with $\iota$.
\end{exercise}

At this stage, we have proved the existence of a product
\begin{equation}
  \star \co H_{*}(\sL \Q; |Q|^{-1})  \otimes H_{*}(\sL \Q; |Q|^{-1})  \to H_{*}(\sL \Q; |Q|^{-1}).
\end{equation}

\subsection{Loop product with other coefficients}
By construction, each element of $ \Tree(y_0; y_1, y_2)   $ corresponds to a triple  of Morse trajectories 
\begin{align}
  \gamma_1 \co (-\infty, 0] & \to \sL^{r_1} \Q \\
\gamma_2 \co (-\infty, 0] & \to \sL^{r_2} \Q  \\
\gamma_0 \co [0,\infty)  & \to \sL^{r_1+r_2} \Q 
\end{align}
such that $\gamma_i$ converges at infinity to $y_i$ and the loop $\geo(\gamma_0(0))$ is obtained by concatenating $ \geo(\gamma_1(0)) $ and $\geo(\gamma_2(0))$.

Assume we have an ordinary local system $\nu$ on $\sL \Q$ (of degree $0$), together with isomorphisms
\begin{equation} \label{eq:multiplicative_structure}
  \nu_{\geo(\gamma_1(0))} \otimes \nu_{\geo(\gamma_2(0))} \to \nu_{\geo(\gamma_0(0))}.
\end{equation}
Using the homotopy between $\geo(y_i)$ and $\geo(\gamma_i(0))$, we obtain maps
\begin{equation}
 \nu_{\gamma} \co  \nu_{\geo(y_1)} \otimes \nu_{\geo(y_2)} \to \nu_{\geo(y_0)}.
\end{equation}
In particular, we can use this map, together with $\star_{\gamma}$, to define the product on Morse chains with coefficients in $\nu$:
\begin{align}
  \star \co CM_{i}(f^{r_1}; \ev_{0}^{*}|Q|^{-1} \otimes \nu ) \otimes CM_{j}(f^{r_2};\ev_{0}^{*}|Q|^{-1}  \otimes \nu ) & \to CM_{i+j}(f^{r_1+r_2}; \ev_{0}^{*}|Q|^{-1}  \otimes \nu ) \\
\star | \left(  \ro_{y_1} \otimes  \ev_{0}^{*}|Q|^{-1}_{y_1}   \otimes \nu \right)  \otimes \left( \ro_{y_2} \otimes  \ev_{0}^{*}|Q|^{-1}_{y_2} \otimes \nu \right)  & \equiv \sum_{ \substack{\ind(y_0) = \ind(y_1) + \ind(y_2) -n  \\ \gamma \in  \Tree(y_0; y_1, y_2)  }  } \star_{\gamma} \otimes  \nu_{\gamma}.
\end{align}
\begin{rem}
In order for the above product to be a chain map, we require that the isomorphism in Equation \eqref{eq:multiplicative_structure} be invariant under homotopies. Moreover, the product will be associative if the isomorphism in  Equation \eqref{eq:multiplicative_structure} is also associative with respect to multiple concatenations.
\end{rem}
\begin{exercise}
Show that the local systems $  \ev_{0}^{*} \left(|\Q|[n]\right)^{-\otimes w(\gamma)} $ and $\s$ are equipped with isomorphisms as in Equation \eqref{eq:multiplicative_structure} (Hint: for the first part, check that $  w(y_0) = w(y_1) + w(y_2) \mod 2$. For the second, use the fact that a $\Spin$ structure on a vector bundle over a pair of pants is determined by its restriction to any two of the three components in its boundary).
\end{exercise}

\section{Guide to the literature}

\subsection{Models for the homology of the loop space}

Among the models we could have chosen for the homology of the free loop space are:
\begin{enumerate}
\item The standard singular complexes, or variants thereof accounting for the smoothness of the manifold $Q$. This is closest to Morse's original point of view, see also \cite{milnor}.
\item The singular complexes of finite dimensional approximations.
\item The Morse homology of an appropriate energy function on a Hilbert manifold model for the loop space. This is developed, for example, in \cite{abbondandolo} or \cite{SW-06}. 
\item The Hochschild homology (or cohomology) of the chains on the based loop space, following the work of Goodwillie \cite{goodwillie}.
\end{enumerate}

In any of these models, one can construct a map from symplectic homology to loop homology; the first two require choices of fundamental chains on higher dimensional moduli spaces of pseudo-holomorphic maps and the last would rely on a similar result for the homology of the based loop space and the Lagrangian Floer cohomology of cotangent fibres;  see \cite{A-loops} for a discussion of both points.  

The deficiencies of these models are as follows:
\begin{enumerate}
\item In order to compare the loop product with the pair of pants product in the singular homology of $\sL \Q$ or its finite approximation, it seems that one needs to show that the Gromov-Floer compactifications of the moduli space admit smooth structures (as manifolds with corners) such that the evaluation map is smooth. There is little doubt that this is true, but a proof does not appear in the literature, and would take us far away from the subject at hand.
\item Morse homology on Hilbert manifolds is a rather delicate object, especially as the Morse function one would naturally study in this setting has barely the amount of regularity required for the theory to be well-defined. This means that one needs careful arguments to make various spaces transverse to each other.
\item  In addition to a heavy dose of homological algebra, the use of  Hochschild homology would require as input a construction proving the analogous result for the based loop space, this time at the chain level. While such a result has been proved in \cite{A-loops}, this seems like a circuitous detour.
\end{enumerate}

\subsection{Operations on loop homology}

Other than the fact that we have used a finite model for the homology of the free loop space, our construction hews quite closely to the original discussion of \cite{CS}, who studied the case of oriented manifolds, with untwisted coefficients, producing a $BV$ structure in this case.

The generalisation of the loop product to the non-orientable case was performed by Cohen and Jones \cite{CJ} in the language of Thom spectra, and by Laudenbach in \cite{Laudenbach}. It does not seem that more general twisted coefficients  have been studied, nor has the circle action $\Delta$ been considered in the non-orientable setting.

There are several proofs of the fact that the $BV$ structure, in the orientable case, depends only on the homotopy type of $\Q$, and not on the underlying smooth structure \cite{CKS}. In all likelihood, the twisted versions that we have introduced will also be homotopy-theoretic invariants.

\subsection{What is missing: Chain-level structure and constant loops}

The homology of the free loop space is expected to admit a $BV_{\infty}$ structure which refines the $BV$ structure on the cohomology; a proof for simply connected manifolds appears in \cite{ward}, relying on Goodwillie's description of the homology of the free loop space as the Hochschild homology of the cochains, and on Poincar\'e duality.    For more general manifolds, one can instead rely on the \emph{derived Poincar\'e duality} satisfied by the homology of the based loop space.  Lurie has sketched an argument in \cite{lurie} that the resulting $BV_{\infty}$ structure is also dependent only on the homotopy type.

There is much interest in extracting invariants from string topology which detect information beyond the homotopy type. In \cite{basu}, Basu has extracted such an invariant, relying on the fact that the homotopy type of configuration spaces distinguishes the $3$-dimensional lens spaces $L(7,1)$ and $L(7,2)$ \cite{LS}.

\chapter{From symplectic cohomology to loop homology} \label{cha:from-sympl-homol}

\section{Introduction}
In this section, we relate the symplectic cohomology of $\TQ$ to the homology of the free loop space.

\begin{thm} \label{thm:viterbo-general}
There is a map of $BV$ algebras from the symplectic cohomology of $\TQ$ to the homology of the free loop space of $\Q$ twisted by $\eta$, with reversed grading, i.e. we have a map
\begin{equation}
  \Vit \co SH^{*}(\TQ; \bZ) \to H_{-*}(\sL \Q; \eta)
\end{equation}
preserving the operations $e$, $\star$, and $\Delta$.
\end{thm}
\begin{rem}
If $\nu$ is a local system on $\sL \Q$, our method of proof similarly yields that the symplectic cohomology of $\TQ$ with coefficients in $\nu$ is isomorphic to the homology of  free loop space of $\Q$ twisted by $\eta \otimes \nu$, up to a grading reversal.  In particular, we have a map
\begin{equation}
  SH^{*}(\TQ, \eta^{-1}) \cong H_{-*}(\sL Q, \bZ).
\end{equation}
If $\nu$ is obtained by transgressing a local system on $Q$, then the map we obtain will preserve the $BV$ structure.
\end{rem}

We shall construct $\Vit$ by constructing a map, for each linear Hamiltonian on $\TQ$, from Floer cohomology to the homology of a finite dimensional approximation of the loop space. We briefly summarise the strategy: the geometric input for such a map are moduli spaces of pseudoholomorphic maps from a disc with one interior puncture to $\TQ$: the time-$1$ orbits of $H$ give asymptotic conditions at the puncture, and the boundary of the disc is assumed to map to the zero section. By evaluation at $r$ equidistant points on the boundary of the disc, we obtain a map from this moduli space to the space of piecewise geodesics. By considering such a moduli space for all orbits, we obtain a chain map from the Floer cochain complex to the Morse chain complex.
\section{The Maslov index for loops} \label{sec:maslov-index-loops}

  \subsection{Topology of the Grassmannian of Lagrangians}
Since loops or paths of Lagrangian subspaces in $\bC^{n}$ will appear in the linearisation of any pseudoholomorphic equation with Lagrangian boundary conditions, one must understand the topology of the Grassmannian of Lagrangian subspaces of  $\bC^{n}$ in order to define orientations for such moduli spaces. We start with the following  well-known result: 
\begin{lem}
The Grassmannian $\Gr(\bC^{n})$ of Lagrangian subspaces in $\bC^n$  is connected, and whenever $3 \leq n$, the first two homotopy groups are
\begin{align}
\pi_{1}(\Gr(\bC^{n})) & \cong \bZ \\
 \pi_{2}(\Gr(\bC^{n})) & \cong \bZ/2 \bZ.
\end{align} 
\end{lem}
\begin{proof}
  The unitary group $U(n)$ acts transitively on the space of Lagrangians, with stabiliser that is homotopy equivalent to $O(n)$, which induces a diffeomorphism
  \begin{equation}
    \Gr(\bC^{n}) \cong U(n)/O(n),
  \end{equation}
as explained in Lemma 2.31 of \cite{MS-baby}. For $n$ greater than $2$, the relevant homotopy groups of $U(n)$ and $O(n)$ are
\begin{alignat}{6}
  \pi_0(U(n)) & = 0 &  \quad  \pi_1(U(n)) & = \bZ  &  \quad \pi_2(U(n)) & = 0 \\
  \pi_0(O(n)) & = \bZ_{2} &  \quad \pi_1(O(n)) & = \bZ_{2}  & \quad \pi_2(O(n)) & = 0.
\end{alignat}
To compute the homotopy groups of $U(n)/O(n)$ from this data and the long exact sequence of a fibration, the only non-trivial piece of information needed is that the image of $\pi_1(U(n))$ into $\pi_1(U(n)/O(n))$ is divisible by $2$. The reader who does not already know this to be true can check it below using Definition \ref{defin:topological_Maslov_index} and Lemma  \ref{lem:pi_1_U(n)}.
\end{proof}

\begin{exercise} \label{ex:grass_Lag_all_subspaces}
Denote the Grassmannian of all $n$ real dimensional subspaces in $\bR^{2n}$ by $O(2n)/( O(n) \times O(n))$. Comparing the maps induced on  long exact sequences of homotopy groups for the fibrations
\begin{equation}
  \xymatrix{ O(n) \ar[r] \ar[d] &  U(n) \ar[r] \ar[d] &  \Gr(\bC^{n}) \ar[d] \\
O(n)  \times O(n) \ar[r] &  O(2n) \ar[r] &  O(2n)/( O(n) \times O(n)),}
\end{equation}
show that the map
\begin{equation}
  \bZ_{2} \cong \pi_{2}( \Gr(\bC^{n}) ) \to \pi_2( O(2n)/( O(n) \times O(n)) ) \cong \bZ_{2}
\end{equation}
is an isomorphism.
\end{exercise}
\begin{rem}
Some care is required when considering the cases $n=1,2$, which are nonetheless quite useful for computations. The reader should consult \cite{seidel-Book}*{Section 11e} for a detailed discussion.
\end{rem}
In terms of cohomology, we obtain isomorphisms
\begin{align}
H^{1}(\Gr(\bC^{n}), \bZ) & \cong \bZ \\ \label{eq:identify_H_2_Z_2}
H^{2}(\Gr(\bC^{n}), \bZ) & \cong \bZ/2 \bZ.
\end{align}

We shall present a concrete construction of the cohomological isomorphisms. First, observe that a Lagrangian subspace $L$ of $\bC^{n}$ has the property that
\begin{equation}
  I L \cap L = \{ 0 \}
\end{equation}
where $I$ is the standard complex structure on $\bC^{n}$. This implies that a (real) basis for $L$ defines a complex basis for $\bC^{n}$. Writing as before $\det$ for the top exterior power of a real vector space, and $\det_{\bC}$ for the top (complex) exterior power, the inclusion of $L$ in $\bC^{n}  $  therefore induces an inclusion
\begin{equation} \label{eq:det_L-maps_to_det_Cn}
 \det_{\bR}(L) \to \det_{\bC}(\bC^{n})
\end{equation}
which is canonical up to homotopy. Fixing an identification of $\det_{\bC}(\bC^{n}) $ with the plane, we obtain a map
\begin{equation} \label{eq:top_malsov_index}
  \Gr(\bC^{n}) \to \bR \bP^{1} = S^{1}
\end{equation}
which maps a Lagrangian subspace to the real line $\det_{\bR}(L)  $, seen as a line in $\det_{\bC}(\bC^{n}) \cong \bC $ via the map in Equation  \eqref{eq:det_L-maps_to_det_Cn}.
\begin{defin} \label{defin:topological_Maslov_index}
  The \emph{topological Maslov index} $\mu$ is the class in $H^{1}(\Gr(\bC^{n}), \bZ) $ associated to Equation  \eqref{eq:top_malsov_index}.
\end{defin}
 More concretely, the Maslov index assigns to a loop in $ \Gr(\bC^{n}) $ the winding number around $S^1$ of its image under Equation  \eqref{eq:top_malsov_index}; the reader should compare this with Equation \eqref{eq:classifying_map_grassmannian}.

Given loops $\Lambda $ in $ \Gr(\bC^{n})  $ and $\Phi $ in $\U(n)$, consider the loop of Lagrangians $\Phi(\Lambda)$ given by
\begin{equation}
  \Phi(\Lambda)_{t} \equiv \Phi_{t}(\Lambda_{t}).
\end{equation}
By comparing the definition of the Maslov index with Lemma \ref{lem:pi_1_U(n)}, we find that
\begin{equation}
  \mu(   \Phi_{t}(\Lambda_{t}) ) = \mu(\Lambda) - 2 \rho(\Phi).
\end{equation}
More abstractly, this implies that the map
\begin{equation}
  H_{1}( \Sp(2n,\bR), \bZ)  \oplus H_{1}( \Gr(\bC^{n}), \bZ) \to H_{1}(\Gr(\bC^{n}), \bZ)
\end{equation}
corresponds, under our chosen trivialisation of the first homology groups of $\Gr(\bC^{n})$ and $\Sp(2n,\bR)$, with the map
\begin{align}
\bZ \oplus \bZ & \to \bZ \\  
a \oplus b & \mapsto b - 2 a.
\end{align}

Next, we consider the second cohomology group of $\Gr(\bC^{n})$. In order to identify this group with $\bZ /2\bZ$, it suffices to give a non-trivial element. We start by noting that the universal coefficient theorem implies that reduction from $\bZ$ to $\bZ/2 \bZ$ yields an isomorphism
\begin{equation}
  H^{2}(\Gr(\bC^{n}), \bZ)  \to H^{2}(\Gr(\bC^{n}), \bZ/ 2 \bZ).
\end{equation}
Consider the tautological bundle $E$ over $ \Gr(\bC^{n}) $ whose fibre at a point corresponding to a Lagrangian subspace $L$ is $L$. This bundle is the pullback of the tautological bundle of the Grassmannian of all real subspaces under the natural embedding
\begin{equation}
   \Gr(\bC^{n}) \to O(2n)/( O(n) \times O(n)),
\end{equation}
which induces an isomorphism on second cohomology by Exercise \ref{ex:grass_Lag_all_subspaces}.   By definition, the second Stiefel-Whitney class of the tautological bundle over $O(2n)/( O(n) \times O(n))$ does not vanish  for $n>1$  (see, e.g. \cite{milnor-stasheff})). We conclude that 
\begin{equation}
  w_{2}(E) \in H^{2}(\Gr(\bC^{n}), \bZ/ 2 \bZ)
\end{equation}
is non-zero. 

Concretely, this means that the isomorphism in Equation \eqref{eq:identify_H_2_Z_2} can be described as follows: the generator of $   H^{2}(\Gr(\bC^{n}), \bZ)  $ evaluates non-trivially on a map from a surface to  $\Gr(\bC^{n}) $ if and only if the pullback of $E \oplus \det^{\oplus 3}(E)$ is non-trivial; here, we are using a geometric interpretation of  the second Stiefel-Whitney class which can be found e.g. in \cite{KT}.

It will be convenient to recast some of these computations in terms of the loop space of $\Gr(\bC^{n})$. Since the spaces of Lagrangians that will appear naturally do not have a canonical basepoint, we shall consider free loops. To state the result, recall that there is a natural \emph{transgression map}
\begin{equation}
H^{k}(X; \bZ) \to H^{k-1}(\sL X; \bZ),
\end{equation}
for any space $X$, which is dual to the map which assigns to any cycle in the loop space the cycle in the base swept by the corresponding family of loops. We also have a natural map
\begin{equation}
 \ev_{0} \co H^{k}(X; \bZ)  \to   H^{k}(\sL X; \bZ)
\end{equation}
induced by the projection to the initial point of a loop.

For $k=0$, the fact that the fundamental group of $\Gr(\bC^{n})$ is abelian implies that the components of $\sL \Gr(\bC^{n})$ are in bijective correspondence with elements of the fundamental group. At the level of cohomology, this implies that the transgression map
\begin{equation}
  \bZ \cong  H^{1}(\Gr(\bC^{n}); \bZ) \to H^{0}(\sL \Gr(\bC^{n}); \bZ )
\end{equation}
is an isomorphism.  We write $ \sL^{\mu} \Gr(\bC^{n}) $ for the components corresponding to loops of Maslov index $\mu$. Our computation of the second cohomology of $\Gr(\bC^{n})$ implies:
\begin{prop}
For $n$ strictly greater than $2$, transgression and evaluation at the initial point induce an isomorphism
\begin{equation} \label{eq:isomorphism_cohomology_loop_space}
 \bZ \oplus \bZ/2\bZ \cong H^{1}(\Gr(\bC^{n}); \bZ) \oplus H^{2}(\Gr(\bC^{n}); \bZ) \to  H^{1}(\sL^{\mu}  \Gr(\bC^{n}); \bZ)
\end{equation} 
\end{prop}
\begin{proof}
We shall use the fact that the inclusion $ \Gr(\bC^{n}) \to \Gr(\bC^{n+1}) $ obtained by taking the product of a Lagrangian in $\bC^{n}$ with a line in $\bC$ induces an isomorphism on all homotopy groups below dimension $n-1$, and that the union of the Grassmannians $ \Gr(\bC^{n}) $  is an $H$-space, which follows from Bott periodicity \cite{milnor}. Since the free loop space of an $H$-space splits as a product of the based loop space and the base, we conclude by the K\"unneth formula that
\begin{equation}
  H^{1}(\sL \Gr(\bC^{n}); \bZ) \cong H^{1}(\Omega\Gr(\bC^{n}); \bZ ) \oplus H^{1}(\Gr(\bC^{n}); \bZ )
\end{equation}
if $2 < n$. Since $\pi_1(\Omega\Gr(\bC^{n}) ) \cong \pi_2( \Gr(\bC^{n})) \cong \bZ/2\bZ$, we conclude that the cohomology of the free loop space is isomorphic to the direct sum
\begin{equation}
   \bZ \oplus \bZ/2\bZ.
\end{equation}
The reader may easily check at this stage that Equation \eqref{eq:isomorphism_cohomology_loop_space} is an isomorphism.
\end{proof}
We obtain the following consequence for local systems on $\sL \Gr(\bC^{n})$: denote by $\ev_{0}^{*}|E|$  the pullback under the evaluation map  of the local system of orientations of the tautological bundle, and by $\sigma^{E}$ the transgression of $E$ to a local system on the free loop space (see the introduction of Chapter \ref{cha:finite-appr-loop}).  Reducing Equation \eqref{eq:isomorphism_cohomology_loop_space} modulo $2$, we find that
\begin{equation}
  H^{1}(\sL^{\mu}  \Gr(\bC^{n}); \bZ/2 \bZ) \cong  \bZ/2\bZ  \oplus \bZ/2\bZ
\end{equation}
for every integer $\mu$.  In the correspondence between elements of this group and  rank-$1$ local systems, $\ev_{0}^{*}|E| $ and $\sigma^{E} $ give rise to distinct (non-vanishing) elements. We conclude:
\begin{cor}
Every non-trivial local system over $\sL^{\mu} \Gr(\bC^{n}) $ is isomorphic to either $\ev_{0}^{*}|E| $,  $\sigma^{E} $, or their tensor product. \qed
\end{cor}

\subsection{Analytic Maslov index and the universal determinant lines}

The construction of the index using Cauchy-Riemann operators proceeds as follows:  consider, for each loop of Lagrangians $\Lambda$, the Sobolev space
\begin{equation}
  W^{1,p}((D^2,S^1), (\bC^{n}, \Lambda))
\end{equation}
consisting of $\bC^{n}$-valued functions $X$ on the disc, of class $W^{1,p}$, such that
\begin{equation}
  X_{e^{2 \pi i t}} \in \Lambda_{t}.
\end{equation}
Equipping the disc with exponential polar coordinates
\begin{equation}
  (s,t) \mapsto e^{s+2\pi i t},
\end{equation}
and letting $I$ denote the standard complex structure on $\bC^{n}$, the Cauchy-Riemann operator
\begin{equation}
  X \mapsto I \partial_{s} X - \partial_{t} X
\end{equation}
defines a Fredholm map
\begin{equation}
  D_{\Lambda} \co   W^{1,p}((D^2,S^1), (\bC^{n}, \Lambda)) \to L^{p}(D^2, \bC^{n}).
\end{equation}
We define the \emph{analytic Maslov index} of $\Lambda$ in terms of the index of the operator $D_{\Lambda}$ as:
\begin{equation} \label{eq:define_mu_analytic}
\mu(\Lambda) \equiv  \ind(D_{\Lambda}) -n. 
\end{equation}
The following result should be considered a toy model for more sophisticated index theorems, and is discussed, for example in \cite[Appendix C.3]{MS}
\begin{lem}
The analytic and topological index agree. \qed
\end{lem}

This result can be refined to a statement about graded local systems on the free loop space of $\Gr(\bC^{n})$. Consider the $\bZ$-graded local system $\delta$ whose fibre at $\Lambda$ is 
\begin{equation}
 | \det(D_{\Lambda})|,
\end{equation}
where the determinant line is defined as in Equation \eqref{eq:determinant_line_definition}. The convention we use is that this local system, which is of rank one, is graded in degree equal to the index of $D_{\Lambda}$.

The following result appears, essentially verbatim, as Lemma 11.17 of \cite{seidel-Book}, but the key computation at the heart of the proof is due independently to da Silva \cite[ Theorem A, p. 118]{da-silva} and to Fukaya, Oh, Ohta, and Ono \cite[Section 8.1.2, p. 684]{FOOO-II}:
\begin{prop} \label{prop:compue_det_line_universal}
  There is an isomorphism of  $\bZ$-graded local systems:
  \begin{equation} \label{eq:topological_description_delta}
  \delta |\sL^{\mu} \Gr(\bC^{n})  \equiv  \sigma^{E} \otimes   \ev_{0}^{*}|E| \otimes \left( \ev_{0}^{*}|E|[-n-1] \right)^{\otimes \mu},
  \end{equation}
where  $\ev_{0}^{*}|E| $  is graded in degree $-n$, and $ \sigma^{E}$ in degree $0 $.
 \qed
\end{prop}

In order to use Proposition \ref{prop:compue_det_line_universal}, we must in fact fix the isomorphism between these two local systems; to this end,  we give a minor variant of the construction  in Lemma 11.17 of \cite{seidel-Book}. In the discussion below, we shall use the fact that an isomorphism between rank $1$ local systems over a connected space is determined by an isomorphism between any of their fibres. 

 The first case to consider is $\mu = 0$. The constant loop 
 \begin{equation}
   \Lambda_{0,t} = \bR^{n}
 \end{equation}
gives a point in $\sL^{0} \Gr(\bC^{n})$ and the operator $D_{\bR^{n}}$  has kernel consisting only of constant functions, which necessarily take value in $\bR^{n}$.  Since this operator is regular (see, e.g. Corollary C.1.10(iii) of \cite{MS}), we obtain a canonical isomorphism:
\begin{equation} \label{eq:isomorphism_disc_operator_constant}
  \det(D_{\bR^{n}}) \cong \det(\ker(D_{\bR^{n}})) \cong \det(\bR^{n}) \cong \ev_{0}^{*} \det(E).
\end{equation}
In order to obtain the isomorphism to the local system in Equation \eqref{eq:topological_description_delta}, we use in addition the fact that the restriction of $\sigma^{E}  $ to constant loops admits a canonical trivialisation.

Next, we consider the component of loops having Maslov index  $-1$, for which we choose the basepoint
\begin{equation}
\Lambda_{-1,t} \equiv   e^{- \pi i t} \bR \oplus \bR^{n-1}.
 \end{equation}
The associated operator is again regular, and all its solutions are constant functions taking value in $\bR^{n-1}$, so we obtain an isomorphism
\begin{equation} \label{eq:iso_delta_n-1}
  \delta_{\Lambda_{-1}} \cong |\bR^{n-1} |.
\end{equation}
The following exercise will allow us to relate this answer to the topological side:
\begin{exercise}[c.f. Lemma 1.2 of \cite{KT}] 
  Let $V$ be a non-orientable vector bundle over the circle, and $W$ a vector space, with $\underline{W}$ the corresponding (trivial) vector bundle over $S^1$. Show that an orientation of $W$ induces a unique trivialisation of
  \begin{equation}
    V \oplus \underline{W} \oplus \det(V \oplus \underline{W})^{\oplus 3} = V \oplus \underline{W} \oplus \left(V \otimes \det(\underline{W}) \right)^{\oplus 3}
  \end{equation}
up to homotopy, which changes if we change the orientation of $W$.
\end{exercise}

Applying the above computation to the fibre of $\sigma^{E}$ at $\Lambda_{-1}$, we conclude that
\begin{equation}\label{eq:iso_sigma_n-1}
  \sigma^{E} \cong |\bR^{n-1} |.
\end{equation}
Combining Equations \eqref{eq:iso_delta_n-1} and \eqref{eq:iso_sigma_n-1} determines the isomorphism in Equation  \eqref{eq:topological_description_delta}  for $\mu = -1$.

The choices we made for $\mu=-1,0$ will fix the isomorphisms for all other components of the loop space if we use complex orientations on the determinant lines of loops of unitary matrices. The key point is to introduce, for each integer $\mu$,  the loop
\begin{equation} \label{eq:standard_loop_maslov_lambda}
\Lambda_{\mu,t} \equiv  e^{i t \pi \mu} \bR \oplus \bR^{n-1}.
\end{equation}
of Lagrangians. Depending on the parity of $\mu$, we can write this loop either as 
\begin{equation}
  \Phi \circ \Lambda_{0} \textrm{ or }   \Phi \circ \Lambda_{-1},
\end{equation}
where $\Phi$ is a loop of unitary matrices. As in Lemma \ref{lem:orient_line_indep_path} we have an isomorphism of determinant lines
\begin{equation} \label{eq:model_isomorphism_change_Maslov_index}
      \det( D_{\Lambda_{\mu}} ) \otimes \det(\bC^{n})^{-1} \otimes  \det(D^{-}_{\Phi^{-1}} )  \cong       \det( D_{\Lambda_{\mu+2k}} ).
\end{equation}
for a loop $\Phi$ of unitary matrices representing $k \in \bZ \cong \pi_{1}(U(n))$. Using complex orientations on $\bC^{n}$ and on the determinant of $D^{-}_{\Phi^{-1}}$ which is a complex linear operator, we obtain isomorphisms
\begin{equation}
\det( D_{\Lambda^{\mu}}) \cong  \det(D_{\Lambda_{0}})) \textrm{ or } \det( D_{\Lambda_{\mu}}) \cong \det(D_{\Lambda_{-1}}))
\end{equation}
Since the right hand side of Equation \eqref{eq:topological_description_delta} depends only on the parity of $\mu$, we obtain the desired isomorphism.

Having constructed a fixed isomorphism in Equation  \eqref{eq:topological_description_delta}, we state a useful consequence:
\begin{lem} \label{lem:inverse_path_det_line_is_inverse}
  For any loop of Lagrangians $\Lambda$, we have a canonical isomorphism of $\bZ$ graded lines:
  \begin{equation} \label{eq:isomorphism_inverse_loop}
    \det( D_{\Lambda^{-1}} ) \cong  \det( D_{\Lambda} ) [2 \mu]  ,
  \end{equation}
\end{lem}
\begin{proof}
Since $\mu$ and $-\mu$ have the same parity, Equation \eqref{eq:model_isomorphism_change_Maslov_index} yields an isomorphism
\begin{equation} \label{eq:model_isomorphism_inverse_loop}
      \det( D_{\Lambda_{\mu}} )  \cong       \det( D_{\Lambda_{-\mu}} ).
\end{equation}
Equation \eqref{eq:isomorphism_inverse_loop} is obtained by composing this isomorphism with the map on orientation lines induced by the choice of a path from $\Lambda$ to $\Lambda_{\mu}$ and the  inverse path from $\Lambda^{-1}$ to $\Lambda_{-\mu}$. To see that the isomorphism above does not depend on the choice of path, note that a different choice yields two loops of loops in $\Gr(\bC^{n})$, the first based at $\Lambda$ and the second at $\Lambda^{-1}$, obtained by concatenating the two possible paths from these loops to the standard ones. These loops represent classes in 
\begin{equation}
  H_{1}(\sL \Gr(\bC^{n}), \bZ) \cong \bZ_{2} \textrm{ if } n \geq 3,
\end{equation}
and, since they differ by the self-homotopy equivalence which sends a loop to its inverse, the classes of the loops are either both trivial, or both non-trivial. We conclude that the isomorphisms in Equation \eqref{eq:isomorphism_inverse_loop}  coming from different choices of paths agree, since the sign difference is the product of the monodromy of the determinant line over both sets of loops.
\end{proof}

\section{Construction of a chain map}
We shall construct in this Section a chain map from the Floer complex of a Hamiltonian on $\TQ$ to the Morse chains with coefficients in $\eta$ of finite dimensional approximations of the free loop space. By showing that the induced map on cohomology commutes with continuation maps, we obtain a map from symplectic cohomology to the homology of the free loop space with these coefficients.

\subsection{Punctured discs with boundary on the zero section} \label{sec:punctured discs-with}

Consider the punctured disc $D^{2} \setminus \{ (0,0) \}$ which we identify with the positive half of the cylinder
\begin{equation}
  Z^{+} =  [0,+\infty) \times S^1.
\end{equation}
Given a linear Hamiltonian $H$ all of whose periodic orbits are non-degenerate, we choose a family $H^{+}_{s,t}$ of Hamiltonians of equal slope, with Hamiltonian flow $X_{s,t}^{+}$ such that 
\begin{equation} \label{eq:flow_vanishes_on_Q}
  \parbox{35em}{ $X_{0,t}^{+}|Q \equiv 0$ and $H_{s,t}^{+} =H_{t}$ whenever $0 \ll s$.}
\end{equation}
We similarly extend $J_{t}$ to a family $J^{+}_{s,t}$ parametrised by points in $Z^{+}$ and define $\Cyl(x)$, for each orbit $x \in \Orbit(H)$, to be the moduli space of maps
\begin{equation}
   u \co  Z^{+} \to \TQ
\end{equation}
solving the differential equation
\begin{equation} \label{eq:operator_half_cylinder}
   \left( du - dt \otimes X_{s,t}^{+} \right)^{0,1} = 0,
\end{equation}
with asymptotic condition $x$ at $s = +\infty$, and with boundary conditions
\begin{equation}
  u(0,t) \in Q.
\end{equation}

By associating to each map $u$ the loop $u(0,t)$, we obtain an evaluation map
\begin{equation}
  \ev \co \Cyl(x)  \to \sL \Q.
\end{equation}

In order to understand the tangent space of $\Cyl(x) $, we first compute its virtual dimension in terms of $\deg(x)$. Recall that $\deg(x)$ is defined in terms of an operator $D_{\Psi_x}$ constructed from $x^{*}(T \TQ)$ with respect to a fixed trivalisation: under this map, the tangent space of $\Q$ maps to a constant loop if $x^{*}(\TQ)$ is orientable, and to a loop of Maslov index $1$ otherwise.

This preferred trivalisation extends to the punctured disc, so we can also use it to linearise the operator associated to Equation \eqref{eq:operator_half_cylinder}, to obtain an operator
\begin{equation}
  D_{u} \co W^{1,p}((Z^{+},S^1), (\bC^{n}, \Lambda_{x})) \to L^{p}(Z^{+}, \bC^{n})
\end{equation}
where the notation indicates that we take $W^{1,p}$ maps from $Z^{+}$ to $\bC^{n}$ with boundary conditions along the loop $\Lambda_{x}$ of Lagrangians in $\bC^{n}$ which is the image of $T\Q$ under the trivalisation.

\begin{figure}[h]
  \centering
  \includegraphics[scale=1]{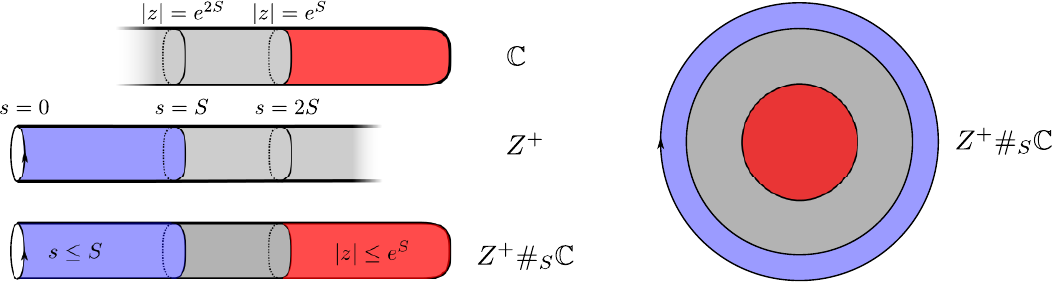}
  \caption{}
  \label{fig:gluing_disc}
\end{figure}

Glue  $D_{u}$ to $D_{\Psi_x}$ along their common end: the result is an operator $D_{u} \# D_{\Psi_x}  $  on a disc with no puncture, and with Lagrangian boundary conditions  $ \Lambda_{x} $ going \emph{clockwise} (see Figure \ref{fig:gluing_disc}).  Since $q \circ u$ defines a homotopy between $q \circ x$ and $\ev(u)$, we obtain an isomorphism of graded lines
\begin{equation}\label{eq:iso_gluing_det_plane_half-cyl}
 \det(T_{u} \Cyl(x))  \otimes   \det(D_{\Psi_x})  \cong \det( D_{\Lambda_{\ev(u)}^{-1}}).
\end{equation}
where $ D_{\Lambda_{\ev(u)}^{-1}} $ is the Cauchy-Riemann operator on the disc with Lagrangian boundary conditions obtained by applying the trivialisation $\Phi$ to $T\Q$ along $\ev(u)$. We now state a description of $  \det( D_{\Lambda_{\ev(u)}^{-1}}) $ that follows from index theory:

\begin{lem} \label{lem:computation_index_disc_Lagrangian_boundary}
The index of $D_{\Lambda_{\ev(u)}^{-1}}$ is $n + w(x)$. Moreover, there is a canonical isomorphism
\begin{equation} \label{eq:iso_det_Lag_boundary_kappa}
  |\det( D_{\Lambda_{\ev(u)}^{-1}})| \cong \eta_{x}[-w(x)].
\end{equation}
\end{lem}
\begin{proof}
If $x^{*}(\TQ) $ is orientable, $w(x)$ vanishes and the trivialisation we have chosen has trivial Maslov index.  This implies that the index of $D_{\Lambda_{\ev(u)}^{-1}}  $ is $n$, which is the is same as the degree of $ \eta_{x} $  as a graded line. For a loop along which $\TQ$ is not orientable, the Maslov index of the loop $ \Lambda_{\ev(u)}$ with respect to the trivialisation from Lemma \ref{lem:fixed_trivialisation} is $1$, so the Maslov index of $ \Lambda_{\ev(u)}^{-1} $ is $-1$, and the Fredholm index of $ D_{\Lambda_{\ev(u)}^{-1}}  $  is $n-1$, yielding the desired index computation since $w(x) = -1$ in this case.

The isomorphism of determinant lines is obtained by comparing the definition of $\eta$ in Equation \eqref{eq:define_eta} with the local system in the right hand side Equation \eqref{eq:topological_description_delta}.
\end{proof}

\begin{lem} \label{lem:vir_dim_moduli_half_cylinders}
The virtual dimension of the moduli space  $\Cyl(x) $ is equal to $n - \deg(x)$. Moreover, there is a natural isomorphism
\begin{equation} \label{eq:iso_det_x_kappa_ev_u}
  |\Cyl(x)| \otimes\ro_{x}[w(x)] \cong \eta_{x} 
\end{equation}
\end{lem}
\begin{proof}
The isomorphism of lines in Equation \eqref{eq:iso_det_x_kappa_ev_u} follows immediately from Equations \eqref{eq:iso_gluing_det_plane_half-cyl} and \eqref{eq:iso_det_Lag_boundary_kappa}.  

To compute the virtual dimension, we use the additivity of the index and Lemma \ref{lem:computation_index_disc_Lagrangian_boundary}, and obtain:
\begin{equation}
  \ind(D_u) = n + w(x) - \ind( D_{\Psi_x} ).
\end{equation}
Since $ \deg(x) = \ind( D_{\Psi_x} ) -w(x) $, we conclude that
\begin{equation}
    \ind(D_u) = n - \deg(x),
\end{equation}
which gives the desired formula for the expected dimension.
\end{proof}

\subsection{Structure of the compactified moduli space}
The Gromov-Floer compactification $\Cylbar(x)  $ of $\Cyl(x)  $ is obtained by adding the strata:
\begin{equation} \label{eq:strata_boundary_half_cylinder}
  \coprod_{\substack{k \\ x_0, \ldots, x_{k-1}}} \Cyl(x_0) \times  \Cyl(x_0,x_{1}) \times \cdots \times \Cyl(x_{k-1},x) .
\end{equation}
The integer $k$ above is called \emph{the virtual codimension} of the
stratum. We begin by assuming that the Floer data are chosen so that
\begin{equation} \label{eq:regularity_moduli_spaces}
  \parbox{35em}{all elements of the moduli spaces $\Cyl(x_0)$ and $\Cyl(x_i,x_{i+1})$ are
    regular, and have image contained in  $\DQ$,}
\end{equation}
which implies that $\Cylbar(x)$ is compact.

For all previously considered operations, we only studied moduli
spaces of dimension $0$ and $1$. For constructing the map from
symplectic cohomology to the homology of the free loop space, we shall
consider moduli spaces of arbitrary dimension.  More precisely, given $r$ marked points $(z_1, \ldots, z_r)$ on the
boundary of the punctured disc, we shall consider the evaluation map
\begin{align}
 \Cylbar(x) & \to  \Q^{r} \\
v & \mapsto (v(z_1), \ldots, v(z_r)).
\end{align}
In addition, we consider a proper map
\begin{equation}
\iota \co  N \to \Q^{r}
\end{equation}
from a manifold with boundary $N$. Whenever the codimension of $N$ equals $n-\deg(x)$, Sard's theorem implies that, for  an open
dense subset of maps $\iota$, the inclusion
\begin{equation} \label{eq:fibre_product_N}
 \Cyl(x) \times_{\Q^{r}} N  \subset  \Cylbar(x) \times_{\Q^{r}} N 
\end{equation}
is an equality, and the left hand side is a compact $0$-dimensional submanifold obtained as a transverse fibre product. The key point is that a generic perturbation of $\iota$
ensures that its image is disjoint from all boundary strata of
$\Cylbar(x)$.  Because $\iota$ is a proper map, we conclude that  the inverse image of $N$ is contained in a compact
subset of $\Cyl(x)$. A standard application of Sard's theorem
therefore implies that the fibre product is generically transverse.

We shall also need to study the case where the virtual dimension of the
fibre product is one:
\begin{prop} \label{prop:gluing_manifold}
Assume that Condition \eqref{eq:regularity_moduli_spaces} holds, and
that the codimension of $N$ is $n+1-\deg(x)$. After a generic
perturbation of $\iota$, $  \Cylbar(x) \times_{\Q^{r}} N $ is a compact manifold of dimension $1$, whose boundary can be naturally decomposed into the following strata:
\begin{align}
&  \coprod_{\deg(x_0) = \deg(x) +1} \left( \Cyl(x_0) \times_{Q^r}
    N\right) \times  \Cyl(x_{0},x) \textrm{ and} \\
&  \Cylbar(x) \times_{\Q^{r}} \partial N.
\end{align}
\end{prop}
\begin{proof}
Following the discussion above, we know that both of the putative
boundary strata are defined by transverse fibre products. From Sard's theorem, we conclude that a neighbourhood of
$\Cylbar(x) \times_{\Q^{r}} \partial N$ in $\Cylbar(x) \times_{\Q^{r}}
N  $  is homeomorphic to a half-open interval. The analogous result for the other boundary stratum follows from an
infinite-dimensional version of Sard's theorem, which is briefly
reviewed in Section \ref{sec:manif-struct-moduli}.

Standard transversality theory also implies that a small perturbation
of $\iota$  ensures transversality with arbitrarily large compact
subsets of $\Cyl(x)$. Combined with the result near the boundary
strata, we conclude the desired result.
\end{proof}
\begin{rem} \label{rem:struct-comp-moduli}
The proof of Proposition \ref{prop:gluing_manifold} and the discussion
of gluing in Section
\ref{sec:manif-struct-moduli} use in a special way the condition that the
virtual codimension of the fibre product is $1$. It would be more
appropriate to prove that the moduli space $ \Cylbar(x) $  admits the
structure of a smooth manifold with corners so that the evaluation map
is smooth. This falls in the class of standard results, which are well-known to experts, but whose proofs have not appeared in print. Since providing such a proof would take us too far afield from the main results which we would like to discuss, we use Morse theory to bypass this gap in the literature.
\end{rem}

\subsection{Construction of the map on Floer cohomology} \label{sec:constr-map-floer}

We shall define a map from the Floer cochain complex to the Morse complex of a finite dimensional approximation for $r$ large enough. The starting point is to consider the map
\begin{align}
  \ev_{r} \co  \Cylbar(x) & \to \Q^{r} \\
\ev_{r}(u) & \mapsto \left(\ev(u)(0), \ldots, \ev(u)\left(\frac{i}{r}\right), \ldots, \ev(u)\left(\frac{r-1}{r}\right) \right).
\end{align}
\begin{lem} \label{lem:evaluation_map_good_r_large}
  If $r$ is large enough, the image of $\ev_{r}$ lies is $\sL^{r} \Q$. Moreover, we have a homotopy commutative diagram:
  \begin{equation}
    \xymatrix{  \Cylbar(x) \ar[r]^{\ev_{r}}  \ar[d]^{\ev} \ar[dr]^{\ev_{r+1}}& \sL^{r} \Q \ar[d]^{\iota}  \\ 
\sL \Q & \ar[l]^{\geo}  \sL^{r+1} \Q.}
  \end{equation}
\end{lem}
\begin{proof}
  The moduli space $ \Cylbar(x)  $ is compact, so there is a uniform bound on the $C^1$ norm of the curves $\ev(u)$ for $u \in \Cylbar(x) $. For each constant $\delta$, we may therefore choose  $r$ large enough so that the restriction of $u$ to any interval of length $\frac{1}{r}$ is less than $\delta$. In particular, the distance between $\frac{i}{r}$ and $\frac{i+1}{r}$ can be assumed bounded by $\delta_{i}^{r}$, and hence the image of $\ev_{r}$ lies in $\sL^{r} \Q$ by Equation \eqref{eq:distance_increases_with_r}.

This argument moreover shows that the path along the image of $u$ between these points is contained within the ball of radius $1$ of either endpoint, and hence is homotopic, within such a ball, to the local geodesic between them. There is a contractible choice of such local homotopies, which implies that we can choose them smoothly over each stratum of $ \Cylbar(x) $, and continuously over the whole space. This implies that $ \geo \circ \ev_{r} $ is homotopic to $\ev$. The same argument produces a homotopy between $\ev_{r+1}$ and $\iota \circ \ev_{r}$.
\end{proof}

Given a Hamiltonian orbit $x$, and a critical point $y$ of $f^{r}$ for $r$ large enough to satisfy the conclusion of Lemma \ref{lem:evaluation_map_good_r_large}, we define
\begin{equation} \label{eq:broke_y_x}
  \Broke(y,x) \equiv  \Cyl(x) \times_{\ev_r}  W^{s}(y).
\end{equation}
We think of this as a hybrid moduli space, consisting of a disc with puncture converging to $x$, followed by a flow line from the boundary of this punctured disc to $y$.
\begin{rem}
 The ordering of the two factors on the right hand side of Equation \eqref{eq:broke_y_x}  is justified as follows: the map we construct from Floer to Morse theory \emph{reverses} the grading, which at the level of orientation lines corresponds to taking the inverse. If we used the same convention in ordering factors as in Chapter \ref{cha:oper-sympl-cohom}, we would have to introduce the appropriate Koszul signs in our constructions to account for this additional operation.  The existence of a natural isomorphism 
  \begin{equation}
    |X \times Y|^{-1} \cong |Y|^{-1} \times |X|^{-1}
  \end{equation}
for manifolds $X$ and $Y$ allows us to simplify these signs.
\end{rem}

\begin{lem}
For a generic function $f^{r}$, the moduli space $   \Broke(y,x)  $ is a manifold of dimension
  \begin{equation} \label{eq:dimension_broken_cylinder}
   n - \deg(x)  - \ind(y).
  \end{equation}
Moreover, every element of $ \Broke(y,x) $ determines a canonical isomorphism
\begin{equation} \label{eq:isomorphism_broken_moduli}
   \ro_{y}  \otimes \ro_{x}[w(x)] \cong  |\Broke(y,x)  |^{-1} \otimes \eta_{x}.
\end{equation}
\end{lem}
\begin{proof}
The dimension of $ \Cyl(x)  $ is $n-\deg(x)$, and the codimension of $ W^{s}(y) $ is $\ind(y)$, hence Equation \eqref{eq:dimension_broken_cylinder}. To check the statement about orientation lines, we start with the short exact sequence
\begin{equation}
  T   \Broke(y,x) \to   T \Cyl(x) \oplus  T W^{s}(y) \to T \sL^{r} \Q
\end{equation}
which yields the isomorphism
\begin{align}
  |\Broke(y,x)  |  \otimes | \sL^{r} \Q | &  \cong | \Cyl(x) |  \otimes |W^{s}(y)| \\
| \sL^{r} \Q | &  \cong |\Broke(y,x)  |^{-1} \otimes  | \Cyl(x) |  \otimes |W^{s}(y)|.
\end{align}
Using Equation \eqref{eq:spliting_tangent_at_crit_point}, we obtain an isomorphism
\begin{equation}
  \ro_{y} \cong  |\Broke(y,x)  |^{-1} \otimes  | \Cyl(x) |.
\end{equation}
Tensoring both sides on the right by $\ro_{x}[w(x)]$, and using Equation \eqref{eq:iso_det_x_kappa_ev_u}, we obtain Equation \eqref{eq:isomorphism_broken_moduli}.
\end{proof}

Let us now consider the situation where $\Broke(y,x)  $ is $0$-dimensional. Using the natural identification of $ \ro_{x}$ with its inverse (see Lemma \ref{lem:iso_inverse_signs}), we obtain, from Equation \eqref{eq:isomorphism_broken_moduli}, an isomorphism 
\begin{equation}
  \Vit_{u} \co   \ro_{x}[w(x)] \cong   \ro_{y}  \otimes \eta_{y}.
\end{equation}

Next, having fixed the Hamiltonian $H$, we choose $r$ large enough so that Lemma \ref{lem:evaluation_map_good_r_large} applies to all time-$1$ periodic orbits of $H$. We then define a map
\begin{align}
  \Vit_{r} \co CF^{*}(H ; \bZ)  &\to CM_{-*}(f^{r}; \eta) \\ \label{eq:vit_map}
\Vit_{r}|  \ro_{x}[w(x)]  & = (-1)^{\deg(x)} \sum_{\substack{ \ind(y) -n = - \deg(x) \\ u \in  \Broke(y,x) }} \Vit_{u}.
\end{align}

  The sign in Equation \eqref{eq:vit_map} is chosen so that we can prove that $\Vit_{r}$ is a chain map. To see this, consider a moduli space $ \Broke(y,x) $ which is $1$-dimensional. The natural compactification of this space is a manifold $ \Brokebar(y,x) $with boundary 
\begin{align}
 & \coprod_{\deg(x_0) = \deg(x) +1 } \Broke(y,x_0) \times \Cyl(x_0,x)  \\ \label{eq:second_boundary_broken}
& \coprod_{\ind(y) = \ind(y_1) +1 }   \Tree(y,y_1) \times  \Broke(y_1,x).
\end{align}

\begin{lem} \label{lem:Vit_is_chain_map}
 $\Vit_{r}$ is a chain map, i.e.
 \begin{equation}
   \partial \circ \Vit_{r} = \Vit_{r} \circ d.
 \end{equation}
\end{lem}
\begin{proof}
We start with the isomorphism
\begin{equation}
  \ro_{x_0}[w(x_0)]   \cong \ro_{y}  \otimes  \eta_{x_0}
\end{equation}
induced by an element of $  \Broke(y,x_0)  $.  Given $u \in \Cyl(x_0,x)$, the map $\partial_{u}$ is induced by the natural isomorphism
\begin{equation}
   \ro_{x_0}[w(x_0)]  \cong  |\partial_{s} u|  \otimes \ro_{x}[w(x)]
\end{equation}
coming from gluing. Composing these two isomorphisms, we obtain the map
\begin{equation}
  |\partial_{s} u|  \otimes \ro_{x}[w(x)]  \cong \ro_{y}  \otimes  \eta_{x}.
\end{equation}
We now note that the vector field $\partial_{s} u$ corresponds to an outward pointing vector along the boundary of $ \Brokebar(y,x)  $.  We can therefore compare this map with Equation \eqref{eq:isomorphism_broken_moduli}, which is the natural map on orientation lines induced by the definition of $ \Brokebar(y,x)  $ as a fibre product. The steps in this comparison, keeping account of the Koszul sign on the right most column, are
\begin{alignat}{4}
   \ro_{y}  \otimes \ro_{x}[w(x)] & \cong  |\partial_{s} u|^{-1} \otimes \eta_{y} &&\quad  \\
 |\partial_{s} u| \otimes  \ro_{y}  \otimes \ro_{x}[w(x)] & \cong   \eta_{y} && \quad 1 \\
 |\partial_{s} u|  \otimes \ro_{x}[w(x)]  & \cong \ro_{y}^{-1}  \otimes  \eta_{y} &\quad & \ind(y) \\
 |\partial_{s} u|  \otimes \ro_{x}[w(x)]  & \cong \ro_{y} \otimes  \eta_{y} &\quad & \ind(y)
\end{alignat}
Keeping into account the sign in Equation \eqref{eq:vit_map}, the natural map induced by $ \partial \Brokebar(y,x)  $ along this stratum agrees with $ - (-1)^n \Vit_{r} \circ d$ since 
\begin{equation}
  \ind(y) +1  + \deg(x_0) - n = 1 \quad  \mod 2.
\end{equation}

Next, we consider the boundary component in \eqref{eq:second_boundary_broken}: Given $\gamma \in \Tree(y,y_1)$, the map $\partial_{\gamma}$ is defined using the isomorphism
\begin{alignat}{4} 
   | \bR \cdot   \partial_{s}\gamma |^{-1} \otimes    \ro_{y_1} & \cong     \ro_{y} & & \\ 
 \ro_{y_1} & \cong    | \bR \cdot   \partial_{s}\gamma | \otimes   \ro_{x} & \quad& 1  \\
\ro_{y_1} & \cong    |  \Brokebar(y,x) | \otimes   \ro_{x}  & \quad & 1.
\end{alignat}
The sign in the middle step arises because we defined $ \ell^{-1}  $ to be the left inverse to $\ell$ in Equation \eqref{eq:inverse_line}, while the last sign comes from the fact that  the vector field $-\partial_{s} \gamma$  corresponds to an inward pointing vector along the boundary of $ \Brokebar(y,x)  $. Using the isomorphism obtained by applying Equation \eqref{eq:isomorphism_broken_moduli} to $\Broke(y_1,x) $, we arrive at the isomorphism for $\Broke(y,x) $. 

The remaining sign arises from the Lemma \ref{lem:iso_inverse_signs}, applied to $\ro_{y} \cong   | \bR \cdot   \partial_{s}\gamma |^{-1} \otimes    \ro_{y_1}$.   The parity of the sign is
\begin{equation}
\ind(y_1).
\end{equation}
Since $ \ind(y_1) = \deg(x)  +n \mod 2$, we conclude that the natural map induced by $\partial \Brokebar(y,x)  $ along this stratum agrees with $ (-1)^{n} \partial  \circ \Vit_{r}$.

Since the sum of the maps associated to $\partial \Brokebar(y,x)   $  vanishes, we conclude that
\begin{equation}
   - \Vit_{r} \circ d + \partial  \circ \Vit_{r} = 0,
\end{equation}
which implies the desired result.
\end{proof}

\subsection{Compatibility with inclusions maps}
Assuming that the maps $\Vit_{r-1}$ and $\Vit_{r}$ are both defined, we claim that the diagram
\begin{equation}
 \xymatrix{ HF^{*}( H; \bZ)  \ar[r]^-{\Vit_{r-1}} \ar[dr]^{\Vit_{r}} &  HM_{-*}(\sL^{r-1} \Q; \eta) \ar[d] \\
& HM_{-*}(\sL^{r} \Q; \eta)}
\end{equation}
commutes. At the chain level, the homotopy is constructed in two steps. First, assuming that $\ind(y) = n - \deg(x)$, we consider the $1$-dimensional manifold consisting of arbitrary flow lines of $-\grad(f^{r})$ emanating from $ \Cyl(x)  $, and whose image under $\iota$ intersects the stable manifold of $y'$:
\begin{equation}
\Broke^{\iota}(y,x)  \equiv \bigcup_{t=0}^{\infty} \Cyl(x)  \times_{  \iota \circ   \psi_{t}^{r} \circ \ev_{r-1}} W^{s}(y).
\end{equation}
By considering the boundaries $t=0$ and $t=\infty$, we see that, for generic data, this is a cobordism between
\begin{equation} \label{eq:boundary_broke_iota_not_compact}
 \Cyl(x)  \times_{  \iota \circ \ev_{r-1}} W^{s}(y)  \textrm{ and } \coprod_{y} \Broke(y',x) \times  W^{u}(y') \times_{  \iota} W^{s}(y).
\end{equation}
\begin{exercise} \label{ex:compact_cobordism_homotopy_inclusion}
If $n - \deg(x) =  \ind(y)$, show that $ \Broke^{\iota}(y,x)$  admits a natural compactification to a one dimensional manifold $ \Brokebar^{\iota}(y,x)$ with boundary given by adding the strata
\begin{equation} 
\coprod_{x_0} \Broke^{\iota}(y,x_0) \times \Cyl(x_0,x)  \textrm{ and } \coprod_{y_1} \Tree(y,y_1) \times \Broke^{\iota}(y_1,x)
\end{equation}
to those in Equation \eqref{eq:boundary_broke_iota_not_compact}.
\end{exercise}
Note that the map induced by the boundary stratum on the right of Equation \eqref{eq:boundary_broke_iota_not_compact} agrees with the composition of $\iota \circ \Vit_{r}$.  On the other hand, Lemma \ref{lem:evaluation_map_good_r_large} implies that $  \iota \circ \ev_{r-1} $  is homotopic to $\ev_{r}$, so the map defined by the right hand side of Equation \eqref{eq:boundary_broke_iota_not_compact} is homotopic to $\Vit_{r-1}$. Exercise \ref{ex:compact_cobordism_homotopy_inclusion} therefore implies that these two maps are homotopic. At the level of cohomology, we conclude:
\begin{lem}
  If $r$ is sufficiently large, the composition
  \begin{equation}  \label{eq:composition_to_loop_space}
  HF^{*}(H; \bZ)  \to  HM_{-*}(\sL^{r} \Q; \eta) \to H_{-*}(\sL \Q; \eta)
  \end{equation}
is independent of $r$. \qed
\end{lem}

So far, we have been abusing notation, as we are denoting by $\Vit_{r}$ a map which, \emph{a priori} depends on the choice of a family of Hamiltonians and almost complex structures in Equation \eqref{eq:operator_half_cylinder}.
\begin{exercise} \label{ex:map_index_of_Floer_data}
Show that the composition in Equation \eqref{eq:composition_to_loop_space}  is independent of the choice of Floer data (Hint: use the same argument as in the proof of Lemma \ref{lem:continuation_independent_choices}.)
\end{exercise}

\subsection{Compatibility with continuation maps} \label{sec:comp-with-cont}

In order to define a map from symplectic cohomology to loop homology, we need to prove the commutativity of the diagram
\begin{equation}
  \xymatrix{  HF^{*}(H; \bZ) \ar[r]^{\cont} \ar[d]^{\Vit_{r}} &  HF^{*}(K; \bZ) \ar[dl]^{\Vit_{r}}\\
 HM_{-*}(\sL^{r} \Q; \eta) }
\end{equation}
whenever $r$ is sufficiently large.

We briefly outline the construction: one can build a family of Cauchy-Riemann equations on the punctured disc $Z^{+}$ interpolating between Equation \eqref{eq:operator_half_cylinder} and the result of gluing this Equation with the continuation map from $H$ to $K$ on a cylinder (see Equation \eqref{eq:continuation_equation}).  We choose all almost complex structures to be convex near the unit sphere bundle, and the Hamiltonians so that, for any fixed equation, the slope does not increase along the $s$ coordinate of the punctured disc.
\begin{exercise}
  Give a precise definition of the family of equations sketched above, using the construction of Section \ref{sec:cont-maps-comm} as a model.
\end{exercise}
We write $\Cont(x)$ for the space of such solutions of this family of equations, with boundary on the zero section, and asymptotic condition given by a Hamiltonian orbit $x$, and $\Contbar(x)  $ for its Gromov-Floer compactification.

By construction, one boundary stratum of $ \Contbar(x)  $ is the space of planes $\Cyl(x) $. At the other end of the moduli space, we obtain a curve that has two components, one a plane and the other a cylinder carrying the continuation equation. We conclude:

\begin{lem} \label{lem:boundary_continuation_value_0-section}
  For regular choices of data, the moduli space $\Cont(x)  $ is a smooth manifold of dimension $n-\deg(x) +1$. The codimension $1$ strata of its compactification are:
  \begin{align}
    & \Cyl(x) \\ \label{eq:boundary_stratum_boundary_2}
& \coprod_{x_- \in \Orbit(K)} \Cyl(x_{-}) \times \Cont(x_{-},x)  \\ \label{eq:boundary_stratum_boundary_1}
& \coprod_{x_0 \in \Orbit(H)} \Cont(x_0 ) \times \Cyl(x_0,x) . 
  \end{align}
 \qed
\end{lem}

As in Section \ref{sec:cont-maps-comm}, an orientation for the moduli space $ \Cont(x)  $ is induced by an orientation of the parametrising interval, which we fix once and for all, as well as a generator of $\ro_{x}$ and $\eta$.  If we choose $r$ large enough,  and the data to be generic, the fibre product
\begin{equation}
  \Cont(x)  \times_{\ev_{r}} W^{s}(y)
\end{equation}
defines a $1$-dimensional cobordism between
\begin{equation}
 \coprod_{\stackrel{x' \in \Orbit(K)}{\deg(x_-) = \deg(x)}}  \Broke(y,x_0) \times \Cont(x_-,x) \textrm{ and }  \Broke(y,x)
\end{equation}
whenever $\ind(y) = n-\deg(x')$.  Elements of the left hand side define the composition $\Vit_{r} \circ \cont$ on $CF^{*}( H ; \bZ) $  while the right hand side defines $\Vit_{r}$. We conclude that these maps are homotopic, and hence the maps induced on homology agree.

In particular, we have a map
  \begin{equation}
  \lim_{\cont} HF^{*}( H ; \bZ) \to  \lim_{r} HM_{-*}( \sL^{r} \Q ; \eta),
  \end{equation}
and conclude:
  \begin{prop}
    The maps $\Vit_{r}$ induce a map
    \begin{equation} \label{eq:viterbo_map}
       \Vit \co SH^{*}(\TQ; \bZ) \to H_{-*}(\sL \Q ; \eta)
    \end{equation}
which is independent of all choices. \qed
  \end{prop}

\section{Compatibility with operations}

We now prove that the map $\Vit$ in Equation \eqref{eq:viterbo_map} is a map of $BV$ algebras. Having proved that $\Vit$ commutes with continuation maps, and recalling that the $BV$ structure on $SH^{*}(\TQ; \bZ)$ was constructed starting with operations
\begin{align}
e \co \bZ & \to HF^{*}(H; \bZ) \\
  \Delta \co HF^{*}( H; \bZ) & \to HF^{*-1}(H; \bZ) \\
 \star \co HF^{*}( H^1; \bZ) \otimes HF^{*}(H^2; \bZ) & \to HF^{*}(H^0; \bZ),
\end{align}
we shall work in the remainder of this section at the level of Floer cohomology groups of linear Hamiltonians. Similarly, having constructed $\Vit $ as a direct limit of the maps $\Vit_{r}$, we shall fix, at each stage, a sufficiently large integer $r$, and show that $\Vit_{r}$ commutes with the operations defined in  Floer and Morse theory.

\begin{rem}
We systematically leave the verification of signs, which follow the pattern of Lemma \ref{lem:Vit_is_chain_map}, to the reader. As should be clear at this stage, the fundamental point in all such arguments is the existence of orientations of the moduli spaces that are coherent at the boundary strata.
\end{rem}

\subsection{Compatibility with the unit}
Recall that we have defined the unit in Floer cohomology by counting elements of the moduli spaces $\Plane(x)$, whenever $x$ is a time-$1$ orbit of a Hamiltonian $H$ satisfying $\deg(x) = 0$; these are rigid pseudo-holomorphic planes in $\TQ$ with asymptotic conditions given by such orbits.

With this in mind, we note that the moduli space
\begin{equation}
  \coprod_{x}   \Cyl(x)  \times  \Plane(x)
\end{equation}
consists of \emph{broken} pseudo-holomorphic discs with boundary condition along the zero section. 

We may glue elements of these two moduli spaces to obtain an equation on a disc $D^2$, which we can deform to have trivial inhomogeneous term, i.e. to be given by
\begin{equation} \label{eq:trivial_homogeneous_term}
 J \frac{\partial u}{\partial s} = \frac{\partial u}{\partial t},
\end{equation}
with respect to some almost complex structure $J$ on $\TQ$. We obtain a family of equations on the disc parametrised by an interval, and define $\Disc(\Q)$ to be the moduli space of solutions with boundary on the zero section. 
\begin{lem}
For generic choices of Cauchy-Riemann equations, $\Disc(\Q)$ is a smooth manifold of dimension $n+1$, with boundary strata
\begin{equation} \label{eq:boundary_discs_htpy}
  \Q \, \textrm{ and }\,   \coprod_{x} \Cyl(x) \times     \Plane(x) .
\end{equation}
\end{lem}
\begin{proof}[Sketch of proof]
  Since the inclusion $\Q \to \TQ$ is a homotopy equivalence, the relative homotopy group $\pi_{2}(\TQ,\Q)$ vanishes; in particular, every disc is homotopic to a constant disc. In particular, for each $u \in \Disc(\Q)$, there is a trivialisation
  \begin{equation}
    u^{*}(T\TQ) \cong \bC^{n}
  \end{equation}
which identifies $\left(u|\partial D\right)^{*}(T\Q)$ with the constant Lagrangian $\bR^{n} \subset \bC^{n}$. This implies that the Fredholm index of the linearisation of the Cauchy-Riemann operator at an element of $\Disc(\Q)  $ is $n$. The dimension of the moduli space is $n+1$ because we are considering a $1$-parametric family of equations.

The space of solutions to Equation \eqref{eq:trivial_homogeneous_term} is naturally diffeomorphic to $\Q$; indeed, all solutions to this equation are constant and the boundary condition implies that they correspond to points of $\Q$. One endpoint of the parametrising family therefore corresponds to the boundary component on the left in Equation  \eqref{eq:boundary_discs_htpy}. The other boundary stratum corresponds to the  other end.
\end{proof}

We have an evaluation map
\begin{equation}
  \ev \co \Disc(\Q) \to \sL \Q,
\end{equation}
which determines an orientation of $  \Disc(\Q)  $ relative $\eta$ as follows: for each map $ u \in    \Disc(\Q)$, we have an isomorphism
\begin{equation}
 | \det(D_{u})| \cong  \ev^{*}(\eta),
\end{equation}
as in Equation \eqref{eq:iso_det_Lag_boundary_kappa}.

In particular, given any critical point $y$ of $f^{r}$ of index $n+1$, an element $(u,\gamma)$ of the fibre product
\begin{equation}
   \Disc(\Q) \times_{\ev_{r}} W^{s}(y)
\end{equation}
induces a canonical map
\begin{equation}
 \cH^{e}_{(u,\gamma)} \co \bZ \to \ro_{y} \otimes \eta_{y}.
\end{equation}
We define a map
\begin{align}
  \cH^{e} \co \bZ & \to CM_{-1}( f^{r}; \eta ) \\
\cH^{e}(1) & = \sum_{\substack{\ind(y) = n-1 \\  (u,\gamma) \in  \Disc(\Q) \times_{\ev_{r}} W^{s}(y) }}  \cH^{e}_{(u,\gamma)}(1).
\end{align}
\begin{rem}
  Note the $CM_{-1}(f^{r},\eta)$ does not necessarily vanish since the homological grading on Morse homology incorporates both the Morse index (which is $n-1$ in this case) and the degree of the graded local system $\eta$.
\end{rem}
The boundary decomposition of $  \Disc(\Q) $ in Equation \eqref{eq:boundary_discs_htpy} implies that $\cH^{e}  $ defines a homotopy in the following square:
 \begin{equation}
   \xymatrix{ \bZ \ar[r] \ar[d] & CF^{*}(H; \bZ) \ar[d] \\
CM_{*}(f^{1}; \eta) \ar[r]^{\iota^r}  & CM_{*}(f^{r}; \eta).}
 \end{equation}
We conclude
\begin{lem}
The map $\Vit$ preserves units:
\begin{equation}
  \Vit \circ e = e.
\end{equation}
\qed
\end{lem}

\subsection{Compatibility with the $BV$ operator}
Given a Hamiltonian orbit $x$ and a critical point $y$, the summands of the compositions  $\Vit_{r+1} \circ \Delta | \ro_{x} $ and $ \Delta \circ \Vit_{r}  | \ro_{x} $ corresponding to $y$ are respectively controlled by the products
\begin{align} \label{eq:first_boundary_BV_htpy}
&  \coprod_{\deg(x') = \deg(x) + 1}   
\Broke(y,x') \times \Cyl_{\Delta}(x',x)  \\ \label{eq:second_boundary_BV_htpy}
&   \coprod_{\ind(y) = \ind(y')-1} \Tree_{\Delta}(y,y') \times \Broke(y',x)   .
\end{align}
To prove that these maps agree, we show that both of these spaces are cobordant to
\begin{equation} \label{eq:product_with_S^1}
\left( S^1 \times \Cyl(x) \right) \times_{a \circ \ev_r}  W^{s}(y),
\end{equation}
which implies that the maps induced on homology are equal.  The map from the first factor above to  $\sL^{r+1} \Q$ is obtained by composing the evaluation map on $\Cyl(x)$  with the family of maps from $\sL^{r} \Q$ into $\sL^{r+1} \Q$ parametrised by the circle, as defined in Equation \eqref{eq:family_circle_embeddings}.

We start with Equation \eqref{eq:first_boundary_BV_htpy}.  The key point is that the moduli space $ \Cyl_{\Delta}(x',x) $ is parametrised over $S^{1}$. By gluing this family to the equation on the punctured disc defining $ \Cyl(x') $, we obtain an $S^{1}$-parametrised family of equations on the punctured disc. To define this gluing, recall that the equations defining $\Delta$ have asymptotic conditions given by Equation \eqref{eq:hamiltonian_family_for_Delta}. In order to glue an element $u$ of $  \Cyl(x') $ to an element $ (\theta,v) $ of $\Cyl_{\Delta}(x',x)  $, we must therefore pre-compose $v$ by rotation by $\theta$. We then choose a homotopy between this $S^1$-parametrised family of glued equations on the punctured disc and the constant family, in such a way that the restriction of each equation to a neighbourhood of the interior puncture is given by:
\begin{equation}
  ( du - dt \otimes X_{H_{t}} )^{0,1} = 0.
\end{equation}

We then define a moduli space $  \Cyl_{\Delta}(x) $ to be the space of
solutions to this family of Cauchy-Riemann equations (parametrised by
$[0,1] \times S^{1}$), which have boundary on the zero section $\Q$,
have asymptotic condition $x$ at infinity, and denote by
$\Cylbar_{\Delta}(x)$ its Gromov-Floer compactification. In
addition to the strata corresponding to the endpoints of the interval
which parametrises the family of equations defining $\Cyl_{\Delta}$,
the boundary of this space consists of concatenations of Floer trajectories and
elements of $\Cyl_{\Delta}(x')$ for orbits $x'$ of degree higher than that
of $x$. We summarise the structure of these moduli spaces in the
following statement:
\begin{lem}
$\Cylbar_{\Delta}(x)$ is a compact space which is stratified by smooth
manifolds; the codimension $1$ strata are
\begin{align}
   \coprod_{\deg(x') = \deg(x) -1} & \Cyl(x') \times  \Cyl_{\Delta}(x',x)  \\
& S^1 \times \Cyl(x) \\
 \coprod_{\deg(x') = \deg(x) + 1}&  \Cyl_{\Delta}(x') \times  \Cyl(x',x).
\end{align}
 \qed
\end{lem}

Assuming $r$ is large enough, we may restrict each element of $\Cylbar_{\Delta}(x)$ to the boundary of the source, and obtain a map
\begin{equation}
  \ev_{r} \co \Cylbar_{\Delta}(x) \to \sL^{r} \Q.
\end{equation}
\begin{lem}
If $r$ is large enough, the restriction of $\ev_{r+1}$ to $ \Cylbar(x') \times \Cylbar_{\Delta}(x',x) $  is homotopic to the composition
  \begin{equation}
\xymatrix{     \Cylbar(x') \times \Cylbar_{\Delta}(x',x)    \ar[r]^-{\ev_{r} \times \pi_{S^1} } &  \sL^{r} \Q \times S^{1} \ar[r] & S^{1} \times \sL^{r} \Q  \ar[r]^{a}&  \sL^{r+1} \Q,}
  \end{equation}
where the middle map is transposition.
\end{lem}
\begin{proof}
It suffices to prove the analogue for the evaluation into the free loop space. We use the fact that the gluing map at this boundary stratum is defined by rotating the element of $  \Cylbar(x)  $ by the angular parameter associated to the element of $ \Cylbar_{\Delta}(x',x)   $. In particular, the evaluation map on $ \Cylbar_{\Delta}(x) $ differs from the evaluation map on this boundary stratum by such a rotation. 
\end{proof}
With this in mind, we consider the composition
\begin{equation}
\xymatrix{   \Cylbar_{\Delta}(x)  \ar[r]^{\ev_{r} \times \pi_{S^1} } & \sL^{r} \Q \times S^{1} \ar[r]^-{a^{-1}} &\sL^{r} \Q}
\end{equation}
where $a^{-1}$ is the composition of $a$ with the map taking $\theta$ to $-\theta$. We denote this map $\ev_{\Delta}$.
\begin{cor}
 The fibre product
 \begin{equation}
   \Cyl_{\Delta}(x) \times_{\ev_{\Delta}} W^{s}(y')
 \end{equation}
is a cobordism between Equations \eqref{eq:first_boundary_BV_htpy} and \eqref{eq:product_with_S^1}. \qed
\end{cor}

We leave the second cobordism as an exercise to the reader:
\begin{exercise}
Consider the map
\begin{align}
\psi^{r} \circ \ev_{r} \co   [0,+\infty) \times \Cyl(x) & \to \sL^{r} \Q \\
(s,u) & \mapsto \psi_{s}^{r}( \ev_{r}(u) ).
\end{align}
Show that the fibre product
  \begin{equation}
  \left( S^{1} \times [0,+\infty) \times \Cyl(x) \right) \times_{a \circ \psi^{r} \circ \ev_{r}} W^{s}(y')
  \end{equation}
defines a cobordism between Equations \eqref{eq:second_boundary_BV_htpy} and \eqref{eq:product_with_S^1}.
\end{exercise}

\subsection{Compatibility with the product}
Let us now consider a pair $(x_1,x_2)$ of Hamiltonian orbits, and a critical point $y$ of $f^{r_1+r_2}$.  The summands of the compositions  $\Vit_{r_1+r_2} \circ \star $ and $ \star \circ \left(\Vit_{r_1} \otimes \Vit_{r_2} \right)$ corresponding to the triple $(y,x_1,x_2)$ are respectively controlled by the products
\begin{align} \label{eq:first_boundary_star_htpy}
&  \coprod_{\deg(x) = \deg(x_1) + \deg(x_2)}  \Pants(x,x_1,x_2) \times  \Broke(y,x) \\ \label{eq:second_boundary_star_htpy}
&   \coprod_{\ind(y_i) = n-\deg(x_i)}  \Broke(y_1,x_1) \times \Broke(y_2,x_2)   \times \Tree(y,y_1,y_2).
\end{align}
To prove that these maps agree, we start by making two independent choices for the Cauchy-Riemann equation on the punctured disc in Equation \eqref{eq:operator_half_cylinder}, yielding moduli spaces  $  \Cyl(x)$ and $ \Cyl'(x) $ with the property that we have a transverse fibre product over $\Q$
\begin{equation}
  \Cyl'(x_2) \times_{\ev_0}  \Cyl(x_1) 
\end{equation}
for every pair of orbits. We map this fibre product to $ \sL^{r_1} \Q  $ by composing $\ev_{r_1} \times \ev_{r_2}$ with the inclusion map in Equation \eqref{eq:map_product_to_more_points} which corresponds to concatenation. Writing $\ev_{r_1+r_2}  $ for the composite, we assume further that the product 
\begin{equation} \label{eq:triple_fibre_product}
\left( \Cyl'(x_1) \times_{\ev_0}  \Cyl(x_2)  \right) \times_{\ev_{r_1+r_2}}  W^{s}(y)
\end{equation}
is transverse; this can be achieved by choosing the Morse function on $ \sL^{r_1} \Q  $ generically.

The construct a cobordism between Equations \eqref{eq:first_boundary_star_htpy} and \eqref{eq:triple_fibre_product}, it suffices to prove that
\begin{equation} \label{eq:first_cobordism_product}
 \coprod_{\deg(x) = \deg(x_1) + \deg(x_2)}  \Pants(x,x_1,x_2) \times \Cyl(x) \textrm{ and } \Cyl'(x_1) \times_{\ev_0}  \Cyl(x_2)
\end{equation}
are cobordant. This is a standard argument in Floer theory: Consider the moduli space of Riemann surfaces of genus $0$, with one boundary component which is a circle carrying a marked point $\xi_{0}$, and two interior punctures which we label $(\xi_1,\xi_2)$ and such that the unique diffeomorphism to the unit disc, taking $\xi_1$ to the origin and $\xi_0$ to $1$, maps $\xi_2$ to a point on the positive real axis.
\begin{figure}[t]
  \centering
    \includegraphics{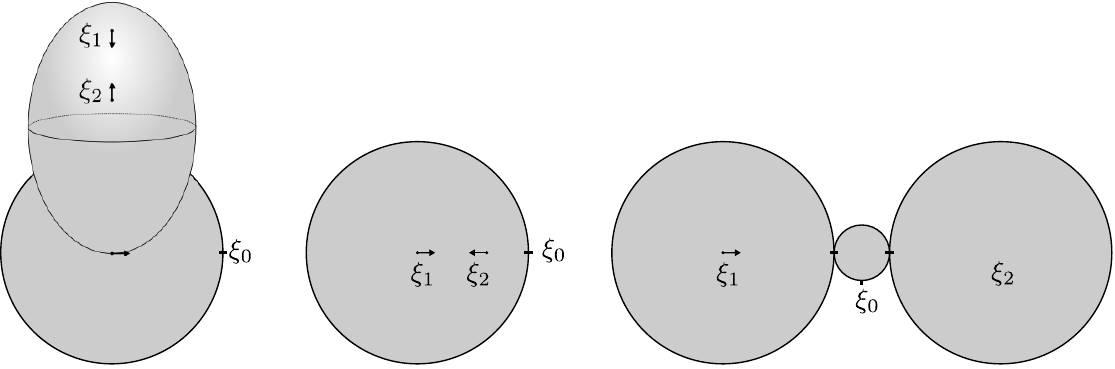}
\caption{}
 \label{fig:products_commute}
\end{figure}

This moduli space is diffeomorphic to open interval, and naturally
compactifies by adding two strata as shown in Figure
\ref{fig:products_commute}; in the limit where the image of $\xi_2$
converges to the origin, we obtain a pair of pants and a punctured
disc (i.e. the source of those maps considered on the left in Equation
\eqref{eq:first_cobordism_product}), while in the limit where $\xi_2 $
converges to the boundary, we obtain two discs with interior punctures attached to a disc with $3$ boundary marked points. In order to
compare this to the right of Equation
\eqref{eq:first_cobordism_product}, note that this space can be
written as
\begin{equation}
 \Cyl'(x_1) \times_{\ev_0} \Q \times_{\ev_0}  \Cyl(x_2).
\end{equation}
If we equip the disc with $3$ boundary marked points with a Cauchy-Riemann
equation with trivial inhomogeneous term, the moduli space of
solutions with boundary on the $0$ section consists only of
constant maps, and is hence naturally diffeomorphic to $Q$. In the literature, this is called a
\emph{ghost bubble}. We now see that the
factors of this triple fibre product correspond exacly to the three
components of the curve on the right of Figure \ref{fig:products_commute}.

To construct a cobordism from this family of Riemann surfaces, we choose
a family of Cauchy-Riemann equations which interpolate between the
choices at the two endpoints. Requiring convexity of the almost
complex structures near $\SQ$, and for the slopes of the Hamiltonians
to satisfy the appropriate monotonicty property, we obtain, for
generic data, the desired cobordism between the two sides in Equation \eqref{eq:first_cobordism_product}.

The next step is to construct a cobordism between Equations
\eqref{eq:second_boundary_star_htpy} and
\eqref{eq:triple_fibre_product}. To this end, we choose a homotopy
between the equation defining $\Cyl'(x)$ and that defining $\Cyl(x)$;
we assume that the homotopy is parametrised by $s \in [0,\infty)$, so that $s=0$ corresponds to the equation for
$\Cyl'(x) $, and the equation for any sufficiently large value of $s$ agrees with that for $\Cyl(x)$.  We then write
\begin{equation}
  \Cyl^{s}(x)
\end{equation}
for the space of solutions for a fixed $s$.

Consider the parametrised moduli space
\begin{equation} \label{eq:cobordism_product_2}
  \bigcup_{s} \psi_{s}^{r_1}  \Cyl^{s}(x_1) \times_{\ev_0} \psi_{s}^{r_2} \Cyl(x_2).
\end{equation}
\begin{rem}
  Note that both factors in Equation \eqref{eq:cobordism_product_2}
  are obtained by applying the gradient flow for the \emph{same} time $s$.
\end{rem}

When $s=0$, we obtain the fibre product $ \Cyl'(x_1) \times_{\ev_0}  \Cyl(x_2)  $. For $s$ sufficiently large, this space consists of pairs $ (\gamma_1,\gamma_2) $ of Morse trajectories, respectively emanating from  $ \Cyl(x_1)  $ and $\Cyl(x_2)  $, such that
\begin{equation} \label{eq:matching_condition_trajectory}
  \ev_{0}(\gamma_1(s)) = \ev_{0}(\gamma_2(s)).
\end{equation}
 In the limit $s \to +\infty$, we obtain broken trajectories; in
 codimension $1$, we can generically assume that such a trajectory
 consists of a negative gradient flow line from $ \Cyl(x_i)  $  to a critical point
 $y_i$, followed by a negative gradient trajectory emanating from $y_i$, i.e. an element of $W^{u}(y_i)$. Equation \eqref{eq:matching_condition_trajectory} corresponds to requiring that the points in $ W^{u}(y_1) $  and $W^{u}(y_2)  $  have the same image in $\Q$ under $\ev_0$. 

In summary, we find that Equation \eqref{eq:cobordism_product_2} defines a cobordism between the manifolds:
\begin{align}
& \Cyl'(x_1) \times_{\ev_0}  \Cyl(x_2)   \\
&  \coprod_{\ind(y_i) = n- deg(x_i)}\left(   \Broke(y_1,x_1) \times \Broke(y_2,x_2)  \right) \times W^{u}(y_1) \times_{\ev_0}  W^{u}(y_2).
\end{align}

If we take the fibre product of the first boundary stratum above with $W^{s}(y)$  over the evaluation map to $\sL^{r_1  + r_2} \Q   $ we obtain Equation \eqref{eq:triple_fibre_product}.  On the other hand, using Equation \eqref{eq:multiplication_moduli_space_2_inputs}, we see that  the fibre product of the second boundary stratum with $W^{s}(y)  $ gives Equation \eqref{eq:second_boundary_star_htpy}. We conclude

\begin{lem}
Given Hamiltonians $H^1$ and $H^2$, the following diagram commutes up to homotopy whenever $r_1$ and $r_2$ are sufficiently large:
 \begin{equation}
   \xymatrix{  CF^{i}(H^{1}; \bZ) \otimes CF^{j}(H^{2}; \bZ) \ar[r] \ar[d] & CF^{i+j}(H^{0}; \bZ) \ar[d] \\
CM_{-i}(f^{r_1}; \eta) \otimes CM_{-j}(f^{r_2}; \eta)  \ar[r] & CM_{-i-j}(f^{r_1+r_2}; \eta).  }
 \end{equation}
\end{lem}

\section{Manifold structure on moduli spaces} \label{sec:manif-struct-moduli}
We give a brief outline of the proof of the gluing result used in
Proposition \ref{prop:gluing_manifold}; there are no new ideas beyond
those that go in proving the fact that the square of the differential on Floer
cohomology vanishes, which is proved for example in \cite{AS}.

Fix a pair of orbits $x$ and $x_0$, such that $\deg(x_0) = \deg(x) +1$. Assuming that  $ \Cyl(x_0,x) $ is  regular, this implies that this manifold is $0$-dimensional; we choose
\begin{equation}
  u \in  \Cyl(x_0,x).
\end{equation}

Next, we fix an integer $r$ and points $(\xi_1, \ldots, \xi_r)$ on the boundary of the punctured disc. By evaluation, we obtain a map $\ev_{r}$ from $ \Cylbar(x) $ to $\Q^{r}$. We also fix a map from a manifold $N$ to $\Q^{r}$ which is transverse to $\Cyl(x_0)$, and such that the codimension of $N$ agrees with the dimension of $x_0$. Since the construction we give below is local, we assume that $N$ is embedded in $\Q^{r}$. We write
\begin{equation}
  \Cylbar_{N}(x) \textrm{ and } \Cyl_{N}(x_0) 
\end{equation}
for the fibre products of the moduli spaces of punctured discs with $N$; i.e. for the subset of punctured discs whose  image under $\ev_{r}$ lies in $N$.

The above assumptions imply that $ \Cyl_{N}(x_0)  $ is $0$-dimensional, we therefore fix an element 
\begin{equation}
v \in  \Cyl_{N}(x_0).
\end{equation}
We shall construct an open embedding
\begin{equation} \label{eq:smoothness_gluing}
 (S,\infty] \to \Cylbar_{N}(x)
\end{equation}
onto a neighbourhood of $(u,v) \in  \Cyl_{N}(x_0) \times  \Cyl(x_0,x) $, whenever $S$ is sufficiently large

\subsection{Construction of the gluing map}
We now give a brief outline of the construction of the map in Equation \eqref{eq:smoothness_gluing}:
\begin{enumerate}
\item Denote by $W^{1,p}_{N}(Z^{+},v^{*}(T\TQ))$ the set  vector fields of Sobolev class $(1,p)$ on $\TQ$, along the image of $v$, whose restriction to points $(\xi_1, \ldots, \xi_r)$ lies in the subspace
  \begin{equation}
    T N \subset \bigoplus_{r} T_{v(\xi_i)} Q.
  \end{equation}
This is the tangent space to the space of maps of Sobolev class $(1,p)$ from $Z^{+}$ to $\TQ$, taking the points  $(\xi_1, \ldots, \xi_r)$ to $N$ and which can be expressed, near the ends, as the exponential of a $W^{1,p}$ section of $x^{*}(T \TQ)$ or $x_{0}^{*}(T \TQ)$.
\item The transversality assumption implies that the restriction of the linearised operator 
  \begin{equation}
 W^{1,p}(Z^{+},v^{*}(T\TQ)) \to  L^{p}(Z^{+},v^{*}(T\TQ)) 
  \end{equation}
to $ W^{1,p}_{N}(Z^{+},v^{*}(T\TQ)) $ is an isomorphism.
\item Choose a hypersurface of $\TQ$, denoted $D$  which is transverse to $u(\bR \times \{0\})$; we require transversality at infinity as well, i.e. that $x(0) \notin D$. Fix a parametrisation of $u$, so that $u(0,0) \in D$ and $ u((0,\infty) \times \{0\}) \cap D = \emptyset $. The restriction of the linearised Cauchy-Riemann operator on $u$ to those vector fields which are tangent to $D$ at $0$ is an isomorphism
  \begin{equation} \label{eq:right_inverse_u}
  W^{1,p}_{D}(Z,u^{*}(T\TQ))  \to L^{p}(Z,u^{*}(T\TQ)).
  \end{equation}

\item Given a positive real number $S$, which is sufficiently large, define a preglued map
  \begin{align}
    v  \#_{S} u \co Z^{+} & \to T \TQ  \\
v \#_{S} u(s,t) & = \begin{cases} v(s,t)  & \textrm{ if } s \in [0 , S/4] \\
\chi( v(s,t),  u(s - S,t) ) & \textrm{ if } s \in [S/4 , 3S/4] \\
u(s - S,t)  & \textrm{ if } s \in [3S/4 , \infty) \\
\end{cases}
  \end{align}
where $\chi$ is constructed using appropriate cutoff functions in exponential coordinates in a neighbourhood of $x_0$. The key point here is that both $u$ and $v$ converge at the appropriate end to the same orbit.
\item Define $W^{1,p}_{N,D}(Z^{+},(v \#_{S} u)^{*}(T\TQ))  $ to be the Sobolev space of vector fields in $\TQ$ of class $(1,p)$ along the pre-glued map, which are tangent to $D$ at $(S,0)$, and whose restriction to $(\xi_1,\ldots, \xi_r)$ is tangent to $N$. Note that this is again a tangent space to a space of maps from the cylinder to $\TQ$;  we fix an exponentiation map which preserves the property that
  \begin{equation}
    (u(S,0), u(\xi_1), \ldots, u(\xi_r)) \in D \times N \subset \TQ^{r+1}.
  \end{equation}
\item By a pregluing construction at the level of tangent vectors, show that the restriction of the linearised Floer equation 
  \begin{equation}
    W^{1,p}_{N,D}(Z^{+},(v \#_{S} u)^{*}(T\TQ))  \to L^{p}(Z^{+},v \#_{S} u)^{*}(T\TQ)) 
  \end{equation}
is invertible.
\item By the implicit function theorem there is a unique solution to the Floer equation in a neighbourhood $v \#_S u  $, which is obtained by exponentiating an element of $ W^{1,p}_{N,D}(Z^{+},(v \#_{S} u)^{*}(T\TQ))  $, and which is a solution to the Floer equation on the punctured disc.  The standard version of the implicit function theorem used here is Floer's Picard Lemma (\cite{floer-Picard}*{Proposition 24} and \cite{AD}).
\end{enumerate}
The above step completes the construction of the gluing map, which assigns to $u$, $S$, and $v$, the unique solution to Floer's equation near $v \#_{S} u$  which is the image under exponentiation of a vector field along $v \#_{S} u $ with the appropriate tangency condition to $D$ and $N$. It remains to show that this map is a homeomorphism onto a neighbourhood of $(v,u)$ in  $ \Cylbar_{N}(x)  $. This requires showing both surjectivity, and injectivity. 

\begin{rem}
The main difficulty in the proof that the gluing map is an embedding is due to the fact that our construction of the right inverse on Sobolev spaces is not smooth in the gluing parameter $S$. Our choice ensures that we treat each $S$ separately.  However, one can show, using either careful analysis (i.e. Sobolev spaces with higher exponents) or more careful choices of right inverses, that the gluing map on moduli spaces is indeed smooth in this parameter as well.
\end{rem}

\subsection{Bijectivity of the gluing parameter}

We now explain how to recover the parameter $S$ for a point that is sufficiently close to the boundary:
\begin{enumerate}
\item If a curve $w$ lies in a sufficiently small neighbourhood $W$  of $ (v,u) $, we can decompose the source into two \emph{thick parts} which are disjoint annuli in $Z^{+}$, separated by thin parts. The key point is that (i) most of the energy is concentrated on the thick parts, and (ii) the restriction of $w$ to each thick part is respectively close to the restrictions of $u$ and $v$ to large open sets of their domains. Let $S(w)$ denote the largest real number so that $w(S(w),0) \in D$. 
\item  To prove injectivity, it suffices to show that the images of the gluing maps for different gluing parameters are disjoint. If $w$ is in the image of the gluing map at $u \#_{S} v$, then 
  \begin{equation}
    S(w) = S,
  \end{equation}
which implies that we can recover the gluing parameter for the glued curve, and hence that the map is injective.
\item If $w$ is sufficiently close to $(v,u)$, then one shows that it is close to $ v \#_{S(w)} u $, and hence can be obtained by exponentiating a tangent vector in
  \begin{equation}
    W^{1,p}(Z^{+},(v \#_{S} u)^{*}(T\TQ)).
  \end{equation}
Moreover, since $v \#_{S(w)} u  $ and $ w(S(w),0) $  intersect $D$ at the same point, and both maps send $(\xi_1, \ldots, \xi_r)$ to $N$, $w$ is in fact in the image of the restriction of the exponential map to $   W^{1,p}_{N,D}(Z^{+},(v \#_{S} u)^{*}(T\TQ)) $. Our previous application of the implicit function theorem implies that it is in the image of the gluing map. This proves surjectivity.
\end{enumerate}

\section{Guide to the Literature} \label{sec:summary-literature}
Viterbo's 1994 ICM address \cite{viterbo-94} starts with a friendly discussion of the definition of symplectic cohomology for cotangent bundles via generating functions, and proceeds  to give a heuristic explanation for why  symplectic cohomology agrees with the generating function invariant, which itself agrees with the homology of the free loop space.  The proofs of these results appear in the later preprint \cite{viterbo-96}, with applications provided in \cite{viterbo-99}.     The map that we construct is modelled after Viterbo's in the sense that we use the homology of  finite-dimensional approximations as an intermediate step. The key difference is that, in  \cite{viterbo-96}, the essential point is to choose Morse functions on finite dimensional approximations whose critical points and gradient trajectories are in bijective correspondence to those associated to a Floer complex computing Hamiltonian Floer cohomology. This is a much more rigid framework than the one we are using.

A slight variant of $\Vit$ was in fact introduced in \cite{A-cotangent-generate}. As noted there,  this map can be considered as an implementation in the setting of symplectic cohomology of a construction introduced by Cieliebak and Latschev \cite{CL}  to relate the contact homology of the unit sphere bundle $\SQ$ to the homology of the quotient of $\sL \Q$ by the circle action, relative the constant loops.  This map was also discussed by Abbondandolo and Schwarz in \cite{AS-on-the-product}.

There are two alternative approaches to relate Floer cohomology to loop homology:
\begin{enumerate}
\item In \cite{SW-06}, Salamon and Weber showed that the heat flow on the loop space arises as a degeneration of the Floer equation. By carefully controlling the limit, they produce a Morse complex on the $L^{2}$ loop space whose critical points and gradient trajectories are also in bijective correspondence with the corresponding Floer data.
\item In \cite{AS}, Abbondandolo and Schwarz relate the Morse homology of the energy functional on a $W^{1,2}$ model of the loop space with Floer homology. They define a chain map from the Morse complex to the Floer complex using a fibre product of moduli spaces constructed using Morse and Floer theory; this is analogous to the construction in Section \ref{sec:constr-map-floer}, with the main difference being that the fibre product is taken over a Hilbert manifold. 
\end{enumerate}
For more discussion about the relation between different approaches, Weber's survey \cite{weber-05} is a useful guide.

The comparison of the circle actions on Floer and loop homology was done by Viterbo in \cite{viterbo-96}, and the comparison of the products by Abbondandolo and Schwarz in \cite{AS-product}.

\subsection{Signs and orientations}

In \cite{kragh-1}, Thomas Kragh constructed a spectrum which refines a homological invariant defined from generating functions. He observed that a natural map associated by Viterbo to an (exact) embedding
\begin{equation}
  D^{*} L \to  \TQ
\end{equation}
can only be defined in his setting if one works with a twisted version of the spectrum. This indicated that symplectic cohomology is not isomorphic to the homology of the free loop space whenever the base is not $\Spin$, as verified by Seidel for $\bC \bP^{2}$  in an unpublished note. The corrected statement of Viterbo's theorem for cotangent bundles of orientable manifolds was given in  \cite{A-cotangent-generate}, together with an explanation for how to prove the result by correcting the constructions in the literature. An independent verification has since appeared in \cite{AS-signs}. The case of non-orientable manifolds has not been considered in the literature.

\chapter{Viterbo's theorem: Surjectivity}
\label{cha:viterbos-theorem}
\section{Introduction}
In this chapter, we  construct a map
\begin{equation}
\Fam \co H_{*}(\sL \Q; \eta) \to SH^{-*}(\TQ; \bZ)
\end{equation}
and prove the following result at the beginning Section \ref{sec:comp-loop-homol}:
\begin{thm} \label{thm:compose_fam_vit_iso}
  The composition
  \begin{equation}
    \xymatrix{ H_{*}(\sL \Q; \eta) \ar[r]^-{\Fam}& SH^{-*}(\TQ; \bZ) \ar[r]^-{\Vit} & H_{*}(\sL \Q; \eta)}
  \end{equation}
is an isomorphism.
\end{thm}
As an immediate consequence, we conclude:
\begin{cor}\label{cor:fam-injective}
  $\Vit$ is surjective and $\Fam$ is injective. \qed
\end{cor}

The construction of $\Fam$ is inspired by ideas of Fukaya on the study of Lagrangian Floer cohomology for families of Lagrangians \cite{Fukaya-family}; the specific family we shall study is the family of cotangent fibres in $\TQ$. At the same time, it may be considered as a finite dimensional variant of the map studied by Abbondandolo and Schwarz in \cite{AS}.

The usual constructions of an isomorphism between loop homology and symplectic cohomology rely on making a special choice of Floer data which ensures that there is a direct relationship between pseudo-holomorphic curves in $\TQ$ and Morse trajectories either in $\sL \Q$ or its finite approximations (see Section \ref{sec:summary-literature} for a discussion).   We shall use instead a cobordism argument involving the moduli space of annuli, which recently has been successfully used to study problems in Lagrangian Floer theory \cites{BC,FOOO-sign,A-generate}.

\begin{rem}
Since we shall ultimately show that $\Vit$ and $\Fam$ are both isomorphisms (and the results of the next Chapter can be used to construct a third isomorphism $\Gam$), some comments about our choice to give $\Vit$ primary status may be in order. The relation between (Hamiltonian) Floer theory on a symplectic manifold and the (Lagrangian) Floer theory of a submanifold is mediated by moduli spaces of holomorphic discs with an interior puncture. This idea has appeared in various incarnations so many times in the literature that we shall necessarily sin by omission in listing potential references: \cite{FOOO,Seidel-ICM,Albers,A-generate}. There are not, as of yet, generalisations of the map $\Fam$ beyond cotangent bundles, so it is in a sense less natural.  At a more technical level, one may also argue that it is easier to show the compatibility of $\Vit$ with various operations than that of $\Fam$.
\end{rem}

\section{Chords, Maslov index, and action}

\subsection{Hamiltonian chords and actions}
Hamiltonian chords are the analogues of orbits for the study of Lagrangian Floer theory. Our use of this theory is quite minimal, but lurking in the background are Lagrangian Floer cohomology groups; in fact, parametrised  Lagrangian Floer cohomology groups.

Consider a pair $L_0$ and $L_1$ of Lagrangians in $\TQ$, and let $H$ be a time-dependent Hamiltonian, parametrised by $t \in [0,1]$.
\begin{defin}
  A time-$1$ \emph{Hamiltonian chord} of $X_{H}$, starting on $L_0$ and ending on $L_1$ is a time-$1$ flow line of $X_{H}$
  \begin{align}
   x \co [0,1] & \to \TQ \\
\frac{dx}{dt} & = X_{H_t}
  \end{align}
whose endpoints lie on $L_i$:
\begin{equation}
  x(i) \in L_i \textrm{ for } i \in \{0,1\}. 
\end{equation}
\end{defin}
We write $\Chord_{H}(L_0,L_1)$ for the set of time-$1$ Hamiltonian chords with these endpoints. The following exercise gives a convenient way of thinking of chords:
\begin{exercise}
Let $\phi$ denote the time-$1$ Hamiltonian flow of $H$. Show that there is a bijective correspondence between elements of $\Chord_{H}(L_{0},L_{1})  $ and intersection points
\begin{equation}
  \phi(L_{0}) \cap  L_{1}.
\end{equation}
\end{exercise}

\begin{defin}
  A chord $x\in \Chord_{H}(L_0,L_1)$ is \emph{non-degenerate} if the corresponding intersection point between $\phi(L_0)$ and $L_1$ is transverse.
\end{defin}

In all cases we shall study, the Lagrangians will be either cotangent fibres, or the $0$-section; these are examples of \emph{exact Lagrangian} submanifolds, on which the primitive $\lambda$ is exact:
\begin{exercise} \label{ex:cotangent_fibre_exact}
  Show that the restriction of  $\lambda$ to $\Q$ and to $\Tq$ vanishes identically.
\end{exercise}
\begin{defin}
If $\lambda|L_i \equiv 0$, the action of a chord $x \in \Chord_{H}(L_0,L_1)$ is the integral
\begin{equation} \label{eq:action_chord}
  \Action(x) \equiv \int H \circ x dt -  x^{*} \lambda.
\end{equation}
\end{defin}
\begin{rem}
On the space of paths $x \co [0,1] \to \TQ$ with endpoints on $L_0$ and $L_1$, Equation \eqref{eq:action_chord} is Floer's action functional; the elements of $\Chord_{H}(L_0,L_1) $ are the critical points of this functional, and can be used to define \emph{Lagrangian Floer cohomology.} It is more common to set up the theory either by dropping the exactness assumption as in \cite{FOOO}, in which case the action functional is only well-defined on a cover of the space of paths, or in an intermediate level of generality as in \cite{seidel-Book}, by requiring only that the class of $\lambda|L_i$ in $H^{1}(L;\bR)$ vanish; in this case, there are additional terms in Equation \eqref{eq:action_chord} coming from choices of primitives for $\lambda|L_i$.
\end{rem}

We shall particularly be interested in the Hamiltonian flow of functions of the form $h \circ \rho$, where $\rho$ is the distance to the zero section  (see Equation \eqref{eq:radial_function}). The most important case is the Hamiltonian flow of  $\frac{\rho^2}{2}$, which is:
\begin{equation}
  \rho X_{\rho}.
\end{equation}
The relation to the geodesic flow is discussed in Section \ref{sec:hamiltonian-orbits}. In particular, the following result follows immediately from Exercise \ref{ex:horizontal-lag}:
\begin{lem} \label{lem:chords_are_based_geodesics}
Projection to $\Q$ defines a bijection between (i) time-$1$ chords of $b X_{\rho}$, starting on $\Tq[0]$ and end ending on $\Tq[1]$, which are contained in a non-zero level set of $\rho$ and (ii) based geodesics of length $b$ starting at $q_0$ and ending at $q_1$. \qed
\end{lem}

The following result will be used essentially to prove compactness for moduli spaces of holomorphic curves:
\begin{exercise} \label{ex:action_chords_y_intercept}
If $h$ is a function on $\TQ$ depending only the distance to the zero section,  show that the action of a chord $x \in \Chord_{h}(\Tq[0],\Tq[1] )$ which is contained in a level set $h(\rho_0)$ is given by
\begin{equation} \label{eq:action_x_y_intercept}
\Action(x) = h(\rho_0) - \rho_0 \frac{dh}{d\rho}(\rho_0).
 \end{equation}
Hint: First, compute that $\lambda(X_{\rho}) = \rho$, then use the chain rule. 
\end{exercise}
Equation \eqref{eq:action_x_y_intercept} has the following useful interpretation: the action of a chord whose distance from the zero section is $\rho_0$ can be computed as the intersection point of the line tangent to the graph of $h$ at $(\rho_0, h(\rho_0))$ with the vertical axis. 
\begin{figure}[h]
  \centering
  \includegraphics{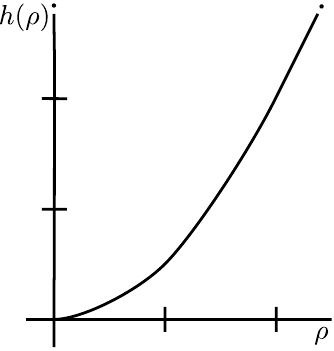}
  \caption{ }
  \label{fig:model_hamiltonian}
\end{figure}
Since $ \frac{\rho^2}{2} $ is not linear at infinity, we shall consider Hamiltonians which agree with it near the $0$ section, and with an affine function of $\rho$ away from $\DQ$.  To formally define this, we choose a function $h$ on $[0,+\infty)$, graphed in Figure \ref{fig:model_hamiltonian}, and which satisfies:
\begin{align}
  h(\rho) &  = \begin{cases}   \frac{\rho^{2}}{2} & \textrm{ if } \rho \in [0,1] \\
 2 \rho -2  & \textrm{ if } 2 \leq \rho.
\end{cases} \\
\frac{d^{2} h}{d \rho^{2}}|(0,2) & > 0.
\end{align}
\begin{exercise}
Show the existence of a function satisfying the above conditions.
\end{exercise}

Given a positive real number $b \in [1,2]$, we can relate, as in Lemma \ref{lem:orbit_is_closed_geodesic}, elements of $\Chord_{b \Hh}(\Tq[0],\Tq[1])  $ to  geodesics from $q_0$ to $q_1$. To state a precise correspondence, note that the Hamiltonian flow of $b \Hh$ preserves the  level sets of $\rho$, and that its restriction to a given level set agrees with the flow of
\begin{equation} \label{eq:flow_of_model_Hamiltonian}
b\frac{d h}{d \rho} X_{\rho}.
\end{equation}
This flow agrees, up to rescaling, with the flow of $\frac{\rho^{2}}{2}$; in particular, time-$1$ flow lines of $b \Hh  $, which are contained in this level set project to geodesics of length $b\frac{d h}{d \rho}$. Since the derivative of $h$ increases monotonically from $0$ to $2$ in the interval $[0,2]$, we can relate chords to geodesics:
\begin{lem} \label{lem:chords_are_geodesics}
There is a bijective correspondence between elements of $\Chord_{b \Hh}(\Tq[0],\Tq[1])  $ contained in $\DQ[2]$ and geodesics between $q_0$ and $q_1$ of length bounded by $2b$. \qed
\end{lem} 
Recalling our normalisation of the metric on $\Q$ to have injectivity radius larger than $4$, we conclude:
\begin{cor} \label{cor:unique_chords}
Given a constant $ b < 2$,    there is a unique element  of $ \Chord_{b \Hh}(\Tq[0],\Tq[1])  $ whenever $d(q_0,q_1) < 2b$. 
\end{cor}
\begin{proof}
 Since  $d(q_0,q_1) <  2b < 4  $  there is a unique geodesic between $q_0$ and $q_1$ whose length is smaller than $2b$.   By Lemma \ref{lem:chords_are_geodesics}, this geodesic is the projection of a Hamiltonian chord with endpoints on the cotangent fibres at $q_0$ and $q_1$.
\end{proof}

\begin{exercise} \label{ex:unique_chord_non-deg}
Under the assumptions of Corollary \ref{cor:unique_chords},  show that the unique element of $ \Chord_{b \Hh}(\Tq[0],\Tq[1])  $  is non-degenerate.
\end{exercise}

\subsection{Strips, half-planes, and triangles}

In Floer theory, non-degenerate Hamiltonian chords are used as asymptotic conditions on boundary punctures of Riemann surfaces with boundary. We start by considering the analogue of the cylinder, which is the strip:
\begin{equation}
  \Strip \equiv \bR \times [0,1],
\end{equation}
equipped with coordinates $(s,t)$, and complex structure $j \partial_{s} = \partial_t $. We define the positive and negative half-strips to be the subsurfaces:
\begin{equation}
   \Strip^{+} \equiv [0,\infty)  \times [0,1] \textrm{ and } \Strip^{-} \equiv (-\infty,0] \times\ [0,1].
\end{equation}

Recall that the orientation line of an orbit was defined as the determinant line of an operator on the plane. The analogue for chords will be an operator on the upper half plane:
\begin{equation}
  \bC_{+} \equiv \{ x + i y \vbar y \geq 0 \} \subset \bC.
\end{equation}
By default, we shall equip $\bC_{+}$ with the negative strip-like end
\begin{align} \label{eq:negative_end_half-plane}
 \epsilon^{-} \co  \Strip^{-} & \to \bC_{+} \\
(s,t) & \mapsto e^{-s-i \pi (1+ t)},
\end{align}
but we shall also sometimes consider the positive end
\begin{align}
 \epsilon^{+} \co  \Strip^{+} & \to \bC_{+} \\
(s,t) & \mapsto e^{s +i t \pi }.
\end{align}
In either case, we equip the upper half-plane with a \emph{strip-like} metric, i.e. so that the restriction to the region $0 \ll |s|$ agrees with the product metric on the half-strip.
  
The surface $\bC_{+}$ is biholomorphic to a disc with one boundary puncture, while $\Strip$ corresponds to a disc with two punctures.

 We shall also consider discs with multiple boundary punctures, most importantly the case of three punctures. It is convenient to use different models for such a surface depending on the context. The following series of exercises shows that all such models are in fact biholomorphic:
\begin{exercise} \label{ex:Riemann-mapping-1}
 Using the Riemann mapping theorem, show that any Riemann surface obtained by removing finitely many points from the boundary of a compact, genus $0$ Riemann surface with one boundary component is biholomorphic to the complement of finitely many points on the boundary of the disc $D^{2} \subset \bC$.
\end{exercise}
\begin{exercise}
Identifying $D^2$ with the union of the upper half-plane with one point at infinity, show that $PSL(2,\bR)$ acts by biholomorphisms.
\end{exercise}
\begin{exercise}\label{ex:Riemann-mapping-3}
Show that the action of $PSL(2,\bR)  $ is triply-transitive on boundary points (i.e. that any two triples of distinct points can be mapped to each other by this action). Conclude that any two surfaces which are obtained by removing $3$ points from the boundary of compact, genus $0$ Riemann surfaces with one boundary component are biholomorphic.
\end{exercise}

One of the models we shall  consider is the punctured Riemann surface
\begin{equation}
\Tria \equiv   \Strip \setminus \{ (0,0)\},
\end{equation}
equipped with the natural positive and negative strip-like ends $\epsilon_{\pm \infty}$ which are given by the inclusions $\Strip^{\pm} \subset \Strip$. We also choose a positive strip-like end $\epsilon_{(0,0)} $ for $\Tria$ at the puncture $(0,0)$ whose image is contained in $[-1,1] \times [0,1]$; with this in mind, we may represent $\Tria $  as either of the two Riemann surfaces in Figure \ref{fig:triangle-1}.
\begin{figure}[h]
  \centering
  \includegraphics{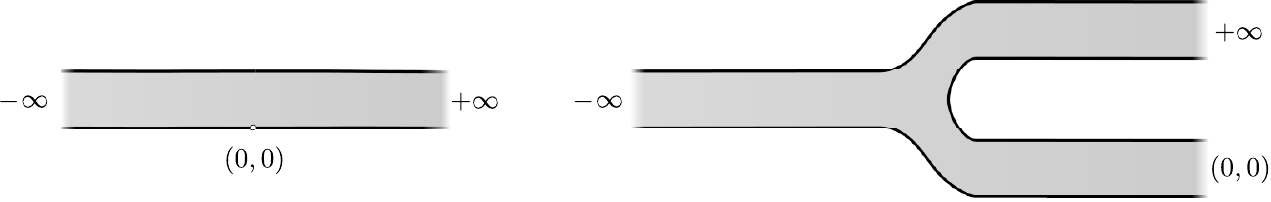}
  \caption{ }
  \label{fig:triangle-1}
\end{figure}

\subsection{The Maslov index for paths} \label{sec:maslov-index-paths-1}

The general construction of $\bZ$-graded determinant lines associated to Hamiltonian chords is analogous to that for Hamiltonian orbits (see Definition \ref{def:determinant_line-CZ}), but is complicated by the fact that it requires a choice of path, in the Grassmannian of Lagrangians, between the tangent spaces of the Lagrangians corresponding to the two endpoints. The general theory is discussed in \cite{seidel-Book}.  We shall give a relatively self-contained account, using only the analogue of Lemma \ref{lem:pi_1_U(n)} for Cauchy-Riemann operators on the disc with Lagrangian boundary conditions, and a basic case of Lemma \ref{lem:computation_index_disc_Lagrangian_boundary}.

Fix a projection $t \co \bC_{+} \to [0,1]$ which agrees with the $t$ coordinate on the image of the strip-like end. Let $\Gr(\bC^{n})$ denote the Grassmannian of Lagrangians in $\bC^{n}$. Consider a pair of paths
\begin{align}
  \Lambda \co [0,1] & \to \Gr(\bC^{n}) \\
A \co [0,1] & \to \Sp(2n,\bR)
\end{align}
with the property that 
\begin{equation} \label{eq:time_0_1_transverse}
  \parbox{35em}{ $A_{1} \Lambda_0$ is transverse to $ A_{0} \Lambda_1$.}
\end{equation}
If the path $A_{t}$ is constant, this condition simplifies to the assumption that $\Lambda_1$ be transverse to $\Lambda_0$.

For any function on the upper-half plane valued in the Lie algebra of the group of symplectic matrices
\begin{equation}
  B^{-} \co \bC_{+} \to \fsp_{2n}
\end{equation}
such that 
\begin{equation}
  B^{-}(\epsilon^{-}(s,t)) = \frac{d A_{t}}{dt}
\end{equation}
whenever $0 \ll |s|$, we obtain a Fredholm operator on the space of $\bC^{n}$ valued functions with $\Lambda$ as Lagrangian boundary conditions:
\begin{align}
D_{A,\Lambda}^{-} \co W^{1,p}((\bC_{+},\bR), (\bC^{n}, \Lambda)) & \to   L^{p}(\bC_{+}, \bC^{n}) \\ 
X & \mapsto \partial_{x} X + (I - B^{-}) \partial_{y} X.
\end{align}
Such operators on the strip were considered by Floer in \cite{Floer-index}; we are using here the extension to Riemann surfaces with strip-like ends discussed in  \cite[Section 8]{seidel-Book}.
When $A$ is the constant path,  we simply write $D_{\Lambda}$  for this operator.

\begin{defin}
The \emph{Maslov index} of the path $\Lambda$ is the Fredholm index of the operator $D_{\Lambda} $. 
\end{defin}

Using a positive strip-like end at infinity, we may define a different operator
\begin{align}
D_{A,\Lambda}^{+} \co W^{1,p}((\bC_{+},\bR), (\bC^{n}, \Lambda)) & \to   L^{p}(\bC_{+}, \bC^{n}) \\ 
X & \mapsto \partial_{x} X + (I - B^{+}) \partial_{y} X,
\end{align}
where $B^{+}$ is a matrix valued function on $\bC_{+}$ such that
\begin{equation}
  B^{+}(\epsilon^{+}(s,t)) = \frac{d A_{t}}{dt}.
\end{equation}
\begin{exercise}
  Let $(\Lambda^{-1},A^{-1})$ denote the paths
  \begin{align}
    \Lambda^{-1}_{t} & \equiv \Lambda_{1-t} \\
    A^{-1}_{t} & \equiv A_{1-t}. 
  \end{align}
Show that there is a canonical up to homotopy isomorphism
\begin{equation}
  \det(D_{A,\Lambda} ) \cong   \det(D^{+}_{A^{-1},\Lambda^{-1}}). 
\end{equation}
\end{exercise}

There are a few formal properties of the Maslov index that shall be useful later. First, note that the result of gluing two copies of the upper half-plane along the ends is biholomorphic to the disc:
\begin{equation}
  \bC_{+} \#_{S} \bC_{+} \cong D^{2};
\end{equation}
here we assume that the first copy of $\bC_{+}$ is equipped with a positive end, and the second copy with a negative end. If the Lagrangian boundary conditions are both given by the path $\Lambda$, then the resulting Lagrangian boundary conditions on the boundary of $D^2$ is the loop of Lagrangians obtained by traversing $\Lambda$ then its inverse. This loop admits a canonical homotopy to the constant loop at $\Lambda_0$. At the level of operators, composing the gluing map with the deformation of determinant lines associated to this homotopy induces a canonical isomorphism:
\begin{equation} \label{eq:glue_positive_negative_disc}
\det(  D^{+}_{ \Lambda }) \otimes  \det(  D^{-}_{ \Lambda } ) \cong \det(D_{\Lambda_0}) \cong \det(\Lambda_0).
\end{equation}
In the last step, we use the isomorphism in Equation \eqref{eq:isomorphism_disc_operator_constant}.

We shall also use the analogue of Proposition \ref{lem:orient_line_indep_path} for paths of Lagrangians:
\begin{exercise}
Let $\Phi$ be a loop of unitary matrices. Show that
  \begin{equation} \label{eq:index_Lag_boundary_loop_symp}
      \det(D_{\Phi(\Lambda) })  \otimes \det_{\bR}(\bC^{n})  \cong   \det_{\bR}(D^{-}_{\Phi^{-1}}) \otimes  \det(D_{\Lambda}).
  \end{equation}
\end{exercise}

\subsection{The Maslov index for paths in dimension $1$} \label{sec:maslov-index-paths}

We would like to compute the Maslov index of the path of Lagrangians
\begin{equation}
 e^{- \pi i \delta t} \bR
\end{equation}
for positive values of $\delta$ smaller than $1$. One can in fact perform these computations relatively explicitly, using for example the methods of \cite{robbin-salamon}. Instead, we shall use formal properties of the Maslov index.

We start in dimension $1$, and consider, for each integer $k$, the path of Lagrangians
\begin{equation}
  \Lambda^{k}(t)  = e^{\pi (\delta_0 + (k - \delta)  t ) i} \bR \in \bR \bP^{1},
\end{equation}
where $\delta_i$ are real numbers such that
\begin{equation} \label{eq:difference_delta_i_small}
  0 < \delta < 1.
\end{equation}
\begin{exercise}
Show that $  \Lambda^{k}(1)  $  is transverse to $\Lambda^{k}(0) $, and that varying $\delta_0$ and $\delta$ does not change the homotopy class of $\Lambda^{k}$ within the class of paths satisfying Condition \eqref{eq:time_0_1_transverse} as long as Equation \eqref{eq:difference_delta_i_small} is satisfied.
\end{exercise}

We shall compute the index of $  \Lambda^{k} $ using various relations:
\begin{exercise}
  Use Lemma \ref{lem:pi_1_U(n)} to show that
\begin{equation} \label{eq:add_2_to_Maslov}
    \ind(D_{\Lambda^{k+2}} ) = 2 + \ind(D_{\Lambda^{k}}).
\end{equation}
\end{exercise}

It remains to compute the two integers $\ind( D_{\Lambda^{0}} ) $ and $\ind( D_{\Lambda^{1}} )$. Consider the specific representatives
\begin{align} \label{eq:explicit_Lambda_0}
  \Lambda^{0}(t)  & = e^{\frac{- t}{2} \pi i} \bR \in \bR \bP^{1} \\
\Lambda^{1}(t)  & = e^{(\frac{-1}{2} +\frac{t}{2}) \pi i} \bR \in \bR \bP^{1}.
\end{align}
Note that these paths have been chosen so that their concatenation is a loop in $\bR \bP^1$ which is homotopic to the constant loop. At the level of operators, we can glue $D_{\Lambda^0}$ and $D_{\Lambda^{1}}$ to obtain an operator on the disc homotopic to $D_{\bR}$. This is precisely the setting of Lemma \ref{lem:inverse_path_det_line_is_inverse}, so we conclude that
\begin{equation} \label{eq:sum_malsov_is_1}
  \ind(D_{\Lambda^{0}} ) +  \ind(D_{\Lambda^{1}} ) = 1.
\end{equation}

\begin{figure}[h]
  \centering
\includegraphics{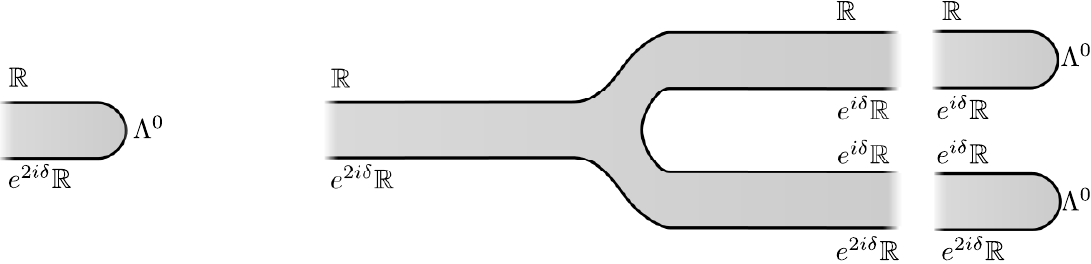}
  \caption{ }
  \label{fig:glue-triangle-1}
\end{figure}
We need to obtain an additional relation. One way to do so is to introduce the operator $D_{\Tria}$ on a thrice boundary punctured disc $\Tria$ with Lagrangian boundary conditions $\bR$, $ e^{\delta i} \bR $, and  $ e^{2 \delta i} \bR $ ordered clockwise.
\begin{exercise} 
  Show that the result of gluing two copies of the operator $D_{\Lambda^0}$  to $D_{\Tria}$ along its two incoming ends is homotopic to $D_{\Lambda^0}$ (see Figure \ref{fig:glue-triangle-1}). Conclude that
  \begin{equation} \label{eq:0-triangle-0}
    \ind(D_{\Lambda^0}) +  \ind(D_{\Tria}) = 0.
  \end{equation}
\end{exercise}
\begin{exercise}
  Show that the result of gluing $D_{\Lambda^1}$ and $D_{\Lambda^1}$ to $D_{\Tria}$ is homotopic to $D_{\Lambda^2}$. Using Equation \eqref{eq:add_2_to_Maslov}, conclude that
  \begin{equation} \label{eq:1-triangle-0}
2 \ind(D_{\Lambda^1}  )  + \ind(D_{\Tria}) =   \ind(D_{\Lambda^0})  + 2.
  \end{equation}
\end{exercise}

All that remains is some elementary arithmetic to show that:
\begin{lem}
  The index of $D_{\Lambda^k}$ is $k$.
\end{lem}
\begin{proof} 
  From Equations \eqref{eq:add_2_to_Maslov} and \eqref{eq:sum_malsov_is_1}, it suffices to show that the index of $D_{\Lambda^0}$ vanishes. Combining Equations \eqref{eq:0-triangle-0} and \eqref{eq:1-triangle-0}, we find that
  \begin{equation} 
    \ind(D_{\Lambda^1}  )   =  \ind(D_{\Lambda^0})  + 1.
  \end{equation}
  Using Equation \eqref{eq:sum_malsov_is_1}, we conclude, indeed that $  \ind(D_{\Lambda^0}) =0 $. 
\end{proof}

\subsection{A regularity result}
Consider a path of Lagrangians
\begin{equation}
   \Lambda_{t} \equiv A_{t} \Lambda_0,
\end{equation}
and assume that the path of  matrices $B_t \in \fsp_{2n}$ generating $A_{t}$ satisfies the following property:
\begin{equation} \label{eq:positive_definite}
  \parbox{30em}{the quadratic form $v \mapsto \omega(v,B_{t}v)$ is negative semi-definite on $\Lambda_t$.}
\end{equation}
\begin{lem} \label{lem:kernel_trivial}
 If Condition \eqref{eq:positive_definite} holds, the operator $D_{\Lambda}$ has trivial kernel.
\end{lem}
\begin{proof}
  Consider an element of the kernel, which is a map
  \begin{equation}
   (X,Y) \co \bC_{+} \to \bR^{n} \oplus \bR^{n} = \bC^{n}
  \end{equation}
mapping a point $z$ on the boundary to $\Lambda_{t(z)}$, and such that
\begin{align}
  \partial_{y}X & = - \partial_{x}Y \\ 
\partial_{y} Y  & = \partial_{x} X
\end{align}
For such a solution the $L^{2}$ energy
\begin{equation}
 \frac{1}{2} \int_{\bC_{+}} ||  (\partial_{x}X, \partial_{x} Y)||^{2}
\end{equation}
is finite and non-negative, vanishes if and only if both $X$ and $Y$ vanish, and agrees with
\begin{equation}
  \int_{\bC_{+}} (X,Y)^{*}(\omega),
\end{equation}
where $\omega$ is the standard Darboux form on $\bC^{n}$.

Applying Stokes's theorem, we find that the above integral is given by
\begin{equation}
  \int_{\bR} (X,Y)^{*}(\sum x_{i} dy_i) = \int_{z \in \bR} \omega( X , \partial_{z} Y ) dz .
\end{equation}
The integrand above is non-positive by Condition \eqref{eq:positive_definite}. We conclude that the $L^{2}$ energy vanishes, hence that every element of the kernel is trivial.
\end{proof}

\begin{cor}\label{cor:determinant_line}
  If the Maslov index of $\Lambda$ vanishes, and Condition \eqref{eq:positive_definite} holds, the determinant line of $D_{\Lambda}$ is canonically trivial.
\end{cor}
\begin{proof}
  If the Fredholm index of $D_{\Lambda} $ vanishes, and its kernel is trivial, then so is its cokernel. This implies that the determinant line of $ D_{\Lambda} $ is the top exterior power of a vector space of rank $0$, which is canonically trivial.
\end{proof}

\subsection{Orientation lines of Hamiltonian chords} \label{sec:orient-lines-hamilt}

Recall that we fixed a Hamiltonian $h$ on the cotangent bundle, which is graphed in Figure \ref{fig:model_hamiltonian}. Let $x \co [0,1] \to \TQ$ be a Hamiltonian chord of $b\Hh$ for $b \in [1,2]$, with endpoints on $\Tq[0]$ and  $\Tq[1]$. By pulling back $T \TQ$ by $x \circ t$, we obtain a symplectic vector bundle over $\bC_{+}$; this bundle is equipped with the Lagrangian boundary conditions
\begin{equation} \label{eq:Lagr_boundary_D_x}
  T^{*}_{q(x(t))} \Q.
\end{equation}

Let $\Lambda^{x}$ denote the image of these boundary conditions under a trivialisation of $(x \circ t)^{*}(T \TQ)$. As in Section \ref{sec:invar-paths-sympl}, let $A_{t}^{x} \in \fsp_{2n}$ be the path of matrices which exponentiates to the path of symplectomorphisms of $\bC^{n}$ induced by the Hamiltonian flow of $\Hh$.

By the construction of Section \ref{sec:maslov-index-paths-1}, we obtain a Cauchy-Riemann operator
\begin{equation}\label{eq:linearisation_Floer}
D_{A^{x},\Lambda^{x}} \co W^{1,p}((\bC_{+},\bR), (\bC^{n}, \Lambda^{x}))  \to   L^{p}(\bC_{+}, \bC^{n}).
\end{equation}

We shall define the orientation line of $x$ to be the determinant line:
\begin{equation} \label{eq:orientation_line_chord}
  \ro_{x} \equiv \det(D_{A^{x},\Lambda^{x}} ).
\end{equation}

In all situations we shall consider, the index of this operator will be trivial:
\begin{lem} \label{lem:orientation_line_trivial}
 If $x$ corresponds to a geodesic which is shorter than the injectivity radius, the index of $D_{A^{x},\Lambda^{x}}$ vanishes, and $\ro_{x}$ admits a canonical trivialisation.
\end{lem}
\begin{proof}
The trivialisation of $D_{A^{x}, \Lambda^{x}}$ will be obtained by deforming $ \det(D_{A^{x}, \Lambda^{x}} ) $, through Fredholm operators, to an operator which is regular, and admits only $0$ as a solution. Start by moving $\vq[1]$ along the geodesic to $\vq[0]$; by Exercise \ref{ex:unique_chord_non-deg}, there is a unique non-degenerate Hamiltonian chord starting at $\Tq[0]$, and ending at the cotangent fibre of every point along the homotopy. There is, therefore, an associated family of Fredholm operators, and hence an induced isomorphism of determinant lines. Next, we deform the metric near $\vq[0]$ to a flat metric; for points in a sufficiently small neighbourhood, there is again a unique non-degenerate geodesic to $\vq[0]$ for all elements of this homotopy.  

Introduce the gauge transformation:
\begin{equation}
  X \mapsto A_{t}^{x} X,
\end{equation}
which maps Equation \eqref{eq:linearisation_Floer} to an operator that can be deformed to the standard Cauchy-Riemann equation with boundary conditions
\begin{equation}
  \tilde{\Lambda}_{t}^{x} \equiv A_{t}^{x} \Lambda_{t}^{x}.
\end{equation}

Having assumed that that the metric is flat near $\vq[0]$, we can identify a neighbourhood of $\Tq[0]$ with $\bC^{n}$ as in Equation \eqref{eq:map_cotangent_to_complex}, and explicitly compute the matrix $A_t^{x}$:
\begin{equation*}
  \begin{pmatrix}
    1 & -t & 0 & 0 &  \cdots & 0 & 0  &0& 0 \\
0 & 1 & 0 & 0 &  \cdots & 0 & 0 & 0 &0 \\
   0 & 0 &  1 & -t & \cdots & 0 & 0  &0& 0 \\
0 & 0 & 0 & 1 &  \cdots & 0 & 0 & 0 &0 \\
\vdots & \vdots & \vdots & \vdots &  \ddots &  \vdots &\vdots & \vdots &\vdots \\
0 & 0 & 0 & 0 &   \cdots & 1 & -t & 0 & 0 \\
0 & 0 & 0 & 0 &   \cdots & 0 & 1 & 0 & 0 \\
0 & 0 & 0 & 0 & \cdots & 0 & 0 & 1 & -t \\
0 & 0 & 0 & 0 & \cdots & 0 & 0  & 0 &1
  \end{pmatrix}= \exp  \begin{pmatrix}
    0 & -t & 0 & 0 &  \cdots & 0 & 0  &0& 0 \\
0 & 0 & 0 & 0 &  \cdots & 0 & 0 & 0 &0 \\
   0 & 0 &  0 & -t & \cdots & 0 & 0  &0& 0 \\
0 & 0 & 0 & 0 &  \cdots & 0 & 0 & 0 &0 \\
\vdots & \vdots & \vdots & \vdots &  \ddots &  \vdots &\vdots & \vdots &\vdots \\
0 & 0 & 0 & 0 &   \cdots & 0 & -t & 0 & 0 \\
0 & 0 & 0 & 0 &   \cdots & 0 & 0 & 0 & 0 \\
0 & 0 & 0 & 0 & \cdots & 0 & 0 & 0 & -t \\
0 & 0 & 0 & 0 & \cdots & 0 & 0  & 0 & 0  \end{pmatrix}.
\end{equation*}
A straightforward computation shows that the assumptions of Corollary \ref{cor:determinant_line} hold, hence we have a trivialisation of $\det(D_{\tilde{\Lambda}^{x}})$.  Deforming the Cauchy-Riemann operator, we obtain the desired trivialisation of $\det(D_{A^{x}, \Lambda^{x}})$.
\end{proof}

We can also associate to $x$ the operator $  D_{A^{x},\Lambda^{x}}^{+}$ which corresponds to equipping the half-plane with a positive strip-like end. We define
\begin{equation}
      \ro^{+}_{x}  \equiv \det(D_{A^{x},\Lambda^{x}}^{+}  ).
\end{equation}
The assumption of Corollary \ref{cor:determinant_line} holds, hence we have:
\begin{lem} \label{lem:iso_negative_orientation_line}
  There is a canonical isomorphism
  \begin{equation}
    \ro^{+}_{x}  \cong |\Tq[0]|.
  \end{equation} \qed
\end{lem}

\subsection{Integrated maximum principle for Riemann surfaces with boundary} \label{sec:integr-maxim-princ}
 The usual maximum principle asserts that pseudo-holomorphic curves cannot escape the sublevel sets of the radial function $\rho$. In most applications, one considers punctured Riemann surfaces with asymptotic conditions that are Hamiltonian chords (or orbits), which are contained within such a level set, and the maximum principle can be used as long as Neumann conditions are satisfied at the boundary of the Riemann surface. However, the constructions we shall later present require a stronger version of the maximum principle: we shall show that, for the appropriate pseudoholomorphic curve equation, there are no solutions some of whose asymptotic conditions lie outside a level set of the radial function. At the heart of the argument for excluding the existence of such solutions are the pointwise estimates that usually enter in the proof of the maximum principle; the main difference is that we shall integrate such quantities over the source Riemann surface, we therefore call this general class of results \emph{integrated maximum principle.} 

Let $\Sigmabar$ be a compact Riemann surface with boundary, $\Sigma$ the complement of finitely many points $\{ z_{i} \}_{i=1}^{e}$ on the boundary of $\Sigma$. Fix a decomposition of the boundary into manifolds with boundary
\begin{equation}
  \partial \Sigma = \partial^{n} \Sigma \cup  \partial^{l} \Sigma,
\end{equation}
such that $\partial^{n} \Sigma \cap \partial^{l} \Sigma  $ consists of finitely many points, and the ends of $\Sigma$ are disjoint from $\partial^{n} \Sigma$. The subscripts refer to \emph{normal} and \emph{Lagrangian} boundary conditions.
\begin{rem}
 In applications, this decomposition of the boundary will appear naturally from the following construction: start with a pseudo-holomorphic map into $\TQ$, with Lagrangian boundary conditions, and define $\Sigma$ to be the inverse image of the complement of the interior of $\DQ[\ell]$. The normal boundary conditions will correspond to the inverse image of $\SQ[\ell]$, while the remaining boundary maps to a collection of Lagrangians. We shall eventually use the results of this section to prove that, if the Floer data and the boundary conditions are appropriately chosen, the image of  pseudo-holomorphic maps with target $\TQ$ must lie in the interior of $ \DQ[\ell] $ by showing that the inverse image of the complement of this disc bundle is empty.
\end{rem}
At each puncture $z_i$, we choose negative strip-like ends
\begin{align}
 \epsilon_{i} \co \Strip^{-} & \to \Sigma.
\end{align}

On the space of maps
\begin{equation}
  u \co \Sigma \to \TQ \setminus \DQ[\ell],
\end{equation}
we shall consider a pseudoholomorphic equation 
\begin{equation} \label{eq:CR-equation-general}
\left(du - \sum_{j=0}^{k} \alpha_j \otimes X_{h_j}   \right)^{0,1} \equiv 0
\end{equation}
determined by a $k$-tuple of  $1$-forms $\alpha_j$ on $\Sigma$, and functions $h_j$ which depend only on $\rho^2$ (alternatively, functions on $[\ell,+\infty)$), together with a family $J_z$ of almost complex structures on $\TQ$. The Floer data are required to satisfy the following conditions:
\begin{align} \label{eq:alpha_restricted_to_end}
& \epsilon_i^{*}(\alpha_j)  = dt \textrm{ for all } 1 \leq i \leq e \textrm{ and } 0 \leq j \leq k.\\ \label{eq:alpha_vanishes_on_lag}
& \alpha_j| \partial^{l} \Sigma =  0  \textrm{ for all } 0 \leq j \leq k.\\
& J_z \textrm{ is convex near } \SQ[\ell]. \\ \label{eq:convexity}
& \textrm{Each function } h_j \textrm{ is non-decreasing, and convex.} \\ \label{eq:positive_function_negative_form}
&  d \alpha_j \leq 0 \textrm{ for all } 0 \leq j \leq k.
\end{align}

We require that $u$ satisfy normal boundary conditions:
\begin{equation}\label{eq:boundary_conditions_normal}
  u \textrm{ is transverse to } \SQ[\ell]  \textrm{, and } u^{-1}( \SQ[\ell]   ) =  \partial^{n} \Sigma.
\end{equation}
To impose Lagrangian boundary conditions, choose, for each component $ I \subset \partial^{l} \Sigma  $ a point $q_{I} \in \Q$. The condition along $\partial^{l} \Sigma$ is:
\begin{equation} \label{eq:boundary_conditions_maximum}
   u(I) \subset \Tq[I]
\end{equation}

There are two notions of energy associated to a map from $\Sigma$ to $\TQ$; the first, which we call \emph{geometric energy} should be thought of as the $L^2$ energy, corrected by the Hamiltonian flow:
\begin{equation}
  E_{geo}(u) \equiv \frac{1}{2} \int_{\Sigma} || du -   \sum_{j=0}^{k} \alpha_j \otimes X_{h_j  } ||^{2}.
\end{equation}
The metric we use on $\TQ$ is the one induced by the symplectic form and the $\Sigma$-dependent family of almost complex structures. We begin with an elementary consequence of the definition:
\begin{lem} \label{lem:vanishing_energy}
  If $E_{geo}(u) $ vanishes, the image of $u$ is contained in a (not-necessarily closed) orbit of $X_{\rho^2}$.
\end{lem}
\begin{proof}
By assumption, $ du -   \sum_{j=0}^{k} \alpha_j \otimes X_{h_j  }  $ vanishes at every point, hence the image of $du$ is contained in the subspace spanned by the vector fields $X_{h_j  } $. Since $h_j$ is a function of $\rho^2$, we conclude that the image of $du$ is contained in the line spanned by $X_{\rho^2}$.
 \end{proof}
We now analyse the consequence of finiteness of energy. 
\begin{lem}
  If $E_{geo}(u)$ is finite, the limit
\begin{equation}
  \lim_{s \to -\infty} u(\epsilon_i(s,t))
\end{equation} 
is well-defined, and converges to a time-$1$ chord $x_i$ of the Hamiltonian
\begin{equation}
 h   \equiv \sum_{j=1}^{k} h_j 
\end{equation}
with endpoints on cotangent fibres.
\end{lem}
\begin{proof}
Equation \eqref{eq:alpha_restricted_to_end} implies that composition of $u$ with $\epsilon_i$ satisfies the pseudoholomorphic curve equation
\begin{equation}
  \left( d(u \circ \epsilon_i) - dt \otimes X_{ h } \right)^{0,1} = 0,
\end{equation}
with respect to a constant almost complex structure. Moreover, since $\partial^{n} \Sigma$ is assumed to be disjoint from the punctures, the boundary of the strip is mapped by $u \circ \epsilon$ to two cotangent fibres. Finiteness of energy for strips with Lagrangian boundary conditions implies convergence to Hamiltonian chords (see, e.g. \cite[Theorem 2]{Floer-gradient} for the case of strips), which implies the desired result.
\end{proof}

The second notion of energy, defined without reference to a metric, is the topological energy:
\begin{equation}
   E_{top}(u) \equiv \int_{\Sigma} u^{*}(\omega) - \sum_{j=0}^{k}  \int_{\Sigma} d \left(u^{*}(h_j  ) \alpha_j   \right).
\end{equation}
\begin{lem}
The geometric and topological energy are related by the following equation
\begin{equation} \label{eq:difference_geo_top}
  E_{geo} (u)= E_{top}(u)  +   \sum_{j=0}^{k} \int_{\Sigma}  u^{*}(h_j  ) d \alpha_j.
\end{equation}
\end{lem}
\begin{proof}
The proof is a straightforward computation which we perform using local holomorphic coordinates $(s,t)$:
\begin{align}
E_{geo}(u) & =\frac{1}{2} \int \omega \left( \partial_{s}u -  \sum_{j=0}^{k}  \alpha_j(\partial_{s})  X_{h_j  } , J (\partial_{s}u -   \sum_{j=0}^{k}  \alpha_j(\partial_{s})  X_{h_j  } ) \right) ds \wedge dt \\
& \qquad +\frac{1}{2}  \int \omega \left( \partial_{t}u -  \sum_{j=0}^{k}  \alpha_j(\partial_{t})  X_{h_j  }, J \left(\partial_{t}u -  \sum_{j=0}^{k}   \alpha_j(\partial_{t})  X_{h_j  }  \right) \right) ds \wedge dt. 
  \end{align}
Equation \eqref{eq:CR-equation-general} implies that the two terms above are equal, and allows us to rewrite the geometric energy as follows:
\begin{align}
& \int \omega \left( \partial_{s}u - \sum_{j=0}^{k}    \alpha_j(\partial_{s}) X_{h_j  }, \partial_{t}u -   \sum_{j=0}^{k}   \alpha_j(\partial_{t})  X_{h_j  }  \right) ds \wedge dt \\
& =  \int \omega(\partial_{s}u, \partial_{t}u) ds \wedge dt -   \sum_{j=0}^{k} \int \omega \left(     \alpha_j(\partial_{s})  X_{h_j  } , \partial_{t}u \right) ds \wedge dt  \\ \notag
& \qquad -  \sum_{j=0}^{k}  \int \omega  \left( \partial_{s}u ,   \alpha_j(\partial_{t})  X_{h_j  }  ) \right) ds \wedge dt \\
& =   \int u^{*}(\omega) +  \sum_{j=0}^{k}  \int  \alpha_j(\partial_{s}) \cdot d (h_j  )(\partial_{t} u) ds \wedge dt   - \sum_{j=0}^{k} \int   \alpha_j(\partial_{t}) \cdot d (h_j  )(\partial_{s} u) ds \wedge dt  \\
& =   \int u^{*}(\omega) - \sum_{j=0}^{k} \int d(   u^{*}(  h_j  )) \wedge \alpha_j  .
  \end{align}
To arrive at Equation \eqref{eq:difference_geo_top}, it remains to use the fact that $d$ is a derivation:
\begin{equation}
d (  u^{*}(h_j  ) \wedge \alpha_j  ) =  d (  u^{*}(h_j  ))  \wedge \alpha_j +  u^{*}(h_i  )   d \alpha_j.
\end{equation}
\end{proof}
\begin{cor} \label{cor:inequality_from_monotonicity}
If Inequality \eqref{eq:positive_function_negative_form} is satisfied, then
\begin{equation} \label{eq:inequality_topological}
  0 \leq E_{top}(u)  +  \sum_{j=0}^{k} h_j(\ell) \int  d \alpha_j
\end{equation}
with equality if and only if the image of $u$ is contained in a level set of $\rho$.
\end{cor}
\begin{proof}
Since each function $h_j$ is non-decreasing,
\begin{equation}
  \int_{ \Sigma} u^{*} (h_{j})  d \alpha_j \leq   h_{j}(\ell) \int_{ \Sigma}  d \alpha_j.
\end{equation}
Applying this in Equation \eqref{eq:difference_geo_top}, we conclude that
\begin{equation}
    E_{geo}(u) \leq E_{top}(u)  +   \sum_{j=0}^{k} h_j(\ell) \int  d \alpha_j.
\end{equation}
By definition, the geometric energy is non-negative, and vanishes only if the image of $du$ at every point is parallel to a linear combination of the Hamiltonian flows $\{ X_{h_j  }\}_{j=0}^{k} $. Since these Hamiltonian vector fields are parallel to a level set of $\rho$, we conclude that $\rho \circ u$ is constant in this case.
\end{proof}

The notation $E_{top}$ is justified by the following result, which implies that $E_{top}$ is invariant by homotopies with the property that the restriction to  $\partial^{n} \Sigma $ is fixed. Recall that $x_i$ are the asymptotic limits of the map $u$ along the ends.
\begin{lem}
 The topological energy of $u$ is given by:
 \begin{equation} \label{eq:Stokes_topological}
   E_{top}(u) = \int_{\partial^{n} \Sigma} u^{*}(\lambda) -  \sum_{j=0}^{k}  \int_{\partial^{n} \Sigma } u^{*}(h_j  ) \alpha_j    + \sum_{i=1}^{e}   \Action(x_i).
 \end{equation}
\end{lem}
\begin{proof}
This is a direct application of Stokes's theorem; the only points that warrant explanation are the vanishing of the contribution of $ \partial^{l} \Sigma$, and the sign of the contribution of the ends. The integral of the primitive vanishes on $\partial^{l} \Sigma $ because of Exercise \ref{ex:cotangent_fibre_exact} and Equation \eqref{eq:alpha_vanishes_on_lag}. To justify the sign on the action, recall that all the ends are negative, and that the vector field $\partial_t$ therefore points clockwise along the boundary of $\Sigma$, which is opposite to the natural orientation. Having defined the action in Equation \eqref{eq:action_chord}, we see that the sign is indeed positive.
\end{proof}

Our next goal is to find an upper bound for $E_{top}(u)$, allowing us to contradict Equation \eqref{eq:inequality_topological}. The first step is to note that the natural orientation of $\partial^{n} \Sigma$ (e.g. when applying Stokes's theorem), is such that
\begin{equation}
  \parbox{35em}{ $j \xi$ is an inward pointing vector on $\Sigma$ whenever $\xi$ is positive along the boundary.}
\end{equation}
In particular, since $\rho \circ u$ reaches its minimum on $\partial^{n} \Sigma$, and $u$ is transverse to the corresponding level set of $\rho$,
\begin{equation}
 d \rho \circ du ( j \xi) > 0.
\end{equation}
Using  the convexity of $J_{z}$ and the fact that $d \rho (X_{\rho}) = 0$, we conclude that 
\begin{equation} \label{eq:positivity_along_boundary}
 \lambda \circ J_{z} \circ \left( du (j \xi)  -  \sum_{j=0}^{k} \alpha_j(j \xi) X_{h_j  } \right) > 0.
\end{equation}
Next, we rewrite the pseudo-holomorphic curve equation satisfied by $u$ as:
\begin{equation}
  J_{z} \left( du (j \xi)  -  \sum_{j=0}^{k} \alpha_j(j \xi) X_{h_j  }  \right) = - \left( du (\xi)  -  \sum_{j=0}^{k} \alpha_j(\xi) X_{h_j  } \right).
\end{equation}
In particular, Equation \eqref{eq:positivity_along_boundary} implies that
\begin{equation}
  \lambda \left( du(\xi) -  \sum_{j=0}^{k} \alpha_j(\xi) X_{h_j  } \right) < 0.
\end{equation}
Integrating over $\partial^{n} \Sigma$, and computing that $  \lambda(X_{h_j  }   )  = \rho h'_j  $, we conclude that
\begin{equation}
  \int_{\partial^{n} \Sigma} u^{*}(\lambda) -   \sum_{j=0}^{k} \ell h_j'(\ell) \alpha_j   < 0.
\end{equation}
Comparing this with the right hand side of Equation \eqref{eq:Stokes_topological}, we arrive at the following result:
\begin{lem} \label{lem:inequality_from_convexity}
If $J_z$ is convex,  
\begin{equation} \label{eq:top_energy_bounded_above}
E_{top}(u)  -   \sum_{j=0}^{k}  \int_{\partial^{n} \Sigma} (\ell h_j'(\ell)   -   h_j(\ell) ) \alpha_j   - \sum_{i=1}^{e}   \Action(x_i) < 0
 \end{equation} \qed
\end{lem}

We shall now prove that this result contradicts Corollary \ref{cor:inequality_from_monotonicity}:
\begin{prop} \label{prop:no_finite_energy_outside_cpct}
  There is no finite energy map satisfying Equation \eqref{eq:CR-equation-general} with boundary conditions given by Equations \eqref{eq:boundary_conditions_normal} and \eqref{eq:boundary_conditions_maximum}, if the Floer data satisfy Equations \eqref{eq:alpha_restricted_to_end}-\eqref{eq:positive_function_negative_form}.
\end{prop}
\begin{proof}
We use Stokes' theorem to compute that
\begin{equation}
  \int_{\partial^{n} \Sigma} \alpha_j =   \int_{\Sigma}d \alpha_j + e,
\end{equation}
where $e$ is the number of negative punctures. Applying this to Equation \eqref{eq:top_energy_bounded_above}, we see that
\begin{equation} \label{eq:partial_step_inequality}
  E_{top}(u)  +  \sum_{j=0}^{k} h_j(\ell) \int_{ \Sigma}  d \alpha_j  
- \sum_{j=0}^{k} \ell h_j'(\ell) \int_{ \Sigma}  d \alpha_j  - \sum_{i=1}^{e}  \left(  \Action(x_i)  -  h(\ell) + \ell h'(\ell)\right) < 0
\end{equation}
By Equation \eqref{eq:action_x_y_intercept}, the action of $x_i$ is given by
\begin{equation}
 \Action(x_i) =  h(\ell_i) - \rho_i h'(\ell_i).
\end{equation}
The fact that $h$ is convex therefore yields the inequality:
\begin{equation} 
 \Action(x_i)  \leq   h(\ell) - \ell h'(\ell).
\end{equation}
Applying this inequality together with the fact that $h_j$ is increasing and that $\alpha_j$ is subclosed, we can remove the last two terms in Equation \eqref{eq:partial_step_inequality}, and conclude:
\begin{equation}
  E_{top}(u)  +   \sum_{j=0}^{k} h_j(\ell) \int  d \alpha_j < 0.
\end{equation}
This directly contradicts Equation \eqref{eq:inequality_topological}, which implies that no map can satisfy these hypotheses.
\end{proof}

\section{From Morse homology to Floer cohomology} \label{sec:from-morse-homology}

The key step in defining $\Fam$ is the construction of a map
\begin{equation}
  \Fam^{r} \co H_{*}(\sL^{r} \Q; \eta) \to HF^{-*}(H^{-}; \bZ)
\end{equation}
whenever the Hamiltonian $H^{-}$ has slope larger than $2r$. By  taking the direct limit over inclusion maps on one side, and continuation  maps on the other, we shall produce the desired map from loop homology to symplectic cohomology.

\begin{rem}
A careful consideration of signs would show that, with the conventions on orientations that we use, the image of the fundamental class $e$ in $ H_{*}( \sL^{r} \Q; \eta) $ under $\Fam^{r}$ is $(-1)^{n(n-1)/2}e$; we shall not prove this directly, but it follows from the results of Section \ref{sec:orient_annuli}. In particular, if we wanted the induced map on direct limits to preserve the $BV$ structure, we would need to change some of the conventions for orienting moduli spaces that are used in this section. These changes would be minor, since they would be uniform over the moduli space, and would not depend on the homotopy class of the curve.
\end{rem}

\subsection{Moduli space of punctured discs with cotangent boundary conditions} \label{sec:moduli-space-half}

The construction of $\Fam$ will involve moduli spaces of punctured discs with Lagrangian boundary conditions which are analogous to the moduli spaces considered in Section \ref{sec:punctured discs-with}. The main differences are that we shall work with discs with an interior puncture carrying a negative end, and multiple boundary  punctures with positive ends; our Lagrangian boundary conditions will consist of several cotangent fibres.

To this end, let $Z^{-}$ denote the punctured disc which we now identify with the negative half of the cylinder:
\begin{equation}
  Z^{-} \equiv (-\infty,0] \times S^1.
\end{equation}

Given a point $\xi$ on the boundary of $Z^{-}$, we define a positive strip-like end at such a point to be a map
\begin{align}
\epsilon_{\xi} \co \Strip^{+} & \to Z^{-} \\ 
\lim_{s \to +\infty} \epsilon_{\xi}(s,t) & = \xi 
\end{align}
 which is a biholomorphism onto a punctured neighbourhood of $\xi$. 

\begin{figure}[h]
  \centering
  \includegraphics{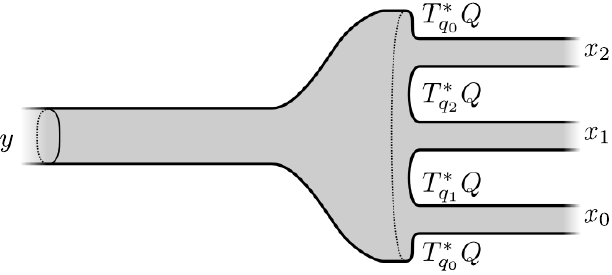}
  \caption{ }
  \label{fig:negative_cylinder_punctured}
\end{figure}

Fix a positive integer $r$,  and denote by $Z^{-}_{r}$ the complement of $r$ uniformly distributed points on the boundary:
\begin{equation}
 Z^{-}_{r} \equiv Z^{-}\setminus \{ (0, \frac{i}{r}) \}_{i=1}^{r}.
\end{equation}
We fix positive strip-like ends $\{ \epsilon_{i} \}_{i=1}^{r}$ at the points $(0, \frac{i}{r})$ whose images are disjoint; see Figure  \ref{fig:negative_cylinder_punctured} for a representation of $Z^{-}_{r}$ with a metric whose restriction to the ends is an isometry.

Given families of Hamiltonians $H_{z}$ and of almost complex structures $J^{-}_{z}$ on $\TQ$, parametrised by $z \in Z^{-}_{r}$, and a $1$-form $\alpha^{-}$ on $Z^{-}_{r}$, we consider the pseudoholomorphic curve equation
\begin{equation} \label{eq:CR-equation-negative}
  \left(du - \alpha^{-} \otimes X_{H} \right)^{0,1} = 0
\end{equation}
on maps from $ Z^{-}_{r} $ to $\TQ$. To fix boundary conditions, we choose a point $\vq \in \sL^{r} \Q$, and require that
\begin{equation}
  u\left(  0 \times \left[\frac{i-1}{r}, \frac{i}{r} \right]  \right)\subset \Tq[i].
\end{equation}

In order for the space of such solutions to be well behaved, we must impose constraints on the Floer data. To state these constraints,  first recall the choice of constant $\delta_{i}^{r}$ which bounds the length of the $i$\th segment of piecewise geodesics lying in  $\sL^{r} \Q$ (see Definition \ref{defin:finite_approx}), and fix constants $b_i$ such that
\begin{equation} \label{eq:b_i_negative_cylinder}
  \delta_{i}^{r} < 2 b_i < 2.
\end{equation}

\begin{exercise} \label{ex:unique_chord_b_i}
By adapting the argument of Corollary \ref{cor:unique_chords} to this situation,  show that there is a unique Hamiltonian chord $x_i$ of $b_{i} \Hh$ with endpoints on $\Tq[i]$ and $\Tq[i+1]$ if $\vq \in \sL^{r} \Q$.
\end{exercise}

Let $b$ denote the sum of the slopes $b_i$. We first impose conditions on the $1$-form:
\begin{equation} \label{eq:condition_on_alpha-}
  \parbox{35em}{ $\alpha^{-}$ is closed. Its restriction to a subset of $Z^{-}_{r}$ with $s \leq -1$ agrees with $bdt$, the restriction to $\partial Z^{-}_{r} $ vanishes,  and the pullback of $\alpha^{-}$ under every strip-like end $\epsilon_i$ agrees with $b_{i}dt$.}
\end{equation}
In addition, we choose a constant $b^{-}$ which is sufficiently large that the following inequality is satisfied:
  \begin{equation} \label{eq:slope_condition_negative}
    2 b < b^{-}.
  \end{equation}
We use the constants $b_{i}$ and $b^{-}$ in our choice of family of Hamiltonians:
\begin{enumerate}
\item The restriction of $H$ to a neighbourhood of the boundary agrees with the model Hamiltonian:
  \begin{equation} \label{eq:restriction_H_negative_s_close_to_boundary}
 H_{(s,t)}  = \Hh  \textrm{ if } -1 \leq s.
  \end{equation}
\item There exists a Hamiltonian $H_{t}^{-}$ of slope $b^{-}$ all of whose orbits are non-degenerate, such that
  \begin{equation}\label{eq:restriction_H_negative_s_close_to_infinity}
    H_{(s,t)} = \frac{H_{t}^{-}}{b} \textrm{ if } s \ll 0.
  \end{equation}
\item There exists a function $f^{-}$ on $\Sigma$, such that
  \begin{equation} \label{eq:Hamiltonian_negative_outside_disc}
    H_{z}| \TQ \setminus \DQ[2] \equiv  2 \rho -  2 + f^{-}\left( \frac{b^{-}}{b} \rho   - 2   \rho + 2 \right).
  \end{equation}
Moreover, we assume that 
\begin{equation} \label{eq:assumption_cutoff_negative}
  \begin{aligned}
\frac{\partial f^{-}}{\partial s} & \leq 0 \\
f^{-}(s,t) & = 0 \textrm{ if } -1 \leq s \\
f^{-}(s,t) & = 1 \textrm{ if } s \ll 0.
\end{aligned}
\end{equation}
\end{enumerate}
\begin{exercise}
Check that Equations \eqref{eq:restriction_H_negative_s_close_to_boundary} and \eqref{eq:restriction_H_negative_s_close_to_infinity} are compatible with Equations \eqref{eq:Hamiltonian_negative_outside_disc} and \eqref{eq:assumption_cutoff_negative}.
\end{exercise}
\begin{exercise}
Show that, for any solution $u$ with finite energy, we have
  \begin{equation}
    \lim_{s \to  - \infty} u(s,t) = x(t)
  \end{equation}
for some orbit $x \in \Orbit(H^{-})$, and that
  \begin{equation}
    \lim_{s \to  +\infty} u \circ \epsilon_{i}(s,t) = x_{i}(t)
  \end{equation}
where $x_{i}$ is the unique element of $\Chord_{b_i \Hh }  (\Tq[i], \Tq[i+1])$ discussed in Exercise \ref{ex:unique_chord_b_i}.
\end{exercise}
We write
\begin{equation}
  \Cyl(x,\vq)
\end{equation}
for the set of solutions to Equation \eqref{eq:CR-equation-negative} with these asymptotic conditions, with respect to a family of almost complex structures which are convex near $\SQ[2]$.

\begin{lem} \label{lem:compactness_negative_cylinders}
Under the above assumptions, all elements of $ \Cyl(x,\vq)$  have image contained in $\DQ[2]$.
\end{lem}
\begin{proof}
This is an application of the integrated maximum principle. Let $u$ be an element of $\Cyl(x,\vq)$, and assume by contradiction that its image is not contained in $\DQ[2]$. In particular, the inverse image of $\DQ[\ell]$ is non-empty for any value of $\ell$ sufficiently close to $2$. We choose $\ell$ greater than $2$ so that $u$ is transverse to $\SQ[\ell]$, and the inverse image $\Sigma $ is non-empty.

By Equation \eqref{eq:Hamiltonian_negative_outside_disc}, the restriction of Equation \eqref{eq:CR-equation-negative} to $\Sigma$ is
\begin{equation}
    \left(du - (2 + ( \frac{b^{-}}{b}- 2)f^{-})  \alpha^{-} \otimes X_{\rho}  \right)^{0,1} = 0,
\end{equation}
where the constant terms in Equation \eqref{eq:Hamiltonian_negative_outside_disc} disappear because the associated Hamiltonian flows vanish. To verify the assumptions of Proposition \ref{prop:no_finite_energy_outside_cpct}, note that $\frac{b^{-}}{b}- 2$ is positive by Equation \eqref{eq:slope_condition_negative}, so it suffices to show that $d  \left( f^{-} \alpha^{-} \right) \leq 0$. The fact that $\alpha^{-}$ is closed and that $f^{-}$ is constant in the region $-1 \leq s$ implies that we need only  to prove this inequality in the region where $\alpha^{-}$ is a constant multiple of $dt$
\begin{equation}
  d  \left( f^{-} \alpha^{-} \right)(s,t) = \frac{b \partial f^{-}}{\partial s}(s,t) ds \wedge dt \textrm{ if } s \leq -1.
\end{equation}
By assumption $f^{-}$ decreases with $s$, so this $2$-form is indeed non-positive. We conclude from Proposition \ref{prop:no_finite_energy_outside_cpct} that the surface $\Sigma$ is empty. Since the constant $\ell$ can be chosen arbitrarily close to $2$,  the image of $ u$ is therefore contained in the disc bundle $\DQ[2]$.
\end{proof}

\subsection{Orienting the moduli space of negative punctured discs} \label{sec:orient-moduli-space-negative}
Our goal in this section is to prove the analogue of Corollary \ref{lem:vir_dim_moduli_half_cylinders}; i.e. provide a canonical orientation of the moduli space $ \Cyl(x,\vq)$, relative to the orientation line of $y$, and the local system $\eta$.

We start with the linearisation $D_{u}$ of Equation \eqref{eq:CR-equation-negative} at a map $u$ with respect to the canonical trivialisation of $u^{*}(T \TQ)$ arising from Lemma \ref{lem:trivalisation_up_to_htpy}. The pullback of this operator under $\epsilon_{i}$ is given by
\begin{equation} \label{eq:linearisation_near_strip_like_end}
  X \mapsto  \partial_{s} X + I \left( \partial_{t} X  - B_{\Hh}^{i} \cdot X \right)
\end{equation}
where $B_{\Hh}^{i}$ is the ``derivative'' of the family of symplectomorphisms associated to $b_{i}\Hh$ (see Section \ref{sec:orient-lines-hamilt} for a precise definition). 

\begin{figure}[h]
  \centering
  \includegraphics{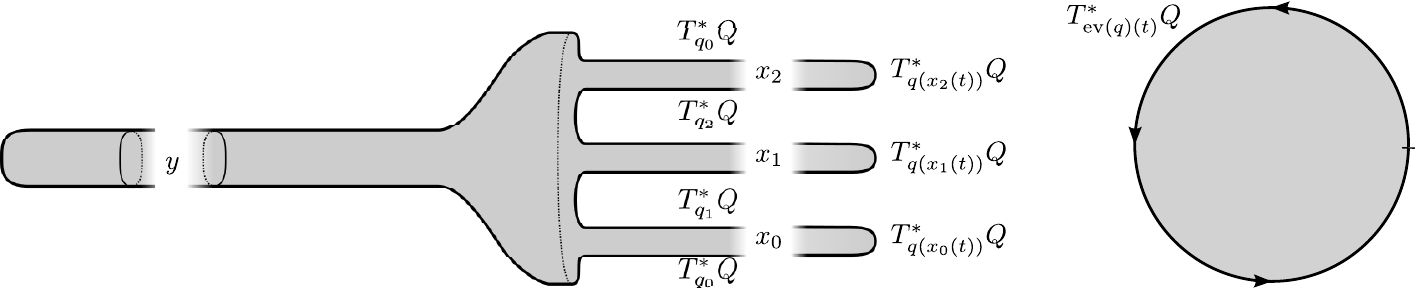}
  \caption{ }
  \label{fig:negative_half_cylinder_glue}
\end{figure}

Associated to each chord $x_{i}$, we have an operator $D_{x_{i}}$ on the upper half plane, with orientation line $\ro_{x_i}$ defined in Equation \eqref{eq:orientation_line_chord}. By gluing $ D_{u}$ at each boundary puncture to the corresponding operator $D_{x_{i}}$, and to the operator $D_{\Psi_{x}}^{+}$ at the interior puncture, we obtain an operator on the disc with Lagrangian boundary conditions (see Figure \ref{fig:negative_half_cylinder_glue}).

\begin{exercise}
Using the fact that the boundary conditions for $D_{x_{i}} $ are given by Equation \eqref{eq:Lagr_boundary_D_x},  show that there is a homotopy, canonical up to contractible choice, from the boundary conditions of the operator obtained by gluing $ D_{u}$,  $D_{\Psi_{x}}^{+}$ and the operators $D_{x_{i}}$, to the image of
\begin{equation}
  T T^{*}_{\ev(\vq)(t)} \Q
\end{equation}
under the trivialisation.
\end{exercise}
Under the natural complex trivialisation of $ x^{*}(T \TQ)$,  the boundary condition of $  \det( D_{\vq}) $ is therefore given by applying the complex structure $I$ to the image of $ T_{\ev(q)(t)} \Q $ under the trivialisation associated to $x$.  We denote by $I \Lambda_{\ev(\vq)}  $ this loop of Lagrangians in $\bC^{n}$.  Lemma \ref{lem:orientation_line_trivial} implies that this gluing induces an isomorphism
\begin{equation} \label{eq:linearize_d_u_glue_half-planes} 
 \det(D_{\Psi_{x}}^{+}) \otimes  \det(D_u) \cong \det( D_{I \Lambda_{\ev(\vq)}} )
\end{equation}
which is canonical up to homotopy.
\begin{exercise} \label{ex:deform_zero-section-to-cotangent}
Let $\Lambda$ be a loop of Lagrangians in $\bC^{n}$. Show that $I \Lambda $ is homotopic to $\Lambda$ ( Hint: split $\bC^{n}$ at every point as $\Lambda_{t} \oplus I \Lambda_{t} $ and show that the diagonal is a loop of Lagrangians which is homotopic to both.)
\end{exercise}

The analogue of  Lemma \ref{lem:computation_index_disc_Lagrangian_boundary} is:
\begin{lem} 
  There is a canonical isomorphism of graded lines
  \begin{equation} \label{eq:orientation_disc_clockwise}
   | \det(D_{I   \Lambda_{\ev(\vq)}})| \cong \eta_{\ev(\vq)}[w(\ev(\vq))].
  \end{equation}
\end{lem}
\begin{proof}
 We start by using Exercise \ref{ex:deform_zero-section-to-cotangent} to identify $  \det(D_{I   \Lambda_{\ev(\vq)}}) \cong   \det(D_{\Lambda_{\ev(\vq)}})$. Next,  note that the image of $ \Lambda_{\ev(\vq)} $  under the Gauss map  for the trivialisation we are using is $w(\ev(\vq))$, while it is $-w(\ev(\vq))$ for the boundary conditions $  \Lambda_{\ev(\vq)} $ considered in Lemma \ref{lem:computation_index_disc_Lagrangian_boundary}.   The sign difference in the contribution of $w(\ev(\vq))$ is due to the fact that the loop of Lagrangians is oriented in opposite ways when the orbit is an input and when it is an output. Compare Figures \ref{fig:gluing_disc} and \ref{fig:negative_half_cylinder_glue}.

Applying Lemma \ref{lem:inverse_path_det_line_is_inverse}, we have a canonical isomorphism
\begin{equation} \label{eq:split-off-sphere}
  \det(D_{\Lambda_{\ev(\vq)}}) \otimes \det_{\bR}(\bC^{n})  \cong   \det_{\bR}(D^{-}_{\Phi^{-1}})\otimes  \det(D_{\Lambda_{\ev(\vq)}^{-1}}) 
\end{equation}
where $\Phi^{-1}$ is a loop of unitary matrices such that $\mu(\Phi^{-1}) = 1 $. Using the complex orientations of $ \det_{\bR}(D^{-}_{\Phi^{-1}}) $  and $ \bC^{n} $, together with Lemma \ref{lem:computation_index_disc_Lagrangian_boundary}, we obtain the desired isomorphism in Equation \eqref{eq:orientation_disc_clockwise}.
\end{proof}

We now have the necessary ingredients to orient the moduli space $  \Cyl(x,\vq)  $. 
\begin{lem} \label{lem:virtual_dim_negative_half_cylinder}
  The virtual dimension of $ \Cyl(x,\vq) $ is $\deg(x) -n$. Moreover, for every element of this moduli space, there is a canonical isomorphism
\begin{equation} \label{eq:iso_det_x_kappa_ev_q}
  |\det(D_{u})|  \cong   \ro_{x} \otimes \eta_{\vq} [w(x)].
\end{equation}
\end{lem}
\begin{proof}
Substituting Equation \eqref{eq:orientation_disc_clockwise} into Equation \eqref{eq:linearize_d_u_glue_half-planes}, we obtain an isomorphism of graded lines
  \begin{equation} \label{eq:gluing_interior_node_Fam}
   \ro^{+}_{x} \otimes | \det(D_u)| \cong \eta_{\vq}[w(x)],
  \end{equation}
where we moreover use the fact that $w(x) = w(\ev(\vq))$. The computation of the virtual dimension follows immediately by computing the degree of the two sides:
\begin{align}
  \ind(D_u) + 2n - |x| & = n - w(x) \\ \label{eq:index_d_u_Fam}
\ind(D_u) & = \deg(x) -n.
\end{align}
The isomorphism in Equation \eqref{eq:iso_det_x_kappa_ev_q} is obtained from Equation \eqref{eq:gluing_interior_node_Fam} by using the isomorphism in Equation \eqref{eq:dual_orientation_lines_orbit}, and the standard orientation of $\bC^{n}$. 
\end{proof}

We now consider the union of these moduli spaces over all piecewise geodesics:
\begin{equation}
  \Cyl(x,\sL^{r} \Q) \equiv \coprod_{\vq \in \sL^{r} \Q}   \Cyl(x,\vq).
\end{equation}
This space admits a natural topology as a parametrised moduli space, embedding in
\begin{equation} \label{eq:parametrised_space_all_families}
 C^{\infty}(Z^{-}_{r}, \TQ) \times \sL^{r} \Q.
\end{equation}
We write
\begin{equation}
  \pi_{\vq} \co   \Cyl(x,\sL^{r} \Q)  \to  \sL^{r} \Q
\end{equation}
for the projection to the parameter space.  By choosing generic choices of  data in  Equation \eqref{eq:CR-equation-negative}, we can ensure that every element $u$ of $ \Cyl(x,\sL^{r} \Q)  $  is regular (see, e.g. \cite[Theorem 5.1]{FHS}),  i.e. that we have a surjection
\begin{equation} \label{eq:surjection_tangent_moduli}
T \sL^{r} \Q  \to \coker(D_u).
\end{equation}
which implies that $ \Cyl(x,\sL^{r} \Q)  $ is a smooth manifold. 
\begin{exercise} \label{ex:dim_universal_family}
 Show that the dimension of $\Cyl(x,\sL^{r} \Q)  $ is $(r-1) n + \deg(x) $, and that there is a canonical isomorphism
   \begin{equation} \label{eq:iso_determinant_line_univeral_moduli}
   |\Cyl(x, \sL^{r} \Q) |  \cong \ro_{x}[w(x)] \otimes \eta_{\vq}  \otimes |T_{\vq} \sL^{r} \Q  |
\end{equation}
at a point of $ \Cyl(x, \sL^{r} \Q) $ whose boundary conditions are given by  $\vq \in \sL^{r} \Q$ (Hint: Use the isomorphism of determinant lines in Equation \eqref{eq:iso_det_x_kappa_ev_q}).
\end{exercise}

\subsection{Construction of the map}

Let $y$ be a critical point of the Morse function $f^{r}$ on $ \sL^{r} \Q$, and let $x$ be a Hamiltonian orbit of a non-degenerate linear Hamiltonian $H^{-}$ whose slope is larger than $2r$. 

By taking the fibre product of the descending manifold of $y$ with the moduli space of negative punctured discs converging to $x$, we obtain a moduli space
\begin{equation} \label{eq:define_broke_x_y}
  \Broke(x,y) \equiv \Cyl(x,\sL^{r} \Q) \times_{\sL^{r} \Q} W^{u}(y),
\end{equation}
where the map from the first factor to $\sL^{r} \Q$ is the projection to the parametrising space in Equation \eqref{eq:parametrised_space_all_families}, and the map from the second factor is the inclusion. Choosing generic choices of  data in  Equation \eqref{eq:CR-equation-negative}, we can ensure that this fibre product is transverse for all pairs $x$ and $y$. Keeping in mind that the dimension of $ W^{u}(y) $ is $\ind(y)$, and that $\ro_{y}$ is its orientation line, the following exercise follows from  Exercise \ref{ex:dim_universal_family}:
\begin{exercise}
 Show that the dimension of $ \Broke(x,y)   $  is
\begin{equation} \label{eq:dimension_moduli_space-Fam}
  \ind(x) + \deg(x) -n
\end{equation}
and that we have a canonical isomorphism
\begin{equation} \label{eq:isomorphism_det_lines_Fam}
|\Broke(x,y) |   \cong   \ro_{x}[w(x)] \otimes \eta_{y} \otimes  \ro_{y}.
\end{equation} 
\end{exercise}
We now assume that $\deg(x) =- ( \ind(y) -n) $. From Equation \eqref{eq:dimension_moduli_space-Fam}, we conclude that $\Broke(x,y)  $ is a $0$-dimensional manifold. From Equation \eqref{eq:isomorphism_det_lines_Fam}, each element of this moduli space yields an isomorphism
\begin{equation}
  \Fam^{u} \co \ro_{y} \otimes \eta_{y}  \cong   \ro_{x}[w(x)],
\end{equation}
where we use the fact that an orientation on $  \eta_{y} \otimes  \ro_{y} $ induces one on its inverse.

Lemma \ref{lem:compactness_negative_cylinders}, together with Gromov-Floer compactness, imply that  $ \Broke(x,y) $ is compact, hence a finite set. We then define a map
\begin{align}
  \Fam^{r} \co CM_{*}(f^{r}; \eta) & \to CF^{-*}(H^{-} ; \bZ) \\
\Fam^{r}| \ro_{y} \otimes \eta_{y} & \equiv \bigoplus_{\deg(x) =- ( \ind(y) -n)  }\sum_{u \in  \Broke(x,y)}  (-1)^{\ind(y)-n}  \Fam^{u}.
\end{align}
The signed contribution of $\ind(y)$ ensures that the differentials commute with $\Fam^{r}$:
\begin{exercise}
Imitate the proof of Lemma \ref{lem:Vit_is_chain_map}, and show that $\Fam^{r}$ is a chain map.
\end{exercise}

We leave, as an exercise to the reader, the proof of the following result, which can be modelled after the discussions in Section \ref{sec:cont-maps-comm} and \ref{sec:comp-with-cont}:
\begin{exercise}\label{ex:Fam-commutes-with-continuation}
  Show that, on homology, the map $\Fam^{r}$ is independent of the choice of Floer data away from the ends, and that there is a commutative diagram
  \begin{equation}
      \xymatrix{HM_{*}(f^{r}; \eta)  \ar[r] \ar[dr] & HF^{-*}(\TQ; H^+) \ar[d]  \\
& HF^{-*}(\TQ; H^-) }
  \end{equation}
whenever $H^{+} \preceq H^{-}$, and both have slope larger than $2r$.
\end{exercise}

\section{From loop homology to symplectic cohomology}

Having constructed a map $\Fam^{r}$ from the Morse homology of a finite dimensional approximation to the Floer cohomology of a Hamiltonian in Section \ref{sec:from-morse-homology}, Exercise \ref{ex:Fam-commutes-with-continuation} implies that the maps are compatible with continuation maps in Floer cohomology, hence yield a map
\begin{equation}
  \Fam^{r} \co HM_{*}(f^{r}; \eta)  \to SH^{-*}(\TQ, \bZ).
\end{equation}

Our goal in this section is to prove the analogous compatibility statement for the inclusion maps $\sL^{r-1} \Q \subset   \sL^{r} \Q$, i.e. that we have a commutative diagram
\begin{equation}
  \xymatrix{HM_{*}(f^{r-1}; \eta)  \ar[r]^{\iota} \ar[dr]^{\Fam^{r-1}} & HM_{*}(f^{r}; \eta) \ar[d]^{\Fam^{r}} \\
&  HF^{-*}(H; \bZ)  }
\end{equation}
whenever the slope of $H$ is larger than $2r$.  We shall show this in Section \ref{sec:constr-homot-assoc}, which will allow us to pass to the direct limit in $r$, and produce a map
\begin{equation}
  H_{*}(\sL \Q; \eta) \cong \lim_{r} HM_{*}(f^{r}; \eta)  \to \lim_{H} HF^{-*}(H; \bZ) \cong SH^{-*}(\TQ, \bZ).
\end{equation}

The key idea is that the inclusion map can be defined using the moduli space of \emph{holomorphic half-planes} with boundary on a cotangent fibre.

\subsection{Half-planes with boundary on a cotangent fibre} \label{sec:cauchy-riem-equat}

We fix a monotonically decreasing smooth function $\chi$ on $\bR$ which vanishes on $\bR^{+}$, and is identically $1$ on $(-\infty, -1] $.  

Given a family $J_{z}$ of almost complex structures parametrised by $\bC_{+}$, whose pullbacks under the negative strip-like end are $s$-independent, and a real number $b \in (0,2)$, we obtain a pseudoholomorphic curve equation
\begin{equation} \label{eq:CR-triangle}
  \left(du - \chi(s) dt \otimes  X_{b\Hh} \right)^{0,1} \equiv 0
\end{equation}
on the space of maps $u$ from $\bC_{+}$ to $\TQ$, mapping the boundary to $\Tq$.  Note that this equation is written in the coordinates of the strip-like end, but since $\chi$ vanishes by assumption when $s$ is positive, it extends naturally to $\bC_{+}$. We write $\Plane(\Tq)$ for the moduli space of finite energy solutions to Equation \eqref{eq:CR-triangle} with this boundary condition. As before, a generic choice of family $J_{z}$ ensures that $\Plane(\Tq) $ is regular.

Having constructed $\Plane(\Tq)  $ as solutions to a pseudoholomorphic curve equation, the following exercises will show that it consists, in fact, of a single element which is the constant map.
\begin{exercise}
Show that, for any $u \in \Plane(\Tq) $, there is a chord
\begin{equation}
x  \in \Chord_{b \Hh}( \Tq, \Tq) 
\end{equation}
which is the limit of $u$ along the strip-like end. 
\end{exercise}
\begin{exercise}
 Using Lemma \ref{lem:chords_are_geodesics}, show that $\Chord_{b\Hh}(\Tq,\Tq) $ consists of a single element which is the constant chord mapping to the intersection of $\Tq$ with the zero section.
\end{exercise}
The above exercise implies that the constant map with image $\Q \cap \Tq$ is an element of $\Plane(\Tq)  $. Indeed, the inhomogeneous term in Equation \eqref{eq:CR-triangle} vanishes in this case because $X_{b\Hh} $ vanishes along the zero section.
\begin{exercise}
Show that the action of the unique element of $ \Chord_{b\Hh}(\Tq,\Tq) $ vanishes. Applying Equation \eqref{eq:Stokes_topological}, conclude that the topological energy of any element of $ \Plane(\Tq) $  vanishes.
\end{exercise}
\begin{exercise} 
Using the monotonicity of $\chi$, and the fact that $b\Hh $ is everywhere non-negative, show that the second term in Equation \eqref{eq:difference_geo_top} is non-positive. Conclude that the geometric energy $E_{geo} (u) $ vanishes for any element of $ \Plane(\Tq) $.
\end{exercise}
\begin{exercise} \label{ex:cotangent_fibre_with_itself_thrice_constant}
Using Lemma \ref{lem:vanishing_energy}, and the fact that the limiting orbit lies on the zero section, show that the only element of $\Plane(\Tq) $ is the constant map.
\end{exercise}

To show that this moduli space is regular, we introduce the following \emph{gauge transformation}. Define
\begin{equation} \label{eq:gauge-transfo}
  \tilde{u}(e^{-s-i \pi (1+ t)}) \equiv \phi^{-b \chi(s) t }(u(e^{-s-i \pi (1+ t)})),
\end{equation}
where $\phi^{t}$ is the time-$t$ Hamiltonian flow of $\Hh$.   Let $\tilde{J}_{z}$ denote the family of almost complex structures on $\TQ$ given by
\begin{equation}
  \tilde{J}_{z} =\phi_{*}^{-b  \chi(s) t} \circ J_{z} \circ \phi_{*}^{b \chi(s) t}.
\end{equation}
\begin{exercise}
Show that $\tilde{u}$ satisfies the homogeneous equation 
\begin{equation} \label{eq:homogeneous_dbar}
\tilde{J}_{z}  \frac{\partial \tilde{u}}{\partial x} =  \frac{\partial \tilde{u}}{\partial y},
\end{equation}
and has boundary conditions given by the path of Lagrangians
\begin{equation}
  \Lambda_{x} \equiv  \begin{cases}  \Tq & \textrm{ if } |x| \leq 0 \\
\phi^{-b \chi(\log(x))}(\Tq) & \textrm{ if } 0 \leq |x|. \end{cases}
\end{equation}
\end{exercise}

The gauge transformation in Equation \eqref{eq:gauge-transfo} can be defined for any smooth curve and maps solutions of Equation \eqref{eq:CR-triangle} to solutions of Equation \eqref{eq:homogeneous_dbar}. In particular, the linearisations of these two equations are intertwined, and the regularity of a pseudo-holomorphic map $u$ is equivalent to that of the gauge transform $\tilde{u}$. Applying Corollary \ref{cor:determinant_line}:
\begin{lem} \label{lem:regular_moduli_space_triple_cotangent}
The constant map is regular as an element of  $\Plane(\Tq)$, and the determinant line admits a canonical trivialisation. \qed
\end{lem}

\subsection{Relating holomorphic punctured discs with $r-1$ and $r$ punctures}

Recall that we have defined an inclusion map $\iota$ in Equation \eqref{eq:map_successive_r}  from the space of piecewise geodesics with $r-1$ segments, to those with $r$ segments.  We shall construct a cobordism between $\Cyl(x,\sL^{r-1} \Q) $ and $ \Cyl(x,\sL^{r} \Q) \times_{\sL^{r} \Q} \iota(\sL^{r-1}) $. The key property that we shall need is that this cobordism can be oriented compatibly with both its boundary components:

\begin{exercise}
  Using Equation \eqref{eq:iso_determinant_line_univeral_moduli}, show that we have a canonical isomorphism
  \begin{equation} \label{eq:isomorphism_inclusion_cylinder}
    | \Cyl(x,\sL^{r} \Q) \times_{\sL^{r} \Q} \iota(\sL^{r-1}) | \cong \ro_{x}[w(x)] \otimes \eta_{\vq}  \otimes |\sL^{r-1} \Q  |.
  \end{equation}
\end{exercise}

\begin{figure}[h]
  \centering
\includegraphics{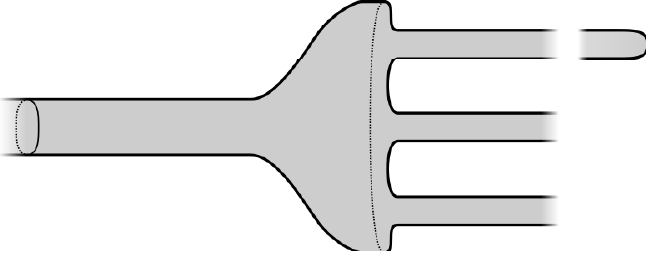}
  \caption{ }
  \label{fig:negative_cylinder-and-plane}
\end{figure}

\begin{lem}
Let $H$ be a Hamiltonian of slope greater than $2 r $. There is a cobordism
\begin{equation}
   \Cyl^{\iota}(x,\sL^{r-1} \Q) 
\end{equation}
with boundary $   \Cyl(x,\sL^{r-1} \Q)  $  and $  \Cyl(x,\sL^{r} \Q) \times_{\sL^{r} \Q} \iota(\sL^{r-1}) $, which is equipped with an isomorphism
\begin{equation} \label{eq:isomorphism_cobordism_inclusion}
  |\Cyl^{\iota}(x,\sL^{r-1} \Q)|  \cong \ro_{x}[w(x)] \otimes \eta_{\vq}  \otimes |T_{\vq} \sL^{r-1} \Q  |
\end{equation}
such that the restriction to the first boundary stratum agrees with Equation \eqref{eq:iso_determinant_line_univeral_moduli}, and the restriction to the second boundary stratum is given by  the opposite of Equation \eqref{eq:isomorphism_inclusion_cylinder}.
\end{lem}
\begin{proof}
By definition, the asymptotic condition at $(0,\frac{1}{r})$ of an element of $ \Cyl(x,\sL^{r} \Q) \times_{\sL^{r} \Q} \iota(\sL^{r-1}) $ is the unique element of
\begin{equation}
  \Chord_{b_1 \Hh}(\Tq,\Tq)
\end{equation}
for some point $q \in \Q$. This is also the asymptotic condition of the unique element of $\Plane(\Tq)$.  By gluing these two surfaces (see Figure \ref{fig:negative_cylinder-and-plane}), we therefore obtain a pseudo-holomorphic equation on a half-cylinder with $r-1$ punctures
\begin{equation} \label{eq:punctured_surface_non-symmetric}
  Z^{-}\setminus \{ (0, \frac{i}{r}) \}_{i=2}^{r}.
\end{equation}
Choose a family $ Z^{-}_{r-1,\tau}$ of punctured half-cylinders, parametrised by $\tau \in [0,1]$, such that $Z^{-}_{r-1,0} = Z^{-}_{r-1}$ and $Z^{-}_{r-1,1}$ is the surface in Equation  \eqref{eq:punctured_surface_non-symmetric}. By equipping this family of surfaces with pseudo-holomorphic equations interpolating between those for $\tau=\{0,1\}$, we obtain a parametrised moduli space.

We write $   \Cyl^{\iota}(x,\sL^{r-1} \Q)  $ for the space of solutions to this family of equations, with output $x \in \Chord(H^{-})$, and inputs points in $\sL^{r-1} \Q$. For generic data, this is a manifold of dimension
\begin{equation}
   (r-1) n + \deg(x) +1.
\end{equation}
The boundary stratum of this moduli space corresponding to $\tau = 0$ is naturally diffeomorphic to $  \Cyl^{\iota}(x,\sL^{r-1} \Q)$. On the other hand, for $\tau = 1$, we obtain the subset of $\Cyl^{\iota}(x,\sL^{r} \Q) $ corresponding to points in the image of $\iota$; so $  \Cyl^{\iota}(x,\sL^{r-1} \Q) $  is indeed a cobordism between $   \Cyl(x,\sL^{r-1} \Q)  $  and $  \Cyl(x,\sL^{r} \Q) \times_{\sL^{r} \Q} \iota(\sL^{r-1}) $.

At every point $(u,\tau)$ of this moduli space, we have a canonical up to homotopy isomorphism
\begin{equation}
  |\Cyl^{\iota}(x,\sL^{r-1} \Q)| \cong |\det(D_u) | \otimes |T_{\tau} [0,1]|.
\end{equation}
Fixing the natural orientation of $[0,1]$ as a subset of $\bR$, and using the same method as in Section \ref{sec:orient-moduli-space-negative}, we obtain the desired orientation of this cobordism.
\end{proof}

\subsection{Construction of the homotopy associated to an inclusion} \label{sec:constr-homot-assoc}

At this stage, we can construct a homotopy for the diagram
\begin{equation} \label{eq:htpy_inclusion_Fam}
  \xymatrix{CM_{*}(f^{r-1}, \eta)  \ar[r] \ar[dr] & CM_{*}(f^{r}, \eta)  \ar[d] \\
& CF^{-*}(H ; \bZ) .}
\end{equation}
Given a critical point $y$ of $f^{r-1}$, and an orbit $x$ of $H$, the relevant moduli space for the construction of the homotopy is the disjoint union
\begin{multline}
  \label{eq:two_moduli_cobordism_inclusion_Fam}
\Broke^{\iota}(x,y) \equiv \Cyl^{\iota}(x,\sL^{r-1} \Q) \times_{\sL^{r-1} \Q} W^{u}(y)  \sqcup \\
  \Cyl(x,\sL^{r} \Q) _{\pi_{\vq}}\times_{  \psi^{r} \circ \iota}  \left( W^{u}(y)) \times [0,+\infty) \right) .
\end{multline}
A few words of explanation about the second component are in order. The fibre product is taken over $\sL^{r} \Q$, with the evaluation map from the first factor corresponding to projection to the parameter space, and the map from the second factor being given by
\begin{equation}
  (\vq,T) \mapsto \psi_{T}^{r}(\iota(\vq)),
\end{equation}
where $\psi^{r}$ is the gradient flow of $f^{r}$.

We can alternatively think of elements of this fibre product as  consisting of the following triples: (i) a positive gradient flow line $\gamma_{0}$  of $f^{r-1}$ in $\sL^{r-1} \Q $,  with domain $[0,+\infty)$, which converges at $+\infty$ to $y$, (ii) a positive gradient flow line $\gamma_1$ of $f^{r}$  with domain $[-T,0]$ starting at $\iota(\gamma_0(0))$, and (iii) a negative punctured disc  $v$ with boundary conditions given by $\gamma_{1}(-T) \in \sL^{r-1} \Q $ and which converges to $x$ at the negative end.
\begin{rem}
The appearance of the second component in Equation \eqref{eq:two_moduli_cobordism_inclusion_Fam} is typical of constructions involving Morse theory because of the need to factor the fibre product of chains over $\sL^{r} \Q$ through ascending and descending manifolds of the critical points of our chosen Morse function. If we were working with a more classical theory of geometric chains (e.g. singular or cubical chains), the first component would suffice to construct the homotopy induced by inclusions.
\end{rem}
\begin{exercise} \label{ex:dimension_moduli_space_htpy_inclusion}
 For generic data, show that the dimension of $\Broke^{\iota}(x,y)$ is $   \ind(y) + \deg(x) -n +1 $. Show that an orientation of the interval $ [0,+\infty) $ induces an isomorphism of the orientation line of both components with
  \begin{equation}
    \ro_{x}[w(x)] \otimes \eta_{\ev(y)} \otimes  \ro_{y}.
  \end{equation}
\end{exercise}
In order to justify introducing this moduli space, recall that the inclusion map on Morse chains was defined using the moduli spaces $\Tree_{\iota}(y',y)  $ which consist of a pair of gradient flow lines in $\sL^{r-1} \Q$ and $ \sL^{r} \Q$ that are matched along the evaluation map $\iota$.
\begin{lem}
The moduli space $\Broke^{\iota}(x,y) $ is the interior of a cobordism between
 \begin{equation} \label{eq:manifolds_two_composition_inclusion_htpy}
    \begin{aligned}
& \coprod_{\ind(y') = \ind(y) }  \Cyl(x,\sL^{r} \Q) \times_{\sL^{r} \Q}  W^{u}(y') \times \Tree_{\iota}(y',y)  \\
& \Cyl(x,\sL^{r-1} \Q) \times_{\sL^{r-1} \Q} W^{u}(y).
\end{aligned}
 \end{equation}
\end{lem}
\begin{proof}
 We first consider enlarging the first space in Equation \eqref{eq:two_moduli_cobordism_inclusion_Fam} by allowing $T=+\infty$. Since gradient flow lines whose length goes to $+\infty$ break at critical points of $f^{r}$, the natural stratum corresponding to $T = +\infty$ is the fibre product of $ \Cyl(x,\sL^{r} \Q)  $ with 
 \begin{equation}
 \coprod_{ \ind(y') = \ind(y) }W^{u}(y') \times \Tree_{\iota}(y',y),
 \end{equation}
which gives the  first space in Equation \eqref{eq:manifolds_two_composition_inclusion_htpy}.

Setting $T = 0$ yields the fibre product
 \begin{equation}
    \Cyl(x,\sL^{r} \Q) \times_{\sL^{r} \Q} \iota(W^{u}(y)).
 \end{equation}
The key observation is that this is also the boundary component of  the second space in Equation \eqref{eq:two_moduli_cobordism_inclusion_Fam} corresponding to $\tau = 1$.  In particular, if we glue these two spaces along this common boundary, we obtain a manifold with boundary given by the stratum $T=+\infty$ discussed above, and $\tau =0$. These correspond exactly to the two spaces in Equation \eqref{eq:manifolds_two_composition_inclusion_htpy}.
\end{proof}

We now consider a pair $(x,y)$ such that
\begin{equation}
  \deg(x) =  - (\ind(y)-n) -1,
\end{equation}
which implies that both spaces in Equation \eqref{eq:two_moduli_cobordism_inclusion_Fam} are $0$-dimensional manifolds.
\begin{exercise} \label{ex:map_homotopy_inclusion_first_factor}
  Using the isomorphism in Equation \eqref{eq:isomorphism_cobordism_inclusion}, construct a map
  \begin{equation}\label{eq:map_homotopy_inclusion_first_factor}
    \cH_{(\tau, u, \gamma)} \co \ro_{y} \otimes \eta_{y} \to \ro_{x}[w(x)]
  \end{equation}
associated to $ (\tau, u, \gamma) \in   \Cyl^{\iota}(x,\sL^{r-1} \Q) \times_{\sL^{r-1} \Q} W^{u}(y)$.
\end{exercise}
\begin{exercise}\label{ex:map_homotopy_inclusion_second_factor}
  Using the isomorphism in Equation \eqref{eq:iso_determinant_line_univeral_moduli}, construct a map
  \begin{equation}
    \cH_{(v, \gamma, T)} \co \ro_{y} \otimes \eta_{y} \to \ro_{x}[w(x)]
  \end{equation}
associated to 
\begin{equation}\label{eq:map_homotopy_inclusion_second_factor}
  (v, \gamma, T) \in \Cyl(x,\sL^{r} \Q) _{\pi_{\vq}}\times_{  \psi^{r} \circ \iota}  \left( W^{u}(y)) \times [0,+\infty) \right).
\end{equation}
\end{exercise}

At this stage, we can define the homotopy in Diagram \eqref{eq:htpy_inclusion_Fam}:
\begin{align}
\cH^{\iota} \co   CM_{*}(f^{r-1}, \eta) & \to CF^{-*-1}(H ; \bZ) \\
\cH^{\iota} | \ro_{y} \otimes \eta_{y}  & \equiv \sum     \cH_{(\tau, u, \gamma)} + \cH_{(v, \gamma, T )},
\end{align}
where the sum is taken over all maps constructed in Equations \eqref{eq:map_homotopy_inclusion_first_factor} and \eqref{eq:map_homotopy_inclusion_second_factor}.
\begin{lem} \label{lem:fam_commutes_inclusion}
  The map $ \cH^{\iota} $ defines a homotopy between the $\Fam^{r-1}$ and the composition $\Fam^{r} \circ \iota $.
\end{lem}
\begin{proof}[Sketch of proof]
Recall that the equation for a homotopy is
\begin{equation} \label{eq:equation_htpy_inclusion_Fam}
  d \circ \cH^{\iota} + \cH^{\iota} \circ \partial = \Fam^{r-1} - \Fam^{r} \circ \iota
\end{equation}
We consider the space $\Broke^{\iota}(x,y) $  whenever $  \deg(x) =  - (\ind(y)-n)$. This manifold has dimension $1$ by Exercise \ref{ex:dimension_moduli_space_htpy_inclusion}, and it is easy to see that the boundary strata  given by Equation \eqref{eq:manifolds_two_composition_inclusion_htpy} respectively correspond to the composition $ \Fam^{r} \circ \iota   $ and to $\Fam^{r-1} $, which give the right hand side of Equation \eqref{eq:equation_htpy_inclusion_Fam}.

It remains to construct a compactification of $ \Broke^{\iota}(x,y)  $ to a manifold with boundary so that the additional boundary strata correspond to the left hand side of Equation \eqref{eq:equation_htpy_inclusion_Fam}.  Since this space is defined as a union of fibre products, the boundary of the compactification is easy to analyse: one possibility is that the punctured disc breaks into two components, yielding in codimension $1$ the union
\begin{equation}
 \coprod_{ \deg(x) = \deg(x') + 1  } \Cyl(x,x') \times \Broke^{\iota}(x',y)
\end{equation}
which corresponds to the composition $ \cH^{\iota} \circ  \partial $. Since the dimension of the moduli space is $1$ in this case, only codimension $1$ breaking can occur.

The other possibility is that the flow line breaks, yielding the strata
\begin{equation}
 \coprod_{ \ind(y') = \deg(y) - 1  }\Broke^{\iota}(x,y') \times \Tree(y',y) 
\end{equation}
which correspond to the composition  $d \circ \cH^{\iota}$.
\end{proof}

\section{Composition on loop homology} \label{sec:comp-loop-homol}

The goal of this section is to prove that $\Fam$ is, up to a sign, a right inverse to $\Vit$. In order to state the sign precisely, we split the Morse homology of $\sL^{r} \Q$ as a direct sum of two groups
\begin{equation} \label{eq:decomposition_morse}
  HM_{*}(f^{r}; \eta) \equiv HM_{*}^{0}(f^{r}; \eta) \oplus HM_{*}^{-1}(f^{r}; \eta)
\end{equation}
corresponding to the components of piecewise geodesics along which $T\Q$ is, or is not, orientable.    We also split the Floer cohomology of a Hamiltonian $H^{-}$ as
\begin{equation} \label{eq:decomposition_Floer}
  HF^{*}(H^{-}; \bZ) \equiv   HF^{*}_{0}(H^{-}; \bZ) \oplus HF^{*}_{-1}(H^{-}; \bZ),
\end{equation}
with the two summands corresponding to the subcomplexes generated by orbits along which $q^{*}(T\Q)$ is orientable or not. Since $\Vit_{r}$ and $\Fam^{r}$ are both defined using moduli spaces of punctured discs, orientability and non-orientability of the pullback of $T \Q$ are preserved by these maps. These maps therefore preserve the decomposition of Morse homology and Floer cohomology, so we obtain maps:
 \begin{equation}
\xymatrix{ HM_{*}^{w}(f^{r}; \eta) \ar[r]^{\Fam^{r}}& HF^{-*}_{w}(H^{-}; \bZ) \ar[r]^{\Vit_{r'}} & HM_{*}^{w}(f^{r'}; \eta) }
\end{equation}
whenever the slope of $H^-$ is larger than $2r$, and $r'$ is sufficiently large.

The following result is proved in Section \ref{sec:constr-homot}:
\begin{prop} \label{lem:compositions_loop_agree_up_to_sign}
  If $r'$ is sufficiently large, the composition
  \begin{equation} \label{eq:composition_vit_fam}
    \Vit_{r'} \circ \Fam^{r} \co HM_{*}^{w}(f^{r}; \eta)  \to HM_{*}^{w}(f^{r'}; \eta) 
  \end{equation}
agrees with 
\begin{equation}
  (-1)^{  \frac{(n+w)(n+w-1)}{2} } \underbrace{ \iota \circ  \iota \circ \cdots  \circ \iota}_{ r' -r }.
\end{equation} 
\end{prop}

The proof of the main theorem of this chapter follows immediately:
\begin{proof}[Proof of Theorem \ref{thm:compose_fam_vit_iso}]
Since continuation maps and inclusion maps preserve the decomposition according to orientability, we have isomorphisms
\begin{align}
  SH^{*}(\TQ; \bZ) & = SH^{*}_{0}(\TQ; \bZ)  \oplus SH^{*}_{-1}(\TQ; \bZ)  \\
H_{*}(\sL \Q; \eta) &=  H_{*}^{0}(\sL \Q; \eta) \oplus H_{*}^{-1}(\sL \Q; \eta),
\end{align}
where each summand in the right hand side is the direct limit of the appropriate groups in Equations \eqref{eq:decomposition_morse} and \eqref{eq:decomposition_Floer}.

Equation \eqref{eq:composition_vit_fam} implies that the composition
  \begin{equation}
    \xymatrix{ H_{*}^{w}(\sL \Q; \eta) \ar[r]^-{\Fam}& SH^{-*}_{w}(\TQ; \bZ) \ar[r]^-{\Vit} & H_{*}^{w}(\sL \Q; \eta)}
  \end{equation}
agrees with multiplication by $ (-1)^{  \frac{(n+w)(n+w-1)}{2} } $. Taking the direct sum over both values of $w$ implies that the composition $\Vit \circ \Fam  $  is an isomorphism.
\end{proof}

There are two main ideas that are required for the proof of Proposition \ref{lem:compositions_loop_agree_up_to_sign}, which will be developed in the rest of this section. The first is a computation of a moduli space of holomorphic \emph{triangles} with boundary on two nearby cotangent fibres and the zero section. If the cotangent fibres are sufficiently close, there is a unique such triangle.

The second idea is to use a cobordism between two possible degenerations of the complex structure on the annulus; at one end, the limit consists of two punctured discs, which give rise to the
composition in Equation \eqref{eq:composition_vit_fam}. At the other end, the annulus breaks into disc components, which  in our case are degenerate, giving rise to iterated compositions of the inclusion map $\iota$.

\subsection{Two cotangent fibres and the zero section} \label{sec:moduli-space-triangl-+}

Consider the punctured positive half-strip
\begin{equation}
  \Tria^{+} \equiv [0,+\infty) \times [0,1] \setminus \{(0,0), (0,1) \}
\end{equation}
\begin{figure}[h]
  \centering
  \includegraphics[scale=1.25]{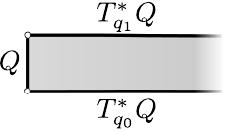}
  \caption{ }
  \label{fig:half-disc-positive}
\end{figure}
equipped with coordinates $(s,t)$, and complex structure $j \partial_{s} = \partial_{t}$.

Choose a function
\begin{equation}
[0,1] \to [0,1]
\end{equation}
which is the identity on the boundary, and is locally constant near the boundary, and let
\begin{equation}
  \tau^{+} \co \Tria^{+} \to [0,1]
\end{equation}
denote the composition with the projection to the second factor of $\Tria^{+}$.

Given points $q_0$ and $q_1$ in $\Q$, define $\Disc (\Tq[0], \Q, \Tq[1])$ to be the space of finite energy maps
\begin{equation}
  u \co  \Tria^{+} \to \TQ
\end{equation}
which solve the differential equation
\begin{equation} \label{eq:equation_positive_triangle}
  \left( du - d\tau^{+}  \otimes X_{\Hh}  \right)^{0,1} = 0
\end{equation}
with respect to a family of almost complex structures $J_{t}$ on $\TQ$, with
the following boundary conditions (see Figure \ref{fig:half-disc-positive}):
\begin{equation} \label{eq:boundary_conditions_Tria+}
  \begin{aligned}
u(\{1\} \times (0,+\infty)) & \subset \Tq[1] \\
  u(\{0\} \times (0,1)) & \subset \Q \\
u(\{0\} \times (0,+\infty)) & \subset \Tq[0].
\end{aligned}
\end{equation}

\begin{rem}
In the definition of $\Fam$, we carefully chose a $1$-form on the
punctured disc whose restriction to the boundary vanishes, while we are
now considering a pseudoholomorphic curve equation where the $1$-form $d\tau^{+}$ does
not vanish on the boundary $\Q$. The reason this does not cause any
real difficulties is that $\Hh$ vanishes to second order on $\Q$,
hence the term $d\tau^{+} \otimes X_{\Hh}  $ in fact vanishes along this
boundary. 
\end{rem}
\begin{exercise}
 Using the same argument as in Exercise \ref{ex:cotangent_fibre_with_itself_thrice_constant}, show that the constant map is the unique element of the moduli space  $\Disc (\Tq, \Q,\Tq)$.
\end{exercise}

Our next goal is to show that $\Disc (\Tq[0], \Q,\Tq[1])$ consists of a unique element whenever $q_0$ and $q_1$ are sufficiently close. In order to prove this, we shall presently see that it suffices to establish the regularity of this moduli space whenever $q_0 = q_1$,

To this end, we introduce the following gauge transform of a map $u \co \Tria^{+} \to \TQ$:
\begin{equation}
  \tilde{u}(s,t) \equiv \phi^{-\tau^{+}}(u(s,t)),
\end{equation}
where $\phi^{t}$ is the time-$t$ Hamiltonian flow of $\Hh$.  
\begin{exercise}
Show that $\tilde{u}$ satisfies a homogeneous equation, and maps the boundary of $\Tria^{+}$ to the triple of Lagrangians $(\Tq[0], \Q, \phi^{-1}(\Tq[1])) $,  the end $s = +\infty$ to  $\Tq[0] \cap \phi^{-1}( \Tq[1]) $, the end $(0,0)$ to $\Q \cap \Tq[0]$,  and  the end $(0,1)$ to $\Q \cap \phi^{-1}( \Tq[1]) $.
\end{exercise}
As in Section \ref{sec:cauchy-riem-equat}, it is easier to analyse the regularity of the equation obtained by gauge transform:

\begin{exercise} \label{ex:constant_triangle_regular_diagonal}
Show that the constant map is regular as an element of $\Disc (\Tq, \Q,\Tq)  $ (Hint: Imitate the proof of Lemma \ref{lem:regular_moduli_space_triple_cotangent}).
\end{exercise}

Let us now consider the space
\begin{equation}
  \Disc (\sL^{2} \Q  ) \equiv \coprod_{(\vq[0], \vq[1]) \in \sL^{2} \Q} \Disc (\Tq[0], \Q,\Tq[1]).
\end{equation}

To topologise this space, we introduce the bundle
\begin{equation}
  C^{\infty}(\Tria^{+}, \TQ; \sL^{2} \Q) \to \sL^{2} \Q
\end{equation}
whose fibre at a point $(\vq[0], \vq[1]) $ is the space
\begin{equation}
  C^{\infty}(\Tria^{+}, \TQ; (\Tq[0], \Q, \Tq[1]) ) 
\end{equation}
of smooth maps from $\Tria^{+}$ to $\TQ$ which decay exponentially to the unique chord connecting $\vq[0]$ to $\vq[1]$ at $+\infty$, and to the intersection point between $\Tq[i]$ and $\Q$ at $(0,i)$ for $i \in \{0,1\}$, and have boundary conditions given by Equation \eqref{eq:boundary_conditions_Tria+}.

This space carries two Banach bundles of interest: the first is the space of $W^{1,p}$ sections of the pullback of $\TQ $ to $\Tria^{+}$, which take value in the tangent space of the Lagrangian boundary conditions along the components of the boundary $\Tria^{+}$, and the second is the space of $L^{p}$ sections of $u^{*}(T\TQ) $. The linearisation of the Cauchy Riemann operator, together with the linearisation of the equation associated to changing the boundary conditions, defines a map
\begin{multline} \label{eq:split_tangent_parametrised_L^1}
  W^{1,p}((\Tria^{+}, \partial \Tria^{+}), (u^{*}(T\TQ),  (u^{*}(\Tq[0]),u^{*}(T \Q), u^{*}(T\Tq[1]))  ) \oplus T \sL^{2} \Q \\ \to   L^{p}(\Tria^{+} , u^{*}(\TQ)).
\end{multline}
Whenever this map is surjective at an element $u$ of $  \Disc (\sL^{2} \Q  )  $, its kernel is the tangent space  $T_{u} \Disc (\sL^{2} \Q  )  $.  Since regularity is an open property for solutions to Cauchy-Riemann operators (because surjectivity is an open condition), Exercise \ref{ex:constant_triangle_regular_diagonal} implies that there is an open neighbourhood in $ \Disc (\sL^{2} \Q  )  $ of the constant triangles, which consists of regular elements, hence  which is a manifold of dimension $2n$. Moreover, for a constant triangle, the restriction of Equation \eqref{eq:split_tangent_parametrised_L^1} to the first factor is an isomorphism. This implies that the kernel of the operator in Equation  \eqref{eq:split_tangent_parametrised_L^1} projects isomorphically onto the second factor, i.e. that the  tangent space of the moduli space in a neighbourhood of constant points projects isomorphically to the tangent space of $\sL^{2} \Q   $ under evaluation.

The restriction of the projection to $\sL^{2} \Q$ is therefore a local diffeomorphism in a neighbourhood of constant triangles; since the set of constant triangles is a submanifold of $ \Disc (\sL^{2} \Q  )  $, on which this projection map is injective,  we conclude that the projection to $\sL^{2} \Q$ is a diffeomorphism from a neighbourhood of the set of constant triangles to a neighbourhood of the diagonal.

In order to give a more precise description of this neighbourhood, we  introduce the subset 
\begin{equation}
    \Disc (\sL^{2}_{ \delta} \Q  )  \subset  \Disc (\sL^{2} \Q  )
\end{equation}
corresponding to points whose distance is bounded by a constant $\delta$.

\begin{lem} \label{lem:moduli_positive_discs_degree_1}
If $\delta$ is sufficiently small, $  \Disc (\sL^{2}_{2 \delta} \Q  )  $ is a smooth manifold of dimension $2n$, and the projection map to $\sL^{2}_{2 \delta} \Q$ is a diffeomorphism. 
\end{lem}
\begin{proof}
  From the above discussion, it suffices to show that, given a neighbourhood of the set of constant triangles, we may choose $\delta$ sufficiently small so that all elements of $  \Disc (\sL^{2}_{2 \delta} \Q  )  $ are contained in it.  We start by noting that all elements of $  \Disc (\sL^{2} \Q  )  $  have image contained in $\DQ$ by the analogue of Lemma \ref{lem:compactness_negative_cylinders}; this will allow us to use Gromov compactness as follows:  consider a sequence of points in $\sL^{2} \Q$ converging to the diagonal.  The Gromov-Floer construction produces a compactification of  $ \Disc (\Tq[0], \Q,\Tq[1]) $ by considering disc bubbles, sphere bubbles, and breakings of Floer trajectories at the corner of the triangle. The first two possibilities are excluded by the exactness of $\TQ$ (and of the Lagrangians $\Q$ and $\Tq[i]$), and the last by the fact that $\Q \cap \Tq[i]$ consists of a single point, and that there is a unique chord of $\Hh$ connecting $\Tq[0]$ and $\Tq[1]$. We conclude that the projection is proper, and that, for $\delta$ sufficiently small, all elements of $   \Disc (\sL^{2}_{2 \delta} \Q  )  $ are close to the constant maps lying over the diagonal in $\sL^{2}_{2 \delta} \Q  $. 
\end{proof}

We shall need slightly more control on the moduli space $  \Disc (\sL^{2}_{2 \delta} \Q  ) $. This space is equipped with an evaluation map to the space of paths between points whose distance is bounded by $2 \delta$, defined by restricting every element of $   \Disc (\sL^{2}_{2 \delta} \Q  )   $  to the segment $\{ 0 \} \times [0,1]$ which is required to map to $\Q$:
\begin{align}
  \Disc (\sL^{2}_{2 \delta} \Q  )  & \to \Path( \sL^{2}_{2 \delta} \Q )  \equiv \{ \gamma \co [0,1] \to \Q \vbar d(\gamma(0), \gamma(1)) \leq 2 \delta \} \\
u & \mapsto u | \{ 0 \} \times [0,1].
\end{align}
In the proof of Lemma \ref{lem:moduli_positive_discs_degree_1}, we required $\delta$ to be sufficiently small so that all elements of $   \Disc (\sL^{2}_{2 \delta} \Q  )    $ are close to constant maps. In particular, we may assume that the image of every element of $  \Disc (\sL^{2}_{2 \delta} \Q  )   $  lies in a geodesically convex ball centered at its starting point, hence is homotopic to the shortest geodesic connecting its two endpoints.  We conclude:
\begin{lem} \label{lem:if_points_close_path_htpic_geo}
  If $\delta$ is sufficiently small, the composition
\begin{equation}
  \xymatrix{\sL^{2} \Q \ar[r]^-{\sim}  & \Disc (\sL^{2}_{2 \delta} \Q  )  \ar[r] &  \Path( \sL^{2}_{2 \delta} \Q )} 
\end{equation}
is homotopic to the map $\sL^{2} \Q  \to  \Path( \sL^{2}_{2 \delta} \Q  ) $  which assigns to two points the unique short geodesic between them. \qed
\end{lem}

\subsection{Moduli space of  degenerate annuli}
 We can associate to every element $\vq \in \sL^{r} \Q$ a moduli space 
\begin{equation}
 \Ann^{0}_{r}(\vq) \equiv  \Disc( \Tq[0], \Q,\Tq[1] )  \times \Disc( \Tq[1], \Q,\Tq[2] ) \times \cdots \times \Disc( \Tq[r-1], \Q,\Tq[0] ),
\end{equation}
whose elements form degenerate annuli by gluing the triangles along their common corners (see the leftmost diagram in Figure \ref{fig:infinite-thin-annuli}); we think of these as being infinitely thin. By taking the union over all $\vq \in \sL^{r} \Q$, we obtain the parametrised moduli space
\begin{equation}
  \Ann^{0}_{r}(\sL^{r} \Q)  \equiv \coprod_{\vq \in \sL^{r} \Q}  \Ann^{0}_{r}(\vq).
\end{equation}
\begin{figure}[h]
  \centering
  \includegraphics{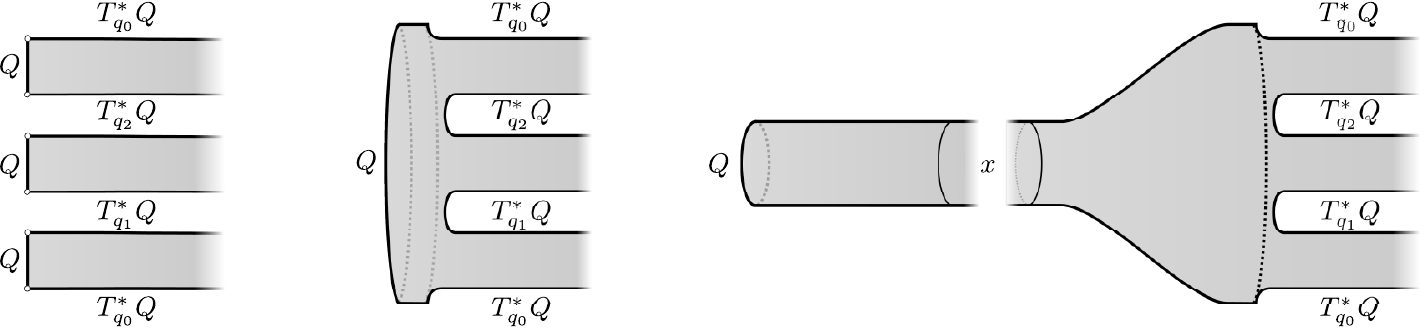}
  \caption{ }
  \label{fig:infinite-thin-annuli}
\end{figure}

We now fix a constant $\delta$ so that the conclusions of Lemmatta \ref{lem:moduli_positive_discs_degree_1} and  \ref{lem:if_points_close_path_htpic_geo} hold, and assume that  the constant $\delta^{r}_{i}$ in the definition of $ \sL^{r} \Q $  satisfies
\begin{equation} \label{eq:constants_in_sL_r_small}
\delta \leq \delta^{r}_{i} \leq 2 \delta.
\end{equation}
As an immediate consequence of Lemma  \ref{lem:moduli_positive_discs_degree_1}, this condition implies
\begin{lem}
   $   \Ann^{0}_{r}(\sL^{r} \Q)  $ is a smooth manifold of dimension $rn$, and the projection to $\sL^{r} \Q$ is a diffeomorphism. \qed
\end{lem} 

By identifying the  the intervals $\left[ \frac{i}{r},\frac{i+1}{r} \right]$  with the segments labelled $\Q$ on the boundary of an element of $   \Ann^{0}_{r}(\sL^{r} \Q)  $, we obtain an evaluation map
\begin{equation} \label{eq:evaluate_thin_annulus_to_circle}
  \begin{aligned}
 \ev \co   \Ann^{0}_{r}(\sL^{r} \Q) & \to \sL \Q \\
\ev(v_0, \ldots, v_r) (t) & \equiv v_{i}(0, r t - i) \textrm{ if } t \in \left[\frac{i}{r},\frac{i+1}{r} \right]
  \end{aligned}
\end{equation}
\begin{exercise}
Check that Equation \eqref{eq:evaluate_thin_annulus_to_circle}  defines a continuous loop in $\Q$.
\end{exercise}
Using Lemma  \ref{lem:if_points_close_path_htpic_geo}, we can compare this evaluation map to the map $ \sL^{r} \Q \to \sL \Q $ defined by piecewise geodesics:  
\begin{lem} \label{lem:ann_0_htpic_to_inclusion}
If Condition \eqref{eq:constants_in_sL_r_small} is satisfied, the composition
\begin{equation}
  \xymatrix{\sL^{r} \Q \ar[r]^-{\sim}  & \ar[r]  \Ann^{0}_{r}(\sL^{r} \Q) &   \sL\Q} 
\end{equation}
is homotopic to the natural map $\sL^{r} Q  \to   \sL \Q $. \qed
\end{lem}

This result shall be used to show that the map induced by the moduli space $ \Ann^{0}_{r}$ on Morse homology agrees, in the limit over $r$, with the identity on the homology of the free loop space. By constructing a cobordism from $\Ann^{0}_{r}$ to the moduli space controlling the composition $\Vit \circ \Fam$, we shall conclude that $ \Vit \circ \Fam $ is the identity on loop space homology.

More precisely, given a Hamiltonian $H^{-}$ of slope larger than $2r$, there is another moduli space of degenerate annuli
\begin{equation}
  \Ann^{\infty}_{r}(\vq) \equiv  \coprod_{x \in \Orbit(H^{-})}  \Cyl(x) \times \Cyl(x,\vq),
\end{equation}
which we think of as consisting of infinitely long annuli (see the rightmost picture in Figure \ref{fig:infinite-thin-annuli}). We define a parametrised version of this space:
\begin{equation} 
  \Ann^{\infty}_{r}(\sL^{r} \Q ) \equiv  \coprod_{x \in \Orbit(H^{-})}  \Cyl(x) \times \Cyl(x,\sL^{r} \Q).
\end{equation}
\begin{exercise}
 Use Lemma \ref{lem:vir_dim_moduli_half_cylinders} and Exercise \ref{ex:dim_universal_family} to show that $  \Ann^{\infty}_{r}(\sL^{r} \Q ) $ is a smooth manifold of dimension $nr$. 
\end{exercise}

\subsection{Cobordism between degenerate annuli}
Consider the moduli space of Riemann surfaces biholomorphic to an annulus
\begin{equation}
\Sigma^{R} \equiv  [0,R] \times S^{1}
\end{equation}
equipped with coordinates $(s,t)$. By removing the points  $z_{i}=  (R, \frac{i}{r})$ from the boundary, we obtain the punctured Riemann surface
\begin{equation}
\Sigma^{R}_{r}  \equiv \Sigma^{R}    \setminus \{ z_{i} \}_{i=1}^{r}.
\end{equation}
 The boundary of $ \Sigma^{R}_{r} $ has $r+1$ boundary components: we denote by $\partial^{0} \Sigma^{R}_{r}$ the boundary component corresponding to setting the first coordinate equal to $0$, and by $\{ \partial^{i} \Sigma^{R}_{r}\}_{i=1}^{r} $   the segment 
\begin{equation}
  s = R, \, \, \, \, t  \in \left[\frac{i-1}{r}, \frac{i}{r} \right].
\end{equation}

There is a natural compactification of this moduli space, which we denote $\Ann_{r} $, corresponding to allowing $R =\{0,+\infty\}$; these degenerate annuli are
\begin{align}
\Sigma^{\infty}_{r} & = Z^{+} \coprod Z^{-}_{r} \\
\Sigma^{0}_{r} & = \underbrace{\Tria^{+} \coprod  \Tria^{+} \coprod \cdots \coprod  \Tria^{+}}_{r}. 
\end{align}
The natural topology at the boundary $\Ann_{r}$ arises from gluing, as explained in  the following two exercises:
\begin{exercise} \label{ex:gluing_infinitely_long_annuli}
For each sufficiently positive real number $S$, define a Riemann surface $Z^{+} \#_{S} Z^{-}$ by gluing the positive and the negative punctured discs along their cylindrical ends. Show that the resulting surface is biholomorphic to $ \Sigma^{2S} $. 
\end{exercise}
\begin{exercise}\label{ex:gluing_infinitely_thin_annuli}
Identify  $\Tria^+$ with $\Strip \setminus \{ (0, 1 )\}$, and equip the strip with its natural strip like ends at $s=+\infty$  and $s= -\infty$. Show that the result of gluing $r$  copies of this Riemann surface end to end for \emph{equal gluing parameter} $S$  is biholomorphic to $ \Sigma^{1/S}$.
\end{exercise}

Equation \eqref{eq:equation_positive_triangle} defines a pseudo-holomorphic curve equation on every component of the surface $\Sigma^{0}_{r} $, and Equations \eqref{eq:operator_half_cylinder} and \eqref{eq:CR-equation-negative} define  pseudo-holomorphic curve equations on the two components of $\Sigma^{\infty}_{r}$.  We shall extend these to other points of the moduli space  $\Ann_{r}$. To start, recall that we have fixed (positive) strip-like ends near the boundary punctures of the two surfaces $\Sigma^{0}_{r} $ and $\Sigma^{\infty}_{r}$. By gluing, we obtain strip-like ends on the nearby elements of $\Ann_{r}$, and we extend these choices to strip-like ends
\begin{equation}
  \epsilon_{i}^{R} \co [0,+\infty) \times [0,1] \to \Sigma^{R}_{r}, \quad 1 \leq i \leq r
\end{equation}
converging to the puncture $z_{i}$.

The pseudo-holomorphic curve equation we shall impose on maps from $\Sigma^{R}_{r}  $ to $\TQ$ will be of the form
\begin{equation} \label{eq:equation_on_annulus}
  (du - X_{H^{R} } \otimes \alpha^{R})^{0,1} = 0,
\end{equation}
where the data will be as follows:
\begin{enumerate}
\item A $1$-form $\alpha^{R}$
\item A family of linear Hamiltonians $H^{R}_{z}$, parametrised by $z \in \Sigma^{R}_{r}$, and
\item A family of almost complex structures $J^{R}_{z}$, parametrised by $z \in \Sigma^{R}_{r}$,  which are convex near $\SQ[2]$.
\end{enumerate}
These Floer data should satisfy the following properties:
\begin{align} \label{eq:condition_data_annulus-1}
& \parbox{33em}{$\alpha^{R}$ is closed, its restriction to the boundary components  $\{ \partial^{i} \Sigma^{R}_{r}\}_{i=1}^{r} $ vanishes, and there are positive strip-like ends at the punctures $z_i$ such that the pullback of $\alpha^{R}$ agrees with $b_{i}^{R}dt$ for real numbers $b_{i}^{R} \in (  \delta_{i}^{r} , 2 )$ (c.f. Equation \eqref{eq:condition_on_alpha-}).} \\
& \parbox{33em}{The Hamiltonian $ H^{R}_{z} $ agrees with $h$ if $z$ lies in a neighbourhood of the punctures $\{ z_i \}_{i=1}^{r}$. If $z$ lies on $ \partial^{0} \Sigma^{R}_{r}  $, we assume that the Hamiltonian flow of $X_{H_{z}^{Z}}$ vanishes on $\Q$ (c.f. Equation \eqref{eq:flow_vanishes_on_Q}).} \\ \label{eq:f-alpha-negative}
& \parbox{33em}{There exists a real number $b^{R}$ such that $  2  \sum_{i=1}^{r} b_{i}^{R} \leq b^{R}$,  and a smooth function $f^{R}$ on $\Sigma^{R}$, which vanishes near the punctures, such that $d f^{R} \wedge \alpha^{R} \leq 0  $.} \\ \label{eq:condition_data_annulus-4}
&    H^{R} | \TQ \setminus \DQ[2]  = 2 \rho - 2 + f^{R} \left(  \left(\frac{b^{R}}{\sum_{i=1}^{r} b_{i}^{R} } - 2 \right)   \rho  + 2 \right).
\end{align}
\begin{exercise}
Show that the space of triples $(\alpha^{R}, H^{R}_{z} )$ statisfying the above properties is contractible (Hint: use the fact that $f^{R} \equiv 0$ is a solution to Equation \eqref{eq:f-alpha-negative}). 
\end{exercise}

We now explain how to produce pseudo-holomorphic curve equations on the surfaces constructed by gluing in Exercises \ref{ex:gluing_infinitely_long_annuli} and \ref{ex:gluing_infinitely_thin_annuli}. For the surface $\Sigma^{\infty}_{r} $, we have assumed that the $1$-form $\alpha^{-}$ on $ Z^{-} $ agrees with $dt$ near the negative puncture, while $Z^{+}$ carries the $1$-form $dt$ in the discussion of Section \ref{sec:punctured discs-with}.  Since the restrictions of these forms to the glued regions agree, we conclude that, for $S$ large enough, the surface $\Sigma^{2S} $ naturally carries a $1$-form that we denote $\alpha^{2S}$. Similarly, the family of functions $H^{-}$ on $  Z^{-}  $  is assumed to agree with a fixed time-dependent Hamiltonian $H_{t}$ along the end, as does the family of Hamiltonians on $ Z^{+} $ fixed in Equation \eqref{eq:flow_vanishes_on_Q}. We therefore obtain a family of Hamiltonians on $\Sigma^{2S}$ that we denote $H^{2S}$.
\begin{exercise}
Check that these data satisfy Conditions \eqref{eq:condition_data_annulus-1}-\eqref{eq:condition_data_annulus-4}.
\end{exercise}

We repeat the same procedure to produce data near the surface $ \Sigma^{0}_{r}$: first, we set
\begin{equation}
  H^{R}_{z} = \Hh
\end{equation}
whenever $R$ is sufficiently close to $0$. Moreover, the $1$-form  $d \tau$ on the surface $\Tria^{+}$ vanishes by assumption near the punctures $(0,0)$ and $(0,1)$. There is therefore a natural $1$-form on the result of gluing $r$ copies of this surface along these ends, which will be the $1$-form $\alpha^{R}$.
\begin{exercise}
Check that these data satisfy Conditions \eqref{eq:condition_data_annulus-1}-\eqref{eq:condition_data_annulus-4}.
\end{exercise}

We now choose data $\alpha^{R}$, $H^{R}$, and $J^{R}$ varying smoothly for $R \in (0,+\infty)$ which, when restricted to a neighbourhood of $0$ and $\infty$ is obtained by the gluing construction, and which satisfy Conditions \eqref{eq:condition_data_annulus-1}-\eqref{eq:condition_data_annulus-4}; to see that this can indeed be achieved, the key observation is that the space of data satisfying these properties is contractible, so there is no obstruction to extending the choice from neighbourhoods of $R = \{ 0, +\infty \}$ to arbitrary $R$.

Given such data, and a point $\vq \in \sL^{r} \Q$, we define $ \Ann^{R}_{r}(\vq) $ to be the space of finite energy maps
\begin{equation}
 w \co \Sigma^{R}_{r} \to \TQ
\end{equation}
which map $\partial^{0} \Sigma$ to $\Q$ and $\partial^{i} \Sigma $ to $\Tq[i]$  and solve Equation \eqref{eq:equation_on_annulus}. 
\begin{exercise}
Imitating the proof of Lemma \ref{lem:compactness_negative_cylinders}, show that the moduli space $ \Ann^{R}_{r}(\vq) $ is compact if $R \neq \infty$.
\end{exercise}

\subsection{Evaluation maps}
Taking the union over all parameters $R$ and points $\vq$, we obtain a moduli space of annuli:
\begin{equation} \label{eq:annuli_arbitrary_boundary}
 \Ann_{r}(\sL^{r} \Q) \equiv \bigcup_{\vq \in  \sL^{r} \Q} \bigcup_{R \in [0,+\infty]}  \Ann^{R}_{r}(\vq) .
\end{equation}
\begin{exercise}
Show that, for generic Floer data, $  \Ann_{r}(\sL^{r} \Q)  $ is a smooth manifold with boundary of dimension $rn+1$, defining a cobordism between $\Ann^{0}_{r}(\sL^{r} \Q)  $  and $ \Ann^{\infty}_{r}(\sL^{r} \Q) $. 
\end{exercise}

Since $\Ann_{r}^{R}(\sL^{r} \Q)$ is compact for $R \neq \infty$, we define
\begin{equation}
  \Annbar_{r}^{R}(\sL^{r} \Q) \equiv \Ann_{r}^{R}(\sL^{r} \Q).
\end{equation}
For $R = \infty$, we use instead the Gromov-Floer compactification
\begin{equation} \label{eq:ann_bar}
\Annbar_{r}^{\infty}(\sL^{r} \Q) \equiv \bigcup_{x \in \Orbit(H^{-})}  \Cylbar(x) \times \Cyl(x,\vq).
\end{equation}
By construction, the complement of $\Ann_{r}^{\infty}(\sL^{r} \Q)   $ in $ \Annbar_{r}^{\infty}(\sL^{r} \Q) $ is covered by smooth manifolds of dimension strictly smaller than $nr$.  The compactification of $ \Ann_{r}(\sL^{r} \Q)  $ is defined to be:
\begin{equation}
  \Annbar_{r}(\sL^{r} \Q)  \equiv \bigcup_{R \in [0,\infty]}  \Annbar_{r}^{R}(\sL^{r} \Q),
\end{equation}
equipped with the Gromov topology.

Consider the evaluation map
\begin{equation} \label{eq:evaluation_annuli}
 \ev \co  \Annbar_{r}(\sL^{r} \Q)  \to  \sL\Q,
\end{equation}
which is defined for $R = 0$ in Equation \eqref{eq:evaluate_thin_annulus_to_circle}, and for other values of $R$ by restricting to $\partial^{0} \Sigma^{R}   $, which is naturally identified with the circle by using polar coordinates.

Since $\Annbar_{r}(\sL^{r} \Q)  $ is compact, there is a uniform bound on  the $C^1$ norm of the curves in the image of Equation \eqref{eq:evaluation_annuli}, which implies, as in Lemma \ref{lem:evaluation_map_good_r_large}, that we may choose $r'$ large enough so that, for each $w \in \Annbar_{r}(\sL^{r} \Q) $ and for each interval $I \subset S^1$ of length less than $1/r'$,  the length in $\Q$ of $\ev(w)(I)$  is less than $\delta  $, where $\delta$ is the constant fixed in Lemma \ref{lem:moduli_positive_discs_degree_1}. We obtain an evaluation map
  \begin{align}
  \ev_{r'} \co  \Annbar_{r}(\sL^{r} \Q)  & \to  \sL^{r'}\Q, \\
w & \mapsto \left( \ev(w)(0), \ev(w)\left(\frac{1}{r'}\right), \ldots, \ev(w)\left(\frac{r'-1}{r'}\right) \right),
  \end{align}
with the property that we have a homotopy commutative diagram:
\begin{equation}
  \xymatrix{\Annbar_{r}(\sL^{r} \Q) \ar[r]^{\ev_{r'}} \ar[dr]^{\ev}&    \sL^{r'}\Q \ar[d] \\
 & \sL \Q.}
\end{equation}
\begin{exercise}
  Show that the restriction of $\ev_{r'}$ to $ \Ann^{\infty}_{r}(\sL^{r} \Q) $ is given by the composition
\begin{equation}
   \coprod_{x \in \Orbit(H^{-})}  \Cyl(x) \times \Cyl(x,\vq) \to  \coprod_{x \in \Orbit(H^{-})}  \Cyl(x)  \to \sL^{r'} \Q,
\end{equation}
where the first map projects to the first factor, and the second is the evaluation map from Lemma \ref {lem:evaluation_map_good_r_large}.
\end{exercise}
The following result is a small generalisation of Exercise \ref{ex:homotopy_different_inclusions}:
\begin{exercise}
Recall that $ \Ann^{0}_{r}(\sL^{r} \Q) \cong \sL^{r} \Q $. Show that the restriction of $\ev_{r'}$ to $ \Ann^{0}_{r}(\sL^{r} \Q)$ is homotopic to the composition of iterates of $\iota$
  \begin{equation}
   \xymatrix{ \Ann^{0}_{r}(\sL^{r} \Q) \ar[r]^{\pi_{\vq}} &  \sL^{r} \Q \ar[r]^{\iota} &   \sL^{r+1} \Q \ar[r]^{\iota} &  \cdots \ar[r]^{\iota} &   \sL^{r'} \Q . }
\end{equation}
\end{exercise}

\subsection{Orienting the moduli space of annuli}
\label{sec:orient_annuli}
By construction, the moduli space of annuli admits a natural map
\begin{equation}
\pi_{\vq} \co   \Ann_{r}(\sL^{r} \Q)   \to  \sL^{r} \Q
\end{equation}
which is the projection to the space of parameters in Equation \eqref{eq:annuli_arbitrary_boundary}.

We shall presently see that this cobordism is naturally oriented relative to $\sL^{r} \Q $, i.e. that we a natural isomorphism
\begin{equation} \label{eq:relative_orientation_ann^0}
  | \Ann_{r}(\sL^{r} \Q)  |  \cong \pi_{\vq}^{*} | \sL^{r} \Q |.
\end{equation}

We start by constructing a  natural relative orientation  for $\Ann^{\infty}_{r}(\sL^{r} \Q)$,  which is compatible with the construction of the operations $\Fam$ and $\Vit$.  Let $(v,u)$ be an element of this moduli space converging to an orbit $x$. The orientation line of the moduli space at this point admits a natural decomposition:
\begin{equation}
   \det( \Ann^{\infty}_{r}( \sL^{r} \Q ) )  \cong \det(D_v)  \otimes \det(D_u)  \otimes   \pi_{\vq}^{*} \det( \sL^{r} \Q ).
\end{equation}
The existence of a relative orientation is equivalent to an identification of $   \det(D_v)  \otimes \det(D_u)  $ with $\bR$. Recall that we have a natural trivialisation of $x^{*}(T\TQ)$ coming from Lemma \ref{lem:trivalisation_up_to_htpy}; this trivialisation extends to $u^{*}(T\TQ)$ and $v^{*}(T\TQ)$, which allow us to associate to the boundary conditions the loops $ I \Lambda_{\ev(\vq)} $ and $ \Lambda_{\ev(\vq)}^{-1}$ whose Maslov index is $\pm w(\ev(\vq))$; more precisely, the maps $u$ and $v$ determine homotopies between the boundary conditions and these Lagrangian loops. Using the isomorphisms which allowed us to construct the maps $\Fam$ and $\Vit$, we obtain the following identifications: 

\begin{alignat}{4} \label{eq:orient_infty-ann-1}
  |\det(D_v)|  \otimes |\det(D_u)|  & \cong   |\det(D_v)|  \otimes |\bC^{n}| \otimes |\bC^{n}|^{-1}  \otimes |\det(D_{v})| &&  \\
& \cong   |\det(D_v)| \otimes |\ro^{+}_{x}|   \otimes  |\bC^{n}|^{-1} \otimes |\ro_{x}| \otimes  |\det(D_{I \Lambda_{\ev(\vq)}})|  &&  \\  \label{eq:two_half_cylinders_to_two_discs}
& \cong   |\det(D_{\Lambda_{\ev(\vq)}^{-1}})|   \otimes  |\bC^{n}|^{-1}  \otimes  |\det(D_{I \Lambda_{\ev(\vq)}})|. & & 
\end{alignat}
The second step used  Equation \eqref{eq:dual_orientation_lines_orbit}. Using the complex orientations of  $\bC^{n}$, together with the isomorphism 
\begin{equation}
   \det(D_{I \Lambda_{\ev(\vq)}^{-1}}) \cong \det(D_{\Lambda_{\ev(\vq)}^{-1}})
\end{equation}
from Exercise \ref{ex:deform_zero-section-to-cotangent} and Equation \eqref{eq:isomorphism_inverse_loop},  we obtain the desired orientation of $  \det(D_v)  \otimes \det(D_u)   $.

We orient the moduli space of annuli $ \Ann_{r}(\sL^{r} \Q)   $, relative to $\sL^{r} \Q  $, in essentially the same way.  Given an element $w \co \Sigma^{R}_{r} \to \TQ$ of this parametrised moduli space, we have a canonical isomorphism
\begin{equation} \label{eq:isomorphism_tangent_moduli_annuli}
  \det( \Ann^{\infty}_{r}( \sL^{r} \Q ) )  \cong \det(D_w)  \otimes  \det([0,+\infty))  \otimes   \pi_{\vq}^{*} \det( \sL^{r} \Q ).
\end{equation}
Fixing the usual orientation on $[0,+\infty)  $,  it remains to trivialise $\det(D_w)$. We shall do this by first considering more general Lagrangian boundary conditions for annuli:

 Over the product of two copies of the loop space of the Grassmannian of Lagrangians in $\bC^{n}$
 \begin{equation}
   \sL \Gr(\bC^n) \times \sL \Gr(\bC^n)
 \end{equation}
we have, for each positive real number $R$, a local system $\eta^{R}$ whose fibre at a pair of loops $(\Lambda_0, \Lambda_1)$ is the space of orientations of the determinant line of the Cauchy-Riemann operator 
\begin{equation}
  D^{R}_{\Lambda_0^{-1}, \Lambda_1} \co W^{1,p}( (\Sigma^{R}, \partial^{0} \Sigma^{R},   \partial^{1} \Sigma^{R}) , (\bC^{n},\Lambda_0^{-1}, \Lambda_1) ) \to L^{p}( \Sigma^{R}, \Omega^{0,1} \otimes \bC^{n})
\end{equation}
with boundary condition given on the boundary component $  \partial^{1} \Sigma^{R} $ by the loop $\Lambda_1$, and on the component $ \partial^{0} \Sigma^{R} $ by the loop $\Lambda_0^{-1} $ which is traversing $\Lambda_0$ backwards (here, the boundary of $\Sigma^{R}$ is given its natural orientation).
\begin{lem} \label{lem:orientation_bundle_moduli_annuli_trivial_diagonal}
The restriction of $\eta^{R}$ to the subset consisting of loops $(\Lambda, I \Lambda)$ is trivial.
\end{lem}
\begin{proof}
Degenerate the annulus by considering the limit $R = \infty$. Gluing theory induces an isomorphism
 \begin{equation}
   \eta^{R} \cong \det(D_{\Lambda^{-1}_{0}}) \otimes \det^{-1}(\bC^{n}) \otimes \det(D_{\Lambda_{1}}) .
 \end{equation}

Assuming that $\Lambda_1 =  I \Lambda_0$, the homotopy provided by Exercise \ref{ex:deform_zero-section-to-cotangent} allows us to use  Lemma \eqref{eq:isomorphism_inverse_loop}, and we conclude the triviality of $\eta^{R}$.
\end{proof}
\begin{cor}
If $w$ is an element of $\Ann_{r}(\sL^{r} \Q)  $, the determinant line $\det(D_w)$ is trivial.
\end{cor}
\begin{proof}
Glue $D_{w}$ to  the operators $D_{x_{i-1,i}}$ associated to the chords labelling the strip-like ends of $ \Sigma^{R}_{r}  $. As in Equation \eqref{eq:linearize_d_u_glue_half-planes}, we obtain an isomorphism
\begin{equation}
  D_{w} \cong D_{ \Lambda_{\ev(\vq)}^{-1}, I \Lambda_{\ev(\vq)}}.
\end{equation}
Lemma \ref{lem:orientation_bundle_moduli_annuli_trivial_diagonal} provides a trivialisation of the right hand side.
\end{proof}

Finally, we orient the moduli space of degenerate annuli corresponding to $R=0$: we start by recalling that the projection of $\Ann^{0}_{r}(\sL^{r} \Q)  $ to $ \sL^{r} \Q $ is a diffeomorphism, which induces the isomorphism in Equation \eqref{eq:relative_orientation_ann^0}. This isomorphism induces an orientation on each fibre $ \Ann^{0}_{r}(\vq)   $ which is the natural (positive) orientation of the point.   It is useful to express this positive orientation in terms of the linearisation: $ \Ann^{0}_{r}(\vq)    $  is a product of moduli spaces of discs, which we orient using the product orientation.  Each moduli space of discs that we consider consists of a unique point which is given its canonical, positive orientation coming from the fact that the $\dbar$ operator is rigid.

\subsection{Discrepancy of the orientation at the boundary}
In this section, we compare the relative orientations of  $\Ann^{0}_{r}(\sL^{r} \Q)  $, and its induced relative orientation as a boundary component of  $ \Ann^{+}_{r}(\sL^{r} \Q) $.  To state the result precisely, denote by $ \sL^{r}_{0} \Q $  the component of $\sL^{r} \Q   $ corresponding to piecewise geodesics along which $\Q$ is orientable, and $  \sL^{r}_{-1} \Q $ the component of loops along which $\Q$ is non-orientable. 
\begin{lem} \label{lem:orientation_boundary_annuli}
  The relative orientation  of $  \Ann^{0}_{r}(\sL^{r}_{w} \Q )  $ differs from the induced orientation as a boundary stratum of $  \Ann_{r}( \sL^{r}_{w} \Q ) $  by a sign whose parity is
  \begin{equation} \label{eq:sign_difference_boundary}
 \frac{(n+w)(n+w-1)}{2} + 1
  \end{equation}
\end{lem}
\begin{rem}
The case $w=0$ was studied, in a slightly different setting, by Fukaya, Oh, Ohta, and Ono in \cite{FOOO-sign}.
\end{rem}

To compare the two orientations, we introduce for each point $\vq \in \sL^{r} \Q$  a Cauchy-Riemann operator $ D_{  \vq}^{\Ann^{R}_{r}} $ on the space of maps from $\Ann^{R}_{r}$ to $\bC^{n}$, with Lagrangian boundary conditions $ \Lambda_{\ev(\vq)}^{-1} $ on $\partial^{0} \Sigma^{R}$ and $I \Lambda_{\ev(\vq)}$ on the other boundary component. By gluing, we obtain a short exact sequence
\begin{equation} \label{eq:fibre_product_Ann_S}
  \ker( D_{  \vq}^{\Ann^{R}_{r}} )  \to \ker(D_{\Lambda_{\ev(\vq)}^{-1}})    \oplus \ker( D_{I \Lambda_{\ev(\vq)}} )  \to \bC^{n},
\end{equation}
which yields an isomorphism
\begin{equation}
  \det(D_{\Lambda_{\ev(\vq)}^{-1}})    \otimes  \det(D_{\Lambda_{\ev(\vq)}})  \otimes  \det^{-1}(\bC^{n}) \to \det(D_{ \vq}^{\Ann^{R}_{r}}).
\end{equation}

\begin{exercise}
  Let $(v,u)$  be an element of the moduli space $ \Ann^{\infty}_{r}(\sL^{r} \Q)  $, and $v \#_{S} u  $ the map from an annulus obtained by gluing. Prove the commutativity of the following diagram, where the vertical maps are defined by gluing, the top horizontal map is Equation \eqref{eq:two_half_cylinders_to_two_discs}, and the bottom horizontal map is obtained  by applying the linearisation:
\begin{equation} \label{eq:commutativity_diagram_annulus_argument}
  \xymatrix{|\det(D_v)| \otimes |\det(D_u)|  \ar[r] \ar[d] &  |\det(D_{\Lambda_{\ev(\vq)}^{-1}})|   \otimes  |\bC^{n}|^{-1}  \otimes  |\det(D_{I \Lambda_{\ev(\vq)}})|  \ar[d]  \\
|\det(D_{v \#_{S} u})| \ar[r] &   |\det(D_{  \Lambda_{\ev(\vq)}}^{\Ann^{S}})| .}
\end{equation}
\end{exercise}

We need an elementary result in linear algebra before completing the proof of Lemma \ref{lem:orientation_boundary_annuli}:
\begin{exercise}
Use multiplication by $i$ to induce an orientation on $ i \bR^{n}$ from an orientation on $\bR^{n}$. Show that the resulting orientation of the direct sum is canonically defined, and differs from  the complex orientation of $\bC^{n}$  by a sign whose parity is
  \begin{equation} \label{eq:gauss_sum}
    \frac{n(n-1)}{2}.
  \end{equation}
\end{exercise}

We can now complete the proof of the main result of this section:
\begin{proof}[Proof of Lemma \ref{lem:orientation_boundary_annuli}]
We first consider the case where $T \Q$ is orientable along $\ev(\vq)$: in this case, the loop of Lagrangians $\Lambda_{\ev(\vq)}$ takes the constant value $\bR^{n}$. The second arrow in Equation \eqref{eq:fibre_product_Ann_S} therefore corresponds to the map of orientation lines associated to the isomorphism of real vector spaces:
\begin{equation}
  \bR^{n} \oplus i \bR^{n} \to \bC^{n}.
\end{equation}
The commutativity of Diagram \eqref{eq:commutativity_diagram_annulus_argument}, implies that the orientation of $  \Ann_{r}( \sL^{r}_{w} \Q ) $ is induced by the orientation of the fibre of this isomorphism corresponding to the choice of complex orientation on the target, and the natural orientation of the direct sum of two copies of the same vector space on the source. The difference with the natural orientation of a $0$-dimensional vector space is given by Equation \eqref{eq:gauss_sum}. Since this is the orientation used for $ \Ann^{0}_{r}(\sL^{r} \Q)  $, Equation \eqref{eq:sign_difference_boundary} follows immediately in the case $w=0$; the additional $+1$ comes from the fact that we have oriented $\Ann^{+}$ so that a positive tangent vector points away from $\Ann^{0}$, which is the opposite of the convention for boundaries.

To prove the result in the non-orientable situation, recall that the trivialisation is chosen so that the loop $ \Lambda_{\ev(\vq)} $ has Maslov index $1$. We choose a specific representative which splits as the product of a loop of Maslov index $1$ in $\bC$ with a constant Lagrangian in $\bC^{n-1}$. We can then choose the loop $\Phi$ in Equation  \eqref{eq:model_isomorphism_inverse_loop} to be represented by the matrices
\begin{equation}
  \begin{pmatrix}
    e^{- 2 \pi t i } & 0 & \cdots &0 \\
0 & 1 & \cdots & 0 \\
\vdots & \vdots & \ddots & \vdots \\
0 & 0 & \cdots &1
  \end{pmatrix}.
\end{equation}
Note that Equation \eqref{eq:gauss_sum}, applied in dimension $n-1$, yields the first term in Equation \eqref{eq:sign_difference_boundary} whenever $w=-1$. It remains therefore to show that, in dimension $1$,  the orientation in the middle term of Equation \eqref{eq:fibre_product_Ann_S} agrees with the complex orientation on the right hand side. To see this, we first observe that  $ D_{\Lambda_{\ev(\vq)}^{-1}}$  is rigid. Gluing theory identifies $ \ker(D_{I \Lambda_{\ev(\vq)}}) $  with the kernel of the evaluation map
\begin{equation}
  \ker(D^{-}_{\Phi^{-1}}  ) \to \bC,
\end{equation}
where $\Phi^{-1}$ has index $1$, and the evaluation map takes place at the point $z_0 \in \bC \bP^{1}$, and we equip $\ker(D_{I \Lambda_{\ev(\vq)}})   $ with the induced orientation. To show that the second map in Equation \eqref{eq:fibre_product_Ann_S} preserves orientations, observe that the evaluation map
\begin{equation}
  \ker(D_{I \Lambda_{\ev(\vq)}})  \to \bC
\end{equation}
corresponds to the evaluation map from $   \ker(D^{-}_{\Phi^{-1}}  ) $ to the fibre at $z_{\infty}$. Since this map is complex linear, we conclude that for non-orientable boundary conditions,  the second arrow in Equation \eqref{eq:fibre_product_Ann_S}  preserves orientation in dimension $1$, and hence, in higher dimensions, that the sign is
\begin{equation}
   \frac{(n-1)(n-2)}{2}.
\end{equation}
As in the orientable case, there is an additional term $1$ coming from our chosen orientation of the abstract moduli space $   \Ann_{r}$, for which the boundary stratum $ \Ann_{r}^{0}$ acquires a negative orientation.
\end{proof}

\subsection{Rigidifying moduli spaces}

Let $y$ be a critical point of $f^{r}$, and $y'$ a critical point of $f^{r'}$. We consider the fibre product
\begin{equation} \label{eq:fibre_product_annuli_ascend_descend}
  \Ann_{r}(y',y) \equiv W^{s}(y') \times_{\ev_r' }  \Ann^{+}_{r}(\sL^{r} \Q) _{\pi}\times W^{u}(y).
\end{equation}
If the Floer data on the annuli and the Morse functions $f^{r}$ and $f^{r'}$ are chosen generically, this fibre product is transverse, hence is a manifold with boundary. In that case, the dimension of $\Ann^{+}_{r}(\sL^{r} \Q) _{\pi}\times W^{u}(y)  $  is $\ind(y)+1 $.  Given that the codimension of $W^{s}(y') $ is $\ind(y')$, we conclude that
\begin{equation} \label{eq:dimension_ann_y,y'}
  \dim(  \Ann_{r}(y',y)  ) = \ind(y) - \ind(y') +1.
\end{equation}
\begin{exercise}
 Use Equation \eqref{eq:relative_orientation_ann^0} to construct a natural isomorphism
 \begin{equation} \label{eq:orient_fibre_product_annuli}
 |  \Ann_{r}(y',y) | \otimes  \ro_{y} \otimes |W^{s}(y')   | \cong | \sL^{r'} \Q|.
 \end{equation}
\end{exercise}

We introduce the standard compactification
\begin{equation}
  \Annbar_{r}(y',y) \equiv  \bar{W}^{s}(y') \times_{\ev_r' }  \Annbar_{r}(\sL^{r} \Q) _{\pi}\times \bar{W}^{u}(y) .
\end{equation}

\begin{lem} \label{lem:annuli_manifold_boundary}
For generic Floer data, $   \Annbar_{r}(y',y) $ is empty whenever 
\begin{equation} \label{eq:difference_index_negative}
  \ind(y) - \ind(y') +1 < 0,
\end{equation}
and is a manifold with boundary whenever $   \ind(y) - \ind(y') +1 $  equals $0$ or $1$.
\end{lem}
\begin{proof}
Recall that the dimension of $\Ann_{r}(\sL^{r} \Q)   $ is $nr +1$, and that its complement  in $\Annbar_{r}(\sL^{r} \Q)  $  is covered by manifolds of dimension smaller than or equal to $  nr-1$. Assuming regularity, the computation leading to Equation \eqref{eq:dimension_ann_y,y'} implies that the inclusion
\begin{equation}
  \Ann_{r}(y',y)   \subset W^{s}(y') \times_{\ev_r' }  \Annbar_{r}(\sL^{r} \Q) _{\pi}\times W^{u}(y)
\end{equation}
is an equality whenever $ \ind(y) - \ind(y') +1 < 2$. 

Next, we use the fact that whenever $ \Treebar(y',y'_1) $ is non-empty, $\ind(y') \leq \ind(y'_1)  $, with equality only if $y' = y'_{1}$, and similarly for $(y_0,y)$ to conclude that the transversality of the fibre product in Equation \eqref{eq:fibre_product_annuli_ascend_descend} implies that $  \Ann_{r}(y'_1,y_0) $ is empty; hence that $ \Annbar_{r}(y',y) $ is empty if Inequality \eqref{eq:difference_index_negative} holds.

Assuming that $ \ind(y) +1 = \ind(y') $, the same argument shows that  $  \Annbar_{r}(y',y) \setminus   \Ann_{r}(y',y)$ is empty, so the result follows from the fact that $\Ann_{r}(y',y)  $ is a $0$-dimensional manifold. If $\ind(y) = \ind(y') $, then the only contributions to $ \Annbar_{r}(y',y) $ which are not empty are:
\begin{equation}
\Ann_{r}(y',y)  \textrm{, } \bigcup_{\ind(y_1') = \ind(y) +1}  \Tree(y',y'_1) \times   \Ann_{r}(y'_1,y)  \textrm{ and } \bigcup_{\ind(y) = \ind(y_0) +1} \Annbar_{r}(y,y_0) \times \Tree(y_0,y).
\end{equation}
The first space is $1$-dimensional manifold with boundary since it is a transverse fibre product, and the other two spaces are $0$-dimensional.  The implicit function theorem implies that, if the fibre product $\Tree(y',y'_1) \times   \Ann_{r}(y'_1,y)   $ is transverse, then a neighbourhood in $ \Annbar_{r}(y',y) $  is a manifold with boundary.
\end{proof}

\subsection{Maps defined by annuli with fixed modular parameter}
Define $\Ann_{r}^{R}(y',y')$ to be the fibre product
\begin{equation}
 \Ann_{r}^{R}(y',y') \equiv  W^{s}(y') \times_{\ev_r' }  \Ann^{R}_{r}(\sL^{r} \Q) _{\pi}\times W^{u}(y).
\end{equation}
Combining Equations \eqref{eq:spliting_tangent_at_crit_point} and \eqref{eq:orient_fibre_product_annuli}, we find a canonical isomorphism
\begin{equation}
     |\det(\Ann_{r}^{R}(y',y))|  \otimes  \ro_{y} \cong \ro_{y'}.
\end{equation}

Assume now that $\ind(y) = \ind(y')$, so that the moduli spaces $\Ann^{R}_{r}(y',y)  $ are $0$-dimensional if they are regular; each element $(\gamma', w, \gamma)$ of these moduli spaces therefore defines a map
\begin{equation}
   \ro_{y} \to \ro_{y'}.
\end{equation}
 The annulus $w$ also induces a map on the local system $\eta$ on $\sL \Q$, hence on the restrictions to $\sL^{r} \Q$ and $\sL^{r'} \Q$; composition with the maps induced by $\gamma$ and $\gamma'$ defines a map
\begin{equation}
   \eta_{y} \to \eta_{y'}.
\end{equation}
By taking the tensor product of these two maps, we obtain 
\begin{equation}
  \iota_{(\gamma',w,\gamma)} \co  \ro_{y} \otimes \eta_{y} \to  \ro_{y'} \otimes \eta_{y'}.
\end{equation}

Taking the direct sum over all rigid elements, we define
\begin{align}
  \iota^{R} \co CM_{*}(f^{r}; \eta) & \to CM_{*}(f^{r'}; \eta)  \\
\iota^{R} | \ro_{y} \otimes \eta_{y} & \equiv  \bigoplus_{\ind(y') = \ind(y)} \sum_{  (\gamma',w,\gamma) \in  \Ann_{r}^{R}(y',y)}   \iota_{(\gamma',w,\gamma)}.
\end{align}

\begin{exercise}
Use Lemma \ref{lem:annuli_manifold_boundary} to  show that $\iota^R$ is a chain map if the moduli spaces of annuli with modulus $R$ are all regular.
\end{exercise}

In general, it may not be possible to achieve regularity for every value of $R$, but the parametrised moduli space that we obtain by taking the union over $R$ gives a homotopy. We shall therefore be particularly interested in the cases $R=0,+\infty$. For $R=\infty$, we have a splitting
\begin{equation}
 \Ann_{r}^{\infty}(y',y) \equiv   \coprod_{x \in \Orbit(H^{-})} \Broke(y',x) \times  \Broke(x,y). 
\end{equation}
\begin{exercise}
Assuming that $\ind(y) = \ind(y')$, show that  if both $  \Broke(y',x) $ and $  \Broke(x,y) $ are non-empty, then $\deg(x) =n - \ind(y)$.
\end{exercise}

The moduli space $  \Ann_{r}^{\infty}(y',y') $ therefore consists of the product of the moduli spaces used to define $ \Vit_{r'} $ and $\Fam^{r} $.   Since the choices made orienting this moduli space in Equations \eqref{eq:orient_infty-ann-1}-\eqref{eq:two_half_cylinders_to_two_discs} use the choice made in defining $ \Vit_{r'} $ and $ \Fam^{r} $, we conclude that
\begin{equation} \label{eq:iota_infty_composition}
  \iota^{+\infty} = \Vit_{r'} \circ \Fam^{r}.
\end{equation}

For $R = 0$,  Lemma \ref{lem:ann_0_htpic_to_inclusion} shows that the composition of $r'-r$ iterates of $\iota$ is homotopic to $\iota^0$, up to an overall sign coming from choices of orientations. By Lemma \ref{lem:orientation_boundary_annuli}, this sign depends on the component of $\sL^{r} \Q$.  To state the sign precisely, we split the Morse complex as a direct sum
\begin{equation}
  CM_{*}(f^{r}; \eta) \equiv CM_{*}^{0}(f^{r}; \eta) \oplus CM_{*}^{-1}(f^{r}; \eta)
\end{equation}
where the first summand consists of those piecewise geodesics along which $\Q$  is orientable.
\begin{lem}
The restriction of $\iota^0$ to $CM_{*}^{w}(f^{r}; \eta) $ is homotopic to
\begin{equation} \label{eq:iota_0_and_inclusion}
  (-1)^{  \frac{(n+w)(n+w-1)}{2} } \iota \circ  \iota \circ \cdots  \circ \iota.
\end{equation} \qed
\end{lem}

\subsection{Construction of the homotopy} \label{sec:constr-homot}
We now repeat the construction of the previous section, removing the constraint on the modulus of annuli. We start with the isomorphism
\begin{equation}
     |\det(\Ann_{r}(y',y))|  \otimes  \ro_{y} \cong \ro_{y'},
\end{equation}
and assume that $\ind(y') = \ind(y) +1  $, which implies that $ \Ann_{r}(y',y) $ is a $0$-dimensional manifold. We obtain a map
\begin{equation}
   \cH_{(\gamma',w,\gamma)} \co \ro_{y} \otimes \eta_{y}  \to  \ro_{y'} \otimes  \eta_{y'}
\end{equation}
associated to every element $ (\gamma',w,\gamma) \in  \Ann_{r}(y',y) $ by tensoring the induced map on orientation lines with the parallel transport map on $\eta$.  Define
\begin{align}
  \cH^{\Vit \circ \Fam} \co CM_{*}(f^{r}; \eta)& \to  CM_{*}(f^{r'}; \eta)  \\
 \cH^{\Vit \circ \Fam}  \vbar  \ro_{y} \otimes \eta_{y} & \equiv \bigoplus_{\ind(y') = \ind(y) + 1} \sum_{  (\gamma',w,\gamma) \in  \Ann_{r}(y',y)}    \cH_{(\gamma',w,\gamma)}.
\end{align}

\begin{exercise} \label{ex:0-and-infty-htpic}
  Show that $   \cH^{\Vit \circ \Fam} $ defines a homotopy between $\iota^{0}$ and $\iota^{\infty}$.
\end{exercise}

\begin{proof}[Proof of Proposition \ref{lem:compositions_loop_agree_up_to_sign}]
  Exercise \ref{ex:0-and-infty-htpic} implies that the maps induced by $\iota^0$ and $\iota^{\infty}$ are equal on homology. The first of these agrees with the composition of inclusion maps, up to the overall sign in Equation \eqref{eq:iota_0_and_inclusion}. By Equation \eqref{eq:iota_infty_composition}, $\iota^{\infty}$ agrees with the composition of $\Vit_{r'}$ and $\Fam^{r}$. The desired result follows immediately. 
\end{proof}

\section{Guide to the Literature}

\subsection{Lagrangian Floer cohomology}

The standard references for Lagrangian Floer cohomology are Fukaya, Oh, Ohta, and Ono's treatise in two volumes \cite{FOOO} which is mostly concerned with the group associated to an embedded closed Lagrangian in a compact symplectic manifold, and Seidel's monograph  \cite{seidel-Book} which considers collections of such Lagrangians in exact symplectic manifolds. Both books study much deeper properties of Lagrangian Floer cohomology than what is used in this Chapter, since they are concerned with the construction of algebraic structures on these groups.

The study of Lagrangian Floer theory for non-compact Lagrangians leads to related, but not necessarily isomorphic theories depending on the required behaviour of the Floer equation at infinity (see \cite{FSS2} for a survey). The closest approach to the one we are taking here considers \emph{infinitesimal} perturbations at infinity, which originated in Oh's work \cite{oh-gap} (a gap in the discussion of compactness in this paper was later addressed in \cite{oh-fix-gap}), which was related in \cite{NZ} to constructible sheaves on the base.

For completeness, we mention the two other approaches: the first is called \emph{wrapped Floer cohomology}, and can be obtained analogously to symplectic cohomology by taking a direct limit over Lagrangian Floer cohomology groups defined using Hamiltonians of increasing slope \cite{ASeidel}. For cotangent fibres, wrapped Floer cohomology was first computed, using a different definition, by Abbondandolo, Portaluri, and Schwarz in \cite{APS}. As for the case of symplectic cohomology, their computation is correct up to sign and wrapped Floer cohomology is isomorphic to the homology of a local system on the based loop space obtained by transgressing the second Stiefel-Whitney class of the base \cite{A-loops}.

The second approach is inspired by ideas of mirror symmetry \cite{kont-ENS}, and uses an open book decomposition of the boundary to define a class of adapted Hamiltonians, with respect to which one can compute Floer cohomology, see \cite{A-HMS-toric,seidel-12}

\subsection{Family Floer cohomology}

The idea of studying Floer cohomology for families of Lagrangians goes back to Kenji Fukaya, who intended to apply it in the subtle case of Lagrangian foliations of closed symplectic manifolds \cite{Fukaya-family}.  Together with Ivan Smith, Fukaya observed that applying such a theory to the easier case of cotangent fibres would yield partial results towards Arnold's nearby Lagrangian conjecture, but the theory of Lefschetz fibrations \cite{FSS} gave a way of bypassing the technical problems of defining family Floer cohomology, so these results were announced but never appeared in writing.

\subsection{Degenerations of annuli}
As noted in the introduction, the use of moduli space of annuli in symplectic topology appeared essentially simultaneously in Biran and Cornea's work on enumerative invariants associated to monotone Lagrangians \cite{BC}, Fukaya, Oh, Ohta and Ono's study of mirror symmetry for toric manifolds \cite{FOOO-sign}, and the development of a generation criterion for wrapped Fukaya categories \cite{A-generate}. In the first two instances, the degeneration of annuli is an instantiation of Poincar\'e duality in the Floer cohomology of closed Lagrangians.

\chapter{Viterbo's theorem: Isomorphism}
\label{cha:viterbos-theorem-II}

\section{Introduction}

The main result of this Chapter is proved in Section \ref{sec:comp-floer-homol-II}:
\begin{thm} \label{thm:Fam_surjective}
$\Fam$ is surjective.
\end{thm}
\begin{cor}
$\Fam$ and $\Vit$ are isomorphisms. 
\end{cor}
\begin{proof}
$\Fam$ is an isomorphism because it is both surjective and injective (see Corollary \ref{cor:fam-injective}). Theorem \ref{thm:compose_fam_vit_iso} therefore implies that $\Vit$ is a left inverse to an isomorphism, hence that $\Vit$ is itself an isomorphism.
\end{proof}

Naively, we would proceed by showing that the composition $\Fam \circ \Vit$ agrees with the identity on $SH^{*}(\TQ; \bZ) $, up to an overall sign. Unfortunately, this composition is difficult to interpret geometrically because the half-cylinders counted in the definition of $\Vit$ are required to map the boundary to the zero-section, while those in the definition of $\Fam$ map the boundary to cotangent fibres. We cannot directly glue moduli spaces on which different boundary conditions have been imposed. 

In order to bypass this problem, we shall construct alternative maps 
\begin{equation}
    \Gam^{r} \co HF^{*}(H; \bZ)  \to HM_{-*}(f^{r}; \eta)
\end{equation}
from Floer cohomology to the Morse cohomology of finite dimensional approximations, using moduli spaces of discs with Lagrangian boundary conditions on cotangent fibres. The main difficulty is to ensure the compactness of these moduli spaces; this will require a delicate use of the integrated maximum principle (Proposition \ref{prop:no_finite_energy_outside_cpct}); in fact, were it not for the constructions of this Chapter, standard versions of the maximum principle would have sufficed. 

\begin{rem}
While it is easy to prove that the maps $\Gam^{r}$ commute with continuation maps in Floer cohomology, proving that they commute with the inclusion maps on the Morse cohomology of finite approximations requires a little bit of work. Since such results are not needed, they are not included in this Chapter.

\end{rem}

With this in mind, the proof of Theorem \ref{thm:Fam_surjective} is analogous to that of Theorem \ref{thm:compose_fam_vit_iso}. The composition $\Fam^{r} \circ \Gam^{r}$ corresponds to pairs of discs with interior punctures with Lagrangian boundary conditions on cotangent fibres. In Chapter \ref{cha:viterbos-theorem}, we glued such discs along the punctures, but we now glue instead along the boundary to obtain the complement of $r$ discs in the cylinder.  The first key idea is to construct a cobordism associated to degenerating these Riemann surfaces to cylinders which are attached to $r$ discs at interior points (see Section \ref{sec:genus-0-surfaces}). The next essential point is that the moduli space of holomorphic discs with boundary conditions on a given cotangent fibre consists only of constant discs. Taking the union over all such fibres, we can therefore represent  $\TQ$ as a parametrised moduli space of holomorphic discs. Attaching elements of this moduli space to a cylinder imposes no constraints, so we obtain a cobordism between the moduli spaces defining $ \Fam^{r} \circ \Gam^{r} $  and those defining the continuation map. Since symplectic cohomology is defined as a direct limit over continuation maps, we readily conclude that $\Fam$ is surjective. 

\begin{rem}
A potential alternative to the construction of $\Gam$ would be to show that $\Fam$ induces an isomorphism by appealing to an action filtration argument. For example, one could show that there is a commutative diagram
\begin{equation}
  \xymatrix{  HM_{*}(f^{r}; \eta)  \ar[d] \ar[r]^-{\Fam}   &HF^{-*}(H; \bZ) \ar[d] \\
H_{*}(\sL \Q; \eta) \ar[r]^-{AS} & SH^{-*}(\TQ; \bZ),}  
\end{equation}
where the bottom horizontal map is the one appearing in \cite{AS}. Note that this would involve on the left passing from a Morse function on $\sL^{r} \Q$ to one on the loop space, while on the right we would have to relate the Floer cohomology of a linear Hamiltonian to that of a quadratic Hamiltonian.  More speculatively, one could attempt to implement the idea of \cite{AS} directly in the framework of linear Hamiltonians.
\end{rem}

\section{From Floer cohomology to Morse homology via families of Lagrangians} \label{sec:from-floer-cohom}
The construction of  $\Gam^{r}$  uses exactly the same method as the construction of $\Fam$, but reversing the roles of inputs and outputs. Starting with a Hamiltonian orbit $x$, the restriction of $\Gam^{r}$ to $\ro_{x}$ is obtained by evaluating a moduli space of discs with an interior puncture equipped with a positive end, and several boundary punctures with negative ends. We shall focus on the aspects of the construction that differ from those of Section \ref{sec:from-morse-homology}; in particular, ensuring compactness of the moduli space (see Lemma \ref{lem:moduli_space_contained_in_disc_bundle}).

\subsection{Discs with a positive interior puncture and negative boundary punctures} \label{sec:punct-posit-half}
Fix a positive integer $r$, and negative strip-like ends $\{ \epsilon_{i} \}_{i=1}^{r}$  on the punctured disc $Z^{+}$,  with asymptotic conditions at the points $(0, \frac{i}{r})$ and disjoint images.  Denote by $Z^{+}_{r}$ the complement of these points:
\begin{equation}
 Z^{+}_{r} \equiv Z^{+}\setminus \{ (0, \frac{i}{r}) \}_{i=1}^{r}.
\end{equation}

As in Section \ref{sec:moduli-space-half}, we  fix some numbers that will determine the slope of the Hamiltonians we shall consider: start by choosing a positive real number $b^{+}$ which will be the slope of the time-dependent Hamiltonian associated to the interior puncture, and real numbers $\{ b_{i} \}_{i=1}^{r}$ satisfying the following properties (recall that the constants $\delta_{i}^{r}$ enter in the definition of $\sL^{r} \Q$ in Section \ref{ex:homotopy_different_inclusions}):
\begin{align} \label{eq:b_i_not_too_big_not_too_small}
\delta_{i}^{r} <  2 b_{i} & \leq 2 \delta_{i}^{r}  \\ \label{eq:slope-bound}
b^{+} & \leq \sum_{i=1}^{r} b_{i} \equiv b.
\end{align}
 Note that there is some tension between the above two inequalities; the first asserts that the numbers $b_i$ are small, and the second that their sum is large. 
\begin{exercise}
Show that for each choice of $b^{+}$, one may find $r$ large enough so that there exists a choice of real numbers $\{ b_{i} \}_{i=1}^{r} $ satisfying these properties (Hint: use Equation \eqref{eq:length_goes_to_infty}).
\end{exercise}

We continue the construction of a Floer equation on $ Z^{+}_{r}$ in complete parallel with the discussion of Section \ref{sec:moduli-space-half}. Given a family of Hamiltonians $H_{z}$, and almost complex structures $J^{+}_{z}$ on $\TQ$, parametrised by $z \in Z^{+}_{r}$, and a $1$-form $\alpha^{+}$ we consider the pseudoholomorphic curve equation
\begin{equation} \label{eq:CR-equation-positive}
  \left(du - \alpha^{+} \otimes X_{H} \right)^{0,1} = 0
\end{equation}
on maps from $ Z^{+}_{r} $ to $\TQ$.

The Floer data are assumed to satisfy the following properties:
\begin{equation} \label{eq:condition_on_alpha+}
  \parbox{35em}{ $\alpha^{+}$ is closed. Its restriction to the subset of $Z^{+}_{r}$ given by $1 \leq s$ agrees with $bdt$, the restriction to $\partial Z^{+}_{r} $ vanishes,  and the pullback of $\alpha^{+}$ under every strip-like end $\epsilon_i$ agrees with $b_{i}dt$.}
\end{equation}
The constraints on the family of Hamiltonians are as follows:
\begin{itemize}
\item The restriction of $H$ to a neighbourhood of the boundary agrees with the model Hamiltonian:
  \begin{equation} \label{eq:restriction_H_positive_s_close_to_boundary}
 H_{(s,t)}  = \Hh  \textrm{ if } s \leq 1.
  \end{equation}
\item There exists a time-dependent Hamiltonian $\{ H_{t}^{+} \}_{t \in S^{1}}$ of slope $b^{+}$ with non-degenerate orbits, such that
  \begin{equation}\label{eq:restriction_H_positive_s_close_to_infinity}
    H_{(s,t)} = \frac{H_{t}^{+}}{b} \textrm{ if } 0 \ll s.
  \end{equation}
\item There exists a function $f^{+}$ on $\Sigma$, such that
  \begin{equation} \label{eq:Hamiltonian_positive_outside_disc}
    H_{z}| \TQ \setminus \DQ \equiv   \frac{b^{+}}{b} \rho  + f^{+}\left( h -  \frac{b^{+}}{b} \rho \right).
  \end{equation}
Moreover, we assume that 
\begin{equation} \label{eq:assumption_cutoff_positive}
  \begin{aligned}
\frac{\partial f^{+}}{\partial s} & \leq 0 \\
f^{+}(s,t) & = 0 \textrm{ if } 0 \ll s \\
f^{+}(s,t) & = 1 \textrm{ if } s \leq 1.
\end{aligned}
\end{equation}
\end{itemize}
\begin{exercise}
Check that Equations \eqref{eq:restriction_H_positive_s_close_to_boundary} and \eqref{eq:restriction_H_positive_s_close_to_infinity} are compatible with Equations \eqref{eq:Hamiltonian_positive_outside_disc} and \eqref{eq:assumption_cutoff_positive}.
\end{exercise}
At the interior puncture, the asymptotic conditions for a finite energy map $u$ are given by an orbit  $x \in \Orbit(H^{+})$, so that
  \begin{equation}
    \lim_{s \to  + \infty} u(s,t) = x(t).
  \end{equation}
To ensure that the moduli space is not empty for tautological reasons, we use the following exercise:
\begin{exercise}
If the distance between $q_i$ and $q_{i+1}$ is less than $\delta_{i}^{r}$, use Equation \eqref{eq:b_i_not_too_big_not_too_small} to show that there is exactly one Hamiltonian chord of $b_i \Hh$ starting on $\Tq[i]$ and ending on $\Tq[i+1]$, and that such a chord lies in $\DQ$. (Hint: use Equation \eqref{eq:flow_of_model_Hamiltonian} relating the flow of $\Hh$ to that of a quadratic Hamiltonian, and Lemma \ref{lem:chords_are_based_geodesics}).
\end{exercise}
\begin{exercise} \label{ex:chords_on_boundary_are_outside}
If the distance between $q_i$ and $q_{i+1}$ is greater than $\delta_{i}^{r}$, show that there is at most one Hamiltonian chord of $b_i \Hh$ starting on $\Tq[i]$ and ending on $\Tq[i+1]$. If such a chord exists, show that it lies in the complement of $\DQ $.  
\end{exercise}
Finally, we assume that 
\begin{equation}
  \parbox{35em}{each  almost complex structure $J_{z}^{+}$ is convex near $\DQ$ and $\DQ[2]$.}
\end{equation}
\begin{rem}
  We only need convexity near $\DQ$ to construct the map $\Gam^{r}$, but convexity near $\DQ[2] $ will be used later to study the composition with $\Fam^{r}$.
\end{rem}

\subsection{Compactness for punctured discs}

Given a point $\vq \in  \Q^{r}$ and a Hamiltonian time-$1$ orbit $x$  of $H^{+}$, define $\Cyl(\vq, x)$ to be the moduli space of finite energy maps
\begin{align}
  u \co  Z^{+}_{r} &  \to \TQ \\
u \left( \frac{i}{r}, \frac{i+1}{r}   \right) & \subset \Tq[i] \\
\lim_{s \to +\infty} u(s,t) & = x(t),
\end{align}
satisfying Equation \eqref{eq:CR-equation-positive}. The main reason for our careful choice of Floer data is the need to prove the following result:
\begin{lem} \label{lem:moduli_space_contained_in_disc_bundle}
Every element of the moduli space $ \Cyl(\vq, x) $  has image contained in $\DQ$.
\end{lem}
\begin{proof}
Let $u$ be an element of this moduli space. Given a real number $\ell$ greater than $1$  such that $\SQ[\ell]$ is transverse to $u$,  denote by $\Sigma \subset Z^{+}_{r}$ the inverse image of $\TQ \setminus \DQ[\ell]$ under $u$, and by $v$ the restriction of $u$ to $\Sigma$.

Since $H^{+}$ is autonomous outside of $\DQ$, and has regular time-$1$ orbits, $x$ is necessarily contained in $\DQ$, which implies that $\Sigma$ is disjoint from the cylindrical end of $Z^{+}_{r}$, and hence is obtained by removing boundary marked points  from a compact Riemann surface with boundary as in Section \ref{sec:integr-maxim-princ}. We denote by $\partial^{n} \Sigma$ the inverse image of $\SQ[\ell]$, and by $\partial^{l} \Sigma$ the intersection of the boundary of $Z^{+}_{r} $ with $\Sigma$.

By Equation \eqref{eq:Hamiltonian_positive_outside_disc}, the restriction of Equation \eqref{eq:CR-equation-positive} to $\Sigma$ can be written as
\begin{equation}
    \left( du -  X_{h_0(\rho)} \otimes \alpha^{+} -  X_{h_1(\rho)} \otimes f_{+}(s) \alpha^{+} \right) ^{(0,1)} \equiv 0,
\end{equation}
where the functions $h_j$ are given by
\begin{align}
  h_{0}(\rho) & =  \frac{b^{+}}{b}\rho  \\
h_{1}(\rho) & =  h(\rho) - \frac{b^{+}}{b}  \rho.
\end{align}
\begin{figure}[h] \label{fig:model_hamiltonian-compare-linear}
  \centering
  \includegraphics{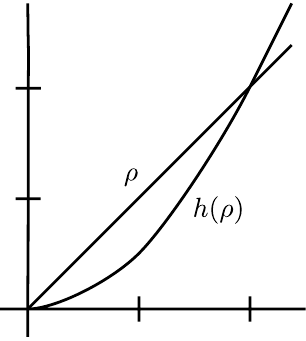}
  \caption{ }
\end{figure}
Having assumed that $1 \leq \ell$, Equation \eqref{eq:slope-bound} implies that the restriction of $h_1$  to $[\ell,+\infty)$ is increasing (see Figure \ref{fig:model_hamiltonian-compare-linear} for the comparison between  $h$ and a linear function). The remaining hypotheses for  Proposition \ref{prop:no_finite_energy_outside_cpct} (listed in Equation \eqref{eq:alpha_restricted_to_end}-\eqref{eq:positive_function_negative_form} and Equation \eqref{eq:boundary_conditions_maximum}) hold by construction.  We conclude that the inverse image of $\TQ \setminus \DQ[\ell]  $   is empty whenever $\SQ[\ell]$ is transverse to the image of $u$ and  $ 1 \leq \ell$. Sard's theorem implies that the set of real numbers $\ell$ satisfying this transversality property is dense, which implies that the image of $u$ is contained in $\DQ$.
\end{proof}
\begin{prop} \label{prop:moduli_positive_cyl_empty_boundary_outside}
  If $\vq$ lies in the complement of $\sL^{r} \Q$, the moduli space $ \Cyl(\vq, x)   $  is empty.
\end{prop}
\begin{proof}
Whenever $\vq$ lies outside $\sL^{r} \Q $, there must be a successive pair of points $q_i$ and $q_{i+1}$ whose distance is greater than $\delta_{i}^{r}$.  If there is no time-$1$ chord of $b_{i} \Hh$ starting on $\Tq[i]$ and ending on $\Tq[i+1]$, then the moduli space we are considering is tautologically empty. Otherwise, such a chord lies in the complement of $\DQ$ by Exercise \ref{ex:chords_on_boundary_are_outside}, in particular, the image of any element of $\Cyl(\vq, x)   $ intersects the complement of the unit disc bundle non-trivially. The moduli space must therefore be empty, for we would otherwise contradict Lemma \ref{lem:moduli_space_contained_in_disc_bundle}.
\end{proof}

\subsection{Orientations} \label{sec:orient-moduli-space}
Given an orbit $x$, consider the union of the moduli spaces $ \Cyl(\vq,x) $ over all piecewise geodesics
\begin{equation}
  \Cyl(\sL^{r} \Q, x) \equiv \coprod_{\vq \in \sL^{r} \Q}   \Cyl(\vq,x),
\end{equation}
which we topologise as a parametrised moduli space.

Our goal in this section is to prove the analogue of Corollary \ref{lem:vir_dim_moduli_half_cylinders}; i.e. provide a canonical orientation of  $  \Cyl(\sL^{r} \Q, x)$ relative to $o_{x}$ and the local system $\eta$. We denote the projection map to the parametrising space
\begin{equation}
  \ev_{r} \co   \Cyl(\sL^{r} \Q, x)  \to   \sL^{r} \Q,
\end{equation}
and define
\begin{equation}
  \ev \co \Cyl(\sL^{r} \Q, x)  \to   \sL \Q
\end{equation}
for the composition with the inclusion of $ \sL^{r} \Q $ into the loop space as piecewise geodesics.

We start with the canonical isomorphism
\begin{equation}
  |T_{u}\Cyl(\sL^{r} \Q, x)| \cong \ev^{*}_{r} | \sL^{r} \Q | \otimes |\det(D_u)|.
\end{equation}
Since $\sL^{r} \Q  $ is a codimension-$0$ submanifold of $\Q^{r}$, assuming that $\ev_{r}(u) = \vq$, we have an isomorphism
\begin{equation}
   |\sL^{r} \Q | \cong | T_{q_0}  \Q | \otimes  | T_{q_1}  \Q  | \otimes \cdots \otimes    | T_{q_{r-1}}  \Q  |. 
\end{equation}
Using Lemma \ref{lem:iso_negative_orientation_line}, and the isomorphism
\begin{equation} \label{eq:orientation_cotangent_fibre}
  | T_{q_i}^{*}  \Q |  \otimes  | T_{q_i} \Q |  \to \bZ
\end{equation}
induced by pairing tangent and cotangent vectors, we obtain an isomorphism
\begin{equation} \label{eq:orient_moduli_space_negative_punctured_cylinder_1}
  |T_{u}\Cyl(\sL^{r} \Q, x)|  \cong \ro_{x_0}^{-} \otimes \cdots \otimes \ro_{x_{r-1}}^{-} \otimes |\det(D_u)|.
\end{equation}
\begin{exercise}
 Show that there is a canonical isomorphism
\begin{equation} \label{eq:gluing_negative_orient_to_cylinder}
 \ro_{x_0}^{-} \otimes \cdots \otimes \ro_{x_{r-1}}^{-} \otimes |\det(D_u)| \otimes \ro_{x}[w(x)] \cong \eta_{\ev(u)}
\end{equation}
induced by gluing (Hint: imitate the proof of Lemma \ref{lem:virtual_dim_negative_half_cylinder}).
\end{exercise}

Combining Equations \eqref{eq:orient_moduli_space_negative_punctured_cylinder_1} and \eqref{eq:gluing_negative_orient_to_cylinder}, we conclude:
\begin{lem}
  There is a canonical isomorphism of graded lines
  \begin{equation} \label{eq:orient_moduli_positive_punctured discs}
  |T_{u} \Cyl(\sL^{r} \Q, x)|  \otimes   \ro_{x}[w(x)] \cong \eta_{\ev(u)}.
  \end{equation}
\end{lem}
As an immediate consequence of this isomorphism of graded lines, we obtain a computation of the dimension of the moduli space which should be compared with Lemma \ref{lem:vir_dim_moduli_half_cylinders}:
\begin{exercise}
  Show that the virtual dimension of $\Cyl(\sL^{r} \Q, x)  $ is $n- \deg(x)$.
\end{exercise}

\subsection{Construction of  $\Gam$}
Given a critical point $y$  of $f^{r}$, and a Hamiltonian orbit  $x$ of $H^{+}$, consider the fibre product

\begin{equation} \label{eq:define_broke_cot}
  \Brokecot(y , x) \equiv W^{s}(y) \times_{ \sL^{r} \Q }  \Cyl(\sL^{r} \Q, x).
\end{equation}
\begin{exercise}
Using the natural orientation of a fibre product, and Equation \eqref{eq:spliting_tangent_at_crit_point}, construct a natural isomorphism of graded lines
  \begin{equation} \label{eq:isomorphism_broken_mod_loop_to_SH}
  \ro_{y} \otimes  |  \Brokecot(y,x) | \otimes   \ro_{x}[w(x)] \cong  \eta_{y} .
  \end{equation}
\end{exercise}
For generic Floer data, we conclude that the moduli space $  \Brokecot(y,x)  $ is a manifold of dimension
  \begin{equation}
    n - \ind(y)  - \deg(x).
  \end{equation}

We shall now restrict attention to the case
\begin{equation}
  \ind(y) = n - \deg(x),
\end{equation}
which implies that $ \Brokecot(y,x)  $ has dimension $0$.  In this case, the moduli space in fact consists only of finitely many points: the key point is that the projection of $\Cyl(\sL^{r} \Q, x) $ to $\sL^{r} \Q$ is disjoint from the boundary by  Proposition \ref{prop:moduli_positive_cyl_empty_boundary_outside}, so the fibre product with $W^{s}(y)$ takes place over a compact subset of the interior of $  \sL^{r} \Q$.  Using Equation \eqref{eq:isomorphism_broken_mod_loop_to_SH}, we associate a map
\begin{equation}
  \Gam_{(\gamma,u)  }  \co  \ro_{x}[w(x)] \to  \ro_{y} \otimes \eta_{y}  
\end{equation}
to every element $(\gamma,u)$ of $  \Brokecot(y,x)  $.

Taking the sum over all elements of these rigid moduli spaces, we define a map
\begin{align}
  \Gam^{r} \co CF^{*}(H^{+}; \bZ) & \to CM_{-*}(f^{r}; \eta) \\
 \Gam^{r}|  \ro_{x}[w(x)] & \equiv \sum_{\ind(y) = n - \deg(x) }   \sum_{(\gamma,u) \in    \Brokecot(y,x)} \Gam_{(\gamma,u)  }.
\end{align}

\begin{exercise}
  Show that $\Gam^{r}$ is a chain map.
\end{exercise}

\section{Composition on Floer cohomology} \label{sec:comp-floer-homol-II}
Let $H^{+}$ and $H^{-}$ be linear Hamiltonians, with slopes $b^{\pm}$, and assume that there exists an integer $r$ so that
\begin{equation} \label{eq:slope_condition}
  b^{+} < \sum_{i=1}^{r} \delta^{r}_{i} < \frac{b^{-}}{2}.
\end{equation}
In this case, we may choose a sequence $\{ b_i \}_{i=1}^{r}$ of positive real numbers, whose sum we denote $b$, such that
\begin{align} \label{eq:b_i-for-homotopy-1}
   \delta_{i}^{r} < 2 b_i  &  < 2   \delta_{i}^{r}  \\ \label{eq:b_i-for-homotopy-2}
2 b^{+} < 2 b  & < b^{-}.
\end{align}
In particular, Equations \eqref{eq:b_i_negative_cylinder} and \eqref{eq:slope_condition_negative} hold, which implies that the map $\Fam^{r}$ with range the Floer cohomology of $H^{-}$ is well-defined, and Equations \eqref{eq:b_i_not_too_big_not_too_small} and \eqref{eq:slope-bound} also hold, which implies that the map $\Gam^{r}$ with domain the Floer cohomology of $H^{+}$ is also well-defined.

To state the next result precisely, recall that we introduced in Equation \eqref{eq:decomposition_Floer} a decomposition of Floer cohomology into summands $  HF^{*}_{w}(H^{+}; \bZ) $, associated to $w \in \{ 0,-1\}$.
\begin{prop} \label{prop:comp-floer-homol}
  The restriction of the composition $\Fam^{r} \circ \Gam^{r}$ to $ HF^{*}_{w}(H^{+}; \bZ) $ agrees with the continuation map
  \begin{equation}
    HF^{*}_{w}(H^{+}; \bZ) \to HF^{*}_{w}(H^{-}; \bZ).
  \end{equation}
 up to a sign that depends only on the triple $(n,r,w)$.
\end{prop}

This result immediately yields a proof of the main result of this Chapter:

\begin{proof}[Proof of Theorem \ref{thm:Fam_surjective}]
Since we have assumed that the total length of the piecewise geodesics goes to infinity, we may choose a sequence $r_{j}$ of integers, and $b^{j}$ of real numbers, such that
\begin{equation}
  b^{j-1} < \sum_{i=1}^{r_j} \delta^{r_j}_{i} < \frac{b^{j}}{2}.
\end{equation}
Choose Hamiltonians $H^{j}$ of slope $b^{j}$ whose Floer cohomology is well-defined. By Lemma \ref{lem:directed_system_H}, symplectic cohomology is the direct limit of the groups $ HF^{*}(H^{i}; \bZ)  $ with respect to these continuation maps. On the other hand, Proposition \ref{prop:comp-floer-homol} implies that continuation factors through the composition of $\Fam^{r_{j}}$ with an isomorphism of Floer cohomology (in the summand corresponding to orientable and non-orientable loops, this isomorphism is either the identity or multiplication by $-1$).  We conclude that every element of $SH^{*}(\TQ; \bZ)$ is in the image of $\Fam^{r_j}$ for $j$ sufficiently large, hence that $\Fam$ is surjective.
\end{proof}

\subsection{Fundamental cycle of $\TQ$ via discs} \label{sec:fundamental-cycle-tq}
We shall use a completely elementary computation of moduli spaces of discs with cotangent boundary conditions to prove Proposition \ref{prop:comp-floer-homol}. The philosophy is analogous to the way the results of Chapter \ref{cha:viterbos-theorem} relied on a local computation of moduli spaces of holomorphic triangles.

Let $J$ be an almost complex structure on $\TQ$, and for any point $q \in \Q$, denote by 
\begin{equation}
  \Disc(\Tq)
\end{equation}
the space of maps
\begin{equation}
  u \co D^{2} \to \TQ
\end{equation}
mapping the boundary to $\Tq$, and solving the pseudoholomorphic curve equation
\begin{equation}
  du^{0,1} \equiv 0.
\end{equation}
\begin{exercise}
  Use Stokes's theorem to show that all elements of $\Disc(\Tq)$ are constant.
\end{exercise}
Evaluation at $0$ defines a map
\begin{equation}
    \Disc(\Tq) \to \TQ
\end{equation}
which is a diffeomorphism onto the cotangent fibre $\Tq$. Letting $q$ vary, we obtain a parametrised moduli space 
\begin{equation}
  \Disc(\sL^{1} \Q) \equiv   \coprod_{q \in \Q}   \Disc(\Tq).
\end{equation}
Since the moduli space $    \Disc(\Tq)   $ consists only of constant discs, and the Lagrangians $\Tq$ fibre $\TQ$, we conclude:
\begin{lem} \label{lem:fundamental_cycle_discs}
  The evaluation map
  \begin{equation}
      \Disc(\sL^{1} \Q) \to \TQ
  \end{equation}
is a diffeomorphism. \qed
\end{lem}
We equip the above moduli space with the orientation coming from this diffeomorphism, and the natural orientation of the right hand side as a symplectic manifold.

\subsection{Genus $0$ surfaces with multiple boundary components} \label{sec:genus-0-surfaces}

The proof  of Proposition \ref{prop:comp-floer-homol} relies on a cobordism of a moduli space of surfaces with $r$ boundary components and $2$ interior marked points.

\begin{figure}[h]
  \centering
  \includegraphics{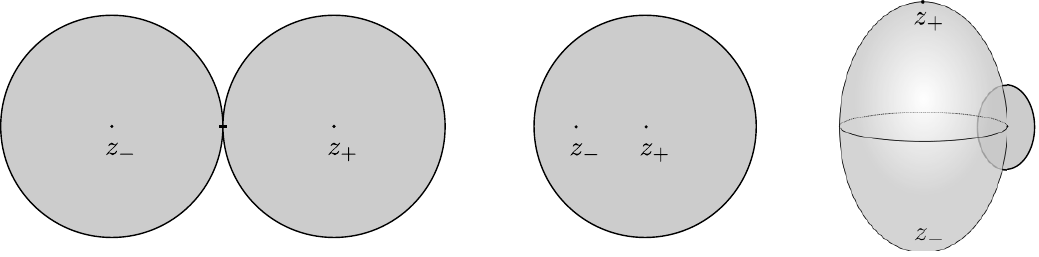}
  \caption{ }
\label{fig:moduli_disc_two_punctures}
\end{figure}

We start with the case $r=1$: let $\Cyl_{2,1}$ denote the moduli space of compact genus $0$ Riemann surfaces with $1$ boundary component and $2$ marked points $(z_+,z_-)$. Any such surface is biholomorphic to a disc with two marked points, and it is convenient to fix the unique parametrisation so that $z_+$ maps to $0$ and $z_-$ to the interval $(-1,0)$ (as in the middle drawing of Figure \ref{fig:moduli_disc_two_punctures}); we write
\begin{equation}
  \Cyl_{2,1}^{R}
\end{equation}
for the unique element of the moduli space $ \Cyl_{2,1} $ corresponding to a point $R \in (-1,0)$.

The \emph{stable compactification} $\Cylbar_{2,1}  $ is a closed interval obtained by adding the endpoints
\begin{equation}
    \Cyl_{2,1}^{-1} \textrm{ and }  \Cyl_{2,1}^{0};
\end{equation}
the first corresponds to two discs each carrying one interior marked point and one boundary marked point, glued along their boundary marked point, while the second is a copy of $\bC \bP^{1}$ with three marked points, two of which correspond to the marked points $z_\pm$, while the third is attached to a \emph{ghost disc} bubble. 

\begin{rem}
The compactification $\Cylbar_{2,1}   $ can be defined more formally by considering Riemann surfaces with interior marked points of two different flavours: ordinary marked points, and ghost disc marked points. Since a stable Riemann surface should have no component whose group of automorphisms is not discrete, and  the disc with one interior marked point has automorphism group $S^1$, we should therefore collapse the disc component on the right of Figure \ref{fig:moduli_disc_two_punctures}, and think of that  stratum more precisely as a genus $0$ compact Riemann surface, with two ordinary marked points, and one  ghost disc marked point.  In practice, we record this data in our figures by drawing the ghost discs. For a general discussion of moduli spaces of Riemann surfaces with boundary, see \cite{Liu}.
\end{rem}

Given an element of $   \Cyl_{2,1} $, we obtain a surface biholomorphic to the complement of a disc in the cylinder by removing the marked points $z_{\pm}$ from the corresponding surface. If we remove the nodal points from the surfaces  corresponding to the boundary point $    \Cyl_{2,1}^{-1}$, we obtain a pair of discs with one interior and one boundary puncture; using notation which is consistent with our choice of ends and the conventions in Sections \ref{sec:moduli-space-half} and \ref{sec:punct-posit-half}, the punctured surface corresponding to this point is
\begin{equation} \label{eq:boundary_two_half_cylinders}
  Z^{-}_{1} \coprod Z^{+}_{1}.
\end{equation}
At the point $   \Cyl_{2,1}^{0} $, the surfaces we obtain are respectively  the complement $Z_1$ of the point  $(0,0)$ on the cylinder, and the complement $Z^{+}$  of an interior point on the disc:
\begin{equation} \label{eq:boundary_cylinder_disc}
  Z_{1} \coprod Z^{+}.
\end{equation}

\begin{figure}[h]
  \centering
  \includegraphics{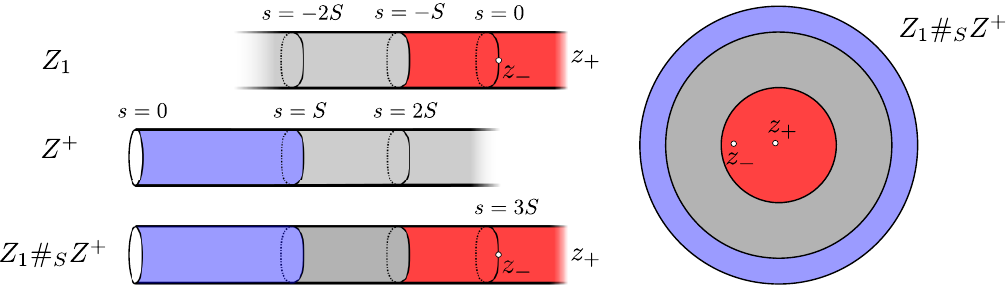}
  \caption{ }
\label{fig:gluing_disc_cylinder_interior}
\end{figure}

We now construct parametrisations of neighbourhoods of the two boundary strata: choose cylindrical ends near $0 \in D^2$, and near the marked point  $(0,0) \in Z$. For each positive  real number $S$, we obtain a Riemann surface
\begin{equation}
  Z_{1} \#_{S} Z^{+}
\end{equation}
by gluing these surfaces along their ends. We claim that we obtain an embedding
\begin{equation} \label{eq:embedding_by_gluing}
  [0,+\infty] \to \Cylbar_{2,1}
\end{equation}
onto a neighbourhood of Equation \eqref{eq:boundary_cylinder_disc}. Instead of proving this in general, we can consider the special case shown in Figure \ref{fig:gluing_disc_cylinder_interior}:
\begin{exercise} \label{ex:good_strip_like_ends_Z_1}
Equip the marked point $(0,0) \in Z$ with a negative cylindrical end that extends to a biholomorphism
\begin{equation}
  Z \setminus \{(0,0) \} \to Z_{1}
\end{equation}
fixing the end $s=+\infty$. Equip the origin in $D^2$ with a positive cylindrical end using exponential polar coordinates. For this choice, show that there is a biholomorphism
\begin{equation} \label{eq:explicit_model_gluing}
    Z_{1} \#_{S} Z^{+}\cong D^{2} \setminus  \{ 0, -e^{-3S} \}.
\end{equation}
\end{exercise}
In Equation \eqref{eq:explicit_model_gluing}, we see that the gluing parameter $S$ is determined by the modulus of the surface $ Z_{1} \#_{S} Z^{+} $, which implies that the gluing map in Equation \eqref{eq:embedding_by_gluing} is indeed an embedding. 

We now implement a similar construction near the other boundary stratum: choose strip-like ends on $Z^{\pm}_{1} $ as in Sections \ref{sec:moduli-space-half} and \ref{sec:punct-posit-half}. By gluing, we obtain, for each positive real number $S$, a Riemann surface
\begin{equation}
  Z^{-}_{1} \#_{S} Z^{+}_{1}
\end{equation}
which gives an element of $  \Cyl_{2,1}$ by filling the two punctures at infinity. One can prove that this yields an embedding 
\begin{equation} \label{eq:embedding_by_gluing_strip}
  [0,+\infty] \to \Cylbar_{2,1}
\end{equation}
for arbitrary choices of strip-like ends, but the following special case will suffice:
\begin{exercise} \label{ex:good_cylindrical_end_cylinder_dics}
Equip  $ Z^{\pm}_{1} $ with strip-like ends using the identifications 
\begin{align}
  Z^{\pm}_{1} & \cong D^{2} \setminus \{ \pm 1 \} \\
D^{2} \setminus \{1, -1\} & \cong \Strip,
\end{align}
where the second map is normalised to take the origin to $(0,1/2)$. For this choice, show that there is a biholomorphism
\begin{equation}
  Z^{-}_{1} \#_{S} Z^{+}_{1}  \cong D^{2} \setminus {\Big \{}\frac{e^{S}-1}{e^{S}+1} , 0 {\Big \}}.
\end{equation}
\end{exercise}
Since the map
\begin{align}
(-\infty,0) & \to (-1,0) \\
S & \mapsto \frac{e^{S}-1}{e^{S}+1} 
\end{align}
is injective, we conclude that the gluing map in Equation \eqref{eq:embedding_by_gluing_strip} is an embedding.

We now define a moduli space $\Cyl_{2,r}$ of compact Riemann surfaces of genus $0$, with $2$ interior marked points, and $r$ boundary components, consisting of those surfaces which are $r$-fold covers of elements of $\Cyl_{2,1} $, branched at the points $z_{\pm}$. By removing the two interior marked points, we obtain an unbranched cover of the complement of a disc in the cylinder, and the covering data is such that the covering space has two punctures and $r$ boundary components. By construction, the natural map
\begin{equation} \label{eq:r-fold_cover_map_moduli_spaces}
 \Cyl_{2,1} \to  \Cyl_{2,r}
\end{equation}
which assigns to a surface with one boundary component its $r$-fold cover is a diffeomorphism; and we write
\begin{equation}
\Cyl_{2,r}^{R}
\end{equation}
for the surface corresponding to a point $R \in (-1,0)$ under this diffeomorphism.
\begin{rem}
  A generic genus $0$ Riemann surfaces with $2$ interior marked points and $r$ boundary components does not represent an element of $ \Cyl_{2,r} $.
\end{rem}
\begin{figure}[h]
  \centering
  \includegraphics{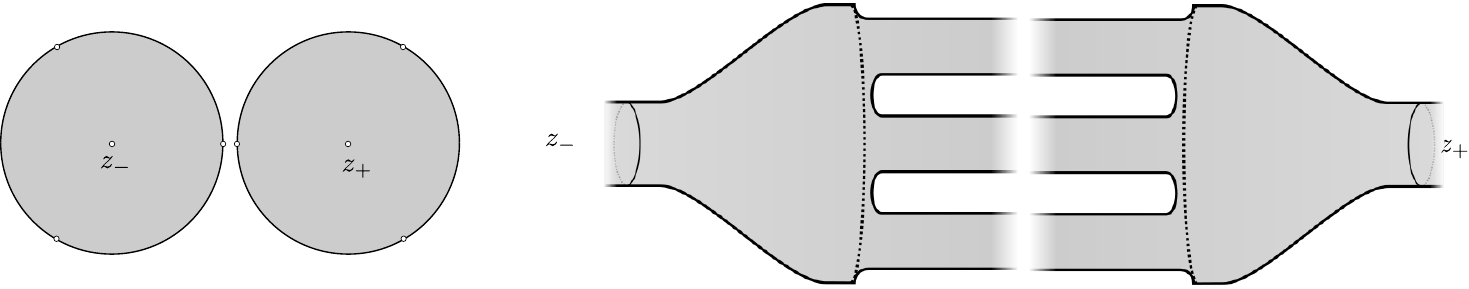}
  \caption{ }
  \label{fig:moduli-disc-two-punctures-1}
\end{figure}
The compactification $\Cylbar_{2,r}$ is obtained by adding two strata
\begin{equation}\label{eq:covering_map_cylbar_r}
\Cyl_{2,r}^{-1}  \textrm{ and } \Cyl_{2,r}^{0}.
\end{equation}
Removing the interior marked point and the nodes, we can identify the surface corresponding to $ \Cyl_{2,r}^{-1} $  (see Figure \ref{fig:moduli-disc-two-punctures-1} for two representations in the case $r=3$) with the Riemann surface
\begin{equation}
  Z^{-}_{r} \coprod  Z^{+}_{r}.
\end{equation}
To see this, observe that, by construction, the surfaces $Z^{\pm}_{r}$ carry an action of the cyclic group $\bZ/r \bZ$, and that the quotient is $Z^{\pm}_{1} $. In particular, if we equip $Z^{\pm}_{r}$ with strip-like ends that are pulled back  by projection  from those defined on $Z^{\pm}_{1}$ in Exercise \ref{ex:good_strip_like_ends_Z_1}, we obtain a chart
\begin{align}
  [0,\infty] & \to \Cylbar_{2,r} \\
S & \mapsto   Z^{-}_{r} \#_{S}  Z^{+}_{r},
\end{align}
where the surface $Z^{-}_{r} \#_{S}  Z^{+}_{r}  $ is obtained by gluing all the matched strip-like ends for the same gluing parameter $S$.  This map is obviously an embedding because it agrees with the composition
\begin{equation}
    [0,\infty]  \to  \Cylbar_{2,1}  \to \Cylbar_{2,r} 
\end{equation}
where the first map is the gluing map for $r=1$, and the second is the map in Equation \eqref{eq:r-fold_cover_map_moduli_spaces}.

\begin{figure}[h]
  \centering
  \includegraphics{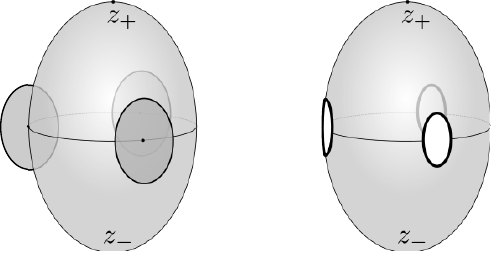}
  \caption{ }
  \label{fig:moduli-disc-two-punctures-0}
\end{figure}

The boundary stratum $ \Cyl_{2,r}^{0} $ represents a surface with $r+1$ components (see the left side of Figure \ref{fig:moduli-disc-two-punctures-0}); after removing the marked points and the nodes,  one of the components is the punctured cylinder
\begin{equation}
  Z_{r} \equiv Z \setminus \{ (0,\frac{i}{r} ) \vbar 1 \leq i \leq r\},
\end{equation}
and the remaining components are punctured (ghost) discs, i.e. copies of $Z^{+}$. We equip $Z_{r}$ with the cylindrical end at each puncture coming from pulling back, via the covering map associated to the natural action by the cyclic group $\bZ/r \bZ$,  the cylindrical end on $Z_1$ which we fixed in Exercise \ref{ex:good_cylindrical_end_cylinder_dics}. We equip the punctured ghost disc with the cylindrical end corresponding to exponential polar coordinates.

By gluing $Z_{r}$ to the ghost discs for equal gluing parameter, we obtain a chart
\begin{align}
[0,\infty] & \to \Cylbar_{2,r} \\
S & \mapsto  Z_{r}  \#_{S} \coprod_{i=1}^{r} Z^{+} ,
\end{align}
which is again an embedding because it factors through the gluing map for $r=1$, and the map which assigns to a surface with one boundary component its $r$-fold cover.

\subsection{Pseudoholomorphic curve equations} \label{sec:pseud-curve-equat}

The Riemann surfaces corresponding to the two boundaries of $ \Cylbar_{2,r}  $ carry natural  pseudoholomorphic curve equations: at the end corresponding to two punctured discs, we use the equations that define the maps $\Fam^{r}$ and $\Gam^{r}$, while at the other end, we use a continuation equation along $Z_{r}$, and a homogeneous equation on the ghost discs.  Later in this section, we shall discuss these equations in more detail.  First, we need to choose auxiliary data to define pseudoholomorphic equations on arbitrary elements of $ \Cylbar_{2,r} $ which interpolate between  those on the ends. 

We start by choosing positive (respectively negative) cylindrical ends $\epsilon_{\pm}$  near the marked point $z_+$ (respectively $z_-$) for each surface in $\Cylbar_{2,r}  $, which vary smoothly in the interior of the moduli space. The surfaces corresponding to the boundary strata carry natural cylindrical ends. The ends chosen for surfaces near the boundary of $ \Cylbar_{2,r}   $  are assumed to agree with those induced by gluing.

Recall that we have fixed, at the beginning of Section \ref{sec:comp-floer-homol-II}, positive real numbers $b^{\pm}$, and $\{ b_{i} \}_{i=1}^{r}$, subject to the constraints of Equations \eqref{eq:b_i-for-homotopy-1} and  \eqref{eq:b_i-for-homotopy-2}.  We then choose a closed $1$-form $\alpha$ on each surface $\Sigma \in \Cylbar_{2,r}$ whose restriction to each cylindrical end agrees with $b dt$, and such that
\begin{equation}
\parbox{35em}{the restriction of $\alpha$ to  $\partial \Sigma$ vanishes.}  
\end{equation}
On the boundary stratum $  \Cyl_{2,r}^{-1} $, we assume that the restriction of $\alpha$ to either of the components of the surface $  Z^{-}_{r} \coprod Z^{+}_{r}$ agrees with the $1$-forms $\alpha^{\pm}$ fixed respectively in Section  \ref{sec:moduli-space-half} and \ref{sec:punct-posit-half}. On  the other boundary stratum, we assume that $\alpha$ vanishes identically on the ghost discs, and 
\begin{equation} \label{eq:1-form_vanish_near_marked_point}
\parbox{35em}{the restriction of $\alpha$ to  $Z_{r}$ vanishes in a neighbourhood of $(0, \frac{i}{r})$ for each $1 \leq i \leq r$.}  
\end{equation}
Moreover, we assume that, for surfaces representing points of $   \Cylbar_{2,r} $ which are sufficiently close to the boundary, $\alpha$ is obtained by gluing the $1$-form fixed above on the nodal surfaces, and that $\alpha$ varies smoothly over the interior of the moduli space.

Next, we choose two functions 
\begin{equation}
  f_{\pm} \co \Sigma \to \bR
\end{equation}
such that
\begin{align}
&d (f_{\pm} \alpha) \leq 0 \\
 &  \parbox{32em}{$f_{\pm}$ vanishes on the positive end, is identically equal to $1$ on the negative end, and is constant on $\partial \Sigma$.}  
\end{align}
Note that the space of functions satisfying these properties is convex, which implies that we can construct them using partitions of unity on the moduli space $ \Cyl_{2,r} $.  On the surface $  Z^{-}_{r} \coprod Z^{+}_{r}$, we assume that these functions agree with those  fixed respectively in Section  \ref{sec:moduli-space-half} and \ref{sec:punct-posit-half}; in particular the restriction of $f_{+}$ to $Z^{-}_{r}$ identically vanishes, while the restriction of $f_{-}$ to $Z^{+}_{r}$ is $1$. For elements of $ \Cyl_{2,r} $ near this boundary stratum, we assume that $f_{\pm}$ is constructed by gluing.  Moreover, we assume that
\begin{equation} \label{eq:condition_function_cylinder_r_points}
  \parbox{35em}{  $f_{+} = f_{-}$ on the surface corresponding to $ \Cyl_{2,r}^{0}$.}
\end{equation}
We also assume that the functions $f_{\pm}$ on surfaces lying in a neighbourhood of $ \Cyl_{2,r}^{0}$  are constructed by gluing.

The functions $f_{\pm}$ determine a class of Hamiltonians $H^{\Sigma}$ on $\TQ$, parametrised by each surface $\Sigma$, which satisfy the following properties:
\begin{align}
  H^{\Sigma} \circ \epsilon_{+}(s,t) & = H^{+}  \textrm{ if } 0 \ll s \\
  H^{\Sigma} \circ \epsilon_{-}(s,t) & = H^{-} \textrm{ if } s \ll 0 \\ \label{eq:use_two_cutoffs_for_H_sigma}
H^{\Sigma} | \Sigma \times (\TQ \setminus \DQ) &  \equiv \frac{b^{+}}{b}\rho + f_{+} \left(\Hh -  \frac{b^{+}}{b}\rho   \right)  + f_{-} \left(    \frac{b^{-}}{b}\rho - \Hh  \right).
\end{align}
\begin{exercise}
  Check that the restriction of the right hand side of Equation \eqref{eq:use_two_cutoffs_for_H_sigma} to  the positive end agrees with the restriction of $H^{+}$, and that the restriction to the negative end agrees with $H^{-}$.
\end{exercise}

Finally, we choose a family of almost complex structures on $\TQ$ parametrised by each curve $\Sigma \in \Cylbar_{2,r}$ which are convex near $\SQ[2]$, and whose pullbacks along the ends agree with the almost complex structures used to define the respective Floer cohomology groups. At the boundary stratum $  \Cyl_{2,r}^{0} $, we assume that the almost complex structure is constant in a neighbourhood of each marked point $(0,\frac{i}{r})$, and agrees with the almost complex structure chosen at each point on the corresponding disc bubble. At the boundary stratum $  \Cyl_{2,r}^{-1} $, we assume that the restrictions to the two components of the family of almost complex structures agree with the choices fixed in Sections  \ref{sec:moduli-space-half} and \ref{sec:punct-posit-half}. By gluing, we obtain families of almost complex structures for each surface corresponding to a point near the boundary of $ \Cylbar_{2,r}$. We extend these choices smoothly to the rest of the moduli space.

\subsection{Moduli spaces of maps}
For each constant $R \in (-1,0)$, we obtain from the choices of Section \ref{sec:pseud-curve-equat} a pseudo-holomorphic curve equation
\begin{equation} \label{eq:psh_homotopy_continuation_F_G}
  \left( du - \alpha \otimes X_{H^{\Sigma}} \right)^{0,1}=0
\end{equation}
on the space of maps from the unique element $\Sigma$ of $  \Cyl_{2,r}^{R} $  to $\TQ$.  We define the moduli space
\begin{equation}
   \Cyl_{2,r}^{R}( x_{-}, x_{+})
\end{equation}
for each pair $(x_{+}, x_{-})$ of orbits of $H^{+}$ and $H^{-}$, to consist of those solutions to  Equation \ref{eq:psh_homotopy_continuation_F_G} which have finite energy, converging at the positive end to $x_+$ and at the negative end to $x_-$, and such that
\begin{equation}
  \parbox{35em}{each component of $\partial \Sigma$ is mapped to a cotangent fibre.}
\end{equation}
\begin{rem}
Since the set of cotangent fibres which can appear as boundary conditions are parametrised by $\Q^{r}$, with one factor corresponding to each boundary component, the moduli space $    \Cyl_{2,r}^{R}( x_{-}, x_{+}) $ is topologised as a parametrised moduli space over $\Q^{r}$.
\end{rem}

The key result in Section \ref{sec:from-floer-cohom} asserts that elements of the moduli space $  \Cyl(\sL^{r} \Q, x_{+}) $ remain in compact subsets of $\DQ$.  That result generalises to the moduli spaces at hand:
\begin{exercise} \label{ex:moduli-spaces-r_in_disc_bundle}
Imitating the proof of Lemma \ref{lem:moduli_space_contained_in_disc_bundle}, show that all elements of $   \Cyl^{R}_{2,r}( x_{-}, x_{+}) $  have image contained in $\DQ[2]$.
\end{exercise}

We define the Gromov-Floer compactification of these moduli spaces to be the union 
\begin{equation}
     \Cylbar_{2,r}^{R}( x_{-}, x_{+}) \equiv \bigcup_{x'_{-},x'_{+}}  \Cylbar(x_{-},x'_{-})    \times \Cyl_{2,r}^{R}( x'_{-}, x'_{+}) \times \Cylbar(x'_{+},x_{+})  
\end{equation}
equipped with the Gromov topology. Exercise \ref{ex:moduli-spaces-r_in_disc_bundle} implies that the compact set $\DQ[2]$ contains the image of all elements of $ \Cylbar_{2,r}^{R}( x_{-}, x_{+}) $, which, by Gromov's compactness theorem, implies that this space is indeed compact.

At the boundary of the moduli space, we first consider the fibre product, over the parameter space of boundary conditions, of the moduli spaces constructed in Sections  \ref{sec:moduli-space-half} and \ref{sec:punct-posit-half}, and define
\begin{align} \label{eq:moduli_space_-1_product}
    \Cyl_{2,r}^{-1}( x_{-}, x_{+})  & \equiv  \Cyl(x_{-}, \sL^{r} \Q) \times_{\sL^{r} \Q}  \Cyl(\sL^{r} \Q, x_{+}) \\
    \Cylbar_{2,r}^{-1}( x_{-}, x_{+})  & \equiv  \Cylbar(x_{-}, \sL^{r} \Q) \times_{\sL^{r} \Q}  \Cylbar(\sL^{r} \Q, x_{+}).
\end{align}
\begin{rem}
A priori, we should be taking the fibre product of moduli spaces parametrised over $\Q^{r}$ (since this is the space parametrising the Lagrangian boundary conditions). However, Proposition \ref{prop:moduli_positive_cyl_empty_boundary_outside} implies that the moduli space $  \Cyl(\vq, x)  $ is empty unless $\vq$ lies in $\sL^{r} \Q$, which justifies our definition of $    \Cyl_{2,r}^{-1}( x_{-}, x_{+})   $.  
\end{rem}

At the other boundary stratum, Equation \eqref{eq:condition_function_cylinder_r_points} implies that Equation \eqref{eq:psh_homotopy_continuation_F_G} becomes a continuation map equation on the cylinder from the Floer equation for $H^{+} $ to the Floer equation for $H^{-}$. On each disc bubble, the requirement in Equation \eqref{eq:1-form_vanish_near_marked_point} that  $\alpha$ vanish implies that the natural pseudo-holomorphic curve equation to impose on each disc bubble is homogeneous. We therefore define
\begin{equation}
   \Cyl_{2,r}^{0}( x_{-}, x_{+})  \equiv   \Cont(x_-,x_+)  \times_{\TQ^{r}}   \underbrace{ \Disc(\sL^{1} \Q)  \times \cdots \times \Disc(\sL^{1} \Q)}_{r}
\end{equation}
where the evaluation map from the space of continuation maps is obtained by considering the images of the cylinder at the points $\{ (0, \frac{i}{r}) \}_{i=1}^{r}$, and the evaluation on each moduli space of discs with cotangent boundary conditions takes place at the origin.

Since Lemma \ref{lem:fundamental_cycle_discs} implies that the space of disc bubbles with boundary conditions on an arbitrary cotangent fibre is a copy of the space $\TQ$, the constraint imposed in the above fibre product is vacuous, and we obtain a natural diffeomorphism
\begin{equation}
   \Cyl_{2,r}^{0}( x_{-}, x_{+})  \cong   \Cont(x_-,x_+).
\end{equation}
In particular, the Gromov-Floer compactification of  this moduli space is the one discussed in Section \ref{sec:continuation-maps}:
\begin{equation}
   \Cylbar_{2,r}^{0}( x_{-}, x_{+})  \cong   \Contbar(x_-,x_+).
\end{equation}

Taking the union of the above moduli spaces over all surfaces in $ \Cylbar_{2,r} $, we obtain the parametrised moduli space
\begin{equation}
   \Cyl_{2,r}( x_{-}, x_{+}) \equiv \coprod_{R \in [-1,0]} \Cyl_{2,r}^{R}( x_{-}, x_{+}) ,
\end{equation}
with compactification
\begin{equation}
   \Cylbar_{2,r}( x_{-}, x_{+}) \equiv \coprod_{R \in [-1,0]} \Cylbar_{2,r}^{R}( x_{-}, x_{+}) ,
\end{equation}
which is equipped with Gromov's topology.

\subsection{Orientations}

As a parametrised moduli space, the tangent space of  $\Cyl_{2,r}( x_{-}, x_{+})   $ at a point $u$ lying over $\vq \in T \sL^{r} \Q$ can be oriented by orienting the base and the fibre:
\begin{equation} \label{eq:iso_parametrised_cylinder_boundary}
  \begin{aligned}
  |  \Cyl_{2,r}( x_{-}, x_{+})  | & \cong  |\det(D_u)|  \otimes | T_{\vq} \sL^{r} \Q | \otimes |  \Cyl_{2,r} | \\
& \cong |\det(D_u)|  \otimes  \bigotimes_{i=1}^{r} | T_{q_i}\Q | \otimes |  \Cyl_{2,r} |.
\end{aligned}
\end{equation}

\begin{exercise}
Show that the identification $ u^{*}( T \TQ) \cong u^{*}(T \Q) \otimes_{\bR} \bC $ induces a unique homotopy class of  trivialisations of $  u^{*}( T \TQ) $ which restricts to the preferred homotopy class of trivialisations of $x_{\pm}^{*}( T \TQ) $ and maps the boundary conditions to loops of vanishing Maslov index. (Hint: review the discussion of Section \ref{sec:conley-zehnder-index})
\end{exercise}

To orient $\det(D_u)$, we use the preferred trivialisation to identify $D_{u}$ with an operator on $\bC^{n}$-valued functions defined on the domain of $u$, with totally real boundary conditions. Such a trivialisation associates to the chords $x_{\pm}$ paths of symplectomorphisms $\Psi_{\pm} $, and to the cotangent fibres $\Tq[i]$ Lagrangian planes $L_{i}$. Deforming the domain of $u$ to $Z_{r}$, we obtain an operator $D_{\Psi_{-}, \Psi_{+}} $ on the cylinder with asymptotic conditions $\Psi_{\pm}$, and  a (homogeneous) Cauchy-Riemann operator $D_{L_i}$ on the disc, with boundary condition $L_i$. Gluing yields a canonical identification:
\begin{equation}
  \det(D_u) \otimes \det(\bC^{n})^{\otimes r} \cong \det(D_{\Psi_{-}, \Psi_{+}}) \otimes \bigotimes_{i=1}^{r} \det(D_{L_{i}}).
\end{equation}
Lemma \ref{lem:isomorphism_glued_det_lines} implies the existence of a canonical isomorphism
\begin{equation}
 | \det(D_{\Psi_{-}, \Psi_{+}})| \otimes \ro_{x_+} \cong  \ro_{x_-}, 
\end{equation}
while Equation \eqref{eq:iso_det_Lag_boundary_kappa}, applied in this very special situation in which the loop of Lagrangians is constant, yields a canonical isomorphism
\begin{equation}
|  \det(D_{L_i})|  \cong  |\Tq_i|.
\end{equation}
Combining these with Equation \eqref{eq:iso_parametrised_cylinder_boundary}, and the pairing in Equation \eqref{eq:orientation_cotangent_fibre}, and using the complex orientations on $\bC^{n}$, we conclude that we have a canonical isomorphism
\begin{align} \notag
   | \Cyl_{2,r}( x_{-}, x_{+})  |  \otimes \ro_{x_+}  &  \cong  |\det(D_u)|  \otimes  \bigotimes_{i=1}^{r} | T_{q_i} \Q | \otimes |  \Cyl_{2,r} | \otimes \ro_{x_+}  \\ \notag
& \cong | \det(D_{\Psi_{-}, \Psi_{+}})| \otimes \bigotimes_{i=1}^{r}| \det(D_{L_{i}})| \otimes |\bC^{n}|^{- \otimes r}  \otimes  \bigotimes_{i=1}^{r} | T_{q_i} \Q | \otimes |  \Cyl_{2,r} | \otimes \ro_{x_+}  \\
&  \cong  |  \Cyl_{2,r} | \otimes \ro_{x_-}
\end{align}

We equip $ \Cyl_{2,r} $ with the natural orientation coming from the projection to $ (-1,0) $, and obtain a map
\begin{equation}\label{eq:orientation_parametrised_moduli_cylinder_boundary}
   | \Cyl_{2,r}( x_{-}, x_{+})  |  \otimes \ro_{x_+}   \cong  \ro_{x_-}. 
\end{equation}

Restricting to the boundary stratum $ \Cyl_{2,r}^{0} $, yields an isomorphism
\begin{equation} \label{eq:orientation_moduli_cylinders_r_marked_points}
   | \Cyl_{2,r}^{0}( x_{-}, x_{+})  |  \otimes \ro_{x_+}   \cong  \ro_{x_-}.
\end{equation}

\begin{exercise} \label{ex:sign_boundary_continuation}
Under the identification of $ \Cyl_{2,r}^{0}( x_{-}, x_{+}) $ with the space of continuation maps, show that the isomorphism in Equation \eqref{eq:orientation_moduli_cylinders_r_marked_points} agrees with the map in Equation \eqref{eq:canonical_maps_orientation_lines_parametrised} up to a sign that depends only on $r$ and on the dimension of $\Q$.
\end{exercise}

The above construction also produces a relative orientation of the other boundary stratum:
\begin{equation} \label{eq:orientation_moduli_cylinders_two_punctured discs}
   | \Cyl_{2,r}^{-1}( x_{-}, x_{+})  |  \otimes \ro_{x_+}   \cong  \ro_{x_-}.
\end{equation}
This space, which splits as a fibre product of moduli spaces of punctured discs, also admits a relative orientation from Equations \eqref{eq:iso_determinant_line_univeral_moduli} and \eqref{eq:orient_moduli_positive_punctured discs}:
\begin{equation} \label{eq:product_orientation_two_half_cylinders}
   |  \Cyl(x_{-}, \sL^{r} \Q) \times_{\sL^{r} \Q}  \Cyl(\sL^{r} \Q, x_{+})  |  \otimes \ro_{x_+}   \cong  \ro_{x_-}.
\end{equation}

The proof of the following result is postponed until Section \ref{sec:comp-orient-line}.
\begin{lem} \label{lem:maps_agree_up_to_uknown_sign}
Under the identification coming from Equation \eqref{eq:moduli_space_-1_product},  the isomorphisms in Equation \eqref{eq:orientation_moduli_cylinders_two_punctured discs} and \eqref{eq:product_orientation_two_half_cylinders} differ by a sign that depends only on the dimension $n$, the number of marked points $r$, and whether $(q \circ x_{\pm})^{*}(T\Q)$ is orientable.
\end{lem}

\subsection{Rigidifying moduli spaces}
For each positive real number $s$, consider the fibre product
\begin{equation}
  \Brokecot^{s}(x_{-},x_{+}) \equiv  \Cyl(x_{-}, \sL^{r} \Q) \times_{\psi_{s}^{r}}  \Cyl(\sL^{r} \Q, x_{+}),
\end{equation}
where the evaluation map on the first factor is given by the projection to $\sL^{r} \Q$, and on the second factor by the composition
\begin{equation}
\xymatrix{   \Cyl(\sL^{r} \Q, x_{+}) \ar[r] & \sL^{r} \Q \ar[r]^-{\psi_{s}^{r}} &   \sL^{r} \Q}
\end{equation}
where the second map is the time-$s$ gradient flow of the Morse function $f^{r}$. We think of elements of this moduli space as consisting of a pair of punctured holomorphic discs, together with a gradient flow line connecting their boundaries.

Taking the union over all positive real numbers, we obtain a moduli space
\begin{equation}
  \Brokecot(x_{-},x_{+}) \equiv  \coprod_{ s \in [0,+\infty) } \Cyl(x_{-}, \sL^{r} \Q) \times_{\psi_{s}^{r}}  \Cyl(\sL^{r} \Q, x_{+}).
\end{equation}
For generic choices of Hamiltonian data, this moduli space is a transverse fibre product, and hence is a smooth manifold of dimension
\begin{equation}
  \deg(x_-) - \deg(x_+) + 1,
\end{equation}
with boundary equal to
\begin{equation}
   \Brokecot^{0}(x_{-},x_{+}) \equiv  \Cyl(x_{-}, \sL^{r} \Q) \times_{\sL^{r} \Q}  \Cyl(\sL^{r} \Q, x_{+}).
\end{equation}

The same strategy used in producing the isomorphism in Equation \eqref{eq:product_orientation_two_half_cylinders} induces an isomorphism
\begin{equation} \label{eq:broken_moduli_space_annuli}
  |   \Brokecot(x_{-},x_{+}) | \otimes \ro_{x_+}   \cong  \ro_{x_-},
\end{equation}
where we fix the standard orientation of the interval $[0,+\infty)$.

There is a natural compactification of $\Brokecot(x_{-},x_{+}) $ obtained as follows: first, for $s \in [0,+\infty) $, define
\begin{equation}
   \Brokecotbar^{s}(x_{-},x_{+}) \equiv  \Cylbar(x_{-}, \sL^{r} \Q) \times_{\psi_{s}^{r}}  \Cylbar(\sL^{r} \Q, x_{+}).
\end{equation}
\begin{exercise}
Show that $   \Brokecot^{s}(x_{-},x_{+})   $ is compact.
\end{exercise}

Next, we define the stratum that corresponds to $s=\infty$: 
\begin{equation}
   \Brokecotbar^{\infty}(x_{-},x_{+}) \equiv \bigcup_{y_-,y_+}   \Broke(x_{-},y_-) \times \Treebar(y_-,y_+)   \times  \Brokecot(y_{+},x_{+}) ,
\end{equation}
where the moduli spaces $\Broke(x_{-},y_-) $ and $\Brokecot(y_{+},x_{+}) $ are respectively defined in Equations \eqref{eq:define_broke_x_y} and \eqref{eq:define_broke_cot}, whose topology is obtained by a combination of the Gromov-Floer topology on the moduli space of curves, and of its analogue for Morse theory on the gradient trajectories. The top dimensional stratum in the above decomposition is
\begin{multline} \label{eq:top_stratum_cyl_orbits_infinite_gradient}
\coprod_{y}  \Broke(x_{-},y) \times \Brokecot(y,x_{+}) = \\ \coprod_{y}    \Cyl(x_{-}, \sL^{r} \Q) \times_{\sL^{r} \Q} W^{u}(y) \times W^{s}(y) \times_{\sL^{r} \Q}  \Cylbar(\sL^{r} \Q, x_{+})   .
\end{multline}
\begin{exercise}
  Show that all components in Equation \eqref{eq:top_stratum_cyl_orbits_infinite_gradient} have dimension equal to $   \deg(x_-) - \deg(x_+) $.
\end{exercise}

Taking the union of these moduli spaces over all possible lengths of the gradient trajectory, we obtain the moduli space
\begin{equation}
    \Brokecotbar(x_{-},x_{+}) \equiv \coprod_{s \in [0,+\infty]}    \Brokecotbar^{s}(x_{-},x_{+}).
\end{equation}
The compactness of the moduli space of gradient trajectories, together with Gromov compactness, implies that $    \Brokecotbar(x_{-},x_{+}) $ is compact.

\subsection{Construction of the homotopy}

For generic choices of almost complex structures, the moduli spaces $  \Cylbar_{2,r}( x_{-}, x_{+}) $ and $\Brokecotbar( x_{-}, x_{+}  )   $  are both $0$-dimensional whenever $\deg(x_{-}) = \deg(x_{+}) - 1$. In this case, Equation \eqref{eq:orientation_parametrised_moduli_cylinder_boundary} associates a map
\begin{equation}
  \cH_{u} \co \ro_{x_+}   \to  \ro_{x_-}
\end{equation}
to every element $u$ of $ \Cylbar_{2,r}( x_{-}, x_{+})   $, while Equation \eqref{eq:broken_moduli_space_annuli} induces an isomorphism
\begin{equation}
  \cH_{(u_-,u_+)} \co \ro_{x_+}   \to  \ro_{x_-}
\end{equation}
for every pair $ (u_-,u_+) \in \Brokecotbar( x_{-}, x_{+}  ) $. 

We define a map
\begin{align} \notag
  \cH \co CF^{*}(H^{+}; \bZ) & \to CF^{*}(H^{-}; \bZ) \\
\cH | \ro_{x_+} & \equiv \bigoplus_{\deg(x_{-}) = \deg(x_{+}) - 1 } \left( \sum_{ u \in \Cylbar_{2,r}( x_{-}, x_{+})  } \cH_{u}  +  \sum_{ (u_-,u_+) \in \Brokecotbar( x_{-}, x_{+}  )  }   \cH_{(u_-,u_+)} \right).
\end{align}

\begin{lem}
Up to an overall sign, $\cH$ is a homotopy between the continuation map and the composition of $\Fam^{r} \circ \Gam^{r} $ with an isomorphism of $ CF^{*}(H^{+}; \bZ)   $.
\end{lem}
\begin{proof}
The overall sign is determined by the sign in Exercise \ref{ex:sign_boundary_continuation} that concerns the difference in orientation between the moduli space of solutions to the continuation map, and the corresponding boundary component of $\Cylbar_{2,r}( x_{-}, x_{+})  $. The isomorphism of $ CF^{*}(H^{+}; \bZ)  $ is given by  multiplication by $\pm 1$ on the two summands corresponding to orbits along which the pullback of $T \Q$ is either orientable or not, as in Lemma \ref{lem:maps_agree_up_to_uknown_sign}.

With this in mind, the argument that $\cH$ defines a homotopy is a standard use of cobordism:  whenever $ \deg(x_{-}) = \deg(x_{+})$, the moduli spaces  $ \Cyl_{2,r}( x_{-}, x_{+})  $  and $  \Brokecot( x_{-}, x_{+}  )  $ have dimension $1$. The compactifications are manifolds with boundary. The boundary strata are 
\begin{align}
  \partial  \Cylbar_{2,r}( x_{-}, x_{+}) & = \begin{cases}
&   \Cont( x_{-}, x_{+})  \\
& \displaystyle{\coprod_{\deg(x_{-}) = \deg(x_{-}^{0}) + 1}} \Cyl(x_{-},x_{-}^{0}) \times \Cyl_{2,r}( x_{-}^{0}, x_{+}) \\
& \displaystyle{\coprod_{\deg(x_{+}^{1}) +1  = \deg(x_{+}) }}  \Cyl_{2,r}( x_{-}, x_{+}^{1}) \times \Cyl(x_{+}^{1},x_{+})  \\
& \Cyl(x_{-}, \sL^{r} \Q) \times_{\sL^{r} \Q}  \Cyl(\sL^{r} \Q, x_{+}). 
\end{cases} \\
  \partial  \Brokecot( x_{-}, x_{+}) & = \begin{cases}
&   \Cyl(x_{-}, \sL^{r} \Q) \times_{\sL^{r} \Q}  \Cyl(\sL^{r} \Q, x_{+})  \\
& \displaystyle{\coprod_{\deg(x_{-}) = \deg(x_{-}^{0}) + 1}} \Cyl(x_{-},x_{-}^{0}) \times \Brokecot( x_{-}^{0}, x_{+}) \\
&\displaystyle{ \coprod_{\deg(x_{+}^{1}) +1  = \deg(x_{+}) }}  \Brokecot(x_{-}, x_{+}^{1}) \times \Cyl(x_{+}^{1},x_{+})  \\
&\displaystyle{ \bigcup_{\ind(y) = n - \deg(x_{+})}}  \Broke(x_{-},y) \times \Brokecot(y,x_{+}).
\end{cases}
\end{align}
Note the appearance of $ \Cyl(x_{-}, \sL^{r} \Q) \times_{\sL^{r} \Q}  \Cyl(\sL^{r} \Q, x_{+}) $  twice; the contributions of these boundary strata cancel.  The first boundary stratum of $  \Cylbar_{2,r}( x_{-}, x_{+})  $  gives rise to the continuation map, and the last boundary stratum of $\Brokecot( x_{-}, x_{+}) $ defines the composition $\Fam^{r} \circ \Gam^{r}  $, multiplied by the sign in the statement of Lemma \ref{lem:maps_agree_up_to_uknown_sign}. The remaining strata correspond to the compositions
\begin{equation}
  \cH \circ d \textrm{ and } d \circ \cH.
\end{equation}
We conclude that the moduli spaces $ \Cylbar_{2,r}( x_{-}, x_{+})  $ and $   \Brokecot( x_{-}, x_{+}) $  indeed define the desired homotopy.
\end{proof}

\subsection{Comparing orientations of the linearised problem} \label{sec:comp-orient-line}

Given a Hamiltonian orbit $x$, recall that $\ro_{x}^{+}$ is the orientation line associated to an operator on the plane with positive cylindrical end, and asymptotic conditions  obtained from $x$. For a surface
\begin{equation}
  \Sigma \in \Cyl^{R}_{2,r}, \quad R \in (-1,0]
\end{equation}
with compactification $\Sigmabar$, and Lagrangian subspaces $\{ L_i \}_{i = 1}^{r}$   of $\bC^{n}$, consider a linear Cauchy-Riemann operator
\begin{equation}
D_{\Sigma}(L_1, \ldots, L_r) \co W^{1,p}\left( (\Sigmabar, \partial \Sigmabar), (\bC^{n}, L_1, \ldots, L_r) \right) \to L^{p}(\Sigmabar, \bC^{n})
\end{equation}
on the space of functions on $\Sigmabar$ whose values at the $i$\th boundary component lie on the Lagrangians $L_i$.

Letting $\Para_{j}$ denote the family of Lagrangian subspaces which are parallel to  $L_j$, we can extend this operator to a map
\begin{equation} \label{eq:parametrised_operator_genus_0_many_holes}
  D_{\Sigmabar}(\Para_{1}, \ldots, \Para_r) \co   W^{1,p}\left( (\Sigmabar, \partial \Sigmabar), (\bC^{n}, L_1, \ldots, L_r) \right)  \oplus \bigoplus_{i=1}^{r} L_{i}^{\vee}  \to L^{p}(\Sigmabar, \bC^{n})
\end{equation}
where $L_i^{\vee}$ is the linear dual of $L_i$, and we have fixed an identification
\begin{equation}
  \bC^{n} \cong L_{i}^{\vee} \oplus L_i
\end{equation}
using the symplectic structure.
\begin{exercise} \label{ex:homogeneous_equation_regular}
Show that  the index of $  D_{\Sigmabar}(\Para_1, \ldots, \Para_r)   $ is $2n$. (Hint: first prove the result in the case $\Sigma$ consists of a cylinder with $r$ ghost discs by expressing the moduli space as a fibre product of regular moduli spaces. Then use gluing and invariance of the index to prove this for nearby surfaces in $\Cyl_{2,r} $).
\end{exercise}

Given a point in $\Sigmabar$, we obtain an evaluation map
  \begin{equation} \label{eq:moduli_space_consists_only_of_points}
 \ker(D_{\Sigmabar}(\Para_1, \ldots, \Para_r) )  \to \bC^{n}.
  \end{equation}
\begin{lem} \label{lem:para_CR_problem_constant}
If the Cauchy-Riemann equation  is homogeneous, $ D_{\Sigmabar}(\Para_1, \ldots, \Para_r) $ is surjective, and the evaluation map in Equation \eqref{eq:moduli_space_consists_only_of_points} is an isomorphism.
\end{lem}
\begin{proof}
Having computed the index in Exercise \ref{ex:homogeneous_equation_regular}, regularity follows from the computation that the real dimension of $  \ker(D_{\Sigmabar}(\Para_1, \ldots, \Para_r) )  $ is $2n$, which we will establish by showing that the evaluation map to $\bC^{n}$ is an isomorphism. To this end, note that $D_{\Sigmabar}(\Para_1, \ldots, \Para_r) $ is obtained by linearising  the moduli space of holomorphic  maps from $\Sigmabar$ to $\bC^{n}$ with boundary on affine Lagrangian subspaces, parallel to $L_i$. To prove the result, it suffices to show that all elements of this moduli space are constant, and that there is a unique such map through every point. Choosing a primitive for the standard symplectic form on $\bC^{n}$, we compute, using Stokes's theorem, that the energy of any element of the moduli space vanishes. All  solutions are therefore constant. One the other hand, since there is a unique subspace parallel to $L_i$ passing through every point, we see that the evaluation map to $\bC^{n}$ is indeed an isomorphism.
\end{proof}

To see the relevance of this parametrised problem to orienting $\Cyl_{2,r}( x_{-}, x_{+})  $, we note that whenever $u$ is an element of this moduli space, gluing the operators associated to $x_{\pm}$ to both ends of $D_{u}$ defines an operator homotopic to $D_{\Sigmabar}(L_1, \ldots, L_r) $ with Lagrangian boundary conditions $L_1, \ldots, L_{r-1}$ which linearise the cotangent boundary conditions.  In particular, there is a canonical isomorphism
 \begin{equation}
   \ro_{x_-}^{+}   \otimes |\det(D_u)| \otimes  \ro_{x_+}   \cong | \det(D_{\Sigmabar}(L_1, \ldots, L_r))  |
 \end{equation}

By considering the parametrised problem, we obtain an isomorphism:
\begin{equation} \label{eq:orient_moduli_space_via_moduli_C^n}
   \ro_{x_-}^{+}  \otimes |\Cyl_{2,r}( x_{-}, x_{+}) | \otimes  \ro_{x_+}   \cong |  \det(D_{\Sigmabar}(L_1, \ldots, L_r))   | \otimes  |\Cyl_{2,r}| .
\end{equation}
\begin{exercise}
Show that the isomorphism in Equation \eqref{eq:orientation_parametrised_moduli_cylinder_boundary} is induced from  Equation \eqref{eq:orient_moduli_space_via_moduli_C^n} and the standard orientation of $\bC^{n}$, via Equations \eqref{eq:dual_orientation_lines_orbit} and \eqref{eq:moduli_space_consists_only_of_points}.
\end{exercise}

In order to compare the orientation at the boundary of $ \Cylbar_{2,r} $ with the interior orientation, we must extend our construction to $\Cylbar_{2,r}^{-1}$.  Equip the complement of the $r$-roots of unity on the disc
\begin{equation}
  D^{2}_{r} \equiv D^{2} \setminus \{ e^{\frac{2 \pi j i}{r}} \}_{j=1}^{r} 
\end{equation}
with positive strip-like ends near each puncture.  Choose paths of symplectic matrices
\begin{equation} \label{eq:path_matrices_end}
  \Psi_{j}^{t} \co [0,1] \to \Sp(2n,\bR)
\end{equation}
starting at the identity with the property that
\begin{equation}
  \Psi_{j}^{1} L_{j} \textrm{ is transverse to } L_{j+1}.
\end{equation}
and consider the ``derivative''
\begin{equation}
  B_{j}^{t} \equiv \frac{d A_{j}^{t}}{dt}, \quad \exp(A_{j}^{t}) \equiv \Psi_{j}^{t}.
\end{equation}
Given a Cauchy-Riemann operator on the disc whose restriction to the $j$\th end is given by
\begin{equation} \label{eq:CR-operator_path_j}
 X \mapsto \partial_{s} X  + I\left(  \partial_{t} X  - B_{j}^{t} X \right),
\end{equation}
we obtain a Fredholm operator
\begin{equation} \label{eq:operator_genus_0_many_holes-minus}
  D_{D^{2}_{r}}^{-}(L_1, \ldots, L_r) \co W^{1,p}\left((D^{2}_{r}, \partial D^{2}_{r}), (\bC^{n}, L_1, \ldots, L_r)\right) \to L^{p}( D^{2}_{r}, \bC^{n}),
\end{equation}
whose source is the space of functions whose values on the arc between the $j$\th and $j+1$\st roots of unity lie in $L_j$. Here, the  Sobolev norms are defined with respect to a metric such that the embedding of each strip is an isometry.
\begin{rem}
The incorporation of the sign $-$ in the notation is justified by the fact that $   D_{D^{2}_{r}}^{-}(L_1, \ldots, L_r)  $ is homotopic to the result of gluing an operator on the plane to the linearisation of a Cauchy-Riemann operator on $Z^{-}_{r}$. 
\end{rem}

Repeating the same construction, using negative instead of positive ends, we obtain an operator which we denote
\begin{equation} \label{eq:operator_genus_0_many_holes-plus}
  D_{D^{2}_{r}}^{+}(L_1, \ldots, L_r) \co W^{1,p}\left((D^{2}_{r}, \partial D^{2}_{r}), (\bC^{n}, L_1, \ldots, L_r)\right) \to L^{p}( D^{2}_{r}, \bC^{n}).
\end{equation}

\begin{rem}
The homotopy class of Equations \eqref{eq:operator_genus_0_many_holes-minus} and \eqref{eq:operator_genus_0_many_holes-plus} in the space of Fredholm operators depends not only on the Lagrangians and whether the ends are positive or negative, but also on the choice of path in Equation \eqref{eq:path_matrices_end}. We elide this choice from the notation.
\end{rem}

There are two gluing constructions that we can apply to these operators. On the one hand, gluing two copies of $D^{2}_{r}  $ along positive and negative strip-like ends produces a surface
\begin{equation}
  D^{2}_{r} \#_{S} D^{2}_{r}  \in \Cyl_{2,r}
\end{equation}
for every positive real number $S$. At the level of determinant lines, we obtain an isomorphism:
\begin{multline} \label{eq:orientation_discs_positive_negative_glue_ends}
  \det\left( D_{D^{2}_{r}}^{+}(L_1, \ldots, L_r)  \right)  \otimes  \det\left( \bigoplus_{i=1}^{r} L_{i}^{\vee}  \right)   \otimes \det \left( D_{D^{2}_{r}}^{-}(L_1, \ldots, L_r)   \right)  \\ \to \det \left( D_{D^{2}_{r} \#_{S} D^{2}_{r} }^{+}(\Para_1, \ldots, \Para_r)   \right)  \cong \det(\bC^{n}) \cong \bR
\end{multline}
where we use Lemma \ref{lem:para_CR_problem_constant} to identify the determinant of the parametrised problem on $  D^{2}_{r} \#_{S} D^{2}_{r}   $  with $\bC^{n}$.

On the other hand, assume we are given paths $\Lambda_{j}^{t}  $ such that $ \Psi_{j}^{t} ( \Lambda_{j}^{t})$ has vanishing Maslov index (see Section \ref{sec:maslov-index-paths-1}).  By concatenating these paths along their endpoints, we obtain a loop of Lagrangians
\begin{equation}
 \Lambda(L_1, \ldots, L_r)  \equiv  \Lambda_{1} \#   \Lambda_{2} \# \cdots \# \Lambda_{r}, 
\end{equation}
to which we can associate a Fredholm operator on the disc:
\begin{equation} \label{eq:moduli_space_positive_negative_discs_C^n}
D_{\Lambda(L_1, \ldots, L_r)} \co W^{1,p}\left( (D^{2}, \partial D^{2}), (\bC^{n}, \Lambda) \right) \to L^{p}( D^{2}, \bC^{n}).
\end{equation}

To the paths $ \Lambda_{j}^{t} $, we also associate operators
\begin{equation}
  D^{\pm}_{ \Lambda_{j}^{t}  } \co W^{1,p}((\bC_{+},\bR), (\bC^{n}, \Lambda_{j}^{t}))  \to   L^{p}(\bC_{+}, \bC^{n})
\end{equation}
as in Section \ref{sec:maslov-index-paths-1}, where the sign $\pm$ indicates whether we are using positive or negative ends. The restrictions of these operators to the ends are given by Equation \eqref{eq:CR-operator_path_j}.  

\begin{exercise} \label{ex:trivialise_deform_product}
Construct a trivialisation of $   \det(D^{-}_{ \Lambda_{j}^{t}  })  $ by deforming this operator to the product of operators valued in $\bC$, and using Corollary \ref{cor:determinant_line}. 
\end{exercise}

By gluing the operator $D^{-}_{ \Lambda_{j}^{t}  } $   to the $j$\th strip-like ends of a disc with $r$ punctures, we obtain a canonical isomorphism
\begin{equation}
  \label{eq:det_isomorphism_disc_r_punctures_positive_glue}
\det(D_{D^{2}_{r}}^{-}(L_1, \ldots, L_r))  \cong \det(D_{\Lambda(L_1, \ldots, L_r)})
\end{equation}
using the trivialisation from Exercise \ref{ex:trivialise_deform_product}.

In order to obtain a similar result for $ D_{D^{2}_{r}}^{-}(L_1, \ldots, L_r) $, we first use Equation \eqref{eq:glue_positive_negative_disc} to induce an isomorphim
\begin{equation}
  \det(D^{+}_{ \Lambda_{j}^{t}  }) \cong \det(L_j)
\end{equation}
from the trivialisation of $ \det(D^{-}_{ \Lambda_{j}^{t}  }) $. By gluing, we obtain a canonical isomorphism
\begin{equation} \label{eq:det_isomorphism_disc_r_punctures_negative_glue}
\det(D_{D^{2}_{r}}^{+}(L_1, \ldots, L_r))  \otimes  \det\left(\bigoplus_{j=1}^{r} L_j \right)  \cong \det(D_{\Lambda^{-1}(L_1, \ldots, L_r)}),
\end{equation}
 where $\Lambda^{-1}(L_1, \ldots, L_r) $ is the path obtained by traversing $\Lambda(L_1, \ldots, L_r)  $ backwards.

Using the isomorphism
\begin{equation}
  \det\left(\bigoplus_{j=1}^{r} L_j \right)  \cong \det\left(\bigoplus_{j=1}^{r} L_j^{\vee} \right) 
\end{equation}
induced by duality, and Lemma \ref{lem:inverse_path_det_line_is_inverse}, we obtain a map
\begin{multline} \label{eq:orientation_discs_positive_negative_loops_dual}
   \det\left( D_{D^{2}_{r}}^{+}(L_1, \ldots, L_r)  \right)  \otimes  \det\left( \bigoplus_{i=1}^{r} L_{i}^{\vee}  \right)   \otimes \det \left( D_{D^{2}_{r}}^{-}(L_1, \ldots, L_r)   \right)  \\ \to \det(D_{\Lambda(L_1, \ldots, L_r)}) \otimes  \det(D_{\Lambda^{-1}(L_1, \ldots, L_r)}) \cong \bR.
\end{multline}

We can associate to Equations \eqref{eq:orientation_discs_positive_negative_glue_ends} and \eqref{eq:orientation_discs_positive_negative_loops_dual} a sign which is $1$ if these maps induce the same maps on orientation spaces, and $-1$ otherwise. Since the two maps we are comparing are both invariant under homotopies; this sign depends only on the homotopy class of the loop $\Lambda(L_1, \ldots, L_r)$. We write
\begin{equation}
\textrm{?`}_{r}^{n}(\mu)
\end{equation}
for the sign associated to loops of Maslov index $\mu$ in dimension $n$, with $r$ points.

\begin{proof}[Proof of Lemma \ref{lem:maps_agree_up_to_uknown_sign}]
Given a Hamiltonian orbit $x$ in $\TQ$, let $\Lambda_{x}$ denote the loop of Lagrangians obtained by applying the preferred trivialisation to $ (q \circ x)^{*}  \TQ $, and denote by $\Lambda_{x}^{-1} $ the inverse loop. Given a pair
\begin{equation} 
 (u_-,u_+) \in   \Cyl(x_{-}, \sL^{r} \Q) \times_{\sL^{r} \Q}  \Cyl(\sL^{r} \Q, x_{+}) , \end{equation}
the isomorphism in Equation \eqref{eq:product_orientation_two_half_cylinders} is induced, via gluing $D_{\Psi_{x_+}}$ and $D_{\Psi_{x_-}}^{+}$ to $D_{u_{\pm}}$, from the isomorphism in Equation \eqref{eq:orientation_discs_positive_negative_loops_dual}. We conclude that the difference in sign between Equation \eqref{eq:orientation_moduli_cylinders_two_punctured discs} and \eqref{eq:product_orientation_two_half_cylinders} is given by $    \textrm{?`}_{r}^{n}(0)$ if $(q \circ x)^{*}  \TQ  $ is orientable, and $  \textrm{?`}_{r}^{n}(1) $ otherwise.
\end{proof}



\begin{bibdiv}
\begin{biblist}

\bib{abbondandolo}{article}{
   author={Abbondandolo, Alberto},
   author={Majer, Pietro},
   title={Lectures on the Morse complex for infinite-dimensional manifolds},
   conference={
      title={Morse theoretic methods in nonlinear analysis and in symplectic
      topology},
   },
   book={
      series={NATO Sci. Ser. II Math. Phys. Chem.},
      volume={217},
      publisher={Springer},
      place={Dordrecht},
   },
   date={2006},
   pages={1--74},
}

\bib{APS}{article}{
   author={Abbondandolo, Alberto},
   author={Portaluri, Alessandro},
   author={Schwarz, Matthias},
   title={The homology of path spaces and Floer homology with conormal
   boundary conditions},
   journal={J. Fixed Point Theory Appl.},
   volume={4},
   date={2008},
   number={2},
   pages={263--293},
   issn={1661-7738},
   review={\MR{2465553 (2009i:53090)}},
   doi={10.1007/s11784-008-0097-y},
}
\bib{AS}{article}{
   author={Abbondandolo, Alberto},
   author={Schwarz, Matthias},
   title={On the Floer homology of cotangent bundles},
   journal={Comm. Pure Appl. Math.},
   volume={59},
   date={2006},
   number={2},
   pages={254--316},
   issn={0010-3640},
   review={\MR{2190223 (2006m:53137)}},
   doi={10.1002/cpa.20090},
}

\bib{AS-on-the-product}{article}{
   author={Abbondandolo, Alberto},
   author={Schwarz, Matthias},
 title={On product structures in Floer homology of cotangent bundles}, 
conference={
title={Global differential geometry},
},
book={
series={Springer Proceedings in Mathematics},
volume={17},
publisher={Springer},
}
date={2012},
pages={491-–521},
}


\bib{AS-product}{article}{
   author={Abbondandolo, Alberto},
   author={Schwarz, Matthias},
   title={Floer homology of cotangent bundles and the loop product},
   journal={Geom. Topol.},
   volume={14},
   date={2010},
   number={3},
   pages={1569--1722},
   issn={1465-3060},
   review={\MR{2679580 (2011k:53126)}},
   doi={10.2140/gt.2010.14.1569},
}

\bib{AS-signs}{article}{
   author={Abbondandolo, Alberto},
   author={Schwarz, Matthias},
title={Corrigendum: On the Floer homology of cotangent bundles},
eprint={arXiv:1309.0148},
}

\bib{A-HMS-toric}{article}{
   author={Abouzaid, Mohammed},
   title={Morse homology, tropical geometry, and homological mirror symmetry
   for toric varieties},
   journal={Selecta Math. (N.S.)},
   volume={15},
   date={2009},
   number={2},
   pages={189--270},
   issn={1022-1824},
   review={\MR{2529936 (2011h:53123)}},
   doi={10.1007/s00029-009-0492-2},
}

\bib{A-generate}{article}{
   author={Abouzaid, Mohammed},
   title={A geometric criterion for generating the Fukaya category},
   journal={Publ. Math. Inst. Hautes \'Etudes Sci.},
   number={112},
   date={2010},
   pages={191--240},
   issn={0073-8301},
   review={\MR{2737980}},
   doi={10.1007/s10240-010-0028-5},
}

\bib{A-loops}{article}{
   author={Abouzaid, Mohammed},
   title={On the wrapped Fukaya category and based loops},
   journal={J. Symplectic Geom.},
   volume={10},
   date={2012},
   number={1},
   pages={27--79},
   issn={1527-5256},
   review={\MR{2904032}},
}
\bib{A-cotangent-generate}{article}{
   author={Abouzaid, Mohammed},
   title={A cotangent fibre generates the Fukaya category},
   journal={Adv. Math.},
   volume={228},
   date={2011},
   number={2},
   pages={894--939},
   issn={0001-8708},
   review={\MR{2822213 (2012m:53192)}},
   doi={10.1016/j.aim.2011.06.007},
}

\bib{ASeidel}{article}{
   author={Abouzaid, Mohammed},
   author={Seidel, Paul},
   title={An open string analogue of Viterbo functoriality},
   journal={Geom. Topol.},
   volume={14},
   date={2010},
   number={2},
   pages={627--718},
   issn={1465-3060},
   review={\MR{2602848 (2011g:53190)}},
   doi={10.2140/gt.2010.14.627},
}

\bib{Albers}{article}{
   author={Albers, Peter},
   title={A Lagrangian Piunikhin-Salamon-Schwarz morphism and two comparison
   homomorphisms in Floer homology},
   journal={Int. Math. Res. Not. IMRN},
   date={2008},
   number={4},
   pages={Art. ID rnm134, 56},
   issn={1073-7928},
   review={\MR{2424172 (2009e:53106)}},
   doi={10.1093/imrn/rnm134},
}

\bib{ABW}{article}{
   author={Albers, Peter},
   author={Bramham, Barney},
   author={Wendl, Chris},
   title={On nonseparating contact hypersurfaces in symplectic 4-manifolds},
   journal={Algebr. Geom. Topol.},
   volume={10},
   date={2010},
   number={2},
   pages={697--737},
   issn={1472-2747},
   review={\MR{2606798 (2011j:53173)}},
   doi={10.2140/agt.2010.10.697},
}

\bib{AD}{book}{
   author={Audin, Mich{\`e}le},
   author={Damian, Mihai},
   title={Th\'eorie de Morse et homologie de Floer},
   language={French},
   series={Savoirs Actuels (Les Ulis). [Current Scholarship (Les Ulis)]},
   publisher={EDP Sciences, Les Ulis},
   date={2010},
   pages={xii+548},
   isbn={978-2-7598-0518-1},
   isbn={978-2-271-07087-6},
   review={\MR{2839638}},
}

\bib{BC}{article}{
   author={Biran, Paul},
   author={Cornea, Octav},
   title={Lagrangian topology and enumerative geometry},
   journal={Geom. Topol.},
   volume={16},
   date={2012},
   number={2},
   pages={963--1052},
   issn={1465-3060},
   review={\MR{2928987}},
   doi={10.2140/gt.2012.16.963},
}
\bib{BH}{article}{
   author={Banyaga, Augustin},
   author={Hurtubise, David E.},
   title={Morse-Bott homology},
   journal={Trans. Amer. Math. Soc.},
   volume={362},
   date={2010},
   number={8},
   pages={3997--4043},
   issn={0002-9947},
   review={\MR{2608393 (2011e:57054)}},
   doi={10.1090/S0002-9947-10-05073-7},
}

\bib{basu}{article}{
author={Basu, Somnath},
title={Of Sullivan models, Massey products, and twisted Pontrjagin products},
eprint={arXiv:1209.3226},
}

\bib{BEE}{article}{
   author={Bourgeois, Fr{\'e}d{\'e}ric},
   author={Ekholm, Tobias},
   author={Eliashberg, Yasha},
   title={Effect of Legendrian surgery},
   note={With an appendix by Sheel Ganatra and Maksim Maydanskiy},
   journal={Geom. Topol.},
   volume={16},
   date={2012},
   number={1},
   pages={301--389},
   issn={1465-3060},
   review={\MR{2916289}},
   doi={10.2140/gt.2012.16.301},
}

\bib{BM}{article}{
   author={Bourgeois, Fr{\'e}d{\'e}ric},
   author={Mohnke, Klaus},
   title={Coherent orientations in symplectic field theory},
   journal={Math. Z.},
   volume={248},
   date={2004},
   number={1},
   pages={123--146},
   issn={0025-5874},
   review={\MR{2092725 (2005g:53173)}},
   doi={10.1007/s00209-004-0656-x},
}


\bib{BO}{article}{
   author={Bourgeois, Fr{\'e}d{\'e}ric},
   author={Oancea, Alexandru},
   title={An exact sequence for contact- and symplectic homology},
   journal={Invent. Math.},
   volume={175},
   date={2009},
   number={3},
   pages={611--680},
   issn={0020-9910},
   review={\MR{2471597 (2010e:53149)}},
   doi={10.1007/s00222-008-0159-1},
}

\bib{BO1}{article}{
   author={Bourgeois, Fr{\'e}d{\'e}ric},
   author={Oancea, Alexandru},
   title={Fredholm theory and transversality for the parametrized and for
   the $S^1$-invariant symplectic action},
   journal={J. Eur. Math. Soc. (JEMS)},
   volume={12},
   date={2010},
   number={5},
   pages={1181--1229},
   issn={1435-9855},
   review={\MR{2677614 (2012a:53174)}},
   doi={10.4171/JEMS/227},
}
\bib{BO2}{article}{
   author={Bourgeois, Fr{\'e}d{\'e}ric},
   author={Oancea, Alexandru},
     title = {$S^1$-equivariant symplectic homology and linearized contact homology},
   eprint = {arXiv:1212.3731},
}

\bib{CS}{article}{
   author={Chas, Moira},
   author={Sullivan, Dennis},
   title={String topology},
eprint={arXiv:9911159},
}


\bib{CL}{article}{
   author={Cieliebak, Kai},
   author={Latschev, Janko},
   title={The role of string topology in symplectic field theory},
   conference={
      title={New perspectives and challenges in symplectic field theory},
   },
   book={
      series={CRM Proc. Lecture Notes},
      volume={49},
      publisher={Amer. Math. Soc.},
      place={Providence, RI},
   },
   date={2009},
   pages={113--146},
   review={\MR{2555935 (2010j:53187)}},
}

\bib{CFH}{article}{
   author={Cieliebak, K.},
   author={Floer, A.},
   author={Hofer, H.},
   title={Symplectic homology. II. A general construction},
   journal={Math. Z.},
   volume={218},
   date={1995},
   number={1},
   pages={103--122},
   issn={0025-5874},
   review={\MR{1312580 (95m:58055)}},
   doi={10.1007/BF02571891},
}
\bib{CJ}{article}{
   author={Cohen, Ralph L.},
   author={Jones, John D. S.},
   title={A homotopy theoretic realization of string topology},
   journal={Math. Ann.},
   volume={324},
   date={2002},
   number={4},
   pages={773--798},
   issn={0025-5831},
   review={\MR{1942249 (2004c:55019)}},
   doi={10.1007/s00208-002-0362-0},
}
\bib{CKS}{article}{
   author={Cohen, Ralph L.},
   author={Klein, John R.},
   author={Sullivan, Dennis},
   title={The homotopy invariance of the string topology loop product and
   string bracket},
   journal={J. Topol.},
   volume={1},
   date={2008},
   number={2},
   pages={391--408},
   issn={1753-8416},
   review={\MR{2399136 (2009h:55004)}},
   doi={10.1112/jtopol/jtn001},
}

\bib{CV}{article}{
   author={Cohen, Ralph L.},
   author={Voronov, Alexander A.},
   title={Notes on string topology},
   conference={
      title={String topology and cyclic homology},
   },
   book={
      series={Adv. Courses Math. CRM Barcelona},
      publisher={Birkh\"auser},
      place={Basel},
   },
   date={2006},
   pages={1--95},
   review={\MR{2240287}},
}

\bib{da-silva}{thesis}{
   author={da Silva, Vin}
title={Products in the Symplectic Floer Homology of Lagrangian Intersections},
type={Ph. D.},
organization={Merton College, Oxford},
}

\bib{EGH}{article}{
   author={Eliashberg, Y.},
   author={Givental, A.},
   author={Hofer, H.},
   title={Introduction to symplectic field theory},
   note={GAFA 2000 (Tel Aviv, 1999)},
   journal={Geom. Funct. Anal.},
   date={2000},
   number={Special Volume},
   pages={560--673},
   issn={1016-443X},
}
\bib{Floer-gradient}{article}{
   author={Floer, Andreas},
   title={The unregularized gradient flow of the symplectic action},
   journal={Comm. Pure Appl. Math.},
   volume={41},
   date={1988},
   number={6},
   pages={775--813},
   issn={0010-3640},
   review={\MR{948771 (89g:58065)}},
   doi={10.1002/cpa.3160410603},
}
\bib{Floer-index}{article}{
   author={Floer, Andreas},
   title={A relative Morse index for the symplectic action},
   journal={Comm. Pure Appl. Math.},
   volume={41},
   date={1988},
   number={4},
   pages={393--407},
   issn={0010-3640},
   review={\MR{933228 (89f:58055)}},
   doi={10.1002/cpa.3160410402},
}

\bib{Floer-JDG}{article}{
   author={Floer, Andreas},
   title={Morse theory for Lagrangian intersections},
   journal={J. Differential Geom.},
   volume={28},
   date={1988},
   number={3},
   pages={513--547},
   issn={0022-040X},
   review={\MR{965228 (90f:58058)}},
}
\bib{floer-Picard}{article}{
   author={Floer, A.},
   title={Monopoles on asymptotically flat manifolds},
   conference={
      title={The Floer memorial volume},
   },
   book={
      series={Progr. Math.},
      volume={133},
      publisher={Birkh\"auser},
      place={Basel},
   },
   date={1995},
   pages={3--41},
   review={\MR{1362821 (96j:58024)}},
}
\bib{FH}{article}{
   author={Floer, A.},
   author={Hofer, H.},
   title={Coherent orientations for periodic orbit problems in symplectic
   geometry},
   journal={Math. Z.},
   volume={212},
   date={1993},
   number={1},
   pages={13--38},
   issn={0025-5874},
   review={\MR{1200162 (94m:58036)}},
   doi={10.1007/BF02571639},
}

\bib{FH-SH}{article}{
   author={Floer, A.},
   author={Hofer, H.},
   title={Symplectic homology. I. Open sets in ${\bf C}^n$},
   journal={Math. Z.},
   volume={215},
   date={1994},
   number={1},
   pages={37--88},
   issn={0025-5874},
   review={\MR{1254813 (95b:58059)}},
   doi={10.1007/BF02571699},
}

\bib{FHS}{article}{
   author={Floer, Andreas},
   author={Hofer, Helmut},
   author={Salamon, Dietmar},
   title={Transversality in elliptic Morse theory for the symplectic action},
   journal={Duke Math. J.},
   volume={80},
   date={1995},
   number={1},
   pages={251--292},
   issn={0012-7094},
   review={\MR{1360618 (96h:58024)}},
   doi={10.1215/S0012-7094-95-08010-7},
}

\bib{Fukaya-family}{article}{
   author={Fukaya, Kenji},
title={Floer homology for families - report of a project in progress},
eprint={http://www.math.kyoto-u.ac.jp/~fukaya/familyy.pdf},
}

\bib{FOOO}{book}{
   author={Fukaya, Kenji},
   author={Oh, Yong-Geun},
   author={Ohta, Hiroshi},
   author={Ono, Kaoru},
   title={Lagrangian intersection Floer theory: anomaly and obstruction.
   Part I},
   series={AMS/IP Studies in Advanced Mathematics},
   volume={46},
   publisher={American Mathematical Society},
   place={Providence, RI},
   date={2009},
   }

\bib{FOOO-II}{book}{
   author={Fukaya, Kenji},
   author={Oh, Yong-Geun},
   author={Ohta, Hiroshi},
   author={Ono, Kaoru},
   title={Lagrangian intersection Floer theory: anomaly and obstruction.
   Part II},
   series={AMS/IP Studies in Advanced Mathematics},
   volume={46},
   publisher={American Mathematical Society},
   place={Providence, RI},
   date={2009},
   pages={i--xii and 397--805},
   isbn={978-0-8218-4837-1},
   review={\MR{2548482 (2011c:53218)}},
}

\bib{FOOO-sign}{article}{
   author={Fukaya, Kenji},
   author={Oh, Yong-Geun},
   author={Ohta, Hiroshi},
   author={Ono, Kaoru},
title={Lagrangian Floer theory and mirror symmetry on compact toric manifolds},
eprint={arXiv:1009.1648},

}

\bib{FSS}{article}{
   author={Fukaya, Kenji},
   author={Seidel, Paul},
   author={Smith, Ivan},
   title={Exact Lagrangian submanifolds in simply-connected cotangent
   bundles},
   journal={Invent. Math.},
   volume={172},
   date={2008},
   number={1},
   pages={1--27},
   issn={0020-9910},
   review={\MR{2385665 (2009a:53142)}},
   doi={10.1007/s00222-007-0092-8},
}

\bib{FSS2}{article}{
   author={Fukaya, Kenji},
   author={Seidel, Paul},
   author={Smith, Ivan},
   title={The Symplectic Geometry of Cotangent Bundles from a Categorical Viewpoint},
   conference={
      title={Homological Mirror Symmetry},
   },
   book={
      series={Lecture Notes in Physics},
      volume={757},
      publisher={Springer},
      place={Berlin / Heidelberg},
   },
   date={2009},
   pages={1--26},
   issn={1616-6361},
}

\bib{GTV}{article}{
   author={G{\'a}lvez-Carrillo, Imma},
   author={Tonks, Andrew},
   author={Vallette, Bruno},
   title={Homotopy Batalin-Vilkovisky algebras},
   journal={J. Noncommut. Geom.},
   volume={6},
   date={2012},
   number={3},
   pages={539--602},
   issn={1661-6952},
   review={\MR{2956319}},
   doi={10.4171/JNCG/99},
}
	


\bib{getzler}{article}{
   author={Getzler, E.},
   title={Batalin-Vilkovisky algebras and two-dimensional topological field
   theories},
   journal={Comm. Math. Phys.},
   volume={159},
   date={1994},
   number={2},
   pages={265--285},
   issn={0010-3616},
   review={\MR{1256989 (95h:81099)}},
}

\bib{goodwillie}{article}{
   author={Goodwillie, Thomas G.},
   title={Cyclic homology, derivations, and the free loopspace},
   journal={Topology},
   volume={24},
   date={1985},
   number={2},
   pages={187--215},
   issn={0040-9383},
   review={\MR{793184 (87c:18009)}},
   doi={10.1016/0040-9383(85)90055-2},
}
\bib{hatcher}{book}{
   author={Hatcher, Allen},
   title={Algebraic topology},
   publisher={Cambridge University Press},
   place={Cambridge},
   date={2002},
   pages={xii+544},
   isbn={0-521-79160-X},
   isbn={0-521-79540-0},
   review={\MR{1867354 (2002k:55001)}},
}
\bib{KT}{article}{
   author={Kirby, R. C.},
   author={Taylor, L. R.},
   title={${\rm Pin}$ structures on low-dimensional manifolds},
   conference={
      title={Geometry of low-dimensional manifolds, 2},
      address={Durham},
      date={1989},
   },
   book={
      series={London Math. Soc. Lecture Note Ser.},
      volume={151},
      publisher={Cambridge Univ. Press},
      place={Cambridge},
   },
   date={1990},
   pages={177--242},
   review={\MR{1171915 (94b:57031)}},
}

\bib{kont-ENS}{article}{
  author={Kontsevich, Maxim},
  title={Lectures at ENS Paris, Spring 1998},
  booktitle={Notes by J. Bellaiche, J.-F. Dat, I. Marin, G. Racinet and H. Randriambololona.},
}

\bib{kragh-1}{article}{
author={Thomas Kragh},
title={The Viterbo Transfer as a Map of Spectra},
eprint={arXiv:0712.2533},
}

\bib{Latour}{article}{
   author={Latour, Fran{\c{c}}ois},
   title={Existence de $1$-formes ferm\'ees non singuli\`eres dans une
   classe de cohomologie de de Rham},
   language={French},
   journal={Inst. Hautes \'Etudes Sci. Publ. Math.},
   number={80},
   date={1994},
   pages={135--194 (1995)},
   issn={0073-8301},
   review={\MR{1320607 (96f:57030)}},
}

\bib{Laudenbach}{article}{
   author={Laudenbach, Fran{\c{c}}ois},
   title={A note on the Chas-Sullivan product},
   journal={Enseign. Math. (2)},
   volume={57},
   date={2011},
   number={1-2},
   pages={3--21},
   issn={0013-8584},
   review={\MR{2850582}},
}
\bib{Liu}{article}{
   author={Liu, Chiu-Chu Melissa},
title={Moduli of J-holomorphic curves with Lagrangian boundary conditions and open Gromov-Witten invariants for an $S^1$-equivariant pair},
eprint={arXiv:0210257},
}

\bib{LS}{article}{
   author={Longoni, Riccardo},
   author={Salvatore, Paolo},
   title={Configuration spaces are not homotopy invariant},
   journal={Topology},
   volume={44},
   date={2005},
   number={2},
   pages={375--380},
   issn={0040-9383},
   review={\MR{2114713 (2005k:55024)}},
   doi={10.1016/j.top.2004.11.002},
}

\bib{lurie}{article}{
   author={Lurie, Jacob},
   title={On the classification of topological field theories},
   conference={
      title={Current developments in mathematics, 2008},
   },
   book={
      publisher={Int. Press, Somerville, MA},
   },
   date={2009},
   pages={129--280},
   review={\MR{2555928 (2010k:57064)}},
}
\bib{may}{book}{
   author={May, J. P.},
   title={The geometry of iterated loop spaces},
   note={Lectures Notes in Mathematics, Vol. 271},
   publisher={Springer-Verlag},
   place={Berlin},
   date={1972},
   pages={viii+175},
   review={\MR{0420610 (54 \#8623b)}},
}

\bib{oancea}{article}{
   author={Oancea, Alexandru},
   title={A survey of Floer homology for manifolds with contact type
   boundary or symplectic homology},
   conference={
      title={Symplectic geometry and Floer homology. A survey of the Floer
      homology for manifolds with contact type boundary or symplectic
      homology},
   },
   book={
      series={Ensaios Mat.},
      volume={7},
      publisher={Soc. Brasil. Mat.},
      place={Rio de Janeiro},
   },
   date={2004},
   pages={51--91},
   review={\MR{2100955 (2005i:53116)}},
}
\bib{oh-gap}{article}{
   author={Oh, Yong-Geun},
   title={Symplectic topology as the geometry of action functional. I.
   Relative Floer theory on the cotangent bundle},
   journal={J. Differential Geom.},
   volume={46},
   date={1997},
   number={3},
   pages={499--577},
   issn={0022-040X},
   review={\MR{1484890 (99a:58032)}},
}
\bib{oh-fix-gap}{article}{
   author={Oh, Yong-Geun},
   title={Floer homology and its continuity for non-compact Lagrangian
   submanifolds},
   journal={Turkish J. Math.},
   volume={25},
   date={2001},
   number={1},
   pages={103--124},
   issn={1300-0098},
   review={\MR{1829082 (2002k:53171)}},
}

\bib{MS-baby}{book}{
   author={McDuff, Dusa},
   author={Salamon, Dietmar},
   title={Introduction to symplectic topology},
   series={Oxford Mathematical Monographs},
   edition={2},
   publisher={The Clarendon Press Oxford University Press},
   place={New York},
   date={1998},
   pages={x+486},
   isbn={0-19-850451-9},
   review={\MR{1698616 (2000g:53098)}},
}


\bib{MS}{book}{
   author={McDuff, Dusa},
   author={Salamon, Dietmar},
   title={$J$-holomorphic curves and symplectic topology},
   series={American Mathematical Society Colloquium Publications},
   volume={52},
   edition={2},
   publisher={American Mathematical Society},
   place={Providence, RI},
   date={2012},
   pages={xiv+726},
   isbn={978-0-8218-8746-2},
   review={\MR{2954391}},
}





\bib{milnor}{book}{
   author={Milnor, J.},
   title={Morse theory},
   series={Based on lecture notes by M. Spivak and R. Wells. Annals of
   Mathematics Studies, No. 51},
   publisher={Princeton University Press},
   place={Princeton, N.J.},
   date={1963},
   pages={vi+153},
   review={\MR{0163331 (29 \#634)}},
}

\bib{milnor-stasheff}{book}{
   author={Milnor, John W.},
   author={Stasheff, James D.},
   title={Characteristic classes},
   note={Annals of Mathematics Studies, No. 76},
   publisher={Princeton University Press},
   place={Princeton, N. J.},
   date={1974},
   pages={vii+331},
   review={\MR{0440554 (55 \#13428)}},
}

\bib{MB}{article}{
   author={Morse, Marston},
   author={Baiada, Emilio},
   title={Homotopy and homology related to the Schoenflies problem},
   journal={Ann. of Math. (2)},
   volume={58},
   date={1953},
   pages={142--165},
   issn={0003-486X},
}   

\bib{NZ}{article}{
   author={Nadler, David},
   author={Zaslow, Eric},
   title={Constructible sheaves and the Fukaya category},
   journal={J. Amer. Math. Soc.},
   volume={22},
   date={2009},
   number={1},
   pages={233--286},
   issn={0894-0347},
   review={\MR{2449059 (2010a:53186)}},
   doi={10.1090/S0894-0347-08-00612-7},
}

\bib{oancea-K}{article}{
   author={Oancea, Alexandru},
   title={The K\"unneth formula in Floer homology for manifolds with
   restricted contact type boundary},
   journal={Math. Ann.},
   volume={334},
   date={2006},
   number={1},
   pages={65--89},
   issn={0025-5831},
   review={\MR{2208949 (2006k:53155)}},
   doi={10.1007/s00208-005-0700-0},
}

\bib{oh}{article}{
   author={Oh, Yong-Geun},
   title={Removal of boundary singularities of pseudo-holomorphic curves
   with Lagrangian boundary conditions},
   journal={Comm. Pure Appl. Math.},
   volume={45},
   date={1992},
   number={1},
   pages={121--139},
   issn={0010-3640},
   review={\MR{1135926 (92k:58065)}},
   doi={10.1002/cpa.3160450106},
}

\bib{ritter}{article}{
   author={Ritter, Alexander F.},
   title={Topological quantum field theory structure on symplectic
   cohomology},
   journal={J. Topol.},
   volume={6},
   date={2013},
   number={2},
   pages={391--489},
   issn={1753-8416},
   review={\MR{3065181}},
   doi={10.1112/jtopol/jts038},
}
\bib{robbin-salamon}{article}{
   author={Robbin, Joel},
   author={Salamon, Dietmar},
   title={The spectral flow and the Maslov index},
   journal={Bull. London Math. Soc.},
   volume={27},
   date={1995},
   number={1},
   pages={1--33},
   issn={0024-6093},
   review={\MR{1331677 (96d:58021)}},
   doi={10.1112/blms/27.1.1},
}
	

\bib{salamon-notes}{article}{
   author={Salamon, Dietmar},
   title={Lectures on Floer homology},
   conference={
      title={Symplectic geometry and topology},
      address={Park City, UT},
      date={1997},
   },
   book={
      series={IAS/Park City Math. Ser.},
      volume={7},
      publisher={Amer. Math. Soc.},
      place={Providence, RI},
   },
   date={1999},
   pages={143--229},
   review={\MR{1702944 (2000g:53100)}},
}

\bib{SW-06}{article}{
   author={Salamon, D. A.},
   author={Weber, J.},
   title={Floer homology and the heat flow},
   journal={Geom. Funct. Anal.},
   volume={16},
   date={2006},
   number={5},
   pages={1050--1138},
   issn={1016-443X},
   review={\MR{2276534 (2007k:53154)}},
   doi={10.1007/s00039-006-0577-4},
}


\bib{Schwarz}{article}{
     author={Schwarz, Matthias},
title={Cohomology Operations from $S^1$-Cobordisms in Floer Homology},
eprint={http://www.math.uni-leipzig.de/~schwarz/diss.pdf},
}
\bib{seidel-Graded}{article}{
   author={Seidel, Paul},
   title={Graded Lagrangian submanifolds},
   language={English, with English and French summaries},
   journal={Bull. Soc. Math. France},
   volume={128},
   date={2000},
   number={1},
   pages={103--149},
   issn={0037-9484},
   review={\MR{1765826 (2001c:53114)}},
}

\bib{Schwarz-equivalence}{article}{
   author={Schwarz, Matthias},
   title={Equivalences for Morse homology},
   conference={
      title={Geometry and topology in dynamics (Winston-Salem, NC, 1998/San
      Antonio, TX, 1999)},
   },
   book={
      series={Contemp. Math.},
      volume={246},
      publisher={Amer. Math. Soc.},
      place={Providence, RI},
   },
   date={1999},
   pages={197--216},
   review={\MR{1732382 (2000j:57070)}},
   doi={10.1090/conm/246/03785},
}

\bib{Seidel-ICM}{article}{
   author={Seidel, Paul},
   title={Fukaya categories and deformations},
   conference={
      title={},
      address={Beijing},
      date={2002},
   },
   book={
      publisher={Higher Ed. Press},
      place={Beijing},
   },
   date={2002},
   pages={351--360},
   review={\MR{1957046 (2004a:53110)}},
}


\bib{seidel-biased}{article}{
   author={Seidel, Paul},
   title={A biased view of symplectic cohomology},
   conference={
      title={Current developments in mathematics, 2006},
   },
   book={
      publisher={Int. Press, Somerville, MA},
   },
   date={2008},
   pages={211--253},
   review={\MR{2459307 (2010k:53153)}},
}

\bib{seidel-Book}{book}{
   author={Seidel, Paul},
   title={Fukaya categories and Picard-Lefschetz theory},
   series={Zurich Lectures in Advanced Mathematics},
   publisher={European Mathematical Society (EMS), Z\"urich},
   date={2008},
   pages={viii+326},
   isbn={978-3-03719-063-0},
   review={\MR{2441780 (2009f:53143)}},
   doi={10.4171/063},
}
\bib{seidel-12}{article}{
   author={Seidel, Paul},
   title={Fukaya $A_\infty$-structures associated to Lefschetz
   fibrations. I},
   journal={J. Symplectic Geom.},
   volume={10},
   date={2012},
   number={3},
   pages={325--388},
   issn={1527-5256},
   review={\MR{2983434}},
}
\bib{viterbo-94}{article}{
   author={Viterbo, Claude},
   title={Generating functions, symplectic geometry, and applications},
   conference={
      title={ 2},
      address={Z\"urich},
      date={1994},
   },
   book={
      publisher={Birkh\"auser},
      place={Basel},
   },
   date={1995},
   pages={537--547},
   review={\MR{1403954 (97m:58084)}},
}

\bib{viterbo-97}{article}{
   author={Viterbo, Claude},
   title={Exact Lagrange submanifolds, periodic orbits and the cohomology of
   free loop spaces},
   journal={J. Differential Geom.},
   volume={47},
   date={1997},
   number={3},
   pages={420--468},
   issn={0022-040X},
   review={\MR{1617648 (99e:58080)}},
}

\bib{viterbo-99}{article}{
   author={Viterbo, C.},
   title={Functors and computations in Floer homology with applications. I},
   journal={Geom. Funct. Anal.},
   volume={9},
   date={1999},
   number={5},
   pages={985--1033},
   issn={1016-443X},
   review={\MR{1726235 (2000j:53115)}},
   doi={10.1007/s000390050106},
}
\bib{viterbo-96}{article}{
   author={Viterbo, Claude},
   title={Functors and computations in Floer cohomology. II.},  
journal={Pr\'epublication Orsay},
volume={98-15},
date={1998},
eprint={http://www.math.ens.fr/∼viterbo/FCFH.II.2003.pdf},
}

\bib{wahl}{article}{
   author={Wahl, Nathalie},
   title={Universal operations in Hochschild homology},
   eprint = {arXiv:1212.6498},
}

\bib{ward}{article}{
   author={Ward, Benjamin},
   title={Cyclic $A_\infty$ Structures and Deligne's Conjecture},
   eprint = {arXiv:1108.4976},
}


\bib{weber-05}{article}{
   author={Weber, Joa},
   title={Three approaches towards Floer homology of cotangent bundles},
   note={Conference on Symplectic Topology},
   journal={J. Symplectic Geom.},
   volume={3},
   date={2005},
   number={4},
   pages={671--701},
   issn={1527-5256},
   review={\MR{2235858 (2007k:53155)}},
}

\bib{weber-06}{article}{
   author={Weber, Joa},
   title={Noncontractible periodic orbits in cotangent bundles and Floer
   homology},
   journal={Duke Math. J.},
   volume={133},
   date={2006},
   number={3},
   pages={527--568},
   issn={0012-7094},
   review={\MR{2228462 (2007h:53140)}},
   doi={10.1215/S0012-7094-06-13334-3},
}

\bib{witten}{article}{
   author={Witten, Edward},
   title={Supersymmetry and Morse theory},
   journal={J. Differential Geom.},
   volume={17},
   date={1982},
   number={4},
   pages={661--692 (1983)},
   issn={0022-040X},
   review={\MR{683171 (84b:58111)}},
}

\end{biblist}
\end{bibdiv}
\end{document}